%% file: TLtower.tex
\documentclass[11pt]{amsart}

\input{TLmacros}

\begin{document}
%===============

\title{A Graphical Calculus for Induction and Restriction on Temperley--Lieb Modules} 

\author{Alistair Savage}
\address[A.S.]{
    Department of Mathematics and Statistics \\
    University of Ottawa \\
    Ottawa, ON, K1N 6N5, Canada
}
\urladdr{\href{https://alistairsavage.ca}{alistairsavage.ca}, \textrm{\textit{ORCiD}:} \href{https://orcid.org/0000-0002-2859-0239}{orcid.org/0000-0002-2859-0239}}
\email{alistair.savage@uottawa.ca}

\begin{abstract}
    We develop a graphical calculus for induction and restriction along the Temperley--Lieb tower at a generic parameter.  The main object, a diagrammatic 2-category $\Dcat$, has one-step generators that model the usual induction and restriction bimodules and additional two-step generators that model the summands that the cup-cap idempotents cut out in the two-strand Temperley--Lieb algebra.  We construct an incarnation 2-functor from $\Dcat$ to the 2-category of Temperley--Lieb bimodules and prove a basis theorem for all 2-morphism spaces; the basis elements are indexed by bridges, a class of paths generalizing Dyck paths.  The basis theorem implies that the incarnation functor is locally full and faithful; after we pass to the additive Karoubi envelope, this functor becomes an equivalence onto the corresponding 2-category of bimodules.  Thus the calculus gives a concrete diagrammatic model for the functorial representation theory of the Temperley--Lieb tower.  We also compute the split Grothendieck ring of the associated monoidal category and its action on the Grothendieck group of Temperley--Lieb modules.  Under this action, homogenized Chebyshev polynomials represent the classes of standard modules and yield a positive integral basis.  In contrast with Heisenberg categorifications, idempotent completion leaves the Grothendieck ring unchanged.
\end{abstract}



\hypersetup{% Set PDF metadata
  pdftitle={A Graphical Calculus for Induction and Restriction on Temperley\texorpdfstring{--}{–}Lieb Modules},
  pdfauthor={Alistair Savage},
  pdfsubject={Graphical calculus for induction and restriction along the Temperley\texorpdfstring{--}{–}Lieb tower at a generic parameter},
  pdfkeywords={Temperley\texorpdfstring{--}{–}Lieb algebras; induction and restriction functors; diagrammatic 2-categories; graphical calculus; bimodules; Grothendieck rings; decategorification; Chebyshev polynomials; 2020 MSC: 18M30, 18N10, 18N25, 16D20, 16D90},
  pdfcreator={LaTeX with hyperref},
  pdfdisplaydoctitle=true
}

\ifboolexpr{togl{comments} or togl{details}}{%
  {\color{magenta}DETAILS OR COMMENTS ON}
}{%
}

\maketitle
\thispagestyle{empty}

\setcounter{tocdepth}{1}
\tableofcontents

%======================================
\section{Introduction\label{sec:intro}}
%======================================

The Temperley--Lieb algebras are among the most important and useful families of diagram algebras.  Introduced by Temperley and Lieb in the study of lattice models in statistical mechanics \cite{TL71}, they subsequently became central in low-dimensional topology, subfactor theory, and knot theory, most notably through Jones' work on subfactors and the Jones polynomial \cite{Jon83} and through Kauffman's diagrammatic approach to link invariants \cite{Kau87}.  Their representation theory is correspondingly rich, but also unusually concrete: it can be described in terms of planar diagrams, Jones--Wenzl projectors, and standard, or cell, modules in the sense of cellular algebra theory \cite{GL96}.  This combination of explicit diagrammatics and strong representation-theoretic structure makes the Temperley--Lieb algebras a natural testing ground for categorical and diagrammatic methods.

For many purposes, it is more useful to study the Temperley--Lieb algebras as a tower
\[
    \TLalg_0 \subseteq \TLalg_1 \subseteq \TLalg_2 \subseteq \dotsb
\]
than to study the individual algebras in isolation.  The inclusions in the tower give rise to induction and restriction functors between module categories, and these functors encode the branching of representations.  Equally important are the natural transformations between iterated induction and restriction functors.  These natural transformations organize the functorial representation theory of the tower, and a diagrammatic presentation of them gives a concrete calculus for computations that would otherwise be expressed in terms of bimodule maps.

This perspective is not specific to Temperley--Lieb algebras.  A tower of algebras often arises from a monoidal category $\cC$ and an object $X \in \cC$: one considers the algebras
\[
    \End_{\cC}(X^{\otimes n}), \qquad n \in \N,
\]
with inclusions induced by tensoring with the identity morphism of $X$.  In the Temperley--Lieb case, the relevant monoidal category is the Temperley--Lieb category $\TLcat(\delta)$, and
\[
    \TLalg_n = \End_\TLcat(\Tobj^{\otimes n}),
\]
where $\Tobj$ is the generating object of the category; see \cref{subsec:TLcat}.  The goal of the present paper is to describe, diagrammatically and completely in the generic case, the induction and restriction calculus associated to this tower.

The best-studied model for this kind of structure is the tower of group algebras of symmetric groups.  Khovanov introduced a diagrammatic category governing induction and restriction for this tower, leading to a categorification of the Heisenberg algebra \cite{Kho14}.  This initiated a substantial body of work on Heisenberg categories and their variants, including quantum and higher-level versions \cite{LS13,MS18,Bru18,BSW20-qheis}.  These categories have also been connected to Kac--Moody 2-categories and quantum-group categorification \cite{BSW20-HKM}.  In that setting, the diagrammatic calculus does more than package known functors: it exposes the algebraic structure governing the entire tower.

The purpose of this paper is to develop an analogous picture for the tower of Temperley--Lieb algebras.  We work under the genericity assumption \cref{generic}, so that the Temperley--Lieb algebras are split semisimple and their standard modules provide the relevant simple objects; see \cref{subsec:TLstandards}.  Even in this generic setting, the Temperley--Lieb tower differs substantially from the symmetric-group tower.  The resulting graphical calculus is therefore not a formal translation of the Heisenberg category, but rather a new diagrammatic category adapted to the Temperley--Lieb branching rules.

A first step in this direction was taken by Harper and Samuelson in \cite{HS25}.  They initiated a diagrammatic study of induction and restriction along the Temperley--Lieb tower and identified important structural features of the corresponding calculus.  Their construction also revealed some significant technical challenges.  In particular, as discussed in \cite{HS25}, the relations in their presentation are considerably more involved than those appearing in the Heisenberg category; notably, the presentation includes infinite families of relations.  Consequently, several results that one would hope for in a diagrammatic treatment remained conjectural there, including basis theorems for morphism spaces, fullness and faithfulness of the incarnation functor to bimodules, and a description of the Grothendieck ring.  The present paper revisits this program with a modified set of generators, and this modification makes it possible to prove these structural results.

We begin by defining a diagrammatic 2-category $\Dcat$ in \cref{subsec:2cat}.  Its objects are indexed by nonnegative integers, and its generating 1-morphisms
\[
    \Sup 1_n,\quad \Sdown 1_{n+1},\quad \Dup 1_n,\quad \Ddown 1_{n+2}, \qquad n \in \N,
\]
should be understood representation-theoretically as follows.  The generators $\Sup$ and $\Sdown$ model one-step induction and restriction, and correspond to the generators $Q_+$ and $Q_-$ in \cite{HS25}.  The additional generators $\Dup$ and $\Ddown$ model certain two-step summands arising from the cup-cap idempotent in the two-string Temperley--Lieb algebra.  The 2-morphisms are generated by cups, caps, and trivalent vertices, subject to a finite list of local relations.  We then pass from the 2-category $\Dcat$ to a monoidal category $\Dmon$ whose objects encode 1-endomorphisms of the formal direct sum of all objects of $\Dcat$; see \cref{Dmondef}.

The diagrammatic 2-category is related to Temperley--Lieb representation theory by an incarnation 2-functor
\[
    \bF \colon \Dcat \to \Bcat
\]
to the 2-category of Temperley--Lieb bimodules, constructed in \cref{sec:functor}.  Under this functor, the generators $\Sup$ and $\Sdown$ are sent to the usual induction and restriction bimodules, while $\Dup$ and $\Ddown$ are sent to the corresponding two-step summands cut out by the cup-cap idempotents.  Tensor product of bimodules then gives an action on
\[
    \Mcat = \bigoplus_{n \in \N} \TLalg_n\md.
\]
On standard modules, this action is given by the branching formulas of \cref{monkey}; explicitly, $\Sup$ and $\Sdown$ act by the usual two-term Temperley--Lieb branching rule, while $\Dup$ and $\Ddown$ preserve the standard-module label and shift the algebra index by two.

The main technical result of the paper is the basis theorem, \cref{basisthm}.  For any two parallel 1-morphisms $f$ and $g$ in $\Dcat$, we construct an explicit set $\bB(f,g)$ and prove that it is a basis for the 2-morphism space $\Dcat(f,g)$.  These basis elements are naturally labelled by bridges, a class of paths generalizing Dyck paths; see \cref{Bdef}.  The proof proceeds by reducing arbitrary diagrams to controlled normal forms and then proving linear independence through the incarnation functor.  This basis theorem is the key point at which the additional generators $\Dup$ and $\Ddown$ become essential: they allow the diagrammatic relations to be local enough for a spanning argument and rigid enough for a linear-independence argument.

As a consequence of the basis theorem, the incarnation functor is locally full and faithful; this is \cref{fullfaithful}.  After passing to the additive Karoubi envelope, the incarnation functor becomes a 2-equivalence
\[
    \Kar(\bF) \colon \Kar(\Dcat) \xrightarrow{\sim} \Bcat
\]
by \cref{paint}.  Equivalently, using the Eilenberg--Watts theorem, the calculus describes the additive, right-exact, direct-sum-preserving functors between the module categories of the Temperley--Lieb algebras.  Thus, in the generic setting considered here, $\Dcat$ gives a concrete graphical model for the functorial representation theory of the Temperley--Lieb tower.

We also compute the decategorification.  In \cref{subsec:Grothring}, we define a graded ring $\Ring$ and a polynomial representation $\Poly$ of it.  The main decategorification result, \cref{sump}, identifies $\Ring$ with the split Grothendieck ring $K_\oplus(\Add(\Dmon))$ and identifies the action of this Grothendieck ring on $K_\oplus(\Mcat)$ with the action of $\Ring$ on $\Poly$.  Under this identification, the classes of the standard Temperley--Lieb modules correspond to homogenized Chebyshev polynomials; see \cref{Polyisom}.  The formulas of \cref{canaction} then show that these classes form a positive integral basis for the action.

One notable contrast with the setting of Heisenberg categorification is that no idempotent completion is needed to obtain the desired Grothendieck ring.  For Heisenberg categories, passing to the Karoubi envelope is essential for the main categorification statement, and the resulting Grothendieck-ring computation is subtle: it was conjectured in \cite{Kho14} and proved later in \cite{BSW-K0}.  By contrast, in the present Temperley--Lieb setting, passing to the additive Karoubi envelope does not change the Grothendieck category of $\Dcat$ or the Grothendieck ring of $\Dmon$; see \cref{convocation1,convocation2}.  This does not mean that the Temperley--Lieb tower is formally simpler in every respect, but it shows that the complexity visible in previous treatments is not an intrinsic obstruction to a complete diagrammatic categorification.

Let us explain more precisely how the present approach differs from \cite{HS25}.  The 2-category defined in \cite[\S 5]{HS25} is the locally full 2-subcategory of $\Dcat$ generated by the 1-morphisms $\Sobj_\pm 1_n$, denoted $Q_\pm 1_n$ there.  The main new feature of the present paper is that we include the generators $\Dupdown 1_n$ from the beginning.  These generators allow certain morphisms in the calculus of \cite{HS25} to be decomposed into smaller pieces factoring through $\Dup$ and $\Ddown$.  As a result, the presentation of $\Dcat$ is substantially simpler: the infinite families of relations \cite[(23), (24)]{HS25} are replaced by the single relations \cref{coal,wood}, for each admissible region label; see \cref{dune} and \cref{thai}\ref{thai1}.  Moreover, the idempotents denoted $C_k$ in \cite{HS25} can be expressed directly in terms of the objects $\Dupdown$; see \cref{hacky}.  Thus these summands are visible before passing to the idempotent completion.

This reorganization is what makes the basis theorem, local full faithfulness, and the Grothendieck-ring computation accessible.  In particular, while the presentation in \cite{HS25} suggested that the Temperley--Lieb tower might be more difficult to treat diagrammatically than the symmetric-group tower, the present paper shows that the difficulty can be removed by choosing the generators appropriately.  In this sense, the Temperley--Lieb tower admits a graphical calculus that is at least as tractable as the Heisenberg calculus and, in some respects, more economical.

We close this introduction by mentioning several directions for further work.  First, it would be natural to extend the construction beyond the generic semisimple case.  The non-generic Temperley--Lieb algebras have substantially richer representation theory, and even the special case $\delta=0$ should already exhibit new phenomena.  Second, the category constructed here is analogous to the Heisenberg categories at central charge $\pm 1$, or equivalently to Khovanov's original Heisenberg category \cite{Kho14} and its quantum analogue \cite{LS13}.  This suggests the existence of higher-central-charge analogues, in the spirit of \cite{MS18,Bru18,BSW20-qheis}, acting on module categories for higher-level cyclotomic quotients of affine Temperley--Lieb algebras.  Third, since Heisenberg categories give rise to actions of Kac--Moody 2-categories \cite{BSW20-HKM}, it is natural to ask whether analogous structures arise from the present Temperley--Lieb calculus and its higher-central-charge variants.  Finally, one can ask for similar diagrammatic calculi for other towers of diagram algebras, including rook-Brauer algebras, Motzkin algebras, Brauer algebras, rook algebras, planar rook algebras, partition algebras, and planar partition algebras; see \cite{Hu20} for an overview of these diagram categories.

%--------------------------
\subsection*{Use of colour}
%--------------------------

The paper is best viewed in colour.  The diagrams use a colour scheme chosen to be distinguishable for most forms of colour blindness. Readers who have difficulty distinguishing the colours can change the corresponding macros in the \LaTeX\ source, available on the arXiv.

%-----------------------------
\subsection*{Acknowledgements}
%-----------------------------

This research was supported by Discovery Grant RGPIN-2023-03842 from the Natural Sciences and Engineering Research Council of Canada.  The author would like to thank M.~Harper and P.~Samuelson for helpful conversations regarding their work \cite{HS25}.

%====================================
\section{The diagrammatic categories}
%====================================

In this section, we define the diagrammatic categories that are the focus of the current work.  We will use the standard string diagram calculus for 2-categories and monoidal categories, which we always assume to be linear over a commutative ring $\kk$.  For 2-categories, this means that 2-morphisms are linear combinations of string diagrams.  In these string diagrams, regions are labelled by objects and strings labelled by 1-morphisms.  The domain of a 1-morphism labelling a string is the object labelling the region to its right, while the codomain is the object labelling the region to its left.  The conventions for monoidal categories are similar, viewing a monoidal category as a 2-category with one object.  The objects and morphisms of the monoidal category are the 1-morphisms and 2-morphisms, respectively, of the corresponding 2-category.  The regions of string diagrams for monoidal categories are not labelled, corresponding to the fact that the corresponding 2-category has only one object.

If $X$ and $Y$ are objects in a 2-category $\cC$, then $\cC(X,Y)$ denotes the category of morphisms from $X$ to $Y$.  Similarly, if $f$ and $g$ are 1-morphisms in $\cC$, then $\cC(f,g)$ denotes the $\kk$-module of 2-morphisms from $f$ to $g$.  We use analogous notation for monoidal categories.  When writing expressions in algebraic notation (as opposed to string diagrams), we use $\otimes$ to denote horizontal composition of 2-morphisms, and juxtaposition to denote vertical composition.  Composition of 1-morphisms is denoted by juxtaposition.

Throughout the paper, if $P$ is a statement, then we define
\[
    \delta_P :=
    \begin{cases}
        1 & \text{if $P$ is true}, \\
        0 & \text{if $P$ is false}.
    \end{cases}
\]
Thus, for example, $\delta_{i,j} := \delta_{i=j}$ is the usual Kronecker delta function.

%----------------------------------------------------------
\subsection{The diagrammatic 2-category\label{subsec:2cat}}
%----------------------------------------------------------

Let $\kk$ be a commutative ring and fix $\delta \in \kk^\times$.   Define $\Delta_n \in \kk$, $n \in \Z_{\ge -1}$, recursively by
\begin{equation} \label{Delta}
    \Delta_{-1} = 0,\qquad
    \Delta_0 = 1,\qquad
    \Delta_{n+1} = \delta \Delta_n - \Delta_{n-1},\qquad
    n \ge 0.
\end{equation}
These are Chebyshev polynomials of the second kind, evaluated at $\delta/2$.  We assume that
\begin{equation} \label{generic}
    \Delta_n \in \kk^\times \qquad \text{for all } n \in \N.
\end{equation}
Note that
\[
    \Delta_n = \frac{q^{n+1}-q^{-n-1}}{q-q^{-1}} \qquad \text{if } \delta = q+q^{-1} \text{ for some } q \in \kk^\times.
\]

We define a $\kk$-linear 2-category $\Dcat$ as follows.  The set of objects is $\{\obj{n} : n \in \N\}$.  We denote the identity $1$-morphism of $\obj{n}$ by $1_n$.  The $1$-morphisms are generated by
\begin{gather*}
    \Sup 1_n = 1_{n+1} \Sup = 1_{n+1} \Sup 1_n \colon \obj{n} \to \obj{n+1},\quad
    \Sdown 1_{n+1} = 1_n \Sdown = 1_n \Sdown 1_{n+1} \colon \obj{n+1} \to \obj{n},
    \\
    \Dup 1_n = 1_{n+2} \Dup= 1_{n+2} \Dup 1_n \colon \obj{n} \to \obj{n+2},\quad
    \Ddown 1_{n+2} = 1_n \Ddown= 1_n \Ddown 1_{n+2} \colon \obj{n+2} \to \obj{n},
\end{gather*}
for $n \in \N$.  We draw the identity $1$-morphisms of the generating objects as follows:
\begin{align*}
    1_{\Sup} &=
    \upstrandreg[S]{n}
    =
    \begin{tikzpicture}[centerzero]
        \draw[S,->] (0,-0.2) -- (0,0.2);
        \node[anchor=east] (0,0) {\regionlabel{n+1}};
    \end{tikzpicture}
    =
    \begin{tikzpicture}[centerzero]
        \draw[S,->] (0,-0.2) -- (0,0.2);
        \node[anchor=east] (0,0) {\regionlabel{n+1}};
        \node[anchor=west] (0,0) {\regionlabel{n}};
    \end{tikzpicture}
    ,&
    1_{\Sdown} &=
    \downstrandreg[S]{n+1}
    =
    \begin{tikzpicture}[centerzero]
        \draw[S,<-] (0,-0.2) -- (0,0.2);
        \node[anchor=east] (0,0) {\regionlabel{n}};
    \end{tikzpicture}
    =
    \begin{tikzpicture}[centerzero]
        \draw[S,<-] (0,-0.2) -- (0,0.2);
        \node[anchor=east] (0,0) {\regionlabel{n}};
        \node[anchor=west] (0,0) {\regionlabel{n+1}};
    \end{tikzpicture}
    , \\
    1_{\Dup} &=
    \upstrandreg[D]{n}
    =
    \begin{tikzpicture}[centerzero]
        \draw[D,->] (0,-0.2) -- (0,0.2);
        \node[anchor=east] (0,0) {\regionlabel{n+2}};
    \end{tikzpicture}
    =
    \begin{tikzpicture}[centerzero]
        \draw[D,->] (0,-0.2) -- (0,0.2);
        \node[anchor=east] (0,0) {\regionlabel{n+2}};
        \node[anchor=west] (0,0) {\regionlabel{n}};
    \end{tikzpicture}
    ,&
    1_{\Ddown} &=
    \downstrandreg[D]{n+2}
    =
    \begin{tikzpicture}[centerzero]
        \draw[D,<-] (0,-0.2) -- (0,0.2);
        \node[anchor=east] (0,0) {\regionlabel{n}};
    \end{tikzpicture}
    =
    \begin{tikzpicture}[centerzero]
        \draw[D,<-] (0,-0.2) -- (0,0.2);
        \node[anchor=east] (0,0) {\regionlabel{n}};
        \node[anchor=west] (0,0) {\regionlabel{n+2}};
    \end{tikzpicture}
    .
\end{align*}
Here, and in what follows, we omit labels for regions when they can be deduced.

The 2-morphisms are generated by
\begin{align*}
    \rightcapreg[S]{n+1} &\colon \Sup \Sdown 1_{n+1} \to 1_{n+1},&
    \rightcupreg[S]{n} &\colon 1_n \to \Sdown \Sup 1_n,
    \\
    \leftcapreg[S]{n} &\colon \Sdown \Sup 1_n \to 1_n,&
    \leftcupreg[S]{n+1} &\colon 1_{n+1} \to \Sup \Sdown 1_{n+1},
    \\
    \rightcapreg[D]{n+2} &\colon \Dup \Ddown 1_{n+2} \to 1_{n+2},&
    \leftcapreg[D]{n} &\colon \Ddown \Dup 1_n \to 1_n,
    \\
    \rightcupreg[D]{n} &\colon 1_n \to \Ddown \Dup 1_n,&
    \leftcupreg[D]{n+2} &\colon 1_{n+2} \to \Dup \Ddown 1_{n+2},
    \\
    \mergeupreg{n}
    =
    \begin{tikzpicture}[centerzero]
        \draw[S,->-=0.55] (-0.3,-0.3) -- (0,0);
        \draw[S,->-=0.55] (0.3,-0.3) -- (0,0);
        \draw[D,->] (0,0) -- (0,0.4);
        \region{0.2,0.1}{n};
    \end{tikzpicture}
    &\colon \Sup \Sup 1_n \to \Dup 1_n,
    &
    \splitupreg{n}
    =
    \begin{tikzpicture}[centerzero]
        \draw[D,->-=0.55] (0,-0.4) -- (0,0);
        \draw[S,->] (0,0) -- (-0.3,0.3);
        \draw[S,->] (0,0) -- (0.3,0.3);
        \region{0.2,-0.1}{n};
    \end{tikzpicture}
    &\colon \Dup 1_n \to \Sup \Sup 1_n.
\end{align*}
subject to relations given below.  Before stating the relations, we point out some conventions that we have used above.  The orientation of the strands incident to a trivalent vertex are uniquely determined by the orientation of any one of the strands: the thin black strands are either both oriented into the vertex or both oriented out of the vertex, while the thick cyan strand is oriented the other way.  In order to reduce clutter, we will often omit the orientation of some strands when these are uniquely determined by orientations that have been specified.  We will also adopt the convention that any diagram with a negative region label (following the rule that region labels increase by one when crossing $\upstrand[S]$ from right to left, etc.) is zero.  Thus, for example,
\[
    \rightbubreg[S]{0} = \rightbubreg[D]{0} = 0
    \qquad \text{and} \qquad
    \rightbubreg[D]{1} = 0.
\]
Finally, we will sometimes omit the notation $1_n$ indicating the domain or codomain of a 1-morphism when it is clear from the context.  For example $\Sdown \Sup 1_n$ is the composition $(\Sdown 1_{n+1})(\Sup 1_n)$.

The first relations we impose are
\begin{gather}\label{adjunction}
    \begin{tikzpicture}[centerzero,S]
        \draw[->] (-0.3,-0.4) -- (-0.3,0) arc(180:0:0.15) arc(180:360:0.15) -- (0.3,0.4);
    \end{tikzpicture}
    =
    \begin{tikzpicture}[centerzero,S]
        \draw[->] (0,-0.4) -- (0,0.4);
    \end{tikzpicture}
    =
    \begin{tikzpicture}[centerzero,S]
        \draw[<-] (-0.3,0.4) -- (-0.3,0) arc(180:360:0.15) arc(180:0:0.15) -- (0.3,-0.4);
    \end{tikzpicture}
    \ ,\qquad
    \begin{tikzpicture}[centerzero,S]
        \draw[<-] (-0.3,-0.4) -- (-0.3,0) arc(180:0:0.15) arc(180:360:0.15) -- (0.3,0.4);
    \end{tikzpicture}
    =
    \begin{tikzpicture}[centerzero,S]
        \draw[<-] (0,-0.4) -- (0,0.4);
    \end{tikzpicture}
    =
    \begin{tikzpicture}[centerzero,S]
        \draw[->] (-0.3,0.4) -- (-0.3,0) arc(180:360:0.15) arc(180:0:0.15) -- (0.3,-0.4);
    \end{tikzpicture}
    \ ,\qquad
    \begin{tikzpicture}[centerzero,D]
        \draw[->] (-0.3,-0.4) -- (-0.3,0) arc(180:0:0.15) arc(180:360:0.15) -- (0.3,0.4);
    \end{tikzpicture}
    =
    \begin{tikzpicture}[centerzero,D]
        \draw[->] (0,-0.4) -- (0,0.4);
    \end{tikzpicture}
    =
    \begin{tikzpicture}[centerzero,D]
        \draw[<-] (-0.3,0.4) -- (-0.3,0) arc(180:360:0.15) arc(180:0:0.15) -- (0.3,-0.4);
    \end{tikzpicture}
    \ ,\qquad
    \begin{tikzpicture}[centerzero,D]
        \draw[<-] (-0.3,-0.4) -- (-0.3,0) arc(180:0:0.15) arc(180:360:0.15) -- (0.3,0.4);
    \end{tikzpicture}
    =
    \begin{tikzpicture}[centerzero,D]
        \draw[<-] (0,-0.4) -- (0,0.4);
    \end{tikzpicture}
    =
    \begin{tikzpicture}[centerzero,D]
        \draw[->] (-0.3,0.4) -- (-0.3,0) arc(180:360:0.15) arc(180:0:0.15) -- (0.3,-0.4);
    \end{tikzpicture}
    \ ,
    \\ \label{swishy}
    \mergedown
    :=
    \begin{tikzpicture}[anchorbase]
        \draw[D,->] (0,0) -- (0,0.1) arc (0:180:0.2) -- (-0.4,-0.6);
        \draw[S] (0,0) to[out=-45, in=left] (0.2,-0.2) arc(-90:0:0.15) -- (0.35,0.4);
        \draw[S] (0,0) to[out=-135,in=up] (-0.2,-0.3) to[out=down,in=down] (0.6,-0.3) -- (0.6,0.4);
    \end{tikzpicture}
    =
    \begin{tikzpicture}[anchorbase,xscale=-1]
        \draw[D,->] (0,0) -- (0,0.1) arc (0:180:0.2) -- (-0.4,-0.6);
        \draw[S] (0,0) to[out=-45, in=left] (0.2,-0.2) arc(-90:0:0.15) -- (0.35,0.4);
        \draw[S] (0,0) to[out=-135,in=up] (-0.2,-0.3) to[out=down,in=down] (0.6,-0.3) -- (0.6,0.4);
    \end{tikzpicture}
    \ ,\qquad
    \splitdown
    :=
    \begin{tikzpicture}[anchorbase,yscale=-1]
        \draw[D] (0,0) -- (0,0.1) arc (0:180:0.2) -- (-0.4,-0.6);
        \draw[S,->] (0,0) to[out=-45, in=left] (0.2,-0.2) arc(-90:0:0.15) -- (0.35,0.4);
        \draw[S,->] (0,0) to[out=-135,in=up] (-0.2,-0.3) to[out=down,in=down] (0.6,-0.3) -- (0.6,0.4);
    \end{tikzpicture}
    =
    \begin{tikzpicture}[anchorbase,xscale=-1,yscale=-1]
        \draw[D] (0,0) -- (0,0.1) arc (0:180:0.2) -- (-0.4,-0.6);
        \draw[S,->] (0,0) to[out=-45, in=left] (0.2,-0.2) arc(-90:0:0.15) -- (0.35,0.4);
        \draw[S,->] (0,0) to[out=-135,in=up] (-0.2,-0.3) to[out=down,in=down] (0.6,-0.3) -- (0.6,0.4);
    \end{tikzpicture}
    \ .
\end{gather}
Above, and throughout the paper, we omit regional labels in equations that hold for any valid labelling of the regions.  The relations \cref{adjunction,swishy} imply that $\Dcat$ is strict pivotal.  Thus, morphisms are invariant under isotopy.  We define
\begin{gather} \label{fishy1}
    \mergeSW
    :=
    \begin{tikzpicture}[anchorbase]
        \draw[D,->] (0,0) -- (0,0.1) arc (0:180:0.2) -- (-0.4,-0.4);
        \draw[S] (0,0) to[out=-45, in=left] (0.2,-0.2) arc(-90:0:0.15) -- (0.35,0.4);
        \draw[S] (0,0) to[out=-135,in=up] (-0.2,-0.4);
    \end{tikzpicture}
    \ ,\qquad
    \mergeSE
    :=
    \begin{tikzpicture}[anchorbase,xscale=-1]
        \draw[D,->] (0,0) -- (0,0.1) arc (0:180:0.2) -- (-0.4,-0.4);
        \draw[S] (0,0) to[out=-45, in=left] (0.2,-0.2) arc(-90:0:0.15) -- (0.35,0.4);
        \draw[S] (0,0) to[out=-135,in=up] (-0.2,-0.4);
    \end{tikzpicture}
    \ ,\qquad
    \splitNE
    :=
    \begin{tikzpicture}[anchorbase]
        \draw[D] (0,-0.4) -- (0,0);
        \draw[S,->] (0,0) -- (-0.3,0.3);
        \draw[S,->] (0,0) to[out=45,in=left] (0.2,0.2) to[out=right,in=up] (0.4,-0.4);
    \end{tikzpicture}
    \ ,\qquad
    \splitNW
    :=
    \begin{tikzpicture}[anchorbase,xscale=-1]
        \draw[D] (0,-0.4) -- (0,0);
        \draw[S,->] (0,0) -- (-0.3,0.3);
        \draw[S,->] (0,0) to[out=45,in=left] (0.2,0.2) to[out=right,in=up] (0.4,-0.4);
    \end{tikzpicture}
    \ ,
    \\ \label{fishy2}
    \splitSE
    :=
    \begin{tikzpicture}[anchorbase,yscale=-1]
        \draw[D] (0,0) -- (0,0.1) arc (0:180:0.2) -- (-0.4,-0.4);
        \draw[S,->] (0,0) to[out=-45, in=left] (0.2,-0.2) arc(-90:0:0.15) -- (0.35,0.4);
        \draw[S,->] (0,0) to[out=-135,in=up] (-0.2,-0.4);
    \end{tikzpicture}
    \ ,\qquad
    \splitSW
    :=
    \begin{tikzpicture}[anchorbase,yscale=-1,xscale=-1]
        \draw[D] (0,0) -- (0,0.1) arc (0:180:0.2) -- (-0.4,-0.4);
        \draw[S,->] (0,0) to[out=-45, in=left] (0.2,-0.2) arc(-90:0:0.15) -- (0.35,0.4);
        \draw[S,->] (0,0) to[out=-135,in=up] (-0.2,-0.4);
    \end{tikzpicture}
    \ ,\qquad
    \mergeNW
    :=
    \begin{tikzpicture}[anchorbase]
        \draw[D,->] (0,0) -- (0,0.4);
        \draw[S] (0,0) -- (-0.3,-0.3);
        \draw[S] (0,0) to[out=-45,in=left] (0.2,-0.2) to[out=right,in=down] (0.4,0.4);
    \end{tikzpicture}
    \, \qquad
    \mergeNE
    :=
    \begin{tikzpicture}[anchorbase,xscale=-1]
        \draw[D,->] (0,0) -- (0,0.4);
        \draw[S] (0,0) -- (-0.3,-0.3);
        \draw[S] (0,0) to[out=-45,in=left] (0.2,-0.2) to[out=right,in=down] (0.4,0.4);
    \end{tikzpicture}
    \ .
\end{gather}
We then impose the following relations:
\begin{gather} \label{fire}
    \begin{tikzpicture}[centerzero]
        \draw[D] (-0.3,-0.5) -- (0,-0.2);
        \draw[S,<-] (0.3,-0.5) -- (0,-0.2);
        \draw[S] (0,-0.2) -- (0,0.2);
        \draw[S] (0,0.2) -- (0.3,0.5);
        \draw[D,->] (0,0.2) -- (-0.3,0.5);
    \end{tikzpicture}
    =
    \begin{tikzpicture}[centerzero]
        \draw[D,->] (-0.2,-0.5) -- (-0.2,0.5);
        \draw[S,<-] (0.2,-0.5) -- (0.2,0.5);
    \end{tikzpicture}
    \ ,\qquad
    \begin{tikzpicture}[centerzero]
        \draw[D] (0.3,-0.5) -- (0,-0.2);
        \draw[S,<-] (-0.3,-0.5) -- (0,-0.2);
        \draw[S] (0,-0.2) -- (0,0.2);
        \draw[S] (0,0.2) -- (-0.3,0.5);
        \draw[D,->] (0,0.2) -- (0.3,0.5);
    \end{tikzpicture}
    =
    \begin{tikzpicture}[centerzero]
        \draw[S,<-] (-0.2,-0.5) -- (-0.2,0.5);
        \draw[D,->] (0.2,-0.5) -- (0.2,0.5);
    \end{tikzpicture}
    \ ,\qquad
    \begin{tikzpicture}[centerzero]
        \draw[S,->] (-0.2,-0.5) -- (-0.2,-0.2) -- (0.2,-0.2) -- (0.2,0.2) -- (-0.2,0.2) -- (-0.2,0.5);
        \draw[D] (0.2,-0.5) -- (0.2,-0.2);
        \draw[D] (-0.2,-0.2) -- (-0.2,0.2);
        \draw[D,->] (0.2,0.2) -- (0.2,0.5);
    \end{tikzpicture}
    =
    \begin{tikzpicture}[centerzero]
        \draw[S,->] (-0.2,-0.5) -- (-0.2,0.5);
        \draw[D,->] (0.2,-0.5) -- (0.2,0.5);
    \end{tikzpicture}
    \ ,\qquad
    \begin{tikzpicture}[centerzero]
        \draw[D,<-] (-0.2,-0.5) -- (-0.2,-0.3) arc(180:0:0.2) -- (0.2,-0.5);
        \draw[D,->] (-0.2,0.5) -- (-0.2,0.3) arc(180:360:0.2) -- (0.2,0.5);
    \end{tikzpicture}
    =\
    \begin{tikzpicture}[centerzero]
        \draw[D,<-] (-0.2,-0.5) -- (-0.2,0.5);
        \draw[D,->] (0.2,-0.5) -- (0.2,0.5);
    \end{tikzpicture}
    \ ,
    \\ \label{smoke}
    \leftbubreg[S]{n} = \delta 1_n,\qquad
    \leftbubreg[D]{n} = 1_n.
\end{gather}
Recursively defining
\begin{equation} \label{sigma0}
    \freeprojectorreg{0}{n}
    := 1_n - \rightbubreg[D]{n},\qquad
    \freeprojectorreg{r+1}{n}
    := \rightbubmultreg[D]{\freeprojector{r}}{n}
    \qquad r,n \in \N,
\end{equation}
the final relations we impose are the following:
\begin{align} \label{coal}
    \leftbubmultreg[S]{\freeprojector{0}}{n}
    &= \frac{\Delta_{n+1}}{\Delta_n}\ \freeprojectorreg{0}{n},
    & n \in \N,
    \\ \label{wood}
    \begin{tikzpicture}[centerzero]
        \draw[S,<-] (-0.4,0.8) -- (-0.4,0.6) to[out=down,in=down,looseness=2] (0.4,0.6) -- (0.4,0.8);
        \projector{0,0.5}{0};
        \draw[S,->] (-0.4,-0.8) -- (-0.4,-0.6) to[out=up,in=up,looseness=2] (0.4,-0.6) -- (0.4,-0.8);
        \projector{0,-0.5}{0};
        \region{0.4,0}{n};
    \end{tikzpicture}
    &=
    \frac{\Delta_{n-1}}{\Delta_{n-2}}\
    \begin{tikzpicture}[centerzero]
        \draw[S,->] (-0.4,-0.8) -- (-0.4,0.8);
        \draw[S,<-] (0.4,-0.8) -- (0.4,0.8);
        \projector{0,0}{0};
        \projector{-0.8,0}{1};
        \projector{0.8,0}{1};
        \region{0.8,0.5}{n};
    \end{tikzpicture}
    + \frac{\Delta_{n-1}}{\Delta_n}\
    \begin{tikzpicture}[centerzero]
        \draw[S,->] (-0.4,-0.8) -- (-0.4,0.8);
        \draw[S,<-] (0.4,-0.8) -- (0.4,0.8);
        \projector{0,0}{0};
        \projector{-0.8,0}{0};
        \projector{0.8,0}{0};
        \region{0.8,0.5}{n};
    \end{tikzpicture}
    \ , & n \ge 1,
\end{align}
where, in \cref{wood}, we interpret the first term on the right-hand side as zero when $n=1$ (even though the coefficient has a denominator of zero).

%---------------------------------------------------------
\subsection{First relations following from the definition}
%---------------------------------------------------------

Let $\Dcat^{\rev,\op}$ be the reversed, opposite 2-category to $\Dcat$, where both horizontal and vertical composition are reversed.  It follows from the fact that $\Dcat$ is strict pivotal that we have an isomorphism of $\kk$-linear 2-categories
\begin{equation} \label{rotate}
    \bR \colon \Dcat \to \Dcat^{\rev,\op}
\end{equation}
that is the identity on objects, acts on the generating $1$-morphisms as
\[
    \Sup 1_n \mapsto 1_n \Sdown,\quad 1_n \Sdown \mapsto \Sup 1_n,\qquad \Dup 1_n \mapsto 1_n \Ddown,\qquad 1_n \Ddown \mapsto \Dup 1_n,
\]
and acts on 2-morphisms by rotating diagrams by $180\degree$.

\begin{lem}
    The following relations hold in $\Dcat$:
    \begin{gather} \label{spark}
        \begin{tikzpicture}[centerzero]
            \draw[D,->] (-0.2,-0.5) -- (-0.2,-0.3) arc(180:0:0.2) -- (0.2,-0.5);
            \draw[D,<-] (-0.2,0.5) -- (-0.2,0.3) arc(180:360:0.2) -- (0.2,0.5);
        \end{tikzpicture}
        =\
        \begin{tikzpicture}[centerzero]
            \draw[D,->] (-0.2,-0.5) -- (-0.2,0.5);
            \draw[D,<-] (0.2,-0.5) -- (0.2,0.5);
        \end{tikzpicture}
        \ ,
        \\ \label{water}
        \begin{tikzpicture}[centerzero]
            \draw[S] (0,-0.5) -- (0,-0.2);
            \draw[D] (0,-0.2) to[out=135,in=-135,looseness=2] (0,0.2);
            \draw[S] (0,-0.2) to[out=45,in=-45,looseness=2] (0,0.2);
            \draw[S,->] (0,0.2) -- (0,0.5);
        \end{tikzpicture}
        =
        \begin{tikzpicture}[centerzero]
            \draw[S,->] (0,-0.5) -- (0,0.5);
            \bubright[D]{-0.4,0};
        \end{tikzpicture}
        \ ,\qquad
        \begin{tikzpicture}[centerzero]
            \draw[S] (0,-0.5) -- (0,-0.2);
            \draw[S] (0,-0.2) to[out=135,in=-135,looseness=2] (0,0.2);
            \draw[D] (0,-0.2) to[out=45,in=-45,looseness=2] (0,0.2);
            \draw[S,->] (0,0.2) -- (0,0.5);
        \end{tikzpicture}
        =
        \begin{tikzpicture}[centerzero]
            \draw[S,->] (0,-0.5) -- (0,0.5);
        \end{tikzpicture}
        \ ,\qquad
        \begin{tikzpicture}[centerzero]
            \draw[D] (0,-0.5) -- (0,-0.2);
            \draw[S] (0,-0.2) to[out=135,in=-135,looseness=2] (0,0.2);
            \draw[S] (0,-0.2) to[out=45,in=-45,looseness=2] (0,0.2);
            \draw[D,->] (0,0.2) -- (0,0.5);
        \end{tikzpicture}
        = \delta\
        \begin{tikzpicture}[centerzero]
            \draw[D,->] (0,-0.5) -- (0,0.5);
        \end{tikzpicture}
        \ .
    \end{gather}
\end{lem}

\begin{proof}
    The relation \cref{spark} follows from the last relation in \cref{fire} after attaching $\rightcap[D]$ to the top of the rightmost strand, attaching a $\leftcup[D]$ to the bottom of the leftmost strand, and then applying \cref{adjunction}.  For the relations in \cref{water}, we compute
    \begin{gather*}
        \begin{tikzpicture}[centerzero]
            \draw[S] (0,-0.5) -- (0,-0.2);
            \draw[D] (0,-0.2) to[out=135,in=-135,looseness=2] (0,0.2);
            \draw[S] (0,-0.2) to[out=45,in=-45,looseness=2] (0,0.2);
            \draw[S,->] (0,0.2) -- (0,0.5);
        \end{tikzpicture}
        =
        \begin{tikzpicture}[centerzero,xscale=-1]
            \draw[S,->] (-0.2,-0.5) -- (0,-0.2) -- (0,0.2) -- (-0.2,0.5);
            \draw[D,-<-=0.5] (0,0.2) to[out=45,in=left] (0.2,0.4) to[out=right,in=right] (0.2,-0.4) to[out=left,in=-45] (0,-0.2);
        \end{tikzpicture}
        \overset{\cref{fire}}{=}
        \begin{tikzpicture}[centerzero]
            \draw[S,->] (0,-0.5) -- (0,0.5);
            \bubright[D]{-0.4,0};
        \end{tikzpicture}
        \ ,\qquad
        \begin{tikzpicture}[centerzero]
            \draw[S] (0,-0.5) -- (0,-0.2);
            \draw[S] (0,-0.2) to[out=135,in=-135,looseness=2] (0,0.2);
            \draw[D] (0,-0.2) to[out=45,in=-45,looseness=2] (0,0.2);
            \draw[S,->] (0,0.2) -- (0,0.5);
        \end{tikzpicture}
        =
        \begin{tikzpicture}[centerzero]
            \draw[S,->] (-0.2,-0.5) -- (0,-0.2) -- (0,0.2) -- (-0.2,0.5);
            \draw[D,-<-=0.5] (0,0.2) to[out=45,in=left] (0.2,0.4) to[out=right,in=right] (0.2,-0.4) to[out=left,in=-45] (0,-0.2);
        \end{tikzpicture}
        \overset{\cref{fire}}{=}
        \begin{tikzpicture}[centerzero]
            \draw[S,->] (0,-0.5) -- (0,0.5);
            \bubleft[D]{0.4,0};
        \end{tikzpicture}
        \overset{\cref{smoke}}{=}
        \begin{tikzpicture}[centerzero]
            \draw[S,->] (0,-0.5) -- (0,0.5);
        \end{tikzpicture}
        \ ,
        \\
        \begin{tikzpicture}[centerzero]
            \draw[D] (0,-0.5) -- (0,-0.2);
            \draw[S] (0,-0.2) to[out=135,in=-135,looseness=2] (0,0.2);
            \draw[S] (0,-0.2) to[out=45,in=-45,looseness=2] (0,0.2);
            \draw[D,->] (0,0.2) -- (0,0.5);
        \end{tikzpicture}
        =
        \begin{tikzpicture}[centerzero]
            \draw[D] (-0.2,-0.5) -- (0,-0.2);
            \draw[S] (0,-0.2) -- (0,0.2);
            \draw[D,->] (0,0.2) -- (-0.2,0.5);
            \draw[S,-<-=0.5] (0,0.2) to[out=45,in=left] (0.2,0.4) to[out=right,in=right] (0.2,-0.4) to[out=left,in=-45] (0,-0.2);
        \end{tikzpicture}
        \overset{\cref{fire}}{=}
        \begin{tikzpicture}[centerzero]
            \draw[D,->] (0,-0.5) -- (0,0.5);
            \bubleft[S]{0.4,0};
        \end{tikzpicture}
        = \delta
        \begin{tikzpicture}[centerzero]
            \draw[D,->] (0,-0.5) -- (0,0.5);
        \end{tikzpicture}
        \ .\qedhere
    \end{gather*}
\end{proof}

For $r \in \Z$ we recursively define
\begin{equation} \label{bubrec}
    \rightbubmultreg[D]{r}{n} := 1_n \text{ for } r \le 0,
    \qquad \text{and} \qquad
    \rightbubmultreg[D]{r+1}{n} := \rightbubmultreg[D]{\rightbubmult[D]{r}}{n}
    \text{ for } r \ge 0.
\end{equation}
By our convention that diagrams with negatively labeled regions are equal to zero, we have
\begin{equation} \label{freeze}
    \rightbubmultreg[D]{r}{n} = 0 \quad \text{if} \quad n < 2r.
\end{equation}
We will denote $f \in \Dcat(1_n,1_n)$ by a free-floating coupon:
\[
    \begin{tikzpicture}[centerzero]
        \coupon{0,0}{f};
        \region{0.4,0}{n};
    \end{tikzpicture}
    \, .
\]

\begin{prop}
    For all 1-endomorphisms $f$ and $g$ in $\Dcat$,
    \begin{equation} \label{onion}
        \leftbubmult[D]{\freecoupon{f}\ \rightbubmult[D]{\freecoupon{g}}}
        = \leftbubmult[D]{\freecoupon{f}}\ \freecoupon{g}\, ,
        \qquad
        \leftbubmult[D]{\freecoupon{f}\ \rightbubmult[S]{\freecoupon{g}}}
        = \leftbubmult[D]{\freecoupon{f}}\ \leftbubmult[S]{\freecoupon{g}}\, ,
        \qquad
        \leftbubmult[S]{\freecoupon{f}\ \rightbubmult[D]{\freecoupon{g}}}
        = \rightbubmult[S]{\leftbubmult[D]{\freecoupon{f}}\ \freecoupon{g}}\, .
    \end{equation}
\end{prop}

\begin{proof}
    We compute
    \begin{gather*}
        \leftbubmult[D]{\freecoupon{f}\ \rightbubmult[D]{\freecoupon{g}}}
        \overset{\cref{fire}}{=}
        \begin{tikzpicture}[centerzero]
            \coupon{0,0}{g};
            \coupon{-0.65,0}{f};
            \draw[D,->] (-0.3,0) -- (-0.3,0.2) to[out=up,in=up] (0.3,0.2) to[out=down,in=down,looseness=1.5] (0.55,0.2) to[out=up,in=up] (-1,0.2) -- (-1,-0.2) to[out=down,in=down] (0.55,-0.2) to[out=up,in=up,looseness=1.55] (0.3,-0.2) to[out=down,in=down] (-0.3,-0.2) -- (-0.3,0);
        \end{tikzpicture}
        =
        \leftbubmult[D]{\freecoupon{f}}\ \freecoupon{g}\ ,
        \\
        \leftbubmult[D]{\freecoupon{f}\ \rightbubmult[S]{\freecoupon{g}}}
        \overset{\cref{fire}}{=}
        \begin{tikzpicture}[centerzero]
            \coupon{0,0}{g};
            \coupon{-0.65,0}{f};
            \draw[D,->] (0.3,0.2) to[out=60,in=up,looseness=1.5] (-1,0.2) -- (-1,0);
            \draw[D] (-1,0) -- (-1,-0.2) to[out=down,in=-60,looseness=1.5] (0.3,-0.2);
            \draw[S] (0.3,-0.2) -- (0.3,0.2) to[out=up,in=up] (-0.3,0.2) -- (-0.3,-0.2) to[out=down,in=down] (0.3,-0.2);
        \end{tikzpicture}
        =
        \begin{tikzpicture}[centerzero]
            \coupon{-0.4,0}{f};
            \coupon{0.4,0}{g};
            \draw[S] (0,-0.2) -- (0,0.2) to[out=60,in=up,looseness=1.5] (0.8,0) to[out=down,in=-60,looseness=1.5] (0,-0.2);
            \draw[D,->] (0,0.2) to[out=120,in=up,looseness=1.5] (-0.8,0);
            \draw[D] (-0.8,0) to[out=down,in=-120,looseness=1.5] (0,-0.2);
        \end{tikzpicture}
        \overset{\cref{fire}}{=}
        \leftbubmult[D]{\freecoupon{f}}\ \leftbubmult[S]{\freecoupon{g}}\, ,
        \\
        \leftbubmult[S]{\freecoupon{f}\ \rightbubmult[D]{\freecoupon{g}}}
        \overset{\cref{fire}}{=}
        \begin{tikzpicture}[centerzero]
            \coupon{0,0}{g};
            \coupon{-0.65,0}{f};
            \draw[S,->] (0.3,-0.2) -- (0.3,0.2) to[out=60,in=up,looseness=1.5] (-1,0.2) -- (-1,0);
            \draw[S] (-1,0) -- (-1,-0.2) to[out=down,in=-60,looseness=1.5] (0.3,-0.2);
            \draw[D] (0.3,0.2) to[out=up,in=up] (-0.3,0.2) -- (-0.3,-0.2) to[out=down,in=down] (0.3,-0.2);
        \end{tikzpicture}
        =
        \begin{tikzpicture}[centerzero]
            \coupon{-0.4,0}{f};
            \coupon{0.4,0}{g};
            \draw[D] (0,-0.2) -- (0,0.2);
            \draw[S,->] (0,0.2) to[out=60,in=up,looseness=1.5] (0.8,0);
            \draw[S] (0.8,0) to[out=down,in=-60,looseness=1.5] (0,-0.2);
            \draw[S,->] (0,0.2) to[out=120,in=up,looseness=1.5] (-0.8,0);
            \draw[S] (-0.8,0) to[out=down,in=-120,looseness=1.5] (0,-0.2);
        \end{tikzpicture}
        =
        \begin{tikzpicture}[centerzero]
            \coupon{0,0}{f};
            \coupon{0.65,0}{g};
            \draw[S,->] (-0.3,-0.2) -- (-0.3,0.2) to[out=120,in=up,looseness=1.5] (1,0.2) -- (1,0);
            \draw[S] (1,0) -- (1,-0.2) to[out=down,in=-120,looseness=1.5] (-0.3,-0.2);
            \draw[D] (-0.3,0.2) to[out=up,in=up] (0.3,0.2) -- (0.3,-0.2) to[out=down,in=down] (-0.3,-0.2);
        \end{tikzpicture}
        \overset{\cref{fire}}{=}
        \rightbubmult[S]{\leftbubmult[D]{\freecoupon{f}}\ \freecoupon{g}}\, .
        \qedhere
    \end{gather*}
\end{proof}

\begin{cor}
    The following relations hold in $\Dcat$:
    \begin{gather} \label{apple1}
        \leftbubmult[D]{\rightbubmult[D]{r}}
        = \rightbubmult[D]{r-1}
        \qquad \text{for all } r \in \Z,
        \\ \label{apple2}
        \leftbubmult[S]{\rightbubmult[D]{r+1}}
        = \rightbubmult[S]{\rightbubmult[D]{r}}
        \qquad \text{for all } r \in \N.
    \end{gather}
\end{cor}

\begin{proof}
    The relation \cref{apple1} is trivial if $r \le 0$.  For $r \ge 1$, it follows from \cref{bubrec,smoke} and the first relation in \cref{onion} with $f=1_{n+2}$ and $g = 
    \rightbubmultreg[D]{r-1}{n}$.  Similarly, relation \cref{apple2} follows from \cref{bubrec,smoke} and the third relation in \cref{onion}.
\end{proof}

\begin{lem}
    The following relation holds in $\Dcat$:
    \begin{equation} \label{headroom}
        \rightbubmult[D]{r} \rightbubmult[D]{s} = \rightbubmult[D]{\max(r,s)}
        \qquad \text{for all } r,s \in \N.
    \end{equation} 
\end{lem}

\begin{proof}
    The result is trivial for $r=0$ or $s=0$.  If it holds for some $r,s \in \N$,
    \[
        \rightbubmult[D]{r+1}\, \rightbubmult[D]{s+1}
        \overset{\cref{bubrec}}{=}
        \rightbubmult[D]{\rightbubmult[D]{r}}\, \rightbubmult[D]{\rightbubmult[D]{s}}
        \overset{\cref{fire}}{=}
        \begin{tikzpicture}[centerzero]
            \bubrightmult[D]{0.4,0}{s};
            \bubrightmult[D]{-0.4,0}{r};
            \draw[D,<-] (0.8,0) to[out=up,in=right] (0.5,0.4) to[out=left,in=right] (0,0.2) to[out=left,in=right] (-0.5,0.4) to[out=left,in=up] (-0.8,0) to[out=down,in=left] (-0.5,-0.4) to[out=right,in=left] (0,-0.2) to[out=right,in=left] (0.5,-0.4) to[out=right,in=down] (0.8,0);
        \end{tikzpicture}
        =
        \begin{tikzpicture}[centerzero]
            \bubrightmult[D]{0,0}{\max(r,s)};
            \draw[D,->] (1.2,0) to[out=down,in=down,looseness=0.8] (-1.2,0) to[out=up,in=up,looseness=0.8] (1.2,0);
        \end{tikzpicture}
        \overset{\cref{bubrec}}{=} \rightbubmult[D]{\max(r,s)+1}
        \, .
    \]
    Since $\max(r,s)+1 = \max(r+1,s+1)$, the result follows by induction on $\min(r,s)$.
\end{proof}

\begin{lem}
    For all 1-endomorphisms $f$ in $\Dcat$, we have
    \begin{equation} \label{smarty}
        \begin{tikzpicture}[centerzero]
            \draw[D,->] (0,-0.5) -- (0,0.5);
            \bubrightmult[D]{-0.6,0}{\freecoupon{f}};
        \end{tikzpicture}
        =
        \begin{tikzpicture}[centerzero]
            \draw[D,->] (0,-0.5) -- (0,0.5);
            \coupon{0.4,0}{f};
        \end{tikzpicture}
        \, ,\qquad
        \begin{tikzpicture}[centerzero]
            \draw[D,->] (0,-0.5) -- (0,0.5);
            \bubrightmult[S]{-0.6,0}{\freecoupon{f}};
        \end{tikzpicture}
        =
        \begin{tikzpicture}[centerzero]
            \draw[D,->] (0,-0.5) -- (0,0.5);
            \bubleftmult[S]{0.5,0}{\freecoupon{f}};
        \end{tikzpicture}
    \end{equation}
\end{lem}

\begin{proof}
    We have
    \begin{gather*}
        \begin{tikzpicture}[centerzero]
            \draw[D,->] (0,-0.7) -- (0,0.7);
            \bubrightmult[D]{-0.6,0}{\freecoupon{f}};
        \end{tikzpicture}
        \overset{\cref{fire}}{=}
        \begin{tikzpicture}[centerzero]
            \draw[D,->] (0,-0.7) -- (0,-0.6) to[out=up,in=right] (-0.2,-0.3) to[out=left,in=right] (-0.6,-0.45) to[out=left,in=left,looseness=1.5] (-0.6,0.45) to[out=right,in=left] (-0.2,0.3) to[out=right,in=down] (0,0.6) -- (0,0.7);
            \coupon{-0.6,0}{f};
        \end{tikzpicture}
        =
        \begin{tikzpicture}[centerzero]
            \draw[D,->] (0,-0.7) -- (0,0.7);
            \coupon{0.4,0}{f};
        \end{tikzpicture}
        \ ,
        \\
        \begin{tikzpicture}[centerzero]
            \draw[D,->] (0,-0.7) -- (0,0.7);
            \bubrightmult[S]{-0.6,0}{\freecoupon{f}};
        \end{tikzpicture}
        \overset{\cref{fire}}{=}
        \begin{tikzpicture}[centerzero]
            \draw[D,->] (0,0.2) -- (0.2,0.7);
            \draw[D] (0,-0.2) -- (0.2,-0.7);
            \coupon{-0.5,0}{f};
            \draw[S,->] (-1,0) -- (-1,0.2) to[out=up,in=135] (0,0.2) -- (0,-0.2) to[out=-135,in=down] (-1,-0.2) -- (-1,0);
        \end{tikzpicture}
        =
        \begin{tikzpicture}[centerzero]
            \draw[D,->] (0,0.2) -- (-0.2,0.7);
            \draw[D] (0,-0.2) -- (-0.2,-0.7);
            \coupon{0.5,0}{f};
            \draw[S,->] (1,0) -- (1,0.2) to[out=up,in=45] (0,0.2) -- (0,-0.2) to[out=-45,in=down] (1,-0.2) -- (1,0);
        \end{tikzpicture}
        \overset{\cref{fire}}{=}
        \begin{tikzpicture}[centerzero]
            \draw[D,->] (0,-0.7) -- (0,0.7);
            \bubleftmult[S]{0.5,0}{\freecoupon{f}};
        \end{tikzpicture}
        \ .\qedhere
    \end{gather*}
\end{proof}

\begin{cor}
    The following relations hold in $\Dcat$:
    \begin{align} \label{crisp1}
        \begin{tikzpicture}[centerzero]
            \draw[D,->] (0,-0.4) -- (0,0.4);
            \bubrightmult[D]{0.45,0}{r};
        \end{tikzpicture}
        &=
        \begin{tikzpicture}[centerzero]
            \draw[D,->] (0,-0.4) -- (0,0.4);
            \bubrightmult[D]{-0.75,0}{r+1};
        \end{tikzpicture}
        \ ,& r \in \Z,
        \\ \label{crisp2}
        \begin{tikzpicture}[centerzero]
            \draw[D,->] (0,-0.5) -- (0,0.5);
            \bubright[S]{-0.4,0};
        \end{tikzpicture}
        &= \delta\
        \begin{tikzpicture}[centerzero]
            \draw[D,->] (0,-0.5) -- (0,0.5);
        \end{tikzpicture}
        \ .
    \end{align}
\end{cor}

\begin{proof}
    When $r<0$, \cref{crisp1} is trivial by \cref{bubrec}.  When $r \ge 0$, \cref{crisp1} follows from taking $f = \rightbubmult[D]{r}$ in the first relation in \cref{smarty}.  Relation \cref{crisp2} follows from taking $f = 1_n$ in the second relation in \cref{smarty}, then using \cref{smoke}.
\end{proof}

%---------------------
\subsection{Crossings}
%---------------------

We define the following crossings:
\begin{equation} \label{crossdefU}
    \upcross{S}{D} :=
    \begin{tikzpicture}[centerzero]
        \draw[S,->] (-0.2,-0.4) -- (-0.2,0) -- (0.2,0) -- (0.2,0.4);
        \draw[D] (0.2,-0.4) -- (0.2,0);
        \draw[D,->] (-0.2,0) -- (-0.2,0.4);
    \end{tikzpicture}
    \ ,\qquad
    \upcross{D}{S} :=
    \begin{tikzpicture}[centerzero]
        \draw[S,->] (0.2,-0.4) -- (0.2,0) -- (-0.2,0) -- (-0.2,0.4);
        \draw[D] (-0.2,-0.4) -- (-0.2,0);
        \draw[D,->] (0.2,0) -- (0.2,0.4);
    \end{tikzpicture}
    \ ,\qquad
    \upcross{D}{D} :=
    \begin{tikzpicture}[centerzero]
        \draw[D,->] (-0.2,-0.4) -- (-0.2,0.4);
        \draw[D,->] (0.2,-0.4) -- (0.2,0.4);
    \end{tikzpicture}
    \ .
\end{equation}
We rotate these crossings using caps and cups, giving:
\begin{gather} \label{crossdefL}
    \leftcross{S}{D} :=
    \begin{tikzpicture}[centerzero]
        \draw[S,<-] (-0.2,-0.4) -- (0,-0.13) -- (0,0.13) -- (0.2,0.4);
        \draw[D] (0.2,-0.4) -- (0,-0.13);
        \draw[D,->] (0,0.13) -- (-0.2,0.4);
    \end{tikzpicture}
    \ ,\qquad
    \leftcross{D}{S} :=
    \begin{tikzpicture}[centerzero]
        \draw[S,->] (0.2,-0.4) -- (0,-0.13) -- (0,0.13) -- (-0.2,0.4);
        \draw[D,<-] (-0.2,-0.4) -- (0,-0.13);
        \draw[D] (0,0.13) -- (0.2,0.4);
    \end{tikzpicture}
    \ ,\qquad
    \leftcross{D}{D} :=
    \begin{tikzpicture}[centerzero]
        \draw[D,<-] (-0.2,-0.4) -- (-0.2,-0.25) to[out=up,in=up,looseness=1.5] (0.2,-0.25) -- (0.2,-0.4);
        \draw[D,<-] (-0.2,0.4) -- (-0.2,0.25) to[out=down,in=down,looseness=1.5] (0.2,0.25) -- (0.2,0.4);
    \end{tikzpicture}
    \ ,
    \\ \label{crossdefR}
    \rightcross{S}{D} :=
    \begin{tikzpicture}[centerzero]
        \draw[S,->] (-0.2,-0.4) -- (0,-0.13) -- (0,0.13) -- (0.2,0.4);
        \draw[D,<-] (0.2,-0.4) -- (0,-0.13);
        \draw[D] (0,0.13) -- (-0.2,0.4);
    \end{tikzpicture}
    \ ,\qquad
    \rightcross{D}{S} :=
    \begin{tikzpicture}[centerzero]
        \draw[S,<-] (0.2,-0.4) -- (0,-0.13) -- (0,0.13) -- (-0.2,0.4);
        \draw[D] (-0.2,-0.4) -- (0,-0.13);
        \draw[D,->] (0,0.13) -- (0.2,0.4);
    \end{tikzpicture}
    \ ,\qquad
    \rightcross{D}{D} :=
    \begin{tikzpicture}[centerzero]
        \draw[D,->] (-0.2,-0.4) -- (-0.2,-0.25) to[out=up,in=up,looseness=1.5] (0.2,-0.25) -- (0.2,-0.4);
        \draw[D,->] (-0.2,0.4) -- (-0.2,0.25) to[out=down,in=down,looseness=1.5] (0.2,0.25) -- (0.2,0.4);
    \end{tikzpicture}
    \ ,
    \\ \label{crossdefD}
    \downcross{S}{D} :=
    \begin{tikzpicture}[centerzero]
        \draw[S,<-] (-0.2,-0.4) -- (-0.2,0) -- (0.2,0) -- (0.2,0.4);
        \draw[D,<-] (0.2,-0.4) -- (0.2,0);
        \draw[D] (-0.2,0) -- (-0.2,0.4);
    \end{tikzpicture}
    \ ,\qquad
    \downcross{D}{S} :=
    \begin{tikzpicture}[centerzero]
        \draw[S,<-] (0.2,-0.4) -- (0.2,0) -- (-0.2,0) -- (-0.2,0.4);
        \draw[D,<-] (-0.2,-0.4) -- (-0.2,0);
        \draw[D] (0.2,0) -- (0.2,0.4);
    \end{tikzpicture}
    \ ,\qquad
    \downcross{D}{D} :=
    \begin{tikzpicture}[centerzero]
        \draw[D,<-] (-0.2,-0.4) -- (-0.2,0.4);
        \draw[D,<-] (0.2,-0.4) -- (0.2,0.4);
    \end{tikzpicture}
    \ .
\end{gather}

\begin{prop}
    The following relations hold in $\Dcat$:
    \begin{gather} \label{reidUU}
        \begin{tikzpicture}[centerzero]
            \draw[D,->] (-0.2,-0.4) to[out=45,in=down] (0.15,0) to[out=up,in=-45] (-0.2,0.4);
            \draw[S,->] (0.2,-0.4) to[out=135,in=down] (-0.15,0) to[out=up,in=225] (0.2,0.4);
        \end{tikzpicture}
        =
        \begin{tikzpicture}[centerzero]
            \draw[D,->] (-0.2,-0.4) -- (-0.2,0.4);
            \draw[S,->] (0.2,-0.4) -- (0.2,0.4);
        \end{tikzpicture}
        \ ,\qquad
        \begin{tikzpicture}[centerzero]
            \draw[S,->] (-0.2,-0.4) to[out=45,in=down] (0.15,0) to[out=up,in=-45] (-0.2,0.4);
            \draw[D,->] (0.2,-0.4) to[out=135,in=down] (-0.15,0) to[out=up,in=225] (0.2,0.4);
        \end{tikzpicture}
        =
        \begin{tikzpicture}[centerzero]
            \draw[S,->] (-0.2,-0.4) -- (-0.2,0.4);
            \draw[D,->] (0.2,-0.4) -- (0.2,0.4);
        \end{tikzpicture}
        \ ,\qquad
        \begin{tikzpicture}[centerzero]
            \draw[D,->] (-0.2,-0.4) to[out=45,in=down] (0.15,0) to[out=up,in=-45] (-0.2,0.4);
            \draw[D,->] (0.2,-0.4) to[out=135,in=down] (-0.15,0) to[out=up,in=225] (0.2,0.4);
        \end{tikzpicture}
        =
        \begin{tikzpicture}[centerzero]
            \draw[D,->] (-0.2,-0.4) -- (-0.2,0.4);
            \draw[D,->] (0.2,-0.4) -- (0.2,0.4);
        \end{tikzpicture}
        \ ,
        \\ \label{reidUD}
        \begin{tikzpicture}[centerzero]
            \draw[D,->] (-0.2,-0.4) to[out=45,in=down] (0.15,0) to[out=up,in=-45] (-0.2,0.4);
            \draw[S,<-] (0.2,-0.4) to[out=135,in=down] (-0.15,0) to[out=up,in=225] (0.2,0.4);
        \end{tikzpicture}
        =
        \begin{tikzpicture}[centerzero]
            \draw[D,->] (-0.2,-0.4) -- (-0.2,0.4);
            \draw[S,<-] (0.2,-0.4) -- (0.2,0.4);
        \end{tikzpicture}
        \ ,\qquad
        \begin{tikzpicture}[centerzero]
            \draw[D,<-] (-0.2,-0.4) to[out=45,in=down] (0.15,0) to[out=up,in=-45] (-0.2,0.4);
            \draw[S,->] (0.2,-0.4) to[out=135,in=down] (-0.15,0) to[out=up,in=225] (0.2,0.4);
        \end{tikzpicture}
        =
        \begin{tikzpicture}[centerzero]
            \draw[D,<-] (-0.2,-0.4) -- (-0.2,0.4);
            \draw[S,->] (0.2,-0.4) -- (0.2,0.4);
            \bubright[D]{0.6,0};
        \end{tikzpicture}
        \ ,\qquad
        \begin{tikzpicture}[centerzero]
            \draw[D,->] (-0.2,-0.4) to[out=45,in=down] (0.15,0) to[out=up,in=-45] (-0.2,0.4);
            \draw[D,<-] (0.2,-0.4) to[out=135,in=down] (-0.15,0) to[out=up,in=225] (0.2,0.4);
        \end{tikzpicture}
        =
        \begin{tikzpicture}[centerzero]
            \draw[D,->] (-0.2,-0.4) -- (-0.2,0.4);
            \draw[D,<-] (0.2,-0.4) -- (0.2,0.4);
        \end{tikzpicture}
        \ ,\qquad
        \begin{tikzpicture}[centerzero]
            \draw[D,<-] (-0.2,-0.4) to[out=45,in=down] (0.15,0) to[out=up,in=-45] (-0.2,0.4);
            \draw[D,->] (0.2,-0.4) to[out=135,in=down] (-0.15,0) to[out=up,in=225] (0.2,0.4);
        \end{tikzpicture}
        =
        \begin{tikzpicture}[centerzero]
            \draw[D,<-] (-0.2,-0.4) -- (-0.2,0.4);
            \draw[D,->] (0.2,-0.4) -- (0.2,0.4);
            \bubright[D]{0.6,0};
        \end{tikzpicture}
        \ ,
        \\ \label{curls}
        \begin{tikzpicture}[centerzero,xscale=-1]
            \draw[D] (0,-0.4) to[out=up,in=180] (0.25,0.15) to[out=0,in=up] (0.4,0);
            \draw[D,->] (0.4,0) to[out=down,in=0] (0.25,-0.15) to[out=180,in=down] (0,0.4);
        \end{tikzpicture}
        = 
        \begin{tikzpicture}[centerzero]
            \draw[D,->] (0,-0.4) -- (0,0.4);
        \end{tikzpicture}
        \ ,\qquad
        \begin{tikzpicture}[centerzero]
            \draw[D,->] (0.4,0) to[out=down,in=0] (0.25,-0.15) to[out=180,in=down] (0,0.4);
            \draw[D] (0,-0.4) to[out=up,in=180] (0.25,0.15) to[out=0,in=up] (0.4,0);
        \end{tikzpicture}
        = 
        \begin{tikzpicture}[centerzero]
            \draw[D,->] (0,-0.4) -- (0,0.4);
            \bubright[D]{0.4,0};
        \end{tikzpicture}
        \ .
    \end{gather}
\end{prop}

\begin{proof}
    For the first relation in \cref{reidUU}, we compute
    \[
        \begin{tikzpicture}[centerzero]
            \draw[D,->] (-0.2,-0.4) to[out=45,in=down] (0.15,0) to[out=up,in=-45] (-0.2,0.4);
            \draw[S,->] (0.2,-0.4) to[out=135,in=down] (-0.15,0) to[out=up,in=225] (0.2,0.4);
        \end{tikzpicture}
        \overset{\cref{crossdefU}}{=}
        \begin{tikzpicture}[centerzero]
            \draw[S,->] (0.2,-0.5) -- (0.2,-0.2) -- (-0.2,-0.2) -- (-0.2,0.2) -- (0.2,0.2) -- (0.2,0.5);
            \draw[D] (-0.2,-0.5) -- (-0.2,-0.2);
            \draw[D] (0.2,-0.2) -- (0.2,0.2);
            \draw[D,->] (-0.2,0.2) -- (-0.2,0.5);
        \end{tikzpicture}
        \overset{\cref{fire}}{=}
        \begin{tikzpicture}[centerzero]
            \draw[S] (0,-0.5) -- (0,-0.2);
            \draw[S] (0,-0.2) to[out=135,in=-135,looseness=2] (0,0.2);
            \draw[D] (0,-0.2) to[out=45,in=-45,looseness=2] (0,0.2);
            \draw[S,->] (0,0.2) -- (0,0.5);
            \draw[D,->] (-0.4,-0.5) -- (-0.4,0.5);
        \end{tikzpicture}
        \overset{\cref{water}}{=}
        \begin{tikzpicture}[centerzero]
            \draw[D,->] (-0.2,-0.4) -- (-0.2,0.4);
            \draw[S,->] (0.2,-0.4) -- (0.2,0.4);
        \end{tikzpicture}
        \ .
    \]
    The second relation in \cref{reidUU} is precisely the third relation in \cref{fire}, while the third relation in \cref{reidUU} is trivial.
    
    For the first relation in \cref{reidUD}, we compute
    \[
        \begin{tikzpicture}[centerzero]
            \draw[D,->] (-0.2,-0.4) to[out=45,in=down] (0.15,0) to[out=up,in=-45] (-0.2,0.4);
            \draw[S,<-] (0.2,-0.4) to[out=135,in=down] (-0.15,0) to[out=up,in=225] (0.2,0.4);
        \end{tikzpicture}
        \overset{\cref{crossdefL}}{\underset{\cref{crossdefR}}{=}}
        \begin{tikzpicture}[centerzero]
            \draw[S] (0,-0.4) -- (0,-0.2);
            \draw[S] (0,-0.2) to[out=135,in=-135,looseness=2] (0,0.2);
            \draw[D] (0,-0.2) to[out=45,in=-45,looseness=2] (0,0.2);
            \draw[S] (0,0.2) -- (0,0.4) -- (0.2,0.6);
            \draw[D,->] (0,0.4) -- (-0.2,0.6);
            \draw[D] (-0.2,-0.6) -- (0,-0.4);
            \draw[S,<-] (0.2,-0.6) -- (0,-0.4);
        \end{tikzpicture}
        \overset{\cref{water}}{=}
        \begin{tikzpicture}[centerzero]
            \draw[D] (-0.3,-0.5) -- (0,-0.2);
            \draw[S,<-] (0.3,-0.5) -- (0,-0.2);
            \draw[S] (0,-0.2) -- (0,0.2);
            \draw[S] (0,0.2) -- (0.3,0.5);
            \draw[D,->] (0,0.2) -- (-0.3,0.5);
        \end{tikzpicture}
        \overset{\cref{fire}}{=}
        \begin{tikzpicture}[centerzero]
            \draw[D,->] (-0.2,-0.5) -- (-0.2,0.5);
            \draw[S,<-] (0.2,-0.5) -- (0.2,0.5);
        \end{tikzpicture}
        \ .
    \]
    For the second relation in \cref{reidUD}, we compute
    \[
        \begin{tikzpicture}[centerzero]
            \draw[D,<-] (-0.2,-0.4) to[out=45,in=down] (0.15,0) to[out=up,in=-45] (-0.2,0.4);
            \draw[S,->] (0.2,-0.4) to[out=135,in=down] (-0.15,0) to[out=up,in=225] (0.2,0.4);
        \end{tikzpicture}
        \overset{\cref{crossdefL}}{\underset{\cref{crossdefR}}{=}}
        \begin{tikzpicture}[centerzero]
            \draw[S] (0,-0.4) -- (0,-0.2);
            \draw[S] (0,-0.2) to[out=135,in=-135,looseness=2] (0,0.2);
            \draw[D] (0,-0.2) to[out=45,in=-45,looseness=2] (0,0.2);
            \draw[S,->] (0,0.2) -- (0,0.4) -- (0.2,0.6);
            \draw[D] (0,0.4) -- (-0.2,0.6);
            \draw[D,<-] (-0.2,-0.6) -- (0,-0.4);
            \draw[S] (0.2,-0.6) -- (0,-0.4);
        \end{tikzpicture}
        \overset{\cref{water}}{=}
        \begin{tikzpicture}[centerzero]
            \draw[D,<-] (-0.3,-0.5) -- (0,-0.2);
            \draw[S] (0.3,-0.5) -- (0,-0.2);
            \draw[S] (0,-0.2) -- (0,0.2);
            \draw[S,->] (0,0.2) -- (0.3,0.5);
            \draw[D] (0,0.2) -- (-0.3,0.5);
            \bubleft[D]{0.6,0};
        \end{tikzpicture}
        \overset{\cref{fire}}{=}
        \begin{tikzpicture}[centerzero]
            \draw[D,<-] (-0.2,-0.4) -- (-0.2,0.4);
            \draw[S,->] (0.2,-0.4) -- (0.2,0.4);
            \bubright[D]{0.6,0};
        \end{tikzpicture}
        \ .
    \]
    For the third relation in \cref{reidUD}, we compute
    \[
        \begin{tikzpicture}[centerzero]
            \draw[D,->] (-0.2,-0.4) to[out=45,in=down] (0.15,0) to[out=up,in=-45] (-0.2,0.4);
            \draw[D,<-] (0.2,-0.4) to[out=135,in=down] (-0.15,0) to[out=up,in=225] (0.2,0.4);
        \end{tikzpicture}
        \overset{\cref{crossdefL}}{\underset{\cref{crossdefR}}{=}}
        \begin{tikzpicture}[centerzero]
            \draw[D,->] (-0.2,-0.4) -- (-0.2,-0.25) to[out=up,in=up,looseness=1.5] (0.2,-0.25) -- (0.2,-0.4);
            \draw[D,<-] (-0.2,0.4) -- (-0.2,0.25) to[out=down,in=down,looseness=1.5] (0.2,0.25) -- (0.2,0.4);
            \bubleft[D]{0.6,0};
        \end{tikzpicture}
        \overset{\cref{smoke}}{\underset{\cref{spark}}{=}}
        \begin{tikzpicture}[centerzero]
            \draw[D,->] (-0.2,-0.5) -- (-0.2,0.5);
            \draw[D,<-] (0.2,-0.5) -- (0.2,0.5);
        \end{tikzpicture}
        \ .
    \]
    For the fourth relation in \cref{reidUD}, we compute
    \[
        \begin{tikzpicture}[centerzero]
            \draw[D,<-] (-0.2,-0.4) to[out=45,in=down] (0.15,0) to[out=up,in=-45] (-0.2,0.4);
            \draw[D,->] (0.2,-0.4) to[out=135,in=down] (-0.15,0) to[out=up,in=225] (0.2,0.4);
        \end{tikzpicture}
        \overset{\cref{crossdefL}}{\underset{\cref{crossdefR}}{=}}
        \begin{tikzpicture}[centerzero]
            \draw[D,<-] (-0.2,-0.4) -- (-0.2,-0.25) to[out=up,in=up,looseness=1.5] (0.2,-0.25) -- (0.2,-0.4);
            \draw[D,->] (-0.2,0.4) -- (-0.2,0.25) to[out=down,in=down,looseness=1.5] (0.2,0.25) -- (0.2,0.4);
            \bubright[D]{0.6,0};
        \end{tikzpicture}
        \overset{\cref{smoke}}{\underset{\cref{fire}}{=}}
        \begin{tikzpicture}[centerzero]
            \draw[D,<-] (-0.2,-0.5) -- (-0.2,0.5);
            \draw[D,->] (0.2,-0.5) -- (0.2,0.5);
            \bubright[D]{0.6,0};
        \end{tikzpicture}
        \ .
    \]
    The relations \cref{curls} follow easily from \cref{crossdefU,smoke}.
\end{proof}

We adopt the following convention for multiple strands:
\[
    \begin{tikzpicture}[centerzero]
        \draw[D,->] (0,-0.3) \botlabel{0} -- (0,0.3);
        \region{0.2,0}{n};
    \end{tikzpicture}
    := 1_n,\qquad
    \begin{tikzpicture}[centerzero]
        \draw[D,->] (0,-0.3) \botlabel{r} -- (0,0.3);
    \end{tikzpicture}
    :=
    \underbrace{
        \begin{tikzpicture}[centerzero]
            \draw[D,->] (-0.3,-0.3) -- (-0.3,0.3);
            \draw[D,->] (0.25,-0.3) -- (0.25,0.3);
            \node at (0,0) {$\cdots$};
        \end{tikzpicture}
    }_{r \text{ strands}},
    \qquad r \ge 1.
\]
We define crossings involving such multiple strands in the obvious way:
\[
    \begin{tikzpicture}[centerzero]
        \draw[D,->] (-0.2,-0.4) \botlabel{r} -- (0.2,0.4);
        \draw[S,->] (0.2,-0.4) -- (-0.2,0.4);
    \end{tikzpicture}
    :=
    \begin{tikzpicture}[centerzero]
        \draw[D,->] (-0.4,-0.4) -- (0,0.4);
        \draw[D,->] (0,-0.4) -- (0.4,0.4);
        \draw[S,->] (0.4,-0.4) -- (-0.4,0.4);
        \node at (-0.1,-0.2) {$\scriptstyle{\cdots}$};
    \end{tikzpicture}
    \ ,\qquad
    \begin{tikzpicture}[centerzero,xscale=-1]
        \draw[D,->] (-0.2,-0.4) \botlabel{r} -- (0.2,0.4);
        \draw[S,->] (0.2,-0.4) -- (-0.2,0.4);
    \end{tikzpicture}
    :=
    \begin{tikzpicture}[centerzero,xscale=-1]
        \draw[D,->] (-0.4,-0.4) -- (0,0.4);
        \draw[D,->] (0,-0.4) -- (0.4,0.4);
        \draw[S,->] (0.4,-0.4) -- (-0.4,0.4);
        \node at (-0.1,-0.2) {$\scriptstyle{\cdots}$};
    \end{tikzpicture}
    \ .
\]

\begin{rem}
    If $\delta = -q-q^{-1}$ for some $q \in \kk^\times$ with a chosen square root $q^{1/2}$, then one can also define black-black crossings:
    \[
        \posupcross := q^{1/2}\ 
        \begin{tikzpicture}[centerzero]
            \draw[S,->] (-0.2,-0.4) -- (-0.2,0.4);
            \draw[S,->] (0.2,-0.4) -- (0.2,0.4);
        \end{tikzpicture}
        + q^{-1/2}\
        \begin{tikzpicture}[centerzero]
            \draw[S] (-0.2,-0.4) -- (0,-0.13) -- (0.2,-0.4);
            \draw[S,<->] (-0.2,0.4) -- (0,0.13) -- (0.2,0.4);
            \draw[D] (0,-0.13) -- (0,0.13);
        \end{tikzpicture}
        \qquad \text{and} \qquad
        \negupcross := q^{-1/2}\ 
        \begin{tikzpicture}[centerzero]
            \draw[S,->] (-0.2,-0.4) -- (-0.2,0.4);
            \draw[S,->] (0.2,-0.4) -- (0.2,0.4);
        \end{tikzpicture}
        + q^{1/2}\
        \begin{tikzpicture}[centerzero]
            \draw[S] (-0.2,-0.4) -- (0,-0.13) -- (0.2,-0.4);
            \draw[S,<->] (-0.2,0.4) -- (0,0.13) -- (0.2,0.4);
            \draw[D] (0,-0.13) -- (0,0.13);
        \end{tikzpicture}
        \ .
    \]
    Then a straightforward computation shows that
    \begin{gather*}
        \begin{tikzpicture}[centerzero]
            \draw[S,->] (0.2,-0.4) to[out=135,in=down] (-0.15,0) to[out=up,in=225] (0.2,0.4);
            \draw[S,over,->] (-0.2,-0.4) to[out=45,in=down] (0.15,0) to[out=up,in=-45] (-0.2,0.4);
        \end{tikzpicture}
        \ =\
        \begin{tikzpicture}[centerzero]
            \draw[S,->] (-0.2,-0.4) -- (-0.2,0.4);
            \draw[S,->] (0.2,-0.4) -- (0.2,0.4);
        \end{tikzpicture}
        \ =\
        \begin{tikzpicture}[centerzero]
            \draw[S,->] (-0.2,-0.4) to[out=45,in=down] (0.15,0) to[out=up,in=-45] (-0.2,0.4);
            \draw[S,over,->] (0.2,-0.4) to[out=135,in=down] (-0.15,0) to[out=up,in=225] (0.2,0.4);
        \end{tikzpicture}
        \ ,\qquad
        \begin{tikzpicture}[centerzero]
            \draw[S,->] (0.4,-0.4) -- (-0.4,0.4);
            \draw[over,->] (0,-0.4) to[out=135,in=down] (-0.32,0) to[out=up,in=225] (0,0.4);
            \draw[over,->] (-0.4,-0.4) -- (0.4,0.4);
        \end{tikzpicture}
        \ =\
        \begin{tikzpicture}[centerzero]
            \draw[S,->] (0.4,-0.4) -- (-0.4,0.4);
            \draw[S,over,->] (0,-0.4) to[out=45,in=down] (0.32,0) to[out=up,in=-45] (0,0.4);
            \draw[S,over,->] (-0.4,-0.4) -- (0.4,0.4);
        \end{tikzpicture}
        \ ,
        \\
        \begin{tikzpicture}[centerzero]
            \draw[S,<-] (0.2,-0.4) to[out=135,in=down] (-0.15,0) to[out=up,in=225] (0.2,0.4);
            \draw[S,over,->] (-0.2,-0.4) to[out=45,in=down] (0.15,0) to[out=up,in=-45] (-0.2,0.4);
        \end{tikzpicture}
        \ =\
        \begin{tikzpicture}[centerzero]
            \draw[S,->] (-0.2,-0.4) -- (-0.2,0.4);
            \draw[S,<-] (0.2,-0.4) -- (0.2,0.4);
        \end{tikzpicture}
        \ ,\qquad
        \begin{tikzpicture}[centerzero]
            \draw[S,->] (0.2,-0.4) to[out=135,in=down] (-0.15,0) to[out=up,in=225] (0.2,0.4);
            \draw[S,over,<-] (-0.2,-0.4) to[out=45,in=down] (0.15,0) to[out=up,in=-45] (-0.2,0.4);
        \end{tikzpicture}
        \ =\
        \begin{tikzpicture}[centerzero]
            \draw[S,<-] (-0.25,-0.6) -- (-0.25,-0.3) to[out=up,in=up,looseness=1.5] (0.25,-0.3) -- (0.25,-0.6);
            \draw[S,->] (-0.25,0.6) -- (-0.25,0.3) to[out=down,in=down,looseness=1.5] (0.25,0.3) -- (0.25,0.6);
            \draw[S] (0.7,0) arc(0:360:0.15);
            \draw[S,->] (0.7,0) -- (0.701,-0.01);
        \end{tikzpicture}
        + q\
        \begin{tikzpicture}[centerzero]
            \draw[S,<-] (-0.27,-0.6) -- (-0.27,-0.3) to[out=up,in=up,looseness=1.5] (0.27,-0.3) -- (0.27,-0.6);
            \draw[S,->] (-0.27,0.6) -- (-0.27,0.3) to[out=down,in=down,looseness=1.5] (0.27,0.3) -- (0.27,0.6);
            \draw[D] (0.12,0.35) arc(0:360:0.15);
            \draw[D,->] (0.12,0.35) -- (0.121,0.34);
        \end{tikzpicture}
        + q^{-1}\
        \begin{tikzpicture}[centerzero]
            \draw[S,<-] (-0.27,-0.6) -- (-0.27,-0.3) to[out=up,in=up,looseness=1.5] (0.27,-0.3) -- (0.27,-0.6);
            \draw[S,->] (-0.27,0.6) -- (-0.27,0.3) to[out=down,in=down,looseness=1.5] (0.27,0.3) -- (0.27,0.6);
            \draw[D] (0.12,-0.38) arc(0:360:0.15);
            \draw[D,->] (0.12,-0.38) -- (0.121,-0.39);
        \end{tikzpicture}
        +\
        \begin{tikzpicture}[centerzero]
            \draw[S,<-] (-0.35,-0.6) -- (-0.35,0.6);
            \draw[S,->] (0.35,-0.6) -- (0.35,0.6);
            \draw[D] (0.15,0) arc(0:360:0.15);
            \draw[D,->] (0.15,0) -- (0.151,-0.01);
        \end{tikzpicture}
        \ .
    \end{gather*}
\end{rem}

%--------------------------
\subsection{More relations}
%--------------------------

We now deduce some additional relations that will be used later in the paper.

\begin{lem}
    The following relations hold in $\Dcat$:
    \begin{equation} \label{puddle}
        \begin{tikzpicture}[centerzero]
            \draw[S,->] (0,-0.4) -- (0,0.4);
            \draw[D,->] (0.4,-0.4) \botlabel{r} -- (0.4,0.4);
            \bubrightmult[D]{-0.6,0}{r};
        \end{tikzpicture}
        =
        \begin{tikzpicture}[centerzero]
            \draw[S,->] (0,-0.4) -- (0,0.4);
            \draw[D,->] (0.4,-0.4) \botlabel{r} -- (0.4,0.4);
        \end{tikzpicture}
        \ ,\qquad
        \begin{tikzpicture}[centerzero]
            \draw[S,<-] (0,-0.4) -- (0,0.4);
            \draw[D,->] (0.4,-0.4) \botlabel{r+1} -- (0.4,0.4);
            \bubrightmult[D]{-0.6,0}{r};
        \end{tikzpicture}
        =
        \begin{tikzpicture}[centerzero]
            \draw[S,<-] (0,-0.4) -- (0,0.4);
            \draw[D,->] (0.4,-0.4) \botlabel{r+1} -- (0.4,0.4);
        \end{tikzpicture}
        \ ,\qquad r \in \N.
    \end{equation}
\end{lem}

\begin{proof}
    To prove the first equality in \cref{puddle}, we compute
    \[
        \begin{tikzpicture}[centerzero]
            \draw[S,->] (0,-0.4) -- (0,0.4);
            \draw[D,->] (0.4,-0.4) \botlabel{r} -- (0.4,0.4);
            \bubrightmult[D]{-0.6,0}{r};
        \end{tikzpicture}
        \overset{\cref{reidUU}}{=}
        \begin{tikzpicture}[centerzero]
            \draw[S,->] (-0.2,-0.4) to[out=45,in=down] (0.15,0) to[out=up,in=-45] (-0.2,0.4);
            \draw[D,->] (0.2,-0.4) \botlabel{r} to[out=135,in=down] (-0.15,0) to[out=up,in=225] (0.2,0.4);
            \bubrightmult[D]{-0.6,0}{r};
        \end{tikzpicture}
        \overset{\cref{crisp1}}{\underset{\cref{bubrec}}{=}}
        \begin{tikzpicture}[centerzero]
            \draw[S,->] (-0.2,-0.4) to[out=45,in=down] (0.15,0) to[out=up,in=-45] (-0.2,0.4);
            \draw[D,->] (0.2,-0.4) \botlabel{r} to[out=135,in=down] (-0.15,0) to[out=up,in=225] (0.2,0.4);
        \end{tikzpicture}
        \overset{\cref{reidUU}}{=}
        \begin{tikzpicture}[centerzero]
            \draw[S,->] (0,-0.4) -- (0,0.4);
            \draw[D,->] (0.4,-0.4) \botlabel{r} -- (0.4,0.4);
        \end{tikzpicture}
        \ .
    \]
    For the second equality in \cref{puddle}, we compute
    \[
        \begin{tikzpicture}[centerzero]
            \draw[S,<-] (0,-0.4) -- (0,0.4);
            \draw[D,->] (0.4,-0.4) \botlabel{r+1} -- (0.4,0.4);
            \bubrightmult[D]{-0.6,0}{r};
        \end{tikzpicture}
        \overset{\cref{fire}}{=}
        \begin{tikzpicture}[centerzero]
            \draw[D] (0.3,-0.5) -- (0,-0.2);
            \draw[S,<-] (-0.3,-0.5) -- (0,-0.2);
            \draw[S] (0,-0.2) -- (0,0.2);
            \draw[S] (0,0.2) -- (-0.3,0.5);
            \draw[D,->] (0,0.2) -- (0.3,0.5);
            \draw[D,->] (0.6,-0.5) \botlabel{r} -- (0.6,0.5);
            \bubrightmult[D]{-0.6,0}{r};
        \end{tikzpicture}
        =
        \begin{tikzpicture}[centerzero]
            \draw[D] (0.3,-0.5) -- (0,-0.2);
            \draw[S,<-] (-0.3,-0.5) -- (0,-0.2);
            \draw[S] (0,-0.2) -- (0,0.2);
            \draw[S] (0,0.2) -- (-0.3,0.5);
            \draw[D,->] (0,0.2) -- (0.3,0.5);
            \draw[D,->] (0.6,-0.5) \botlabel{r} -- (0.6,0.5);
        \end{tikzpicture}
        \overset{\cref{fire}}{=}
        \begin{tikzpicture}[centerzero]
            \draw[S,<-] (0,-0.4) -- (0,0.4);
            \draw[D,->] (0.4,-0.4) \botlabel{r+1} -- (0.4,0.4);
        \end{tikzpicture}
        \ ,
    \]
    where the second equality follows from the first relation in \cref{puddle}.
\end{proof}

\begin{cor}
    For all $r,s \in \Z$, the following relation holds in $\Dcat$:
    \begin{equation} \label{tulum}
        \begin{tikzpicture}[centerzero]
            \draw[->] (0,-0.5) -- (0,0.5);
            \bubrightmult[D]{-0.5,0}{r};
            \bubrightmult[D]{0.5,0}{s};
        \end{tikzpicture}
        =
        \begin{cases}
            \begin{tikzpicture}[centerzero]
                \draw[->] (0,-0.5) -- (0,0.5);
                \bubrightmult[D]{-0.5,0}{r};
            \end{tikzpicture}
            & \text{if } r > s,
            \\
            \begin{tikzpicture}[centerzero]
                \draw[->] (0,-0.5) -- (0,0.5);
                \bubrightmult[D]{0.5,0}{s};
            \end{tikzpicture}
            & \text{if } r \le s.
        \end{cases}
    \end{equation}
\end{cor}

\begin{proof}
    If $r \le 0$ or $s \le 0$, then the result is trivial, using \cref{bubrec}.  If $r,s > 0$, then the result follows from \cref{puddle}.
\end{proof}

\begin{lem}
    The following relation holds in $\Dcat$ for all $f \in \Dcat(1_n,1_n)$, $n \in \N$:
    \begin{equation} \label{switchy}
        \rightbubmult[D]{\leftbubmult[S]{\freecoupon{f}}}
        =
        \rightbub[D]\ \leftbubmult[S]{\rightbubmult[D]{\freecoupon{f}}}\ .
    \end{equation}
\end{lem}

\begin{proof}
    We have
    \[
        \rightbubmult[D]{\leftbubmult[S]{\freecoupon{f}}}
        \overset{\cref{reidUD}}{=}
        \begin{tikzpicture}[centerzero]
            \coupon{0,0}{f};
            \draw[S,->] (-0.6,0) arc(-180:180:0.6);
            \draw[D,->] (0,-0.8) to[out=left,in=down] (-0.6,-0.4) to[out=up,in=down] (-0.4,0) to[out=up,in=down] (-0.6,0.4) to[out=up,in=left] (0,0.8) to[out=right,in=up] (0.6,0.4) to[out=down,in=up] (0.4,0) to[out=down,in=up] (0.6,-0.4) to[out=down,in=right] (0,-0.8);
        \end{tikzpicture}
        \overset{\cref{reidUD}}{\underset{\cref{bubrec}}{=}} \rightbub[D]\ \leftbubmult[S]{\rightbubmult[D]{\freecoupon{f}}}\ .
        \qedhere
    \]
\end{proof}

\begin{prop}
    The following relations hold in $\Dcat$ for all $n \in \Z$:
    \begin{align} \label{leftbull}
        \leftbubmultreg[S]{\rightbubmult[D]{r}}{n}
        &= \frac{\Delta_{n-2r+3}}{\Delta_{n-2r+2}}\ \rightbubmultreg[D]{r}{n} + \frac{\Delta_{n-2r+1}}{\Delta_{n-2r+2}}\ \rightbubmultreg[D]{r-1}{n}
        \ ,& r \in \Z,\ 2r \le n+2,
        \\ \label{rightbull}
        \rightbubmultreg[S]{\rightbubmult[D]{r}}{n}
        &= \frac{\Delta_{n-2r-1}}{\Delta_{n-2r}}\ \rightbubmultreg[D]{r}{n} + \frac{\Delta_{n-2r+1}}{\Delta_{n-2r}}\ \rightbubmultreg[D]{r+1}{n}
        \ ,& r \in \N,\ 2r \le n.
    \end{align}
\end{prop}

\begin{proof}
    Since the statements are trivial if $n < 0$, we assume $n \ge 0$.  We prove \cref{leftbull} by induction on $r$.  For $r \le 0$, it follows from \cref{bubrec,Delta,smoke}.  For $r = 1$, we have
    \[
        \leftbubmultreg[S]{\rightbub[D]}{n}
        \overset{\cref{sigma0}}{=} \leftbubreg[S]{n} - \leftbubmultreg[S]{\freeprojector{0}}{n}
        \overset{\cref{smoke}}{\underset{\cref{coal}}{=}} \delta 1_n - \frac{\Delta_{n+1}}{\Delta_n}
        \begin{tikzpicture}[centerzero]
            \projector{0,0}{0};
            \region{0.4,0}{n};
        \end{tikzpicture}
        \overset{\cref{sigma0}}{\underset{\cref{Delta}}{=}}
        \frac{\Delta_{n+1}}{\Delta_n}\ \rightbubreg[D]{n} + \frac{\Delta_{n-1}}{\Delta_n}\ 1_n,
    \]
    verifying \cref{leftbull}.  Now suppose that \cref{leftbull} holds for some $r \in \N$.  Taking $f = \rightbubmult[D]{r}$ in \cref{switchy}, we have
    \[
        \rightbubmultreg[D]{\leftbubmult[S]{\rightbubmult[D]{r}}}{n+2}
        = \rightbub[D]\ \leftbubmultreg[S]{\rightbubmult[D]{r+1}}{n+2}
        \overset{\cref{puddle}}{=} \leftbubmultreg[S]{\rightbubmult[D]{r+1}}{n+2}.
    \]
    Thus, using the induction hypothesis, we have
    \[
        \leftbubmultreg[S]{\rightbubmult[D]{r+1}}{n+2}
        = \rightbubmultreg[D]{\leftbubmult[S]{\rightbubmult[D]{r}}}{n+2}
        \overset{\cref{bubrec}}{=} \frac{\Delta_{n-2r+3}}{\Delta_{n-2r+2}}\ \rightbubmultreg[D]{r+1}{n+2} + \frac{\Delta_{n-2r+1}}{\Delta_{n-2r+2}}\ \rightbubmultreg[D]{r}{n+2},
    \]
    completing the proof of the induction step.
    
    To prove \cref{rightbull}, we compute
    \[
        \rightbubmult[S]{\rightbubmult[D]{r}}
        \overset{\cref{apple2}}{=} \leftbubmult[S]{\rightbubmult[D]{r+1}}
        \overset{\cref{leftbull}}{=} \frac{\Delta_{n-2r-1}}{\Delta_{n-2r}}\ \rightbubmultreg[D]{r}{n} + \frac{\Delta_{n-2r+1}}{\Delta_{n-2r}}\ \rightbubmultreg[D]{r+1}{n}\ .\qedhere
    \]
\end{proof}

We extend the definition \cref{sigma0} to negative values of $r$ by setting $\freeprojectorreg{r}{n} = 0$ for $r < 0$.  It follows from \cref{sigma0,bubrec} that
\begin{equation} \label{sigma}
    \freeprojectorreg{r}{n}
    = \rightbubmultreg[D]{r}{n} - \rightbubmultreg[D]{r+1}{n},
    \qquad r,n \in \Z.
\end{equation}

\begin{prop}
    The following relations hold in $\Dcat$:
    \begin{gather} \label{abyss}
        \freeprojectorreg{r}{n}
        = 0, \qquad r,n \in \Z \text{ and } (n < 2r \text{ or } r < 0),
        \\ \label{spread}
        \sum_{r=0}^{\lfloor n/2 \rfloor} \freeprojectorreg{r}{n}
        = 1_n, \qquad n \in \N,
        \\ \label{absorb}
        \freeprojector{r}\ \freeprojector{s} = \delta_{r,s}\ \freeprojector{r}\, ,\qquad r,s \in \Z,
        \\ \label{smush}
        \begin{tikzpicture}[centerzero]
            \draw[D,->] (0,-0.4) -- (0,0.4);
            \projector{0.4,0}{r};
        \end{tikzpicture}
        =
        \begin{tikzpicture}[centerzero]
            \draw[D,->] (0,-0.4) -- (0,0.4);
            \projector{-0.6,0}{r+1};
        \end{tikzpicture}
        \ ,\qquad r \in \Z,
        \\ \label{portis}
        \begin{tikzpicture}[centerzero]
            \draw[S,->] (0,-0.4) -- (0,0.4);
            \projector{-0.4,0}{r};
        \end{tikzpicture}
        =
        \begin{tikzpicture}[centerzero]
            \draw[S,->] (0,-0.4) -- (0,0.4);
            \projector{-0.4,0}{r};
        \end{tikzpicture}
        \left( \freeprojector{r} + \freeprojector{r-1} \right),
        \quad
        \begin{tikzpicture}[centerzero]
            \draw[S,->] (0,-0.4) -- (0,0.4);
            \projector{0.4,0}{r};
        \end{tikzpicture}
        = \left( \freeprojector{r} + \freeprojector{r+1} \right)
        \begin{tikzpicture}[centerzero]
            \draw[S,->] (0,-0.4) -- (0,0.4);
            \projector{0.4,0}{r};
        \end{tikzpicture}
        \, , \qquad r \in \Z,
        \\ \label{bluewind}
        \leftbubmult[D]{\freeprojector{r}} = \freeprojector{r-1}\, ,\qquad
        \rightbubmult[D]{\freeprojector{s}} = \freeprojector{s+1}\, ,\qquad
        r \in \Z,\ s \in \N.
    \end{gather}
\end{prop}

\begin{proof}
    Relation \cref{abyss} follows from \cref{freeze,bubrec}.  For \cref{spread}, we have
    \[
        \sum_{r=0}^{\lfloor n/2 \rfloor} \freeprojectorreg{r}{n}
        \overset{\cref{sigma}}{=} \rightbubmultreg[D]{0}{n} - \rightbubmultreg[D]{\lfloor n/2 \rfloor + 1}{n}
        \overset{\cref{bubrec}}{\underset{\cref{abyss}}{=}} 1_n.
    \]
    Relation \cref{absorb} follows from \cref{headroom}.  Relation \cref{smush} is trivial when $r < -1$, by \cref{abyss}.  When $r=-1$, both sides of \cref{smush} are zero by \cref{abyss,crisp1}.  For $r>0$, \cref{smush} follows from \cref{crisp1}.  
    
    For the first equality in \cref{portis}, we compute
    \[
        \begin{tikzpicture}[centerzero]
            \draw[S,->] (0,-0.4) -- (0,0.4);
            \projector{-0.4,0}{r};
        \end{tikzpicture}
        \left( \freeprojector{r} + \freeprojector{r-1} \right)
        \overset{\cref{sigma}}{=}
        \left( \rightbubmult[D]{r} - \rightbubmult[D]{r+1} \right)
        \begin{tikzpicture}[centerzero]
            \draw[S,->] (0,-0.4) -- (0,0.4);
        \end{tikzpicture}
        \left( \rightbubmult[D]{r-1} - \rightbubmult[D]{r+1} \right)
        \overset{\cref{tulum}}{=}
        \left( \rightbubmult[D]{r} - \rightbubmult[D]{r+1} \right)
        \begin{tikzpicture}[centerzero]
            \draw[S,->] (0,-0.4) -- (0,0.4);
        \end{tikzpicture}
        =
        \begin{tikzpicture}[centerzero]
            \draw[S,->] (0,-0.4) -- (0,0.4);
            \projector{-0.4,0}{r};
        \end{tikzpicture}.
    \]
    The proof of the second equality in \cref{portis} is analogous.
    \details{
        We have
        \[
            \left( \freeprojector{r} + \freeprojector{r+1} \right)
            \begin{tikzpicture}[centerzero]
                \draw[S,->] (0,-0.4) -- (0,0.4);
                \projector{0.4,0}{r};
            \end{tikzpicture}
            \overset{\cref{sigma}}{=}
            \left( \rightbubmult[D]{r} - \rightbubmult[D]{r+2} \right)
            \begin{tikzpicture}[centerzero]
                \draw[S,->] (0,-0.4) -- (0,0.4);
            \end{tikzpicture}
            \left( \rightbubmult[D]{r-1} - \rightbubmult[D]{r+1} \right)
            \overset{\cref{tulum}}{=} \left( \rightbubmult[D]{r} - \rightbubmult[D]{r+1} \right)
            \begin{tikzpicture}[centerzero]
                \draw[S,->] (0,-0.4) -- (0,0.4);
            \end{tikzpicture}
            =
            \begin{tikzpicture}[centerzero]
                \draw[S,->] (0,-0.4) -- (0,0.4);
                \projector{-0.4,0}{r};
            \end{tikzpicture}.
        \]
    }
    
    The relations \cref{bluewind} follow easily from \cref{sigma,bubrec,apple1}.
\end{proof}

\begin{prop} \label{dune}
    The following relations hold in $\Dcat$ for all $n \in \Z$:
    \begin{align} \label{wind}
        \rightbubmultreg[S]{\freeprojector{r-1}}{n}
        &= \leftbubmultreg[S]{\freeprojector{r}}{n}
        = \frac{\Delta_{n-2r+1}}{\Delta_{n-2r}}\ \freeprojectorreg{r}{n}
        + \frac{\Delta_{n-2r+1}}{\Delta_{n-2r+2}}\ \freeprojectorreg{r-1}{n},
        &r \in \Z,\ 2r \le n+1,
        \\ \label{camel1}
        \begin{tikzpicture}[centerzero]
            \draw[S,<-] (-0.5,0.8) -- (-0.5,0.6) to[out=down,in=down,looseness=1.7] (0.5,0.6) -- (0.5,0.8);
            \projector{0,0.5}{r-1};
            \draw[S,->] (-0.5,-0.8) -- (-0.5,-0.6) to[out=up,in=up,looseness=1.7] (0.5,-0.6) -- (0.5,-0.8);
            \projector{0,-0.5}{r-1};
            \region{0.4,0}{n};
        \end{tikzpicture}
        &= \frac{\Delta_{n-2r+1}}{\Delta_{n-2r}}\
        \begin{tikzpicture}[centerzero]
            \draw[S,->] (-0.5,-0.8) -- (-0.5,0.8);
            \draw[S,<-] (0.5,-0.8) -- (0.5,0.8);
            \projector{0,0}{r-1};
            \projector{-0.85,0}{r};
            \projector{0.85,0}{r};
            \region{0.8,0.5}{n};
        \end{tikzpicture}        
        + \frac{\Delta_{n-2r+1}}{\Delta_{n-2r+2}}\
        \begin{tikzpicture}[centerzero]
            \draw[S,->] (-0.5,-0.8) -- (-0.5,0.8);
            \draw[S,<-] (0.5,-0.8) -- (0.5,0.8);
            \projector{0,0}{r-1};
            \projector{-1,0}{r-1};
            \projector{1,0}{r-1};
            \region{0.8,0.5}{n};
        \end{tikzpicture}
        \ ,& r \in \Z,\ 2r \le n+1,
        \\ \label{camel2}
        \begin{tikzpicture}[centerzero]
            \draw[S,->] (-0.4,0.8) -- (-0.4,0.6) to[out=down,in=down,looseness=2] (0.4,0.6) -- (0.4,0.8);
            \projector{0,0.5}{r};
            \draw[S,<-] (-0.4,-0.8) -- (-0.4,-0.6) to[out=up,in=up,looseness=2] (0.4,-0.6) -- (0.4,-0.8);
            \projector{0,-0.5}{r};
            \region{0.4,0}{n};
        \end{tikzpicture}
        &= \frac{\Delta_{n-2r+1}}{\Delta_{n-2r}}\
        \begin{tikzpicture}[centerzero]
            \draw[S,<-] (-0.4,-0.8) -- (-0.4,0.8);
            \draw[S,->] (0.4,-0.8) -- (0.4,0.8);
            \projector{0,0}{r};
            \projector{-0.7,0}{r};
            \projector{0.7,0}{r};
            \region{0.8,0.5}{n};
        \end{tikzpicture}
        + \frac{\Delta_{n-2r+1}}{\Delta_{n-2r+2}}\
        \begin{tikzpicture}[centerzero]
            \draw[S,<-] (-0.4,-0.8) -- (-0.4,0.8);
            \draw[S,->] (0.4,-0.8) -- (0.4,0.8);
            \projector{0,0}{r};
            \projector{-0.9,0}{r-1};
            \projector{0.9,0}{r-1};
            \region{0.8,0.5}{n};
        \end{tikzpicture}
        \ ,& r \in \Z,\ 2r \le n+1,
        \\ \label{camel3}
        \begin{tikzpicture}[centerzero]
            \draw[S,<->] (-0.4,0.8) -- (-0.4,0.6) to[out=down,in=left] (0,0.15) to[out=right,in=down] (0.4,0.6) -- (0.4,0.8);
            \projector{0,0.5}{r};
            \draw[S] (-0.4,-0.8) -- (-0.4,-0.6) to[out=up,in=left] (0,-0.15) to[out=right,in=up] (0.4,-0.6) -- (0.4,-0.8);
            \projector{0,-0.5}{r};
            \draw[D] (0,-0.15) -- (0,0.15);
            \region{0.4,0}{n};
        \end{tikzpicture}
        &= \frac{\Delta_{n-2r+1}}{\Delta_{n-2r}}\
        \begin{tikzpicture}[centerzero]
            \draw[S,->] (-0.4,-0.8) -- (-0.4,0.8);
            \draw[S,->] (0.4,-0.8) -- (0.4,0.8);
            \projector{0,0}{r};
            \projector{-0.9,0}{r+1};
            \projector{0.7,0}{r};
            \region{0.8,0.5}{n};
        \end{tikzpicture}
        + \frac{\Delta_{n-2r+1}}{\Delta_{n-2r+2}}\
        \begin{tikzpicture}[centerzero]
            \draw[S,->] (-0.4,-0.8) -- (-0.4,0.8);
            \draw[S,->] (0.4,-0.8) -- (0.4,0.8);
            \projector{0,0}{r};
            \projector{-0.7,0}{r};
            \projector{0.9,0}{r-1};
            \region{0.8,0.5}{n};
        \end{tikzpicture}
        \ ,& r \in \Z,\ 2r \le n+1,
    \end{align}
    where we interpret the first terms on the right-hand sides as zero when $2r=n+1$ (even though the coefficients have a denominator of zero).
\end{prop}

\begin{proof}
    Since all the equalities are trivial when $n<0$, we assume $n \in \N$.  To prove the first equality in \cref{wind}, we compute
    \[
        \rightbubmultreg[S]{\freeprojector{r-1}}{n}
        \overset{\cref{sigma}}{=} \rightbubmultreg[S]{\rightbubmult[D]{r-1}}{n} - \rightbubmultreg[S]{\rightbubmult[D]{r}}{n}
        \overset{\cref{apple2}}{=} \leftbubmultreg[S]{\rightbubmult[D]{r}}{n} - \leftbubmultreg[S]{\rightbubmult[D]{r+1}}{n}
        \overset{\cref{sigma}}{=} \leftbubmultreg[S]{\freeprojector{r}}{n}.
    \]
    For the second equality in \cref{wind}, when $2r \le n$, we compute
    \begin{align*}
        \leftbubmultreg[S]{\freeprojector{r}}{n}
        \ &\overset{\mathclap{\cref{sigma}}}{=}\ \leftbubmultreg[S]{\rightbubmult[D]{r}}{n} - \leftbubmultreg[S]{\rightbubmult[D]{r+1}}{n}
        \\
        &\overset{\mathclap{\cref{leftbull}}}{=}\ \frac{\Delta_{n-2r+3}}{\Delta_{n-2r+2}}\ \rightbubmultreg[D]{r}{n}
        + \frac{\Delta_{n-2r+1}}{\Delta_{n-2r+2}}\ \rightbubmultreg[D]{r-1}{n}
        - \frac{\Delta_{n-2r+1}}{\Delta_{n-2r}}\ \rightbubmultreg[D]{r+1}{n}
        - \frac{\Delta_{n-2r-1}}{\Delta_{n-2r}}\ \rightbubmultreg[D]{r}{n}
        \\
        &\overset{\mathclap{\cref{sigma}}}{=}\ \frac{\Delta_{n-2r+1}}{\Delta_{n-2r}}\ \freeprojectorreg{r}{n}
        + \left( \frac{\Delta_{n-2r+3}}{\Delta_{n-2r+2}} - \frac{\Delta_{n-2r-1}}{\Delta_{n-2r}} - \frac{\Delta_{n-2r+1}}{\Delta_{n-2r}}\right)\ \rightbubmultreg[D]{r}{n}
        + \frac{\Delta_{n-2r+1}}{\Delta_{n-2r+2}}\ \rightbubmultreg[D]{r-1}{n}
        \\
        &\overset{\mathclap{\cref{Delta}}}{=}\ \frac{\Delta_{n-2r+1}}{\Delta_{n-2r}}\ \freeprojectorreg{r}{n}
        + \left( \frac{\Delta_{n-2r+3}}{\Delta_{n-2r+2}} - \delta \right)\ \rightbubmultreg[D]{r}{n}
        + \frac{\Delta_{n-2r+1}}{\Delta_{n-2r+2}}\ \rightbubmultreg[D]{r-1}{n}
        \\
        &\overset{\mathclap{\cref{Delta}}}{=}\ \frac{\Delta_{n-2r+1}}{\Delta_{n-2r}}\ \freeprojectorreg{r}{n}
        - \frac{\Delta_{n-2r+1}}{\Delta_{n-2r+2}}\ \rightbubmultreg[D]{r}{n}
        + \frac{\Delta_{n-2r+1}}{\Delta_{n-2r+2}}\ \rightbubmultreg[D]{r-1}{n}
        \\
        &\overset{\mathclap{\cref{sigma}}}{=}\ \frac{\Delta_{n-2r+1}}{\Delta_{n-2r}}\ \freeprojectorreg{r}{n}
        + \frac{\Delta_{n-2r+1}}{\Delta_{n-2r+2}}\ \freeprojectorreg{r-1}{n}.
    \end{align*}
    \details{
        We need the condition $2r \le n$ above since we use \cref{leftbull} with $r$ replaced by $r+1$.
    }
    When $2r = n+1$, we have
    \[
        \leftbubmultreg[S]{\rightbubmult[D]{r+1}}{n}
        = \rightbubmultreg[D]{r+1}{n}
        = \rightbubmultreg[D]{r}{n}
        = \freeprojectorreg{r}{n}
        = 0
        \quad \text{and} \quad
        \freeprojectorreg{r-1}{n}
        = \rightbubmultreg[D]{r-1}{n} - \rightbubmultreg[D]{r}{n}
        = \rightbubmultreg[D]{r-1}{n}.
    \]
    Thus,
    \[
        \leftbubmultreg[S]{\freeprojector{r}}{n}
        \overset{\mathclap{\cref{sigma}}}{=}\ \leftbubmultreg[S]{\rightbubmult[D]{r}}{n}
        \overset{\mathclap{\cref{leftbull}}}{=}\ \frac{\Delta_{n-2r+1}}{\Delta_{n-2r+2}}\ \rightbubmultreg[D]{r-1}{n}
        \overset{\mathclap{\cref{sigma}}}{=}\ \frac{\Delta_{n-2r+1}}{\Delta_{n-2r+2}}\ \freeprojectorreg{r-1}{n}.
    \]
    
    Next, we prove \cref{camel1} by induction on $r$.  It is trivial when $r \le 0$, and the $r=1$ case is \cref{wood}.  Now assume that \cref{camel1} holds for some $r \ge 1$.  We apply the following operations to both sides of \cref{camel1}:
    \begin{itemize}
        \item compose on the left with $\upstrandreg[D]{n}$,
        \item compose on the right with $\downstrandreg[D]{n+2}$,
        \item compose on top with
            \(
                \begin{tikzpicture}[anchorbase]
                    \draw[D,->] (-0.5,0) to[out=up,in=up,looseness=1] (0.5,0);
                    \draw[S,->] (-0.2,0) -- (-0.2,0.5);
                    \draw[S,<-] (0.2,0) -- (0.2,0.5);
                    \region{0.7,0.3}{n+2};
                \end{tikzpicture}
            \),
        \item compose on bottom with
            \(
                \begin{tikzpicture}[anchorbase]
                    \draw[D,<-] (-0.5,0.5) to[out=down,in=down,looseness=1] (0.5,0.5);
                    \draw[S,->] (-0.2,0) -- (-0.2,0.5);
                    \draw[S,<-] (0.2,0) -- (0.2,0.5);
                    \region{0.7,0.2}{n+2};
                \end{tikzpicture}
            \).
    \end{itemize}
    On the left-hand side, we obtain
    \[
        \begin{tikzpicture}[centerzero]
            \draw[S,<-] (-0.5,1.1) -- (-0.5,0.6) to[out=down,in=down,looseness=1.7] (0.5,0.6) -- (0.5,1.1);
            \projector{0,0.5}{r-1};
            \draw[S,->] (-0.5,-1.1) -- (-0.5,-0.6) to[out=up,in=up,looseness=1.7] (0.5,-0.6) -- (0.5,-1.1);
            \projector{0,-0.5}{r-1};
            \region{0.4,0}{n};
            \draw[D,->] (0.7,0) -- (0.7,-0.5) to[out=down,in=right] (0,-0.9) to[out=left,in=down] (-0.7,-0.5) -- (-0.7,0.5) to[out=up,in=left] (0,0.9) to[out=right,in=up] (0.7,0.5) -- (0.7,0);
        \end{tikzpicture}
        \overset{\cref{smush}}{\underset{\cref{reidUD}}{=}}
        \begin{tikzpicture}[centerzero]
            \draw[S,<-] (-0.45,1.3) -- (-0.45,0.6) to[out=down,in=down,looseness=1.5] (0.45,0.6) -- (0.45,1.3);
            \projector{0,0.6}{r};
            \draw[S,->] (-0.45,-1.3) -- (-0.45,-0.6) to[out=up,in=up,looseness=1.5] (0.45,-0.6) -- (0.45,-1.3);
            \projector{0,-0.6}{1};
            \region{0.5,0}{n+2};
            \bubright[D]{0.8,0.5};
            \bubright[D]{0,1.1};
            \bubright[D]{0,-1.1};
        \end{tikzpicture}
        \overset{\cref{sigma}}{\underset{\cref{headroom}}{=}}
        \begin{tikzpicture}[centerzero]
            \draw[S,<-] (-0.4,1.1) -- (-0.4,0.6) to[out=down,in=down,looseness=1.5] (0.4,0.6) -- (0.4,1.1);
            \projector{0,0.6}{r};
            \draw[S,->] (-0.4,-1.1) -- (-0.4,-0.6) to[out=up,in=up,looseness=1.5] (0.4,-0.6) -- (0.4,-1.1);
            \projector{0,-0.6}{r};
            \region{0.5,0}{n+2};
            \bubright[D]{0.8,0.5};
        \end{tikzpicture}
        \overset{\cref{sigma}}{\underset{\cref{puddle}}{=}}
        \begin{tikzpicture}[centerzero]
            \draw[S,<-] (-0.4,1.1) -- (-0.4,0.6) to[out=down,in=down,looseness=1.5] (0.4,0.6) -- (0.4,1.1);
            \projector{0,0.6}{r};
            \draw[S,->] (-0.4,-1.1) -- (-0.4,-0.6) to[out=up,in=up,looseness=1.5] (0.4,-0.6) -- (0.4,-1.1);
            \projector{0,-0.6}{r};
            \region{0.5,0}{n+2};
        \end{tikzpicture}
        .
    \]
    On the right-hand side, we obtain
    \begin{multline*}
        \frac{\Delta_{n-2r+1}}{\Delta_{n-2r}}\
        \begin{tikzpicture}[centerzero]
            \draw[S,->] (-0.5,-0.8) -- (-0.5,0.8);
            \draw[S,<-] (0.5,-0.8) -- (0.5,0.8);
            \projector{0,0}{r-1};
            \projector{-0.9,0}{r};
            \projector{0.9,0}{r};
            \region{1.1,0.65}{n+2};
            \draw[D,->] (1.4,0) to[out=down,in=right,looseness=0.8] (0,-0.5) to[out=left,in=down,looseness=0.8] (-1.4,0) to[out=up,in=left,looseness=0.8] (0,0.5) to[out=right,in=up,looseness=0.8] (1.4,0);
        \end{tikzpicture}
        +
        \frac{\Delta_{n-2r+1}}{\Delta_{n-2r+2}}\
        \begin{tikzpicture}[centerzero]
            \draw[S,->] (-0.5,-0.8) -- (-0.5,0.8);
            \draw[S,<-] (0.5,-0.8) -- (0.5,0.8);
            \projector{0,0}{r-1};
            \projector{-1,0}{r-1};
            \projector{1,0}{r-1};
            \region{1.1,0.65}{n+2};
            \draw[D,->] (1.6,0) to[out=down,in=right,looseness=0.8] (0,-0.5) to[out=left,in=down,looseness=0.8] (-1.6,0) to[out=up,in=left,looseness=0.8] (0,0.5) to[out=right,in=up,looseness=0.8] (1.6,0);
        \end{tikzpicture}
        \\
        \overset{\substack{\cref{smush} \\ \cref{bluewind}}}{\underset{\cref{reidUU}}{=}}
        \frac{\Delta_{n-2r+1}}{\Delta_{n-2r}}\
        \begin{tikzpicture}[centerzero]
            \draw[S,<-] (-0.4,-0.8) -- (-0.4,0.8);
            \draw[S,->] (0.4,-0.8) -- (0.4,0.8);
            \projector{0,0}{r};
            \projector{-0.9,0}{r+1};
            \projector{0.9,0}{r+1};
            \region{0.9,0.6}{n+2};
        \end{tikzpicture}
        + \frac{\Delta_{n-2r+1}}{\Delta_{n-2r+2}}\
        \begin{tikzpicture}[centerzero]
            \draw[S,<-] (-0.4,-0.8) -- (-0.4,0.8);
            \draw[S,->] (0.4,-0.8) -- (0.4,0.8);
            \projector{0,0}{r};
            \projector{-0.8,0}{r};
            \projector{0.8,0}{r};
            \region{0.9,0.6}{n+2};
        \end{tikzpicture}
        ,
    \end{multline*}
    completing the proof of the induction step.
    
    Next, we prove \cref{camel2}.  For $r \ge 1$, we apply the following operations to both sides of \cref{camel1}:
    \begin{itemize}
        \item compose on top with 
            \(
                \begin{tikzpicture}[centerzero]
                    \draw[S] (-0.2,-0.3) -- (-0.2,0.3);
                    \draw[S,<->] (0.2,-0.3) -- (0.2,0.3);
                    \draw[D] (-0.2,0) -- (0.2,0);
                \end{tikzpicture}
            \)

        \item compose on bottom with 
            \(
                \begin{tikzpicture}[centerzero]
                    \draw[S,<->] (-0.2,-0.3) -- (-0.2,0.3);
                    \draw[S] (0.2,-0.3) -- (0.2,0.3);
                    \draw[D] (-0.2,0) -- (0.2,0);
                \end{tikzpicture}
            \)
    \end{itemize}
    On the left-hand side, we obtain
    \[
        \begin{tikzpicture}[centerzero]
            \draw[S,->] (-0.5,1) -- (-0.5,0.6) to[out=down,in=down,looseness=1.7] (0.5,0.6) -- (0.5,1);
            \projector{0,0.5}{r-1};
            \draw[S,<-] (-0.5,-1) -- (-0.5,-0.6) to[out=up,in=up,looseness=1.7] (0.5,-0.6) -- (0.5,-1);
            \projector{0,-0.5}{r-1};
            \draw[D] (-0.5,0.8) -- (0.5,0.8);
            \draw[D] (-0.5,-0.8) -- (0.5,-0.8);
            \region{0.4,0}{n};
        \end{tikzpicture}
        \overset{\cref{smush}}{\underset{\cref{water}}{=}}
        \begin{tikzpicture}[centerzero]
            \draw[S,->] (-0.4,1) -- (-0.4,0.6) to[out=down,in=down,looseness=2] (0.4,0.6) -- (0.4,1);
            \projector{0,0.6}{r};
            \draw[S,<-] (-0.4,-1) -- (-0.4,-0.6) to[out=up,in=up,looseness=2] (0.4,-0.6) -- (0.4,-1);
            \projector{0,-0.6}{r};
            \region{0.4,0}{n};
        \end{tikzpicture}
        \ .
    \]
    On the right-hand side, we obtain
    \begin{multline*}
        \frac{\Delta_{n-2r+1}}{\Delta_{n-2r}}\
        \begin{tikzpicture}[centerzero]
            \draw[S,<-] (-0.5,-0.8) -- (-0.5,0.8);
            \draw[S,->] (0.5,-0.8) -- (0.5,0.8);
            \draw[D] (-0.5,0.5) -- (0.5,0.5);
            \draw[D] (-0.5,-0.5) -- (0.5,-0.5);
            \projector{0,0}{r-1};
            \projector{-0.9,0}{r};
            \projector{0.9,0}{r};
            \region{0.8,0.5}{n};
        \end{tikzpicture}
        + \frac{\Delta_{n-2r+1}}{\Delta_{n-2r+2}}\
        \begin{tikzpicture}[centerzero]
            \draw[S,<-] (-0.5,-0.8) -- (-0.5,0.8);
            \draw[S,->] (0.5,-0.8) -- (0.5,0.8);
            \draw[D] (-0.5,0.5) -- (0.5,0.5);
            \draw[D] (-0.5,-0.5) -- (0.5,-0.5);
            \projector{0,0}{r-1};
            \projector{-1,0}{r-1};
            \projector{1,0}{r-1};
            \region{0.9,0.5}{n};
        \end{tikzpicture}
        \\
        \overset{\cref{fire}}{\underset{\cref{bluewind}}{=}}
        \frac{\Delta_{n-2r+1}}{\Delta_{n-2r}}\
        \begin{tikzpicture}[centerzero]
            \draw[S,<-] (-0.4,-0.8) -- (-0.4,0.8);
            \draw[S,->] (0.4,-0.8) -- (0.4,0.8);
            \projector{0,0}{r};
            \projector{-0.8,0}{r};
            \projector{0.8,0}{r};
            \region{0.8,0.5}{n};
        \end{tikzpicture}
        + \frac{\Delta_{n-2r+1}}{\Delta_{n-2r+2}}\
        \begin{tikzpicture}[centerzero]
            \draw[S,<-] (-0.4,-0.8) -- (-0.4,0.8);
            \draw[S,->] (0.4,-0.8) -- (0.4,0.8);
            \projector{0,0}{r};
            \projector{-1,0}{r-1};
            \projector{1,0}{r-1};
            \region{0.8,0.5}{n};
        \end{tikzpicture}
        \ ,
    \end{multline*}
    proving \cref{camel2}.  (Our assumption $r \ge 1$ was needed when we applied \cref{bluewind} above.)  This proves \cref{camel2} for $r \ge 1$.  Since it is trivial for $r < 0$ by \cref{abyss}, it remains to prove it for $r=0$.  For this, we apply the following operations to both sides of the $r=1$ case of \cref{camel2}, with $n$ replaced by $n+2$:
    \begin{itemize}
        \item compose on the left with $\downstrandreg[D]{n+2}$,
        \item compose on the right with $\upstrandreg[D]{n}$,
        \item compose on top with
            \(
                \begin{tikzpicture}[anchorbase]
                    \draw[D,<-] (-0.5,0) to[out=up,in=up,looseness=1] (0.5,0);
                    \draw[S,<-] (-0.2,0) -- (-0.2,0.5);
                    \draw[S,->] (0.2,0) -- (0.2,0.5);
                    \region{0.7,0.3}{n};
                \end{tikzpicture}
            \),
        \item compose on bottom with
            \(
                \begin{tikzpicture}[anchorbase]
                    \draw[D,->] (-0.5,0.5) to[out=down,in=down,looseness=1] (0.5,0.5);
                    \draw[S,<-] (-0.2,0) -- (-0.2,0.5);
                    \draw[S,->] (0.2,0) -- (0.2,0.5);
                    \region{0.7,0.2}{n};
                \end{tikzpicture}
            \).
    \end{itemize}
    On the left-hand side, we obtain
    \[
        \begin{tikzpicture}[centerzero]
            \draw[S,->] (-0.4,1.1) -- (-0.4,0.6) to[out=down,in=down,looseness=2] (0.4,0.6) -- (0.4,1.1);
            \projector{0,0.5}{1};
            \draw[S,<-] (-0.4,-1.1) -- (-0.4,-0.6) to[out=up,in=up,looseness=2] (0.4,-0.6) -- (0.4,-1.1);
            \projector{0,-0.5}{1};
            \region{0.2,0}{n+2};
            \draw[D,<-] (0.7,0) -- (0.7,-0.5) to[out=down,in=right] (0,-0.9) to[out=left,in=down] (-0.7,-0.5) -- (-0.7,0.5) to[out=up,in=left] (0,0.9) to[out=right,in=up] (0.7,0.5) -- (0.7,0);
        \end{tikzpicture}
        \overset{\cref{smush}}{\underset{\cref{reidUD}}{=}}
        \begin{tikzpicture}[centerzero]
            \draw[S,->] (-0.4,1.1) -- (-0.4,0.6) to[out=down,in=down,looseness=2] (0.4,0.6) -- (0.4,1.1);
            \projector{0,0.5}{0};
            \draw[S,<-] (-0.4,-1.1) -- (-0.4,-0.6) to[out=up,in=up,looseness=2] (0.4,-0.6) -- (0.4,-1.1);
            \projector{0,-0.5}{0};
            \region{0.5,0}{n};
            \bubleft[D]{0.8,0.5};
        \end{tikzpicture}
        \overset{\cref{smoke}}{=}
        \begin{tikzpicture}[centerzero]
            \draw[S,->] (-0.4,1.1) -- (-0.4,0.6) to[out=down,in=down,looseness=2] (0.4,0.6) -- (0.4,1.1);
            \projector{0,0.5}{0};
            \draw[S,<-] (-0.4,-1.1) -- (-0.4,-0.6) to[out=up,in=up,looseness=2] (0.4,-0.6) -- (0.4,-1.1);
            \projector{0,-0.5}{0};
            \region{0.5,0}{n};
        \end{tikzpicture}
        .
    \]
    On the right-hand side, we obtain
    \[
        \frac{\Delta_{n+1}}{\Delta_{n}}\
        \begin{tikzpicture}[centerzero]
            \draw[S,<-] (-0.4,-0.8) -- (-0.4,0.8);
            \draw[S,->] (0.4,-0.8) -- (0.4,0.8);
            \projector{0,0}{1};
            \projector{-0.7,0}{1};
            \projector{0.7,0}{1};
            \region{0.9,0.6}{n};
            \draw[D,<-] (1.1,0) to[out=down,in=right] (0,-0.5) to[out=left,in=down] (-1.1,0) to[out=up,in=left] (0,0.5) to[out=right,in=up] (1.1,0);
        \end{tikzpicture}
        + \frac{\Delta_{n+1}}{\Delta_{n+2}}\
        \begin{tikzpicture}[centerzero]
            \draw[S,<-] (-0.4,-0.8) -- (-0.4,0.8);
            \draw[S,->] (0.4,-0.8) -- (0.4,0.8);
            \projector{0,0}{1};
            \projector{-0.7,0}{0};
            \projector{0.7,0}{0};
            \region{0.9,0.6}{n};
            \draw[D,<-] (1.1,0) to[out=down,in=right] (0,-0.5) to[out=left,in=down] (-1.1,0) to[out=up,in=left] (0,0.5) to[out=right,in=up] (1.1,0);
        \end{tikzpicture}
        \\
        \overset{\substack{\cref{smush} \\ \cref{bluewind}}}{\underset{\substack{\cref{reidUD} \\ \cref{abyss}}}{=}} \frac{\Delta_{n+1}}{\Delta_{n}}\
        \begin{tikzpicture}[centerzero]
            \draw[S,<-] (-0.4,-0.8) -- (-0.4,0.8);
            \draw[S,->] (0.4,-0.8) -- (0.4,0.8);
            \projector{0,0}{0};
            \projector{-0.7,0}{0};
            \projector{0.7,0}{0};
            \region{0.9,0.6}{n};
        \end{tikzpicture}
        \ ,
    \]
    giving the $r=0$ case of \cref{camel2}.
    
    To prove \cref{camel3}, we apply the following operations to both sides of \cref{camel1} with $r$ and $n$ replaced by $r+1$ and $n+2$, respectively:
    \begin{itemize}
        \item compose on the right with $\upstrandreg[D]{n}$,
        \item compose on top with
            \(
                \begin{tikzpicture}[anchorbase]
                    \draw[D] (0.3,-0.3) -- (0,0);
                    \draw[S,<-] (-0.3,-0.3) -- (0,0);
                    \draw[S,->] (0,0) -- (0,0.4);
                    \draw[S,->] (-0.5,-0.3) -- (-0.5,0.4);
                    \region{0.5,0.05}{n};
                \end{tikzpicture}
            \)
        \item compose on bottom with
            \(
                \begin{tikzpicture}[anchorbase]
                    \draw[S] (0,-0.4) -- (0,0);
                    \draw[D,->] (0,0) -- (0.3,0.3);
                    \draw[S] (0,0) -- (-0.3,0.3);
                    \draw[S,->] (-0.5,-0.4) -- (-0.5,0.3);
                    \region{0.5,-0.05}{n};
                \end{tikzpicture}
            \)
    \end{itemize}
    On the left-hand side, we obtain precisely the left-hand side of \cref{camel3}.  On the right-hand side, we obtain
    \[
        \frac{\Delta_{n-2r+1}}{\Delta_{n-2r}}\
        \begin{tikzpicture}[centerzero]
            \draw[S,->] (-0.35,-0.8) -- (-0.35,0.8);
            \draw[S,->] (0.9,-0.8) -- (0.9,-0.5) to[out=left,in=down] (0.35,0) to[out=up,in=left] (0.9,0.5) -- (0.9,0.8);
            \draw[D] (0.9,-0.5) to[out=right,in=down] (1.45,0) to[out=up,in=right] (0.9,0.5);
            \projector{0,0}{r};
            \projector{-0.85,0}{r+1};
            \projector{0.9,0}{r+1};
            \region{1.5,0.5}{n};
        \end{tikzpicture}
        + \frac{\Delta_{n-2r+1}}{\Delta_{n-2r+2}}\
        \begin{tikzpicture}[centerzero]
            \draw[S,->] (-0.35,-0.8) -- (-0.35,0.8);
            \draw[S,->] (0.8,-0.8) -- (0.8,-0.4) to[out=left,in=down] (0.4,0) to[out=up,in=left] (0.8,0.4) -- (0.8,0.8);
            \draw[D] (0.8,-0.4) to[out=right,in=down] (1.2,0) to[out=up,in=right] (0.8,0.4);
            \projector{0,0}{r};
            \projector{-0.7,0}{r};
            \projector{0.8,0}{r};
            \region{1.3,0.5}{n};
        \end{tikzpicture}
        .
    \]
    By \cref{smush,water}, this is equal to the right-hand side of \cref{camel3}.
\end{proof}

\begin{rem}
    \Cref{dune} shows that one can recover the relations from \cite[\S4.1]{HS25} from the simpler relations of \cref{subsec:2cat}.  In particular, \cite[(23)]{HS25} and \cite[(24)]{HS25} correspond to \cref{camel2} and \cref{camel3}, respectively.
\end{rem}

%-----------------------------------------------------------------
\subsection{The diagrammatic monoidal category\label{subsec:Dmon}}
%-----------------------------------------------------------------

We now form a $\kk$-linear strict monoidal category $\Dmon'$ whose objects encode 1-endomorphisms of the formal sum $\bigoplus_{n \in \N} \obj{n}$.  Precisely, the objects of $\Dmon'$ are matrices
\[
    f = (f_{m,n})_{m,n \in \N},\qquad f_{m,n} \in \Dcat(\obj{n},\obj{m}),
\]
that are both row finite and column finite.  We will often use summation notation for these matrices:
\[
    f = (f_{m,n})_{m,n \in \N} = \sum_{m,n \in \N} f_{m,n}.
\]
Tensor product is given by matrix multiplication:
\[
    (f \otimes g)_{m,n} = \sum_{k \in \N} f_{m,k} g_{k,n},
\]
where the right-hand side is composition in $\Dcat$.  Morphisms in $\Dmon'$ are matrices of 2-morphisms in $\Dcat$
\[
    \alpha = (\alpha_{m,n})_{m,n \in \N} \colon f \to g,\qquad
    \alpha_{m,n} \colon f_{m,n} \Rightarrow g_{m,n}.
\]
Composition is componentwise: for $\alpha \colon f \to g$ and $\beta \colon g \to h$,
\[
    \beta\alpha \colon f \to h,\qquad
    (\beta \alpha)_{m,n} = \beta_{m,n} \alpha_{m,n},
\]
where, on the right-hand side, juxtaposition is vertical composition of 2-morphisms in $\Dcat$.  Tensor product of morphisms in $\Dmon'$ is given by matrix multiplication:
\[
    (\alpha \otimes \beta)_{m,n} = \sum_{k \in \N} \alpha_{m,k} \otimes \beta_{k,n},
\]
where, on the right-hand side, $\otimes$ denotes horizontal composition of 2-morphisms in $\Dcat$.  The unit object of $\Dmon'$ is the identity matrix
\[
    \one = (\one_{m,n})_{m,n \in \N},\qquad \one_{m,n} = \delta_{m,n} 1_n.
\]

\begin{defin} \label{Dmondef}
    Let $\Dmon$ be the full $\kk$-linear strict monoidal subcategory of $\Dmon'$ generated by the objects
    \[
        \Sup := \sum_{n=0}^\infty \Sup 1_n,\qquad
        \Sdown := \sum_{n=1}^\infty \Sdown 1_n,\qquad
        \Dup := \sum_{n=0}^\infty \Dup 1_n,\qquad
        \Ddown := \sum_{n=2}^\infty \Ddown 1_n.
    \]
\end{defin}

A string diagram in $\Dcat$ with the region labels removed defines an element of $\Dmon$ by summing over all labels of the rightmost region.  For example
\[
    \mergeup \in \Dmon(\Sup \Sup, \Dup)
    \quad \text{is defined to be} \quad
    \sum_{n \in \N} \mergeupreg{n}.
\]

The 2-functor $\bR$ from \cref{rotate} induces an isomorphism of $\kk$-linear monoidal categories
\begin{equation} \label{rotatemon}
    \bR \colon \Dmon \to \Dmon^{\rev,\op}.
\end{equation}

\begin{rem}
    The Heisenberg categories of \cite{Kho14,LS13,MS18,Bru18,BSW20-qheis} are all directly defined as monoidal categories.  In contrast, the monoidal category $\Dmon$ is defined via the 2-category $\Dcat$.  The reason for this is the dependence of the defining relations \cref{coal,wood} on $n$.  The authors of \cite{HS25} take a different approach, defining a monoidal category with extra generators taking the place of the coefficients appearing in \cref{coal,wood}.  This leads to some complications, including difficulty in determining bases for morphism spaces; see \cref{HSbasis}.
\end{rem}

%======================================
\section{Isomorphisms\label{sec:isoms}}
%======================================

In this section we deduce some isomorphisms that hold in the 2-category $\Dcat$ and the monoidal category $\Dmon$.  These will play an important role in our discussion of Grothendieck rings in \cref{sec:decat}.

For a monoidal category $\cC$, let $\Add(\cC)$ denote its additive envelope, and let $\Kar(\cC)$ denote its additive Karoubi envelope (the idempotent completion of $\Add(\cC)$).  Similarly, for a 2-category $\fC$, let $\Add(\fC)$ denote its additive envelope, and let $\Kar(\fC)$ denote the additive Karoubi envelope.  By definition, these have the same objects as $\fC$, and have morphism categories given by the additive envelopes and additive Karoubi envelopes, respectively, of the morphism categories of $\fC$.  In particular, this means that we have a zero 1-morphism between any two objects.

In $\Kar(\Dcat)$, we let $0_n$, $n \in \N$, denote the zero 1-endomorphism in $\Kar(\Dcat)$ of $\obj{n}$.  We adopt the convention that any 1-morphism expression involving $1_n$, $n<0$, is the zero 1-morphism.  We have the corresponding zero object $\zero$ in $\Kar(\Dmon)$, which can be identified with the matrix whose components are all zero $1$-morphisms in $\Dcat$.

If $f = (f_{mn})_{m,n \in \N}$ is a morphism in $\Dmon$, we will identify $f 1_n$, $n \in \N$, with the sum $\sum_{m \in \N} f_{mn}$, which is finite since $f$ is row finite.  In this way, we also view $f 1_n$ as a 1-morphism in $\Add(\Dcat)$.  To prove that two objects $f$ and $g$ in $\Dmon$ are isomorphic, it suffices to show that $f 1_n \cong g 1_n$ in $\Add(\Dcat)$ for all $n \in \N$.  This will be used repeatedly without mention below.

\begin{prop}
    The following isomorphisms hold in $\Dmon$:
    \begin{align} \label{isoms1}
        \Sup \Dup &\cong \Dup \Sup,&
        \Sdown \Ddown &\cong \Ddown \Sdown,
        \\ \label{isoms2}
        \Sdown \Dup &\cong \Sup,&
        \Ddown \Sup &\cong \Sdown.
    \end{align}
\end{prop}

\begin{proof}
    The morphisms
    \[
        \upcrossreg{S}{D}{n} \colon \Sup \Dup 1_n \to \Dup \Sup 1_n
        \qquad \text{and} \qquad
        \upcrossreg{D}{S}{n} \colon \Dup \Sup 1_n \to \Sup \Dup 1_n
    \]
    are mutually inverse by \cref{reidUU}.  The second isomorphism in \cref{isoms1} then follows from an application of the functor $\bR$ of \cref{rotate}.

    The morphisms
    \[
        \splitNWreg{n} \colon \Sdown \Dup 1_n \to \Sup 1_n
        \qquad \text{and} \qquad
        \mergeNEreg{n} \colon \Sup \Sdown 1_n \to \Dup 1_n
    \]
    are mutually inverse by the second relation in \cref{fire} and the second relation in \cref{water}.  The second isomorphism in \cref{isoms2} then follows from an application of the functor $\bR$ of \cref{rotate}.
\end{proof}

For $r,n \in \Z$ define the following 1-morphisms in $\Kar(\Dcat)$:
\[
    \Theta_r 1_n := \left( 1_n, \rightbubmultreg[D]{r}{n} \right),\qquad
    \Sigma_r 1_n := \left( 1_n, \freeprojectorreg{r}{n} \right).
\]
We then have the corresponding objects in $\Kar(\Dmon)$:
\[
    \Theta_r := \sum_{n \in \N} \Theta_r 1_n,\qquad
    \Sigma_r := \sum_{n \in \N} \Sigma_r 1_n.
\]
Note that, for $r \in \Z$, $n \in \N$,
\[
    \bR(\Theta_r) = \Theta_r,\qquad
    \bR(\Theta_r 1_n) = \Theta_r 1_n, \qquad
    \bR(\Sigma_r) = \Sigma_r,\qquad
    \bR(\Sigma_r 1_n) = \Sigma_r 1_n.
\]

\begin{lem}
    The following isomorphisms hold in $\Kar(\Dmon)$:
    \begin{align} \label{wasp1}
        \Theta_r &\cong \one,
        & r \le 0,
        \\ \label{wasp2}
        \Theta_r \Theta_s &\cong \Theta_{\max(r,s)},
        & r,s \in \Z,
        \\ \label{wasp3}
        \Sigma_r \Sigma_s &\cong \delta_{r,s} \Sigma_r
        & r,s, \in \Z,
        \\ \label{wasp4}
        \Sigma_r \Theta_s \cong \Theta_s \Sigma_r &\cong
        \begin{cases}
            \Sigma_r & \text{if } s \le r, \\
            \zero & \text{if } s > r,
        \end{cases}
        & r,s \in \Z,
        \\ \label{disect}
        \Theta_r &\cong \Sigma_r \oplus \Theta_{r+1},
        & r \in \Z.
    \end{align}
\end{lem}

\begin{proof}
    The isomorphism~\cref{wasp1} holds by \cref{bubrec}.  The isomorphism~\cref{wasp2} follows from \cref{headroom}.  The isomorphism~\cref{wasp3} holds by \cref{absorb}.  The isomorphism~\cref{wasp4} holds by \cref{sigma,headroom}.  Finally, the isomorphism~\cref{disect} follows from \cref{sigma}.
\end{proof}

\begin{lem}
    The following isomorphisms holds in $\Kar(\Dcat)$ for all $n \in \N$:
    \begin{align} \label{hornet1}
        \Theta_r 1_n & \cong 0_n,
        & r \in \Z,\ 2r > n,
        \\ \label{hornet2}
        \Sigma_r 1_n &\cong 0_n,
        & r \in \Z,\ \big( r < 0 \text{ or } 2r > n \big).
    \end{align}
\end{lem}

\begin{proof}
    Relation~\cref{hornet1} follows from the fact that $\rightbubmultreg[D]{r}{n}$ contains a negatively-labelled region when $2r>n$.  Relation~\cref{hornet2} holds by \cref{abyss}.
\end{proof}

\begin{prop}
    The following isomorphisms hold in $\Kar(\Dmon)$:
    \begin{align} \label{frisbee1}
        \Theta_{r+1} \Dup &\cong \Dup \Theta_r,&
        \Theta_r \Ddown &\cong \Ddown \Theta_{r+1},&
        r \in \Z,
        \\ \label{frisbee2}
        \Sigma_{r+1} \Dup &\cong \Dup \Sigma_r,&
        \Sigma_r \Ddown &\cong \Ddown \Sigma_{r+1},&
        r \in \Z,
        \\ \label{hacky}
        \Ddown[r] \Dup[r] &\cong \one,&
        \Dup[r] \Ddown[r] &\cong \Theta_r,&
        r \in \N.
    \end{align}
\end{prop}

\begin{proof}
    The isomorphisms \cref{frisbee1} follow from \cref{crisp1}.  The isomorphisms \cref{frisbee2} follow from \cref{smush}.
    
    The first isomorphism in \cref{hacky} follows from the fact that the morphisms
    \[
        \leftcapreg[D]{n} \colon \Ddown \Dup 1_n \to 1_n
        \qquad \text{and} \qquad
        \rightcupreg[D]{n} \colon 1_n \to \Ddown \Dup 1_n
    \]
    are mutually inverse by the last relation in \cref{fire} and the second relation in \cref{smoke}.  
    
    We prove the second isomorphism in \cref{hacky} by induction on $r$.  The isomorphism
    \[
        \Dup[r] \Ddown[r] 1_n \cong \Theta_r 1_n
    \]
    is trivial for $n < 2r$ and for $r=0$.  Thus, we assume that $n \ge 2r$ and $r \ge 1$.  The base case $r=1$ follows from the fact that the morphisms
    \[
        \rightcapreg[D]{n}
        \overset{\cref{smarty}}{=}
        \begin{tikzpicture}[centerzero]
            \draw[D,->] (-0.25,-0.3) -- (-0.25,0) to[out=up,in=up,looseness=2] (0.25,0) -- (0.25,-0.3);
            \bubright[D]{-0.7,0};
            \region{0.6,0}{n};
        \end{tikzpicture}
        \colon \Dup \Ddown 1_n \to \Theta_1 1_n
    \]
    and
    \[
        \leftcupreg[D]{n}
        \overset{\cref{smarty}}{=}
        \begin{tikzpicture}[centerzero]
            \draw[D,<-] (-0.25,0.3) -- (-0.25,0) to[out=down,in=down,looseness=2] (0.25,0) -- (0.25,0.3);
            \bubright[D]{-0.7,0};
            \region{0.6,0}{n};
        \end{tikzpicture}
        \colon \Theta_1 1_n \to \Dup \Ddown 1_n
    \]
    are mutually inverse by \cref{bubrec,headroom}.  Then, assuming the result for some $r \ge 1$, we have
    \[
        \Dup[r+1] \Ddown[r+1] 1_n
        \cong \Dup \Theta_r \Ddown 1_n
        \overset{\cref{frisbee1}}{\cong} \Theta_{r+1} \Dup \Ddown 1_n
        \overset{\cref{hacky}}{\cong} \Theta_{r+1} \Theta_1 1_n
        \overset{\cref{wasp3}}{\cong} \Theta_{r+1} 1_n.
        \qedhere
    \]
\end{proof}

\begin{prop} \label{zebra}
    The following isomorphisms hold in $\Kar(\Dmon)$:
    \begin{align} \label{zebra1}
        \Sigma_r \Sdown \Sigma_r \Sup \Sigma_r &\cong \Sigma_r,&
        r \in \Z,
        \\ \label{zebra2}
        \Sigma_r \Sup \Sigma_{r-1} \Sdown \Sigma_r &\cong \Sigma_r,&
        r \in \Z,
        \\ \label{zebra3}
        \Sigma_{r+1} \Sup \Sigma_r \Sup \Sigma_r
        &\cong \Dup \Sigma_r,&
        r \in \Z,
        \\ \label{zebra4}
        \Sigma_r \Sdown \Sigma_r \Sdown \Sigma_{r+1} &\cong \Sigma_r \Ddown,&
        r \in \Z.
    \end{align}
\end{prop}

\begin{proof}
    After adding $1_n$ on the right of all the expressions, the isomorphisms are trivial when $2r>n$.  Thus, we assume that $2r \le n$.  By \cref{wind,camel2,absorb}, we have
    \[
        \freeprojector{r}\ \leftbubmultreg[S]{\freeprojector{r}}{n}
        = \frac{\Delta_{n-2r+1}}{\Delta_{n-2r}}\ \freeprojectorreg{r}{n}
        \quad \text{and} \quad
        \begin{tikzpicture}[centerzero]
            \draw[S,->] (-0.4,0.8) -- (-0.4,0.6) to[out=down,in=down,looseness=2] (0.4,0.6) -- (0.4,0.8);
            \projector{0,0.5}{r};
            \draw[S,<-] (-0.4,-0.8) -- (-0.4,-0.6) to[out=up,in=up,looseness=2] (0.4,-0.6) -- (0.4,-0.8);
            \projector{0,-0.5}{r};
            \region{0.4,0}{n};
            \projector{-0.7,0}{r};
        \end{tikzpicture}
        = \frac{\Delta_{n-2r+1}}{\Delta_{n-2r}}\
        \begin{tikzpicture}[centerzero]
            \draw[S,<-] (-0.4,-0.8) -- (-0.4,0.8);
            \draw[S,->] (0.4,-0.8) -- (0.4,0.8);
            \projector{0,0}{r};
            \projector{-0.7,0}{r};
            \projector{0.7,0}{r};
            \region{0.8,0.5}{n};
        \end{tikzpicture}
        \ .
    \]
    It follows that
    \[
        \begin{tikzpicture}[anchorbase]
            \draw[S,->] (-0.4,0.8) -- (-0.4,0.6) to[out=down,in=down,looseness=2] (0.4,0.6) -- (0.4,0.8);
            \projector{0,0.5}{r};
            \region{0.6,0.5}{n};
            \projector{-0.7,0.5}{r};
        \end{tikzpicture}
        \colon \Sigma_r 1_n \to \Sigma_r \Sdown \Sigma_r \Sup \Sigma_r 1_n
    \]
    and
    \[
        \frac{\Delta_{n-2r}}{\Delta_{n-2r+1}}\
        \begin{tikzpicture}[anchorbase]
            \draw[S,<-] (-0.4,-0.8) -- (-0.4,-0.6) to[out=up,in=up,looseness=2] (0.4,-0.6) -- (0.4,-0.8);
            \projector{0,-0.5}{r};
            \region{0.6,-0.5}{n};
            \projector{-0.8,-0.5}{r};
        \end{tikzpicture}
        \colon \Sigma_r \Sdown \Sigma_r \Sup \Sigma_r 1_n \to \Sigma_r 1_n
    \]
    are mutually inverse isomorphisms.  This proves \cref{zebra1}.
    
    Now we prove \cref{zebra2}.  By \cref{wind,camel1,absorb}, we have, for $2r \le n$,
    \[
        \freeprojector{r}\ \rightbubmultreg[S]{\freeprojector{r-1}}{n}
        = \frac{\Delta_{n-2r+1}}{\Delta_{n-2r}}\ \freeprojectorreg{r}{n}
        \quad \text{and} \quad
        \begin{tikzpicture}[centerzero]
            \draw[S,<-] (-0.5,0.8) -- (-0.5,0.6) to[out=down,in=down,looseness=1.7] (0.5,0.6) -- (0.5,0.8);
            \projector{0,0.5}{r-1};
            \draw[S,->] (-0.5,-0.8) -- (-0.5,-0.6) to[out=up,in=up,looseness=2] (0.5,-0.6) -- (0.5,-0.8);
            \projector{0,-0.5}{r-1};
            \region{0.5,0}{n};
            \projector{-0.7,0}{r};
        \end{tikzpicture}
        = \frac{\Delta_{n-2r+1}}{\Delta_{n-2r}}\
        \begin{tikzpicture}[centerzero]
            \draw[S,->] (-0.5,-0.8) -- (-0.5,0.8);
            \draw[S,<-] (0.5,-0.8) -- (0.5,0.8);
            \projector{0,0}{r-1};
            \projector{-0.85,0}{r};
            \projector{0.85,0}{r};
            \region{0.8,0.5}{n};
        \end{tikzpicture}
        \ .
    \]
    Thus,
    \[
        \begin{tikzpicture}[anchorbase]
            \draw[S,<-] (-0.5,0.8) -- (-0.5,0.6) to[out=down,in=down,looseness=2] (0.5,0.6) -- (0.5,0.8);
            \projector{0,0.5}{r-1};
            \region{0.7,0.5}{n};
            \projector{-0.8,0.5}{r};
        \end{tikzpicture}
        \colon \Sigma_r 1_n \to \Sigma_r \Sup \Sigma_{r-1} \Sdown \Sigma_r 1_n
    \]
    and
    \[
        \frac{\Delta_{n-2r}}{\Delta_{n-2r+1}}\
        \begin{tikzpicture}[anchorbase]
            \draw[S,->] (-0.5,-0.8) -- (-0.5,-0.6) to[out=up,in=up,looseness=2] (0.5,-0.6) -- (0.5,-0.8);
            \projector{0,-0.5}{r-1};
            \region{0.7,-0.5}{n};
            \projector{-0.9,-0.5}{r};
        \end{tikzpicture}
        \colon \Sigma_r \Sup \Sigma_{r-1} \Sdown \Sigma_r 1_n \to \Sigma_r 1_n
    \]
    are mutually inverse isomorphisms.
        
    Next, we prove \cref{zebra3}.  By \cref{camel3,absorb},
    \[
        \begin{tikzpicture}[centerzero]
            \draw[S,<->] (-0.4,0.8) -- (-0.4,0.6) to[out=down,in=left] (0,0.15) to[out=right,in=down] (0.4,0.6) -- (0.4,0.8);
            \projector{0,0.5}{r};
            \draw[S] (-0.4,-0.8) -- (-0.4,-0.6) to[out=up,in=left] (0,-0.15) to[out=right,in=up] (0.4,-0.6) -- (0.4,-0.8);
            \projector{0,-0.5}{r};
            \draw[D] (0,-0.15) -- (0,0.15);
            \region{0.6,0.4}{n};
            \projector{0.4,0}{r};
        \end{tikzpicture}
        = \frac{\Delta_{n-2r+1}}{\Delta_{n-2r}}\
        \begin{tikzpicture}[centerzero]
            \draw[S,->] (-0.4,-0.8) -- (-0.4,0.8);
            \draw[S,->] (0.4,-0.8) -- (0.4,0.8);
            \projector{0,0}{r};
            \projector{-0.9,0}{r+1};
            \projector{0.7,0}{r};
            \region{0.8,0.5}{n};
        \end{tikzpicture}
        \ .
    \]
    We also have
    \[
        \begin{tikzpicture}[centerzero]
            \draw[D] (0,-0.7) -- (0,-0.4);
            \draw[S] (0,-0.4) to[out=right,in=down] (0.4,0) to[out=up,in=right] (0,0.4);
            \draw[S] (0,-0.4) to[out=left,in=down] (-0.4,0) to[out=up,in=left] (0,0.4);
            \draw[D,->] (0,0.4) -- (0,0.7);
            \projector{0,0}{r};
            \projector{0.7,-0.35}{r};
        \end{tikzpicture}
        \overset{\cref{fire}}{=}
        \begin{tikzpicture}[centerzero]
            \draw[D,->] (0,-0.7) -- (0,0.7);
            \bubleftmult[S]{0.6,0.2}{\freeprojector{r}};
            \projector{0.6,-0.4}{r};
        \end{tikzpicture}
        \overset{\cref{wind}}{=} \frac{\Delta_{n-2r+1}}{\Delta_{n-2r}}\
        \begin{tikzpicture}[centerzero]
            \draw[D,->] (0,-0.7) -- (0,0.7);
            \projector{0.5,0}{r};
        \end{tikzpicture}
        \ .
    \]
    Thus,
    \[
        \begin{tikzpicture}[anchorbase]
            \draw[S,<->] (-0.4,0.8) -- (-0.4,0.6) to[out=down,in=left] (0,0.15) to[out=right,in=down] (0.4,0.6) -- (0.4,0.8);
            \projector{0,0.5}{r};
            \draw[D] (0,-0.3) -- (0,0.15);
            \region{0.6,0.4}{n};
            \projector{0.4,0}{r};
        \end{tikzpicture}
        \overset{\cref{absorb}}{\underset{\cref{smush}}{=}}
        \begin{tikzpicture}[anchorbase]
            \draw[S,<->] (-0.4,0.8) -- (-0.4,0.6) to[out=down,in=left] (0,0.15) to[out=right,in=down] (0.4,0.6) -- (0.4,0.8);
            \projector{0,0.5}{r};
            \draw[D] (0,-0.3) -- (0,0.15);
            \region{0.6,0.4}{n};
            \projector{0.4,0}{r};
            \projector{-0.6,0}{r+1};
        \end{tikzpicture}
        \colon \Dup \Sigma_r 1_n \to \Sigma_{r+1} \Sup \Sigma_r \Sup \Sigma_r 1_n
    \]
    and
    \[
        \frac{\Delta_{n-2r}}{\Delta_{n-2r+1}}\ 
        \begin{tikzpicture}[centerzero]
            \draw[S] (-0.4,-0.8) -- (-0.4,-0.6) to[out=up,in=left] (0,-0.15) to[out=right,in=up] (0.4,-0.6) -- (0.4,-0.8);
            \projector{0,-0.5}{r};
            \draw[D,->] (0,-0.15) -- (0,0.4);
            \region{0.6,-0.4}{n};
            \projector{0.4,0}{r};
        \end{tikzpicture}
        \overset{\cref{absorb}}{\underset{\cref{smush}}{=}} \frac{\Delta_{n-2r}}{\Delta_{n-2r+1}}\ 
        \begin{tikzpicture}[centerzero]
            \draw[S] (-0.4,-0.8) -- (-0.4,-0.6) to[out=up,in=left] (0,-0.15) to[out=right,in=up] (0.4,-0.6) -- (0.4,-0.8);
            \projector{0,-0.5}{r};
            \draw[D,->] (0,-0.15) -- (0,0.4);
            \region{0.6,-0.4}{n};
            \projector{0.4,0}{r};
            \projector{-0.6,0}{r+1};
        \end{tikzpicture}
        \colon \Sigma_{r+1} \Sup \Sigma_r \Sup \Sigma_r 1_n \to \Dup \Sigma_r 1_n
    \]
    are mutually inverse isomorphisms.  Finally, the isomorphism \cref{zebra4} follows from \cref{zebra3} by applying $\bR$.
\end{proof}

\begin{prop} \label{donkey}
    The following isomorphisms hold in $\Kar(\Dcat)$ for all $n \in \N$:
    \begin{align} \label{donkey1}
        \Sigma_r \Sdown \Sigma_{r+1} \Sup \Sigma_r 1_n &\cong \Sigma_r 1_n,&
        r \in \Z,\ 2r \ne n,
        \\ \label{donkey2}
        \Sigma_r \Sup \Sigma_r \Sdown \Sigma_r 1_n &\cong \Sigma_r 1_n,&
        r \in \Z,\ 2r \ne n,
        \\ \label{donkey3}
        \Sigma_{r+1} \Sup \Sigma_{r+1} \Sup \Sigma_r 1_n
        &\cong \Dup \Sigma_r 1_n,&
        r \in \Z,\ 2r \ne n,
        \\ \label{donkey4}
        \Sigma_{r-1} \Sdown \Sigma_r \Sdown \Sigma_r 1_n &\cong \Sigma_r \Ddown 1_n,&
        r \in \Z,\ 2r \ne n+2.
    \end{align}
\end{prop}

\begin{proof}
    Since the isomorphisms are trivial when $2r>n$, we assume that $2r \le n$.  First, we prove \cref{donkey1}.  By \cref{wind,camel2,absorb}, we have, for $2r < n$,
    \[
        \freeprojector{r}\ \leftbubmultreg[S]{\freeprojector{r+1}}{n}
        = \frac{\Delta_{n-2r-1}}{\Delta_{n-2r}}\ \freeprojectorreg{r}{n}
        \quad \text{and} \quad
        \begin{tikzpicture}[centerzero]
            \draw[S,->] (-0.5,0.8) -- (-0.5,0.6) to[out=down,in=down,looseness=1.7] (0.5,0.6) -- (0.5,0.8);
            \projector{0,0.5}{r+1};
            \draw[S,<-] (-0.5,-0.8) -- (-0.5,-0.6) to[out=up,in=up,looseness=2] (0.5,-0.6) -- (0.5,-0.8);
            \projector{0,-0.5}{r+1};
            \region{0.5,0}{n};
            \projector{-0.7,0}{r};
        \end{tikzpicture}
        = \frac{\Delta_{n-2r-1}}{\Delta_{n-2r}}\
        \begin{tikzpicture}[centerzero]
            \draw[S,<-] (-0.5,-0.8) -- (-0.5,0.8);
            \draw[S,->] (0.5,-0.8) -- (0.5,0.8);
            \projector{0,0}{r+1};
            \projector{-0.85,0}{r};
            \projector{0.85,0}{r};
            \region{0.8,0.5}{n};
        \end{tikzpicture}
        \ .
    \]
    (Note that the condition $2r<n$ arises from the fact that we use \cref{wind,camel2} with $r$ replaced by $r+1$.  This is why we have the condition $2r \ne n$ in \cref{donkey1}.)  It follows that
    \[
        \begin{tikzpicture}[anchorbase]
            \draw[S,->] (-0.5,0.8) -- (-0.5,0.6) to[out=down,in=down,looseness=2] (0.5,0.6) -- (0.5,0.8);
            \projector{0,0.5}{r+1};
            \region{0.7,0.5}{n};
            \projector{-0.8,0.5}{r};
        \end{tikzpicture}
        \colon \Sigma_r 1_n \to \Sigma_r \Sdown \Sigma_{r+1} \Sup \Sigma_r 1_n
    \]
    and
    \[
        \frac{\Delta_{n-2r}}{\Delta_{n-2r-1}}\
        \begin{tikzpicture}[anchorbase]
            \draw[S,<-] (-0.5,-0.8) -- (-0.5,-0.6) to[out=up,in=up,looseness=2] (0.5,-0.6) -- (0.5,-0.8);
            \projector{0,-0.5}{r+1};
            \region{0.7,-0.5}{n};
            \projector{-0.9,-0.5}{r};
        \end{tikzpicture}
        \colon \Sigma_r \Sdown \Sigma_{r+1} \Sup \Sigma_r 1_n \to \Sigma_r 1_n
    \]
    are mutually inverse isomorphisms.  The proof of \cref{donkey2} is analogous, using \cref{wind,camel1,absorb}.
    \details{
        We prove \cref{donkey2}.  By \cref{wind,camel1,absorb}, we have, for $2r<n$,
        \[
            \freeprojector{r}\ \rightbubmultreg[S]{\freeprojector{r}}{n}
            = \frac{\Delta_{n-2r-1}}{\Delta_{n-2r}}\ \freeprojectorreg{r}{n}
            \quad \text{and} \quad
            \begin{tikzpicture}[centerzero]
                \draw[S,<-] (-0.4,0.8) -- (-0.4,0.6) to[out=down,in=down,looseness=2] (0.4,0.6) -- (0.4,0.8);
                \projector{0,0.5}{r};
                \draw[S,->] (-0.4,-0.8) -- (-0.4,-0.6) to[out=up,in=up,looseness=2] (0.4,-0.6) -- (0.4,-0.8);
                \projector{0,-0.5}{r};
                \region{0.4,0}{n};
                \projector{-0.7,0}{r};
            \end{tikzpicture}
            = \frac{\Delta_{n-2r+1}}{\Delta_{n-2r}}\
            \begin{tikzpicture}[centerzero]
                \draw[S,->] (-0.4,-0.8) -- (-0.4,0.8);
                \draw[S,<-] (0.4,-0.8) -- (0.4,0.8);
                \projector{0,0}{r};
                \projector{-0.7,0}{r};
                \projector{0.7,0}{r};
                \region{0.8,0.5}{n};
            \end{tikzpicture}
            \ .
        \]
        Thus
        \[
            \begin{tikzpicture}[anchorbase]
                \draw[S,<-] (-0.4,0.8) -- (-0.4,0.6) to[out=down,in=down,looseness=2] (0.4,0.6) -- (0.4,0.8);
                \projector{0,0.5}{r};
                \region{0.6,0.5}{n};
                \projector{-0.7,0.5}{r};
            \end{tikzpicture}
            \colon \Sigma_r 1_n \to \Sigma_r \Sup \Sigma_r \Sdown \Sigma_r 1_n
        \]
        and
        \[
            \frac{\Delta_{n-2r}}{\Delta_{n-2r+1}}\
            \begin{tikzpicture}[anchorbase]
                \draw[S,->] (-0.4,-0.8) -- (-0.4,-0.6) to[out=up,in=up,looseness=2] (0.4,-0.6) -- (0.4,-0.8);
                \projector{0,-0.5}{r};
                \region{0.6,-0.5}{n};
                \projector{-0.8,-0.5}{r};
            \end{tikzpicture}
            \colon \Sigma_r \Sup \Sigma_r \Sdown \Sigma_r 1_n \to \Sigma_r 1_n
        \]
        are mutually inverse isomorphisms.
    }
    
    The proof of \cref{donkey3} is similar to the proof of \cref{zebra3}.
    \details{
        By \cref{camel3,absorb}, we have, for $2r < n$,
        \[
            \begin{tikzpicture}[centerzero]
                \draw[S,<->] (-0.6,0.8) -- (-0.6,0.6) to[out=down,in=left] (0,0.15) to[out=right,in=down] (0.6,0.6) -- (0.6,0.8);
                \projector{0,0.5}{r+1};
                \draw[S] (-0.6,-0.8) -- (-0.6,-0.6) to[out=up,in=left] (0,-0.15) to[out=right,in=up] (0.6,-0.6) -- (0.6,-0.8);
                \projector{0,-0.5}{r+1};
                \draw[D] (0,-0.15) -- (0,0.15);
                \region{0.8,0.4}{n};
                \projector{0.6,0}{r};
            \end{tikzpicture}
            = \frac{\Delta_{n-2r-1}}{\Delta_{n-2r}}\
            \begin{tikzpicture}[centerzero]
                \draw[S,->] (-0.5,-0.8) -- (-0.5,0.8);
                \draw[S,->] (0.55,-0.8) -- (0.55,0.8);
                \projector{0,0}{r+1};
                \projector{-1,0}{r+1};
                \projector{0.9,0}{r};
                \region{0.8,0.5}{n};
            \end{tikzpicture}
            \ .
        \]
        We also have
        \[
            \begin{tikzpicture}[centerzero]
                \draw[D] (0,-0.7) -- (0,-0.4);
                \draw[S] (0,-0.4) to[out=right,in=down] (0.5,0) to[out=up,in=right] (0,0.4);
                \draw[S] (0,-0.4) to[out=left,in=down] (-0.5,0) to[out=up,in=left] (0,0.4);
                \draw[D,->] (0,0.4) -- (0,0.7);
                \projector{0,0}{r+1};
                \projector{0.7,-0.35}{r};
            \end{tikzpicture}
            \overset{\cref{fire}}{=}
            \begin{tikzpicture}[centerzero]
                \draw[D,->] (0,-0.7) -- (0,0.7);
                \bubleftmult[S]{0.8,0.2}{\freeprojector{r+1}};
                \projector{0.6,-0.4}{r};
            \end{tikzpicture}
            \overset{\cref{wind}}{=} \frac{\Delta_{n-2r-1}}{\Delta_{n-2r}}\
            \begin{tikzpicture}[centerzero]
                \draw[D,->] (0,-0.7) -- (0,0.7);
                \projector{0.5,0}{r};
            \end{tikzpicture}
            \ .
        \]
        Thus
        \[
            \begin{tikzpicture}[anchorbase]
                \draw[S,<->] (-0.5,0.8) -- (-0.5,0.6) to[out=down,in=left] (0,0.15) to[out=right,in=down] (0.5,0.6) -- (0.5,0.8);
                \projector{0,0.5}{r+1};
                \draw[D] (0,-0.3) -- (0,0.15);
                \region{0.7,0.4}{n};
                \projector{0.5,0}{r};
            \end{tikzpicture}
            \overset{\cref{absorb}}{\underset{\cref{smush}}{=}}
            \begin{tikzpicture}[anchorbase]
                \draw[S,<->] (-0.5,0.8) -- (-0.5,0.6) to[out=down,in=left] (0,0.15) to[out=right,in=down] (0.5,0.6) -- (0.5,0.8);
                \projector{0,0.5}{r+1};
                \draw[D] (0,-0.3) -- (0,0.15);
                \region{0.7,0.4}{n};
                \projector{0.5,0}{r};
                \projector{-0.7,0}{r+1};
            \end{tikzpicture}
            \colon \Dup \Sigma_r 1_n \to \Sigma_{r+1} \Sup \Sigma_{r+1} \Sup \Sigma_r 1_n
        \]
        and
        \[
            \frac{\Delta_{n-2r}}{\Delta_{n-2r-1}}\ 
            \begin{tikzpicture}[centerzero]
                \draw[S] (-0.5,-0.8) -- (-0.5,-0.6) to[out=up,in=left] (0,-0.15) to[out=right,in=up] (0.5,-0.6) -- (0.5,-0.8);
                \projector{0,-0.5}{r+1};
                \draw[D,->] (0,-0.15) -- (0,0.4);
                \region{0.7,-0.4}{n};
                \projector{0.5,0}{r};
            \end{tikzpicture}
            \overset{\cref{absorb}}{\underset{\cref{smush}}{=}} \frac{\Delta_{n-2r}}{\Delta_{n-2r-1}}\ 
            \begin{tikzpicture}[centerzero]
                \draw[S] (-0.5,-0.8) -- (-0.5,-0.6) to[out=up,in=left] (0,-0.15) to[out=right,in=up] (0.5,-0.6) -- (0.5,-0.8);
                \projector{0,-0.5}{r+1};
                \draw[D,->] (0,-0.15) -- (0,0.4);
                \region{0.7,-0.4}{n};
                \projector{0.5,0}{r};
                \projector{-0.7,0}{r+1};
            \end{tikzpicture}
            \colon \Sigma_{r+1} \Sup \Sigma_{r+1} \Sup \Sigma_r 1_n \to \Dup \Sigma_r 1_n
        \]
        are mutually inverse isomorphisms.
    }
    Finally, \cref{donkey4} follows from applying the functor $\bR$ to \cref{donkey3}.
\end{proof}

\begin{cor} \label{kona}
    The following isomorphism holds in $\Kar(\Dmon)$:
    \begin{gather} \label{ccr}
        \Sdown \Sup \cong \Sup \Sdown \oplus \Sigma_0.
    \end{gather}
    Furthermore, the following isomorphism holds in $\Add(\Dmon)$:
    \begin{gather} \label{ccrmon}
        \Sdown \Sup \oplus \Dup \Ddown \cong \Sup \Sdown \oplus \one.
    \end{gather}
\end{cor}

\begin{proof}
    We have
    \begin{equation} \label{kona1}
        \Sdown \Sigma_0 \Sup
        \overset{\cref{portis}}{\underset{\cref{abyss}}{\cong}} \Sigma_0 \Sdown \Sigma_0 \Sup \Sigma_0
        \overset{\cref{zebra1}}{\cong} \Sigma_0.
    \end{equation}
    Thus,
    \[
        \Sdown \Sup
        \overset{\cref{wasp1}}{\cong} \Sdown \Theta_0 \Sup
        \overset{\cref{disect}}{\underset{\cref{hacky}}{\cong}} \Sdown \Dup \Ddown \Sup \oplus \Sdown \Sigma_0 \Sup
        \overset{\cref{isoms2}}{\underset{\cref{kona1}}{\cong}} \Sup \Sdown \oplus \Sigma_0,
    \]
    proving \cref{ccr}.
    
    In $\Kar(\Dcat)$, we then have
    \[
        \Sdown \Sup \oplus \Dup \Ddown
        \overset{\cref{ccr}}{\cong} \Sup \Sdown \oplus \Sigma_0 \oplus \Dup \Ddown
        \overset{\cref{disect}}{\underset{\cref{hacky}}{\cong}} \Sup \Sdown \oplus \Theta_0
        \overset{\cref{wasp1}}{\cong} \Sup \Sdown \oplus \one,
    \]
    implying \cref{ccrmon}.
\end{proof}

\begin{rem} \label{thai}
    \begin{enumerate}[wide]
        \item \label{thai1} The isomorphism \cref{zebra1} is fundamental for us, since it is used in the proof of \cref{kona}, which is one of the key isomorphisms used in our analysis of the Grothendieck ring of $\Dmon$; see the proof of \cref{Grothsurj}.  The other isomorphisms in \cref{zebra,donkey} are not used elsewhere in the current paper.  We include them to show that all of the isomorphisms of \cite[Table~1]{HS25} can be proved directly in the setting of the current paper.  
        \item \label{thai2} The conditions on $r$ in \cref{donkey1,donkey2,donkey3,donkey4} are necessary.  For the excluded values of $r$, it follows from \cref{abyss} that the left side of the equation is zero.  However, it will follow from our basis theorem (\cref{basisthm}) that the right sides are all nonzero for these values of $r$.
    \end{enumerate}
\end{rem}

%=================================
\section{Temperley--Lieb algebras}
%=================================

In this section we collect the crucial facts about Temperley--Lieb algebras that we will need.  Many of the basic facts (properties of Jones--Wenzl projectors, classification of modules, etc.) are well known.  For a thorough discussion of the results on induction, restriction, and matrix units, we refer the reader to \cite{Qui17} and \cite[Appendix]{HS25}.  We make the same assumptions on the ground field as in \cref{subsec:2cat}.

%----------------------------------------------------------------------------------------
\subsection{The Temperley--Lieb category and Jones--Wenzl projectors\label{subsec:TLcat}}
%----------------------------------------------------------------------------------------

Define $\TLcat(\delta)$ to be the $\kk$-linear strict monoidal category generated by the object $\Tobj$ and morphisms
\[
    \Tcap \colon \Tobj \otimes \Tobj \to \one,\qquad
    \Tcup \colon \one \to \Tobj \otimes \Tobj,
\]
subject to the relations
\begin{equation} \label{TLrels}
    \begin{tikzpicture}[centerzero,TL]
        \draw (-0.3,-0.4) -- (-0.3,0) arc(180:0:0.15) arc(180:360:0.15) -- (0.3,0.4);
    \end{tikzpicture}
    =
    \begin{tikzpicture}[centerzero,TL]
        \draw (0,-0.4) -- (0,0.4);
    \end{tikzpicture}
    =
    \begin{tikzpicture}[centerzero,TL]
        \draw (-0.3,0.4) -- (-0.3,0) arc(180:360:0.15) arc(180:0:0.15) -- (0.3,-0.4);
    \end{tikzpicture}
    \ ,\qquad
    \unbub[TL] = \delta 1_\one.
\end{equation}
For $n \in \N$,
\[
    \TLalg_n = \TLalg_n(\delta) := \TLcat(\Tobj^{\otimes n}, \Tobj^{\otimes n})
\]
is the Temperley--Lieb algebra on $n$ strands.  In Temperley--Lieb diagrams, we adopt the convention that
\begin{center}
    strings are numbered $1,2,\dotsc$ from \emph{right to left}.
\end{center}

We denote the unit element of $\TLalg_n$ by $1_n$ and the Jones--Wenzl projector by $\JW_n$.  We will sometimes denote multiple strands by a thick strand labelled by the multiplicity.  For example,
\[
    \begin{tikzpicture}[centerzero,TL]
        \draw[multi] (0,-0.2) \botlabel{n} -- (0,0.2);
    \end{tikzpicture}
    \, = 1_n.
\]
We will sometimes omit the multiplicity label when it can be deduced from the rest of the diagram.  We have the following recursive formula for the Jones--Wenzl projectors:
\begin{equation} \label{JWrecursion}
    \JW_0 = 1_\one,\qquad
    \JW_1 = 1_1,\qquad
    \begin{tikzpicture}[centerzero,TL]
        \draw[multi] (0,-1) -- (0,1);
        \coupon{0,0}{\JW_{n+1}};
    \end{tikzpicture}
    =
    \begin{tikzpicture}[centerzero,TL]
        \draw[multi] (0,-1) -- (0,1);
        \draw (-0.4,-1) -- (-0.4,1);
        \coupon{0,0}{\JW_n};
    \end{tikzpicture}
    - \frac{\Delta_{n-1}}{\Delta_n}
    \begin{tikzpicture}[centerzero,TL]
        \draw[multi] (0.15,-0.5) -- (0.15,0.5);
        \draw[multi] (0,0.5) -- (0,1);
        \draw[multi] (0,-0.5) -- (0,-1);
        \draw (-0.15,0.5) -- (-0.15,0.25) to[out=down,in=down,looseness=1.5] (-0.5,0.25) -- (-0.5,1);
        \draw (-0.15,-0.5) -- (-0.15,-0.25) to[out=up,in=up,looseness=1.5] (-0.5,-0.25) -- (-0.5,-1);
        \genbox{-0.3,-0.3}{0.3,-0.7}{\JW_n};
        \genbox{-0.3,0.3}{0.3,0.7}{\JW_n}; 
    \end{tikzpicture}
    \ ,\qquad n \ge 1.
\end{equation}
Using \cref{Delta,TLrels,JWrecursion}, we see that
\begin{equation} \label{JWtrace}
    \begin{tikzpicture}[centerzero,TL]
        \draw[multi] (0,-0.6) -- (0,0.6) \toplabel{n};
        \draw (-0.2,0) -- (-0.2,0.25) to[out=up,in=up,looseness=1.5] (-0.6,0.25) -- (-0.6,-0.25) to[out=down,in=down,looseness=1.5] (-0.2,-0.25) -- (-0.2,0);
        \genbox{-0.4,-0.2}{0.4,0.2}{\JW_{n+1}};
    \end{tikzpicture}
    =
    \frac{\Delta_{n+1}}{\Delta_n}\,
    \begin{tikzpicture}[centerzero,TL]
        \draw[multi] (0,-0.6) -- (0,0.6);
        \coupon{0,0}{\JW_n};
    \end{tikzpicture}
    \qquad \text{and} \qquad
    \begin{tikzpicture}[centerzero,TL]
        \draw[multi] (0,0) -- (0,0.2) to[out=up,in=up,looseness=1.5] (-0.5,0.2) -- (-0.5,-0.2) to[out=down,in=down,looseness=1.5] (0,-0.2) -- (0,0);
        \coupon{0,0}{\JW_n};
    \end{tikzpicture}
    = \Delta_n 1_\one,
\end{equation}
where the second equality follows by induction.  We also have
\begin{equation} \label{JWkill}
    \begin{tikzpicture}[centerzero,TL]
        \draw (-0.15,0.2) -- (-0.15,0.3) to[out=up,in=up,looseness=1.5] (0.15,0.3) -- (0.15,0.2);
        \draw[multi] (-0.3,0.2) -- (-0.3,0.5) \toplabel{r};
        \draw[multi] (0.3,0.2) -- (0.3,0.5) \toplabel{s};
        \draw[multi] (0,-0.2) -- (0,-0.5);
        \genbox{-0.5,-0.2}{0.5,0.2}{\JW_n};
    \end{tikzpicture}
    = 0 =
    \begin{tikzpicture}[centerzero,TL]
        \draw (-0.15,-0.2) -- (-0.15,-0.3) to[out=down,in=down,looseness=1.5] (0.15,-0.3) -- (0.15,-0.2);
        \draw[multi] (-0.3,-0.2) -- (-0.3,-0.5) \botlabel{r};
        \draw[multi] (0.3,-0.2) -- (0.3,-0.5) \botlabel{s};
        \draw[multi] (0,0.2) -- (0,0.5);
        \genbox{-0.5,-0.2}{0.5,0.2}{\JW_n};
    \end{tikzpicture}
    \qquad \text{for all } r,s \in \N,\ n=r+s+2,
\end{equation}
and
\begin{equation} \label{JWabsorb}
    \begin{tikzpicture}[centerzero,TL]
        \draw[multi] (0,-0.15) -- (0,1);
        \draw (-0.4,-0.15) -- (-0.4,1);
        \draw (0.4,-0.15) -- (0.4,1);
        \draw[multi] (0,-0.55) -- (0,-1);
        \coupon{0,0.35}{\JW_m};
        \genbox{-0.6,-0.15}{0.6,-0.55}{\JW_n};
    \end{tikzpicture}
    =
    \begin{tikzpicture}[centerzero,TL]
        \draw[multi] (0,-1) -- (0,1);
        \coupon{0,0}{\JW_n};
    \end{tikzpicture}
    =
    \begin{tikzpicture}[centerzero,TL,yscale=-1]
        \draw[multi] (0,-0.15) -- (0,1);
        \draw (-0.4,-0.15) -- (-0.4,1);
        \draw (0.4,-0.15) -- (0.4,1);
        \draw[multi] (0,-0.55) -- (0,-1);
        \coupon{0,0.35}{\JW_m};
        \genbox{-0.6,-0.15}{0.6,-0.55}{\JW_n};
    \end{tikzpicture}
    \qquad \text{for all } 0 \le m \le n.
\end{equation}

For $0 \le k < n$, we define
\begin{equation} \label{ecupcap}
    e_k =
    \begin{tikzpicture}[centerzero,TL]
        \draw[multi] (-0.7,-0.4) \botlabel{k-1} -- (-0.7,0.4);
        \draw[multi] (0.7,-0.4) \botlabel{n-k-1} -- (0.7,0.4);
        \draw (-0.15,-0.4) -- (-0.15,-0.25) to[out=up,in=up,looseness=1.5] (0.15,-0.25) -- (0.15,-0.4);
        \draw (-0.15,0.4) -- (-0.15,0.25) to[out=down,in=down,looseness=1.5] (0.15,0.25) -- (0.15,0.4);
    \end{tikzpicture}
    .
\end{equation}

%------------------------------------------------------
\subsection{Standard modules\label{subsec:TLstandards}}
%------------------------------------------------------

Our assumption \cref{generic} implies that the Temperley--Lieb algebras $\TLalg_n$ are split semisimple.  However, since we do not assume that $\kk$ is a field, one needs to take care with the notion of a \emph{simple} module.  In this subsection, we recall the classification of the standard modules (also called cell modules).  These are the absolutely irreducible $\kk$-forms of the generic simple modules, and coincide with the simple modules when $\kk$ is a field.

The isomorphism classes of standard $\TLalg_n$-modules are in bijection with elements of the set
\begin{equation} \label{Lambda}
    \Lambda_n := \{ k \in \N : k \le n,\ k \equiv n \mod 2\}.
\end{equation}
The standard module corresponding to $k \in \Lambda_n$ is
\[
    L^n_k := \TLcat(\Tobj^{\otimes k},\Tobj^{\otimes n}) \JW_k.
\]
We adopt the convention that
\[
    L^n_k = 0 \quad \text{if } k<0 \text{ or } k>n.
\]
We depict elements of $L^n_k$ by string diagrams with $\JW_k$ at the bottom and $n$ strings at the top:
\[
    \begin{tikzpicture}[centerzero,TL]
        \draw[multi] (0,-0.7) -- (0,0.7) \toplabel{n};
        \coupon{0,0.3}{f};
        \coupon{0,-0.3}{\JW_k};
    \end{tikzpicture}
    \ ,\qquad f \in \TLcat(\Tobj^{\otimes k}, \Tobj^{\otimes n}).
\]
The action of $\TLalg_n$ is by composition on the top of such diagrams.

%----------------------------------------------------------
\subsection{Induction and restriction\label{subsec:indres}}
%----------------------------------------------------------

For $n,m \in \N$, we have a natural inclusion of algebras
\[
    \TLalg_m \hookrightarrow \TLalg_{m+n},\qquad
    \begin{tikzpicture}[centerzero,TL]
        \draw[multi] (0,-0.5) -- (0,0.5);
        \coupon{0,0}{f};
    \end{tikzpicture}
    \mapsto
    \begin{tikzpicture}[centerzero,TL]
        \draw[multi] (0,-0.5) -- (0,0.5);
        \draw[multi] (-0.3,-0.5) \botlabel{n} -- (-0.3,0.5);
        \coupon{0,0}{f};
    \end{tikzpicture}
    \ .
\]
For $0 \le m,l \le n$, let $\TLbim{m}{\TLalg_n}{l}$ be $\TLalg_n$, viewed as a $(\TLalg_m,\TLalg_l)$-bimodule.  For $n \in \N$, we denote the tensor product over $\TLalg_n$ by $\otimes_n$.  When working with (bi)modules, an unadorned tensor product denotes a tensor product over $\kk$.  We have the induction functor
\[
    \Ind_n^{n+1} \colon \TLalg_n\md \to \TLalg_{n+1}\md,\qquad
    M \mapsto \TLbim{n+1}{\TLalg_{n+1}}{n} \otimes_n M
\]
and the restriction functor
\[
    \Res_n^{n+1} \colon \TLalg_{n+1}\md \to \TLalg_n\md,\qquad
    M \mapsto \TLbim{n}{\TLalg_{n+1}}{n+1} \otimes_{n+1} M.
\]
As noted in \cite[\S A.2]{HS25} (see also \cite[\S 3.3, \S 3.4]{Qui17}), we have an isomorphism
\begin{equation} \label{resformula}
    \Res_n^{n+1}(L^{n+1}_k) \xrightarrow{\cong} L^n_{k+1} \oplus L^n_{k-1},\qquad
    \begin{tikzpicture}[anchorbase,TL]
        \draw[multi] (0,-0.8) -- (0,0.8) \toplabel{n+1};
        \coupon{0,0.3}{f};
        \coupon{0,-0.3}{\JW_k};
    \end{tikzpicture}
    \mapsto
    \left(
        \begin{tikzpicture}[anchorbase,TL]
            \draw[multi] (0,-0.3) -- (0,0.2);
            \draw[multi] (0.1,0.4) -- (0.1,0.8) \toplabel{n};
            \draw[multi] (-0.15,-0.9) -- (-0.15,-0.5);
            \draw (-0.05,0.4) -- (-0.05,0.6) to[out=up,in=up,looseness=1.5] (-0.35,0.6) -- (-0.35,-0.3);
            \genbox{-0.55,-0.6}{0.25,-0.2}{\JW_{k+1}};
            \genbox{-0.2,0.1}{0.2,0.5}{f};
        \end{tikzpicture}
        \, , \frac{\Delta_{k-1}}{\Delta_k}\,
        \begin{tikzpicture}[anchorbase,TL]
            \draw (-0.15,0) -- (-0.15,0.25) to[out=up,in=up,looseness=1.5] (-0.5,0.25) -- (-0.5,-0.95) to[out=down,in=down,looseness=1.5] (-0.15,-0.95) -- (-0.15,-0.7);
            \draw[multi] (0.15,0) -- (0.15,0.6) \toplabel{n};
            \draw[multi] (0.15,-0.7) -- (0.15,-1.2);
            \draw[multi] (0,-0.7) -- (0,0);
            \genbox{-0.3,-0.2}{0.3,0.2}{f};
            \genbox{-0.3,-0.5}{0.3,-0.9}{\JW_k};
        \end{tikzpicture}
    \right),
\end{equation}
with inverse
\begin{equation} \label{resinv}
    L^n_{k+1} \oplus L^n_{k-1} \xrightarrow{\cong} \Res_n^{n+1}(L^{n+1}_k),\qquad
    \left(
        \begin{tikzpicture}[anchorbase,TL]
            \draw[multi] (0,-0.8) -- (0,0.8) \toplabel{n};
            \coupon{0,0.3}{f};
            \coupon{0,-0.3}{\JW_{k+1}};
        \end{tikzpicture}
        \ ,\
        \begin{tikzpicture}[anchorbase,TL]
            \draw[multi] (0,-0.8) -- (0,0.8) \toplabel{n};
            \coupon{0,0.3}{g};
            \coupon{0,-0.3}{\JW_{k-1}};
        \end{tikzpicture}
    \right)
    \mapsto
    \begin{tikzpicture}[anchorbase,TL]
        \draw[multi] (0,-0.3) -- (0,0.8) \toplabel{n};
        \draw (-0.2,-0.5) to[out=down,in=down,looseness=2] (-0.6,-0.5) -- (-0.6,0.8);
        \draw[multi] (0.2,-0.5) -- (0.2,-0.9);
        \coupon{0,0.3}{f};
        \genbox{-0.4,-0.5}{0.4,-0.1}{\JW_{k+1}};
    \end{tikzpicture}
    +
    \begin{tikzpicture}[anchorbase,TL]
        \draw[multi] (0,-0.3) -- (0,0.2);
        \draw[multi] (0,0.4) -- (0,0.8) \toplabel{n};
        \draw[multi] (-0.15,-0.9) -- (-0.15,-0.5);
        \draw (-0.3,-0.3) -- (-0.3,0.8);
        \genbox{-0.55,-0.6}{0.25,-0.2}{\JW_k};
        \coupon{0,0.3}{g};
    \end{tikzpicture}
    \ ,
\end{equation}
and an isomorphism
\begin{equation} \label{indformula}
    \Ind_n^{n+1}(L^n_k) \xrightarrow{\cong} L^{n+1}_{k+1} \oplus L^{n+1}_{k-1},\qquad
    \begin{tikzpicture}[centerzero,TL]
        \draw[multi] (0,-0.8) -- (0,0.3);
        \draw[multi] (-0.2,0.3) -- (-0.2,0.8) \toplabel{n+1};
        \draw (-0.4,-0.8) -- (-0.4,0.3);
        \genbox{-0.6,0.1}{0.2,0.5}{f};
        \coupon{0,-0.3}{\JW_k};
    \end{tikzpicture}
    \mapsto
    \left(
        \begin{tikzpicture}[centerzero,TL]
            \draw[multi] (0,-0.8) -- (0,0.8) \toplabel{n+1};
            \coupon{0,0.3}{f};
            \coupon{0,-0.3}{\JW_{k+1}};
        \end{tikzpicture}
        \, , \frac{\Delta_{k-1}}{\Delta_k}\
        \begin{tikzpicture}[centerzero,TL]
            \draw[multi] (0.1,-0.9) -- (0.1,-0.3);
            \draw[multi] (0,-0.3) -- (0,0.3);
            \draw[multi] (-0.25,0.3) -- (-0.25,0.8) \toplabel{n+1};
            \draw (-0.1,-0.3) -- (-0.1,-0.5) to[out=down,in=down,looseness=2] (-0.5,-0.5) -- (-0.5,0.3);
            \genbox{-0.7,0.1}{0.2,0.5}{f};
            \genbox{-0.3,-0.5}{0.3,-0.1}{\JW_k};
        \end{tikzpicture}
    \right),
\end{equation}
with inverse
\begin{equation} \label{indinv}
    L^{n+1}_{k+1} \oplus L^{n+1}_{k-1} \xrightarrow{\cong} \Ind_n^{n+1}(L^n_k),\qquad
    \left(
        \begin{tikzpicture}[anchorbase,TL]
            \draw[multi] (0,-0.8) -- (0,0.8) \toplabel{n+1};
            \coupon{0,0.3}{f};
            \coupon{0,-0.3}{\JW_{k+1}};
        \end{tikzpicture}
        \ ,\
        \begin{tikzpicture}[anchorbase,TL]
            \draw[multi] (0,-0.8) -- (0,0.8) \toplabel{n+1};
            \coupon{0,0.3}{g};
            \coupon{0,-0.3}{\JW_{k-1}};
        \end{tikzpicture}
    \right)
    \mapsto
    \begin{tikzpicture}[anchorbase,TL]
        \draw[multi] (0,-0.8) -- (0,0.8) \toplabel{n+1};
        \coupon{0,0.3}{f};
        \coupon{0,-0.3}{\JW_{k+1}};
    \end{tikzpicture}
    + 
    \begin{tikzpicture}[anchorbase,TL]
        \draw[multi] (0,-0.3) -- (0,0.8) \toplabel{n+1};
        \draw[multi] (-0.1,-0.8) -- (-0.1,-0.3);
        \draw (-0.2,-0.1) to[out=up,in=up,looseness=2] (-0.6,-0.1) -- (-0.6,-0.8);
        \coupon{0,0.3}{g};
        \genbox{-0.4,-0.5}{0.2,-0.1}{\JW_k};
    \end{tikzpicture}
    \ .
\end{equation}
Note that we have rescaled the maps (A.7) and (A.8) of \cite{HS25} to avoid the square root appearing there.

%--------------------------------------
\subsection{Bases for standard modules}
%--------------------------------------

Since the direct sum decompositions appearing in \cref{resformula,indformula} are multiplicity free, we have a basis for $L^n_k$ corresponding to paths in the branching graph.  For $j,k \in \Z$, let
\begin{gather*}
    P^n := \{ p = (p_0,p_1,\dotsc,p_n) \in \N^{n+1} : p_0 = 0,\ p_{i+1} = p_i \pm 1 \text{ for all } 0 \le i < n \}, \\
    P^n_k := \{ p \in P^n : p_n = k \},\qquad
    P^n_{j,k} := \{ p \in P^n_k : p_{n-1} = j \},\qquad
    P := \bigcup_{n=0}^\infty P^n.
\end{gather*}
We define $P^n_k = \varnothing$ for $n<0$.  By definition, $P^n_k = \varnothing$ if $k<0$ or $k>n$.  For $p \in P^n_k$, $n>0$, define
\begin{equation} \label{pcut}
    \bar{p} = (p_0,\dotsc,p_{n-1}) \in P^{n-1}_{k-1} \cup P^{n-1}_{k+1}.
\end{equation}
We then define $v_p \in L^n_k$, for $p \in P^n_k$, by induction on $n$ via
\begin{equation} \label{v-initial}
    v_{(0)} = 1 \in \kk = L^0_0,
\end{equation}
and then, for $p \in P^n_k$,
\begin{equation} \label{v-recursion}
    v_p =
    \begin{tikzpicture}[anchorbase,TL]
        \draw[multi] (0,-0.3) -- (0,0.8);
        \draw (-0.2,-0.5) to[out=down,in=down,looseness=2] (-0.6,-0.5) -- (-0.6,0.8);
        \draw[multi] (0.2,-0.5) -- (0.2,-0.9);
        \coupon{0,0.3}{v_{\bar{p}}};
        \genbox{-0.4,-0.5}{0.4,-0.1}{\JW_{k+1}};
    \end{tikzpicture}
    \text{ if } p_{n-1} = k+1,
    \qquad \text{and} \qquad
    v_p =
    \begin{tikzpicture}[anchorbase,TL]
        \draw[multi] (0,-0.3) -- (0,0.2);
        \draw[multi] (0,0.4) -- (0,0.8);
        \draw[multi] (-0.15,-0.9) -- (-0.15,-0.5);
        \draw (-0.3,-0.3) -- (-0.3,0.8);
        \genbox{-0.55,-0.6}{0.25,-0.2}{\JW_k};
        \genbox{-0.2,0.1}{0.2,0.5}{v_{\bar{p}}};
    \end{tikzpicture}
    \text{ if } p_{n-1} = k-1.
\end{equation}
(Compare to \cref{resinv}.)  Then $v_p$, $p \in P^n_k$, is a basis for $L^n_k$ for all $k,n \in \Z$.

For $0 \le k \le n$ and $p \in P^n$, define
\begin{equation}
    p[k] := (p_0,p_1,\dotsc,p_k) \in P^k.
\end{equation}
In particular $p[n-1] = \bar{p}$.

\begin{lem} \label{cloud}
    For $p \in P^{n+2}$, we have
    \[
        \begin{tikzpicture}[TL,centerzero]
            \draw[multi] (0,-0.5) -- (0,0);
            \draw[multi] (0.3,0) -- (0.3,0.5) \toplabel{n};
            \draw (-0.3,0.1) -- (-0.3,0.2) to[out=up,in=up,looseness=2] (0,0.2) -- (0,0.1);
            \genbox{-0.5,-0.3}{0.5,0.1}{v_p};
        \end{tikzpicture}
        = 0 \quad \text{if} \quad p_{n+2} \ne p_n.
    \]
\end{lem}

\begin{proof}
    If $(p_n,p_{n+1},p_{n+2}) = (k-2,k-1,k)$ for some $k \ge 2$, then
    \[
        \begin{tikzpicture}[TL,centerzero]
            \draw[multi] (0,-0.5) -- (0,0);
            \draw[multi] (0.3,0) -- (0.3,0.5) \toplabel{n};
            \draw (-0.3,0.1) -- (-0.3,0.2) to[out=up,in=up,looseness=2] (0,0.2) -- (0,0.1);
            \genbox{-0.5,-0.3}{0.5,0.1}{v_p};
        \end{tikzpicture}
        \overset{\cref{v-recursion}}{=}
        \begin{tikzpicture}[anchorbase,TL]
            \draw[multi] (0,-0.3) -- (0,0.8);
            \draw[multi] (-0.15,-0.9) -- (-0.15,-0.5);
            \draw (-0.5,-0.2) -- (-0.5,-0.1) to[out=up,in=up,looseness=2] (-0.8,-0.1) -- (-0.8,-0.2);
            \genbox{-1,-0.6}{0.25,-0.2}{\JW_k};
            \coupon{0,0.3}{v_{p[n]}};
        \end{tikzpicture}
        \overset{\cref{JWkill}}{=} 0.
    \]
    If $(p_n,p_{n+1},p_{n+2}) = (k+2,k+1,k)$ for some $k \in \N$, then
    \[
        \begin{tikzpicture}[TL,centerzero]
            \draw[multi] (0,-0.5) -- (0,0);
            \draw[multi] (0.3,0) -- (0.3,0.5) \toplabel{n};
            \draw (-0.3,0.1) -- (-0.3,0.2) to[out=up,in=up,looseness=2] (0,0.2) -- (0,0.1);
            \genbox{-0.5,-0.3}{0.5,0.1}{v_p};
        \end{tikzpicture}
        \overset{\cref{v-recursion}}{=}
        \begin{tikzpicture}[anchorbase,TL]
            \draw[multi] (0,-0.3) -- (0,0.8);
            \draw (-0.3,-0.5) to[out=down,in=down,looseness=2] (-0.8,-0.5) -- (-0.8,0) to[out=up,in=up,looseness=2] (-1.1,0) -- (-1.1,-0.5) to[out=down,in=down,looseness=2] (0,-0.5);
            \draw[multi] (0.3,-0.5) -- (0.3,-1.2);
            \coupon{0,0.3}{v_{p[n]}};
            \genbox{-0.6,-0.5}{0.6,-0.1}{\JW_k};
        \end{tikzpicture}
        \overset{\cref{JWkill}}{=} 0.
    \]
    Since the only other possibilities are that $(p_n,p_{n+1},p_{n+2}) = (k, k \pm 1, k)$ for some $k \in \N$, the result follows.
\end{proof}

We have an isomorphism of monoidal categories
\[
    \Omega_\updownarrow \colon \TLcat \to \TLcat^\op,\qquad
    \Tcap \mapsto \Tcup\, ,\quad \Tcup \mapsto \Tcap\, .
\]
Intuitively, $\Omega_\updownarrow$ flips diagram across a horizontal axis.  We will sometimes denote $\Omega_\updownarrow(f)$ by $f^\updownarrow$.  It follows from \cref{JWrecursion} and induction on $n$ that
\begin{equation}
    \JW_n^\updownarrow = \JW_n \qquad \text{for all} \quad n \in \N.
\end{equation}

\begin{lem}
    For all $0 \le k \le n$, and $p,r \in P^n_k$,
    \begin{equation} \label{crab}
        v_p^\updownarrow v_{r}
        = \delta_{p,r} \Delta_p\
        \begin{tikzpicture}[centerzero,TL]
            \draw[multi] (0,-0.5) -- (0,0.5);
            \coupon{0,0}{\JW_k};
        \end{tikzpicture}
        \ ,\qquad \text{where} \quad
        \Delta_p := \prod_{i : p_{i-1} > p_i} \frac{\Delta_{p_{i-1}}}{\Delta_{p_i}}.
    \end{equation}
\end{lem}

\begin{proof}
    We prove the result by induction on $n$.  For $n=0$, \cref{crab} follows immediately from \cref{v-initial,JWrecursion}.  Now suppose that $p,r \in P^n_k$ for some $n > 0$, and that the result holds for $n-1$.  If $p_{n-1} = r_{n-1} = k+1$, then
    \[
        v_p^\updownarrow v_r
        \overset{\cref{v-recursion}}{=}
        \begin{tikzpicture}[anchorbase,TL]
            \draw[multi] (0,-1.3) -- (0,1.4);
            \draw (-0.25,1.1) to[out=up,in=up,looseness=1.5] (-0.7,1.1) -- (-0.7,-1) to[out=down,in=down,looseness=1.5] (-0.25,-1);
            \coupon{0,0.3}{v_{\bar{p}}^\updownarrow};
            \coupon{0,-0.3}{v_{\bar{r}}};
            \genbox{-0.5,0.7}{0.2,1.1}{\JW_{k+1}};
            \genbox{-0.5,-0.6}{0.2,-1}{\JW_{k+1}};
        \end{tikzpicture}
        = \delta_{\bar{p}, \bar{r}} \Delta_{\bar{p}}\
        \begin{tikzpicture}[anchorbase,TL]
            \draw[multi] (0,-1.3) -- (0,1.4);
            \draw (-0.25,1.1) to[out=up,in=up,looseness=1.5] (-0.7,1.1) -- (-0.7,-1) to[out=down,in=down,looseness=1.5] (-0.25,-1);
            \coupon{0,0}{\JW_k};
            \genbox{-0.5,0.7}{0.2,1.1}{\JW_{k+1}};
            \genbox{-0.5,-0.6}{0.2,-1}{\JW_{k+1}};
        \end{tikzpicture}
        \overset{\cref{JWabsorb}}{\underset{\cref{JWtrace}}{=}} \delta_{\bar{p},\bar{r}} \Delta_{\bar{p}} \frac{\Delta_{k+1}}{\Delta_k}\ 
        \begin{tikzpicture}[centerzero,TL]
            \draw[multi] (0,-0.5) -- (0,0.5);
            \coupon{0,0}{\JW_k};
        \end{tikzpicture}
        = \delta_{p,r} \Delta_p\
        \begin{tikzpicture}[centerzero,TL]
            \draw[multi] (0,-0.5) -- (0,0.5);
            \coupon{0,0}{\JW_k};
        \end{tikzpicture}
        \ .
    \]
    If $p_{n-1} = r_{n-1} = k-1$, then
    \[
        v_p^\updownarrow v_r
        \overset{\cref{v-recursion}}{=}
        \begin{tikzpicture}[anchorbase,TL]
            \draw[multi] (0,-1.3) -- (0,1.4);
            \draw (-0.35,0.7) -- (-0.35,-0.6);
            \coupon{0,0.3}{v_{\bar{p}}^\updownarrow};
            \coupon{0,-0.3}{v_{\bar{r}}};
            \genbox{-0.5,0.7}{0.2,1.1}{\JW_k};
            \genbox{-0.5,-0.6}{0.2,-1}{\JW_k};
        \end{tikzpicture}
        = \delta_{\bar{p}, \bar{r}} \Delta_{\bar{p}}\
        \begin{tikzpicture}[anchorbase,TL]
            \draw[multi] (0,-1.3) -- (0,1.4);
            \draw (-0.45,0.7) -- (-0.45,-0.6);
            \genbox{-0.35,-0.2}{0.35,0.2}{\JW_{k-1}};
            \genbox{-0.6,0.7}{0.3,1.1}{\JW_k};
            \genbox{-0.6,-0.6}{0.3,-1}{\JW_k};
        \end{tikzpicture}
        \overset{\cref{JWabsorb}}{=} \delta_{\bar{p},\bar{r}} \Delta_{\bar{p}} \ 
        \begin{tikzpicture}[centerzero,TL]
            \draw[multi] (0,-0.5) -- (0,0.5);
            \coupon{0,0}{\JW_k};
        \end{tikzpicture}
        = \delta_{p,r} \Delta_p\
        \begin{tikzpicture}[centerzero,TL]
            \draw[multi] (0,-0.5) -- (0,0.5);
            \coupon{0,0}{\JW_k};
        \end{tikzpicture}
        \ .
    \]
    If $p_{n-1} = k+1$ and $r_{n-1} = k-1$, then
    \[
        v_p^\updownarrow v_{r}
        \overset{\cref{v-recursion}}{=}
        \begin{tikzpicture}[anchorbase,TL]
            \draw[multi] (0,-1.3) -- (0,1.4);
            \draw (-0.25,1.1) to[out=up,in=up,looseness=1.5] (-0.6,1.1) -- (-0.6,-0.6);
            \coupon{0,0.3}{v_{\bar{p}}^\updownarrow};
            \coupon{0,-0.3}{v_{\bar{r}}};
            \genbox{-0.5,0.7}{0.2,1.1}{\JW_{k+1}};
            \genbox{-0.8,-0.6}{0.2,-1}{\JW_k};
        \end{tikzpicture}
        \overset{\cref{JWkill}}{=} 0,
    \]
    where the last equality follows from the fact that, since only $k-1$ strands from the $\JW_k$ coupon are connected to the bottom of the $\JW_{k+1}$ coupon, there must be a cup attached to the bottom of the $\JW_{k+1}$ coupon.  The remaining case of $p_{n-1} = k-1$ and $r_{n-1} = k+1$ follows from the last case above after applying $\Omega_\updownarrow$.
\end{proof}

%----------------------
\subsection{Trace maps}
%----------------------

For $0 \le k \le n$, define the \emph{partial trace map}
\[
    \tr^n_{n-k} \colon \TLalg_n \to \TLalg_{n-k},\quad
    \begin{tikzpicture}[centerzero,TL]
        \draw[multi] (0,-0.5) -- (0,0.5);
        \coupon{0,0}{f};
    \end{tikzpicture}
    \mapsto
    \begin{tikzpicture}[centerzero,TL]
        \draw[multi] (0,0.2) to[out=up,in=up,looseness=1.5] (-0.4,0.2) -- (-0.4,0) node[anchor=east] {\strandlabel{k}} -- (-0.4,-0.2) to[out=down,in=down,looseness=1.5] (0,-0.2);
        \draw[multi] (0.3,-0.5) -- (0.3,0.5);
        \genbox{-0.2,-0.2}{0.5,0.2}{f};
    \end{tikzpicture}
    \ .
\]
Note that
\begin{equation} \label{telescope}
    \tr^{n-k}_{n-l} \tr^n_{n-k} = \tr^n_{n-l}\
    \qquad \text{for all} \quad 0 \le k \le l \le n,
\end{equation}
and
\begin{equation} \label{pingpong}
    \tr^n_{n-k} \big( (x \otimes 1_{n-k}) y) = \tr^n_{n-k} \big( y (x \otimes 1_{n-k}) \big)
    \qquad \text{for all } x \in \TLalg_k,\ y \in \TLalg_n.
\end{equation}

\begin{lem}
    For all $p,r \in P^n_k$, $n > 0$, $0 \le k \le n$,
    \begin{equation} \label{spiral}
        \tr^n_{n-1} \big( v_p v_r^\updownarrow \big)
        = \delta_{p_{n-1},r_{n-1}} \left( \frac{\Delta_k}{\Delta_{k-1}} \right)^{\delta_{p_{n-1},k-1}} v_{\bar{p}} v_{\bar{r}}^\updownarrow.
    \end{equation}
\end{lem}

\begin{proof}
    We prove the result by induction on $n$.  The case $n=1$ follows from \cref{Delta,TLrels}.  Now suppose that $p,r \in P^n_k$ for some $n>1$, and that the result holds for $n-1$.  If $p_{n-1} = r_{n-1}  = k-1$, then
    \[
        \tr^n_{n-1} \big( v_p v_r^\updownarrow \big)
        \overset{\cref{v-recursion}}{=}
        \begin{tikzpicture}[centerzero,TL]
            \draw[multi] (0,-1) -- (0,1);
            \draw (-0.3,0.2) to[out=up,in=up,looseness=1.5] (-0.75,0.2) -- (-0.75,-0.2) to[out=down,in=down,looseness=1.5] (-0.3,-0.2);
            \coupon{0,0.65}{v_{\bar{p}}};
            \coupon{0,-0.65}{v_{\bar{r}}};
            \genbox{-0.5,-0.2}{0.3,0.2}{\JW_k};
        \end{tikzpicture}
        \overset{\cref{JWtrace}}{=} \frac{\Delta_k}{\Delta_{k-1}}\
        \begin{tikzpicture}[centerzero,TL]
            \draw[multi] (0,-1) -- (0,1);
            \coupon{0,0.65}{v_{\bar{p}}};
            \coupon{0,-0.65}{v_{\bar{r}}};
            \coupon{0,0}{\JW_{k-1}};
        \end{tikzpicture}
        = \frac{\Delta_k}{\Delta_{k-1}} v_{\bar{p}} v_{\bar{r}}^\updownarrow.
    \]
    If $p_{n-1} = r_{n-1} = k+1$, then
    \[
        \tr^n_{n-1} \big( v_p v_r^\updownarrow \big)
        \overset{\cref{v-recursion}}{=}
        \begin{tikzpicture}[centerzero,TL]
            \draw[multi] (0,-1.35) -- (0,1.35);
            \coupon{0,1}{v_{\bar{p}}};
            \coupon{0,-1}{v_{\bar{r}}};
            \draw (-0.3,0.2) to[out=down,in=down,looseness=1.5] (-0.6,0.2) -- (-0.6,0.7) to[out=up,in=up,looseness=1.5] (-0.85,0.7) -- (-0.85,-0.7) to[out=down,in=down,looseness=1.5] (-0.6,-0.7) -- (-0.6,-0.2) to[out=up,in=up,looseness=1.5] (-0.3,-0.2);
            \genbox{-0.5,0.2}{0.3,0.6}{\JW_{k+1}};
            \genbox{-0.5,-0.2}{0.3,-0.6}{\JW_{k+1}};
        \end{tikzpicture}
        \overset{\cref{JWabsorb}}{=}
        \begin{tikzpicture}[centerzero,TL]
            \draw[multi] (0,-1.35) -- (0,1.35);
            \coupon{0,0.7}{v_{\bar{p}}};
            \coupon{0,-0.7}{v_{\bar{r}}};
            \coupon{0,0}{\JW_{k+1}};
        \end{tikzpicture}
        = v_{\bar{p}} v_{\bar{r}}.
    \]
    If $p_{n-1}=k-1$ and $r_{n-1} = k+1$, then
    \[
        \tr^n_{n-1} \big( v_p v_r^\updownarrow \big)
        \overset{\cref{v-recursion}}{=}
        \begin{tikzpicture}[centerzero,TL]
            \draw[multi] (0,-1.35) -- (0,1.35);
            \coupon{0,1}{v_{\bar{p}}};
            \coupon{0,-1}{v_{\bar{r}}};
            \draw (-0.3,0.6) -- (-0.3,0.8) to[out=up,in=up,looseness=1.5] (-0.85,0.8) -- (-0.85,-0.7) to[out=down,in=down,looseness=1.5] (-0.6,-0.7) -- (-0.6,-0.2) to[out=up,in=up,looseness=1.5] (-0.3,-0.2);
            \genbox{-0.5,0.2}{0.3,0.6}{\JW_k};
            \genbox{-0.5,-0.2}{0.3,-0.6}{\JW_{k+1}};
        \end{tikzpicture}
        \overset{\cref{JWabsorb}}{\underset{\cref{JWkill}}{=}} 0.
    \]
    The remaining case of $p_{n-1} = k+1$ and $r_{n-1} = k-1$ is analogous; we simply reflect all diagrams in a horizontal axis.
\end{proof}

We have a symmetric bilinear form determined by
\begin{equation} \label{formdef}
    \langle -,- \rangle \colon L^n_k \otimes L^n_k \to \kk,\qquad
    \langle v,w \rangle 1_\one
    = \tr^n_0 \big( v w^\updownarrow \big).
\end{equation}

\begin{lem}
    For all $p,r \in P^n_k$, $0 \le k \le n$
    \begin{equation}
        \langle v_p, v_r \rangle = \Delta_k \Delta_p \delta_{p,r},
    \end{equation}
    where $\Delta_p$ is defined as in \cref{crab}.
\end{lem}

\begin{proof}
    We have
    \[
        \langle v_p, v_r \rangle 1_\one
        \overset{\cref{formdef}}{=} \tr^n_0 \left( v w^\updownarrow \right) 1_\one
        \overset{\cref{pingpong}}{=} \tr^n_0 \left( w^\updownarrow v \right) 1_\one
        \overset{\cref{crab}}{=} \delta_{p,r} \Delta_p\
        \begin{tikzpicture}[centerzero,TL]
            \draw[multi] (0,0) -- (0,0.2) to[out=up,in=up,looseness=1.5] (-0.5,0.2) -- (-0.5,-0.2) to[out=down,in=down,looseness=1.5] (0,-0.2) -- (0,0);
            \coupon{0,0}{\JW_k};
        \end{tikzpicture}
        \overset{\cref{JWtrace}}{=} \delta_{p,r} \Delta_k \Delta_p 1_\one.
    \]
    (One can also give a proof using \cref{spiral} and the fact that $\tr^n_0 = \tr^1_0 \tr^2_1 \dotsm \tr^n_{n-1}$, by \cref{telescope}.)
\end{proof}

%------------------------
\subsection{Matrix units}
%------------------------

Define the \emph{matrix units}
\begin{equation} \label{matrixunits}
    e_{p,r} = \Delta_p^{-1} v_p v_r^\updownarrow \in \TLalg_n,\qquad
    p,r \in P^n_k,\quad 0 \le k \le n,\ n \in \N.
\end{equation}
Then $e_{p,r}$, $p,r \in P^n_k$, $0 \le k \le n$, form a basis for $\TLalg_n$, and it follows from \cref{crab} that
\begin{equation} \label{mumult}
    e_{p,r} e_{s,t} = \delta_{r,s} e_{p,t},\qquad
    \text{for all} \quad
    p,r,s,t \in P^n, \text{ such that } p_n = r_n,\ s_n=t_n.
\end{equation}
\details{
    Recall that, under our assumption \cref{generic}, $\TLalg_n$ is split semisimple.  By~\cref{mumult}, for each $k$, the elements $\{e_{p,r}\}_{p,r\in P_k^n}$ satisfy the matrix–unit relations and hence span a subalgebra isomorphic to $M_{|P_k^n|}(\Bbbk) \cong \End_\kk(L^n_k)$.  Taking the union over all admissible $k$ yields a basis of $\TLalg_n$.
}
It follows from \cref{v-recursion,crab} that
\begin{equation} \label{funspiel}
    e_{(0,1,2,\dotsc,n),(0,1,2,\dotsc,n)} = J_n = v_{(0,1,2,\dotsc,n)}.
\end{equation}

Define
\[
    p^+(i_1,\dotsc,i_m) := (p_0,p_1,\dotsc,p_n,i_1,\dotsc,i_m) \in P^{n+m}
\]
for $p \in P^n$, $m \in \N$, and $i_1,\dotsc,i_m$ such that $(p_0,p_1,\dotsc,p_n,i_1,\dotsc,i_m) \in P^{n+m}$.  Then we have the following generalization of \cref{mumult}.

\begin{prop}[General matrix unit multiplication]
    For all $p,r \in P^n$, $s,t \in P^m$ such that $p_n=r_n$, $s_m=t_m$, $n \ge m$,
    \begin{equation} \label{mumultgen}
        e_{p,r} e_{s,t} = \delta_{r[m],s} e_{p,t^+(r_{m+1},\dotsc,r_n)}
        \qquad \text{and} \qquad
        e_{t,s} e_{r,p} = \delta_{r[m],s} e_{t^+(r_{m+1},\dotsc,r_n),p}.
    \end{equation}
\end{prop}

\begin{proof}
    We prove the first equality in \cref{mumultgen} by induction on $n-m$.  When $n=m$, the result is \cref{mumult}.  Now suppose $m=n-1$ and $p_n=r_n=k$.  If $r_{n-1} = k+1$, then
    \[
        e_{p,r} e_{s,t}
        \overset{\cref{matrixunits}}{=} \Delta_p^{-1} \Delta_s^{-1}\
        \begin{tikzpicture}[anchorbase,TL]
            \draw[multi] (-0.15,0.3) -- (-0.15,1.4);
            \draw[multi] (0,-1.4) -- (0,0.3);
            \draw (-0.35,-1.4) -- (-0.35,0.3);
            \genbox{-0.5,0}{0.2,0.6}{v_r^\updownarrow};
            \coupon{0,-0.35}{v_{s}};
            \genbox{-0.5,0.8}{0.2,1.2}{v_p};
            \coupon{0,-1}{v_t^\updownarrow};
        \end{tikzpicture}
        \overset{\cref{v-recursion}}{\underset{\cref{crab}}{=}} \delta_{\bar{r},s} \Delta_p^{-1}\
        \begin{tikzpicture}[anchorbase,TL]
            \draw[multi] (0,-1.4) -- (0,1.4);
            \draw (-0.2,0.2) to[out=up,in=up,looseness=2] (-0.6,0.2) -- (-0.6,-1.4);
            \genbox{-0.5,-0.2}{0.2,0.2}{\JW_{k+1}};
            \coupon{0,1}{v_p};
            \coupon{0,-0.9}{v_t^\updownarrow};
        \end{tikzpicture}
        \overset{\cref{v-recursion}}{=} \delta_{\bar{r},s} \Delta_p^{-1}\
        \begin{tikzpicture}[anchorbase,TL]
            \draw[multi] (0,-1.4) -- (0,1.4);
            \coupon{0,0.5}{v_p};
            \coupon{0,-0.5}{v_{t^+(k)}^\updownarrow};
        \end{tikzpicture}
        \overset{\cref{matrixunits}}{=} \delta_{\bar{r},s} e_{p,t^+(k)},
    \]
    as desired.  On the other hand, if $r_{n-1}=k-1$, then
    \[
        e_{p,r} e_{s,t}
        \overset{\cref{matrixunits}}{=} \Delta_p^{-1} \Delta_s^{-1}\
        \begin{tikzpicture}[anchorbase,TL]
            \draw[multi] (-0.15,0.3) -- (-0.15,1.4);
            \draw[multi] (0,-1.4) -- (0,0.3);
            \draw (-0.35,-1.4) -- (-0.35,0.3);
            \genbox{-0.5,0}{0.2,0.6}{v_r^\updownarrow};
            \coupon{0,-0.35}{v_{s}};
            \genbox{-0.5,0.8}{0.2,1.2}{v_p};
            \coupon{0,-1}{v_t^\updownarrow};
        \end{tikzpicture}
        \overset{\cref{v-recursion}}{\underset{\cref{crab}}{=}} \delta_{\bar{r},s} \Delta_p^{-1}\
        \begin{tikzpicture}[anchorbase,TL]
            \draw[multi] (-0.15,0.3) -- (-0.15,1.4);
            \draw[multi] (0,-1.4) -- (0,0.3);
            \draw (-0.35,-1.4) -- (-0.35,0.3);
            \genbox{-0.5,0}{0.2,0.4}{\JW_k};
            \genbox{-0.5,0.8}{0.2,1.2}{v_p};
            \coupon{0,-0.8}{v_t^\updownarrow};
        \end{tikzpicture}
        \overset{\cref{v-recursion}}{=} \delta_{\bar{r},s} \Delta_p^{-1}\
        \begin{tikzpicture}[anchorbase,TL]
            \draw[multi] (0,-1.4) -- (0,1.4);
            \coupon{0,0.5}{v_p};
            \coupon{0,-0.5}{v_{t^+(k)}^\updownarrow};
        \end{tikzpicture}
        \overset{\cref{matrixunits}}{=} \delta_{\bar{r},s} e_{p,t^+(k)},
    \]
    as desired.
    
    For $n-m \ge 2$, we have, using the induction hypothesis,
    \[
        e_{p,r} e_{s,t}
        = e_{p,r} e_{r[n-1],r[n-1]} e_{s,t}
        = \delta_{r[m],s} e_{p,r} e_{r[n-1],t^+(r_{m+1},\dotsc,r_{n-1})}
        = \delta_{r[m],s} e_{p,t^+(r_{m+1},\dotsc,r_n)},
    \]
    completing the proof of the induction step.  The proof of the second equality in \cref{mumultgen} is analogous.
\end{proof}

\begin{lem}
    For all $0 \le k < n$, and $p,r \in P^n$ such that $p_n=r_n$,
    \begin{equation} \label{mutrace}
        \tr^n_k(e_{p,r})
        = \frac{\Delta_{p_n}}{\Delta_{p_k}} \delta_{p_{n-1},r_{n-1}} \dotsm \delta_{p_k,r_k} e_{p[k],r[k]}.
    \end{equation}
\end{lem}

\begin{proof}
    We prove the result by reverse induction on $k$.  Suppose $p,r \in P^n$ with $p_n=r_n$. The case $k=n-1$ follows from \cref{spiral}.
    \details{
        We have
        \[
            \tr^n_{n-1} \left( e_{p,r} \right)
             \overset{\cref{matrixunits}}{\underset{\cref{spiral}}{=}} \delta_{p_{n-1}, r_{n-1}} \frac{\Delta_{\bar{p}}}{\Delta_p} \left( \frac{\Delta_k}{\Delta_{k-1}} \right)^{\delta_{p_{n-1},k-1}} e_{\bar{p},\bar{r}}
            = \delta_{p_{n-1}, r_{n-1}} \frac{\Delta_{p_n}}{\Delta_{p_{n-1}}} e_{\bar{p},\bar{r}}.
        \]
    }
    Now suppose that $k \le n-2$ and that the result holds for $k+1$.  Then we have
    \begin{multline*}
        \tr^n_k(e_{p,r})
        = \tr^{k+1}_k \tr^n_{k+1}(e_{p,r})
        = \frac{\Delta_{p_n}}{\Delta_{p_{k+1}}} \delta_{p_{n-1},r_{n-1}} \dotsm \delta_{p_{k+1},r_{k+1}} \tr^{k+1}_k \left( e_{p[k+1],r[k+1]} \right) \\
        = \frac{\Delta_{p_n}}{\Delta_{p_k}} \delta_{p_{n-1},r_{n-1}} \dotsm \delta_{p_k,r_k} e_{p[k],r[k]}.
        \qedhere
    \end{multline*}
\end{proof}

%-------------------------------------------------------
\subsection{Dual bases for the Temperley--Lieb algebras}
%-------------------------------------------------------

For $n>0$, define
\begin{equation} \label{salt}
    \bB_n = \left\{ e_{p,r} : p,r \in P^n,\ p_n = r_n \right\} \subseteq \TLalg_n.
\end{equation}
For $p,r \in P^n$ satisfying $p_n=r_n$, and $1 \le k \le n$, define
\begin{equation} \label{mudual}
    e_{p,r}^{\vee,k}
    := \frac{\Delta_{r_{n-k}}}{\Delta_{r_n} |P^{n-k}_{r_{n-k}}|} e_{r,p}
    \in \TLalg_n,
    \qquad e_{p,r}^\vee := e_{p,r}^{\vee,1}.
\end{equation}

\begin{lem} \label{Froblem}
    We have
    \begin{equation} \label{Frobfall}
        \sum_{b \in \bB_n} \tr^n_{n-k}(xb) b^{\vee,k}
        = x
        = \sum_{b \in \bB_n} b \tr^n_{n-k}(b^{\vee,k} x)
        \qquad \text{for all } x \in \TLalg_n,\ 1 \le k \le n.
    \end{equation}    
\end{lem}

\begin{proof}
    It suffices to prove the result for $x = e_{p,r}$, $p,r \in P^n_l$, $0 \le l \le n$.  We have
    \begin{multline*}
        \sum_{b \in \bB_n} \tr^n_{n-k}(xb) b^{\vee,k}
        \overset{\cref{salt}}{\underset{\cref{mudual}}{=}} \sum_{\substack{s,t \in P^n \\ s_n = t_n}} \frac{\Delta_{t_{n-k}}}{\Delta_{t_n} |P^{n-k}_{t_{n-k}}|} \tr^n_{n-k} \left( e_{p,r} e_{s,t} \right) e_{t,s}
        \\
        \overset{\cref{mumult}}{=} \sum_{t \in P^n_l} \frac{\Delta_{t_{n-k}}}{\Delta_l |P^{n-k}_{t_{n-k}}|} \tr^n_{n-k} \left( e_{p,t} \right) e_{t,r}
        \overset{\cref{mutrace}}{=} \sum_{t \in P^n_{p_{n-k},\dotsc,p_{n-1},l}} \frac{1}{|P^{n-k}_{p_{n-k}}|} e_{p[n-k],t[n-k]} e_{t,r}
        \\
        \overset{\cref{mumultgen}}{=} \sum_{t \in P^n_{p_{n-k},\dotsc,p_{n-1},l}} \frac{1}{|P^{n-k}_{p_{n-k}}|} e_{p,r}
        = e_{p,r}.
    \end{multline*}
    The proof of the second equality in \cref{Frobfall} is analogous.
    \details{
        We have
        \begin{multline*}
            \sum_{b \in \bB_n} b \tr^n_{n-k}(b^{\vee,k} x)
            = \sum_{\substack{s,t \in P^n \\ s_n = t_n}} \frac{\Delta_{t_{n-k}}}{\Delta_{t_n} |P^{n-k}_{t_{n-k}}|} e_{s,t} \tr^n_{n-k} \left( e_{t,s} e_{p,r} \right)
            \overset{\cref{mumult}}{=} \sum_{t \in P^n_l} \frac{\Delta_{t_{n-k}}}{\Delta_l |P^{n-k}_{t_{n-k}}|} e_{p,t} \tr^n_{n-k} \left( e_{t,r} \right)
            \\
            \overset{\cref{mutrace}}{=} \sum_{t \in P^n_{r_{n-k},\dotsc,r_{n-1},l}} \frac{1}{|P^{n-k}_{r_{n-k}}|} e_{p,t} e_{t[n-k],r[n-k]}
            \overset{\cref{mumultgen}}{=} \sum_{t \in P^n_{r_{n-k},\dotsc,r_{n-1},l}} \frac{1}{|P^{n-k}_{r_{n-k}}|} e_{p,r}
            = e_{p,r}.
        \end{multline*}
    }
\end{proof}

\Cref{Froblem} implies that $\TLalg_n$ is a Frobenius extension of $\TLalg_{n-k}$.  It follows that
\begin{equation} \label{teleport}
    \sum_{b \in \bB_n} xb \otimes_{n-k} b^{\vee,k} = \sum_{b \in \bB_n} b \otimes_{n-k} b^{\vee,k} x
    \qquad \text{for all } x \in \TLalg_n,\ 1 \le k \le n.
\end{equation}
\details{
    We have
    \begin{multline*}
        \sum_{b \in \bB_n} xb \otimes_{n-k} b^{\vee,k}
        \overset{\cref{Frobfall}}{=} \sum_{b,c \in \bB_n} c \tr^n_{n-k}(c^{\vee,k} xb) \otimes_{n-k} b^{\vee,k}
        \\
        = \sum_{b,c \in \bB_n} c \otimes_{n-k} \tr^n_{n-k}(c^{\vee,k} xb) b^{\vee,k}
        \overset{\cref{Frobfall}}{=} \sum_{c \in \bB_n} c \otimes_{n-k} c^{\vee,k} x.
    \end{multline*}
}

We conclude this section with some results that will be needed later in proofs.

\begin{lem}
    For all $p,r \in P^n$, $s,t \in P^{n-k}$ such that $p_n=r_n$, $s_{n-k}=t_{n-k}$, $k \ge 1$,
    \begin{equation} \label{dualmult}
        e_{s,t}^\vee e_{p,r}^{\vee,k}
        = \frac{\delta_{r[n-k],s}}{|P^{n-k}_{r_{n-k}}|} e_{p,t^+(r_{n-k+1},\dotsc,r_n)}^{\vee,k+1}.
    \end{equation}
\end{lem}

\begin{proof}
    We compute
    \begin{multline*}
        e_{s,t}^\vee e_{p,r}^{\vee,k}
        \overset{\cref{mudual}}{=} \frac{\Delta_{t_{n-k-1}} \Delta_{r_{n-k}}}{\Delta_{t_{n-k}} \Delta_{r_n} |P^{n-k-1}_{t_{n-k-1}}| |P^{n-k}_{r_{n-k}}|} e_{t,s} e_{r,p}
        \\
        \overset{\cref{mumultgen}}{=} \delta_{r[n-k],s} \frac{\Delta_{t_{n-k-1}}}{\Delta_{r_n} |P^{n-k-1}_{t_{n-k-1}}| |P^{n-k}_{r_{n-k}}|} e_{t^+(r_{n-k+1},\dotsc,r_n),p}
        \overset{\cref{mudual}}{=} \frac{\delta_{r[n-k],s}}{|P^{n-k}_{r_{n-k}}|} e_{p,t^+(r_{n-k+1},\dotsc,r_n)}^{\vee,k+1}.
        \qedhere
    \end{multline*}
\end{proof}

\begin{cor}
    For $0 \le k < n$, we have
    \begin{equation} \label{rose}
        \sum_{(b_{(n)},\dotsc,b_{(n-k)}) \in B_n \times \dotsb \times B_{n-k}} b_{(n)} b_{(n-1)} \dotsm b_{(n-k)} \otimes b_{(n-k)}^\vee \dotsm b_{(n-1)}^\vee b_n^\vee
        = \sum_{b \in B_n} b \otimes b^{\vee,k+1}.
    \end{equation}
\end{cor}

\begin{proof}
    We prove the result by induction on $k$.  The case $k=0$ is a tautology.  For $k \ge 1$, using the induction hypothesis, we have
    \begin{multline*}
        \sum_{(b_{(n)},\dotsc,b_{(n-k)}) \in B_n \times \dotsb \times B_{n-k}} b_{(n)} b_{(n-1)} \dotsm b_{(n-k)} \otimes b_{(n-k)}^\vee \dotsm b_{(n-1)}^\vee b_n^\vee
        \\
        = \sum_{\substack{p,r \in P^n,\ s,t \in P^{n-k+1} \\ p_n=r_n,\ s_{n-k+1}=t_{n-k+1}}} e_{p,r} e_{s,t} \otimes e_{s,t}^\vee e_{p,r}^{\vee,k}
        \\
        \overset{\cref{mumultgen}}{\underset{\cref{dualmult}}{=}} \sum_{\substack{p,r \in P^n,\ s,t \in P^{n-k+1} \\ p_n=r_n,\ s_{n-k+1}=t_{n-k+1}}} \frac{\delta_{r[n-k+1],s}}{|P^{n-k+1}_{r_{n-k+1}}|} e_{p,t^+(r_{n-k+2},\dotsc,r_n)} \otimes e_{p,t^+(r_{n-k+2},\dotsc,r_n)}^{\vee,k+1}
        \\
        = \sum_{\substack{p,r \in P^n,\ t \in P^{n-k+1}_{r_{n-k+1}} \\ p_n=r_n}} \frac{1}{|P^{n-k+1}_{r_{n-k+1}}|} e_{p,t^+(r_{n-k+2},\dotsc,r_n)} \otimes e_{p,t^+(r_{n-k+2},\dotsc,r_n)}^{\vee,k+1}
        \\
        = \sum_{\substack{p,r \in P^n \\ p_n=r_n}} e_{p,r} \otimes e_{p,r}^{\vee,k+1}
        = \sum_{b \in B_n} b \otimes b^{\vee,k+1},
    \end{multline*}
    completing the proof of the induction step.
\end{proof}

%===================================================
\section{The incarnation functor\label{sec:functor}}
%===================================================

In this section, we relate the diagrammatic 2-category $\Dcat$ to the representation theory of Temperley--Lieb algebras.  The algebras $\TLalg_n$, $n \in \N$, bimodules between them, and bimodule homomorphisms form a bicategory under relative tensor product.  We will suppress associativity and unit constraints, using the standard coherence theorem for bicategories, and hence view this as a 2-category.  (Equivalently, one may replace this bicategory by a biequivalent strict 2-category.)  We will use the term \emph{2-functor} in what follows to mean \emph{pseudofunctor}.

%--------------------------------
\subsection{Functor to bimodules}
%--------------------------------

We let $\Bcat$ be the $\kk$-linear $2$-category of bimodules over the Temperley--Lieb algebras.  Thus, the set of objects of $\Bcat$ is $\{\obj{n} : n \in \N\}$, and the $1$-morphisms categories are
\[
    \Bcat(\obj{m},\obj{n}) = (\TLalg_n, \TLalg_m)\bimod.
\]
Horizontal composition is given by tensor product of bimodules.  Recall, from \cref{subsec:indres}, that, for $0 \le m,l \le n$, $\TLbim{m}{\TLalg_n}{l}$ is $\TLalg_n$, viewed as a $(\TLalg_m,\TLalg_l)$-bimodule.  Recall also the definition of $e_k$ from \cref{ecupcap}.

We define a $\kk$-linear $2$-functor
\[
    \bF \colon \Dcat \to \Bcat
\]
as follows.  On objects,
\[
    \bF(\obj{n}) = \obj{n},\qquad n \in \N.
\]
On 1-morphisms,
\begin{align*}
    \bF \left( \Sup 1_n \right) &= \TLbim{n+1}{\TLalg_{n+1}}{n},&
    \bF \left( \Sdown 1_{n+1} \right) &= \TLbim{n}{\TLalg_{n+1}}{n+1},
    \\
    \bF \left( \Dup 1_n \right) &= \TLbim{n+2}{\TLalg_{n+2} e_{n+1}}{n},&
    \bF \left( \Ddown 1_{n+2} \right) &= \TLbim{n}{e_{n+1} \TLalg_{n+2}}{n+2}.
\end{align*}
On 2-morphisms
\begin{align*}
    \bF \left( \rightcapreg[S]{n} \right) &\colon \TLbim{n}{\TLalg_n}{n-1} \TLbimr{\TLalg_n}{n} \to \TLbim{n}{\TLalg_n}{n},
    & x \otimes_{n-1} y &\mapsto xy,
    \\
    \bF \left( \rightcupreg[S]{n} \right) &\colon \TLbim{n}{\TLalg_n}{n} \to \TLbim{n}{\TLalg_{n+1}}{n},
    & x &\mapsto x,
    \\
    \bF \left( \leftcapreg[S]{n} \right) &\colon \TLbim{n}{\TLalg_{n+1}}{n} \to \TLbim{n}{\TLalg_n}{n},
    & x &\mapsto \tr^{n+1}_n(x),
    \\
    \bF \left( \leftcupreg[S]{n} \right) &\colon \TLbim{n}{\TLalg_n}{n} \to \TLbim{n}{\TLalg_n}{n-1} \TLbimr{\TLalg_n}{n},
    & x &\mapsto x \sum_{b \in \bB_n} b \otimes_{n-1} b^\vee,
    \\
    \bF \left( \rightcapreg[D]{n} \right) &\colon \TLbim{n}{\TLalg_n e_{n-1}}{n-2} \TLbimr{e_{n-1}\TLalg_n}{n} \to \TLbim{n}{\TLalg_n}{n},
    & x \otimes_{n-2} y &\mapsto xy,
    \\
    \bF \left( \rightcupreg[D]{n} \right) &\colon \TLbim{n}{\TLalg_n}{n} \to \TLbim{n}{e_{n+1}\TLalg_{n+2}e_{n+1}}{n},
    & x &\mapsto \delta^{-1} e_{n+1} x = \delta^{-1} x e_{n+1},
    \\
    \bF \left( \leftcapreg[D]{n} \right) &\colon \TLbim{n}{e_{n+1}\TLalg_{n+2}e_{n+1}}{n} \to \TLbim{n}{\TLalg_n}{n},
    & x &\mapsto \tr^{n+2}_n(x),
    \\
    \bF \left( \leftcupreg[D]{n} \right) &\colon \TLbim{n}{\TLalg_n}{n} \to \TLbim{n}{\TLalg_n e_{n-1}}{n-2} \TLbimr{e_{n-1}\TLalg_n}{n},
    & x &\mapsto \delta^{-1} x \sum_{b \in \bB_n} b e_{n-1} \otimes_{n-2} b^{\vee,2},
    \\
    \bF \left( \mergeupreg{n} \right) &\colon \TLbim{n+2}{\TLalg_{n+2}}{n} \to \TLbim{n+2}{\TLalg_{n+2} e_{n+1}}{n},
    & x &\mapsto x e_{n+1},
    \\
    \bF \left( \splitupreg{n} \right) &\colon \TLbim{n+2}{\TLalg_{n+2} e_{n+1}}{n} \to \TLbim{n+2}{\TLalg_{n+2}}{n},
    & x &\mapsto x,
\end{align*}

We now verify that $\bF$ respects the defining relations of $\Dcat$.  For $x \in \TLalg_{n+1}$, we have
\[
    \bF
    \left(
        \begin{tikzpicture}[centerzero,S]
            \draw[->] (-0.3,-0.4) -- (-0.3,0) arc(180:0:0.15) arc(180:360:0.15) -- (0.3,0.4);
            \region{-0.1,-0.2}{n};
        \end{tikzpicture}
    \right)
    \colon x \mapsto x \otimes_n 1 \mapsto x
    =
    \bF
    \left(
        \begin{tikzpicture}[centerzero,S]
            \draw[->] (0,-0.4) -- (0.,0.4);
            \region{0.2,0}{n};
        \end{tikzpicture}
    \right)
    (x)
\]
and
\[
    \bF
    \left(
        \begin{tikzpicture}[centerzero,S]
            \draw[<-] (-0.3,0.4) -- (-0.3,0) arc(180:360:0.15) arc(180:0:0.15) -- (0.3,-0.4);
            \region{-0.1,0.2}{n};
        \end{tikzpicture}
    \right)
    \colon x \mapsto \sum_{b \in \bB_{n+1}} b \otimes_n b^\vee x
    \\
    \mapsto \sum_{b \in \bB_{n+1}} b \tr^{n+1}_n (b^\vee x)
    \overset{\cref{Frobfall}}{=} x
    =
    \bF
    \left(
        \begin{tikzpicture}[centerzero,S]
            \draw[->] (0,-0.4) -- (0.,0.4);
            \region{0.2,0}{n};
        \end{tikzpicture}
    \right)
    (x).
\]
For $x \in \TLalg_{n+2}$, we have
\[
    \bF
    \left(
        \begin{tikzpicture}[centerzero,D]
            \draw[->] (-0.3,-0.4) -- (-0.3,0) arc(180:0:0.15) arc(180:360:0.15) -- (0.3,0.4);
            \region{-0.1,-0.2}{n};
        \end{tikzpicture}
    \right)
    \colon x e_{n+1} \mapsto \delta^{-1} x e_{n+1} \otimes_n e_{n+1} \mapsto x e_{n+1}
    =
    \bF
    \left(
        \begin{tikzpicture}[centerzero,D]
            \draw[->] (0,-0.4) -- (0.,0.4);
            \region{0.2,0}{n};
        \end{tikzpicture}
    \right)
    (x e_{n+1}),
\]
and
\begin{multline*}
    \bF
    \left(
        \begin{tikzpicture}[centerzero,D]
            \draw[<-] (-0.3,0.4) -- (-0.3,0) arc(180:360:0.15) arc(180:0:0.15) -- (0.3,-0.4);
            \region{-0.1,0.2}{n};
        \end{tikzpicture}
    \right)
    \colon x e_{n+1} \mapsto \sum_{b \in \bB_{n+2}} b e_{n+1} \otimes_n e_{n+1} b^{\vee,2} x e_{n+1}
    \\
    \mapsto \delta^{-2} \sum_{b \in \bB_{n+2}} b e_{n+1} \tr^{n+2}_n (e_{n+1} b^{\vee,2} x e_{n+1})
    \overset{\cref{pingpong}}{=} \delta^{-1} \sum_{b \in \bB_{n+2}} b e_{n+1} \tr^{n+2}_n (b^{\vee,2} x e_{n+1})
    \\
    = \delta^{-1} \sum_{b \in \bB_{n+2,2}} b \tr^{n+2}_n (b^\vee x e_{n+1}) e_{n+1}
    \overset{\cref{Frobfall}}{=} \delta^{-1} x e_{n+1}^2
    = x e_{n+1}
    =
    \bF
    \left(
        \begin{tikzpicture}[centerzero,D]
            \draw[->] (0,-0.4) -- (0.,0.4);
            \region{0.2,0}{n};
        \end{tikzpicture}
    \right)
    (x e_{n+1}),
\end{multline*}
where the second equality uses the fact that $e_{n+1}$ commutes with elements of $\TLalg_n$.  Thus, we have verified the relations \cref{adjunction}.

Next, for $x \in \TLbim{n}{e_{n+1} \TLalg_{n+2}}{n+2}$,
\begin{multline*}
    \bF
    \left(
        \begin{tikzpicture}[anchorbase]
            \draw[D,->] (0,0) -- (0,0.1) arc (0:180:0.2) -- (-0.4,-0.6);
            \draw[S] (0,0) to[out=-45, in=left] (0.2,-0.2) arc(-90:0:0.15) -- (0.35,0.4);
            \draw[S] (0,0) to[out=-135,in=up] (-0.2,-0.3) to[out=down,in=down] (0.6,-0.3) -- (0.6,0.4);
            \region{0.15,0.2}{n};
        \end{tikzpicture}
    \right)
    \colon x
    \mapsto \sum_{b \in \bB_{n+2},\ c \in \bB_{n+1}} x bc \otimes_n c^\vee b^\vee
    \overset{\cref{rose}}{=} \sum_{b \in \bB_{n+2}} xb \otimes_n b^{\vee,2}
    \mapsto \sum_{b \in \bB_{n+2}} x b e_{n+1} \otimes_n b^{\vee,2}
    \\
    \mapsto \sum_{b \in \bB_{n+2}} \tr^{n+2}_n (x b e_{n+1}) b^{\vee,2}
    \overset{\cref{pingpong}}{=} \delta \sum_{b \in \bB_{n+2}} \tr^{n+2}_n (e_{n+1} x b) b^{\vee,2}
    \overset{\cref{Frobfall}}{=} e_{n+1} x
    = \delta x
\end{multline*}
and
\[
    \bF
    \left(
        \begin{tikzpicture}[anchorbase,xscale=-1]
            \draw[D,->] (0,0) -- (0,0.1) arc (0:180:0.2) -- (-0.4,-0.6);
            \draw[S] (0,0) to[out=-45, in=left] (0.2,-0.2) arc(-90:0:0.15) -- (0.35,0.4);
            \draw[S] (0,0) to[out=-135,in=up] (-0.2,-0.3) to[out=down,in=down] (0.6,-0.3) -- (0.6,0.4);
            \region{-0.2,0.1}{n};
        \end{tikzpicture}
    \right)
    \colon x
    \mapsto 1 \otimes_n x
    \mapsto e_{n+1} \otimes_n x
    \mapsto e_{n+1} x
    = \delta x,
\]
verifying the first relation in \cref{swishy}.  Similarly, for $x \in \TLbim{n}{\TLalg_{n+2}}{n+2}$,
\[
    \bF
    \left(
        \begin{tikzpicture}[anchorbase,yscale=-1]
            \draw[D] (0,0) -- (0,0.1) arc (0:180:0.2) -- (-0.4,-0.6);
            \draw[S,->] (0,0) to[out=-45, in=left] (0.2,-0.2) arc(-90:0:0.15) -- (0.35,0.4);
            \draw[S,->] (0,0) to[out=-135,in=up] (-0.2,-0.3) to[out=down,in=down] (0.6,-0.3) -- (0.6,0.4);
            \region{0.18,0.1}{n};
        \end{tikzpicture}
    \right)
    \colon x \mapsto \delta^{-1} e_{n+1} \otimes_n x
    \mapsto \delta^{-1} e_{n+1} x
\]
and
\begin{multline*}
    \bF
    \left(
        \begin{tikzpicture}[anchorbase,xscale=-1,yscale=-1]
            \draw[D] (0,0) -- (0,0.1) arc (0:180:0.2) -- (-0.4,-0.6);
            \draw[S,->] (0,0) to[out=-45, in=left] (0.2,-0.2) arc(-90:0:0.15) -- (0.35,0.4);
            \draw[S,->] (0,0) to[out=-135,in=up] (-0.2,-0.3) to[out=down,in=down] (0.6,-0.3) -- (0.6,0.4);
            \region{-0.2,0.05}{n};
        \end{tikzpicture}
    \right)
    \colon x \mapsto \delta^{-2} \sum_{b \in \bB_{n+2}} x b e_{n+1} \otimes_n e_{n+1} b^{\vee,2}
    \mapsto \delta^{-2} \sum_{b \in \bB_{n+2}} \tr^{n+2}_n(x b e_{n+1}) e_{n+1} b^{\vee,2}
    \\
    \overset{\cref{pingpong}}{=} \delta^{-2} \sum_{b \in \bB_{n+2}} e_{n+1} \tr^{n+2}_n(e_{n+1} x b) b^{\vee,2}
    \overset{\cref{Frobfall}}{=} \delta^{-2} e_{n+1} e_{n+1} x
    = \delta^{-1} e_{n+1} x,
\end{multline*}
where the first equality uses the fact that $e_{n+1}$ commutes with elements of $\TLalg_n$.  This verifies the second relation in \cref{swishy}.

Using \cref{fishy1,fishy2}, we now compute
\begin{gather*}
    \bF \left( \mergeSWreg{n} \right)
    =
    \bF
    \left(
        \begin{tikzpicture}[anchorbase]
            \draw[D,->] (0,0) -- (0,0.1) arc (0:180:0.2) -- (-0.4,-0.4);
            \draw[S] (0,0) to[out=-45, in=left] (0.2,-0.2) arc(-90:0:0.15) -- (0.35,0.4);
            \draw[S] (0,0) to[out=-135,in=up] (-0.2,-0.4);
            \region{0,-0.25}{n};
        \end{tikzpicture}
    \right)
    \colon \TLbim{n-1}{e_n \TLalg_{n+1}}{n} \to \TLbim{n-1}{\TLalg_n}{n},\quad
    x \mapsto \sum_{b \in \bB_n} \tr^{n+1}_{n-1}(xb) b^\vee,
    \\
    \bF \left( \mergeSEreg{n} \right)
    =
    \bF
    \left(
        \begin{tikzpicture}[anchorbase,xscale=-1]
            \draw[D,->] (0,0) -- (0,0.1) arc (0:180:0.2) -- (-0.4,-0.4);
            \draw[S] (0,0) to[out=-45, in=left] (0.2,-0.2) arc(-90:0:0.15) -- (0.35,0.4);
            \draw[S] (0,0) to[out=-135,in=up] (-0.2,-0.4);
            \region{0.2,0.1}{n};
        \end{tikzpicture}
    \right)
    \colon \TLbim{n-1}{\TLalg_{n-1}}{n-2} \TLbimr{e_{n-1} \TLalg_n}{n} \to \TLbim{n-1}{\TLalg_n}{n},\quad
    x \otimes_{n-2} y \mapsto x e_{n-1} y
    \\
    \bF \left( \splitNEreg{n} \right)
    =
    \bF
    \left(
        \begin{tikzpicture}[anchorbase]
            \draw[D] (0,-0.4) -- (0,0);
            \draw[S,->] (0,0) -- (-0.3,0.3);
            \draw[S,->] (0,0) to[out=45,in=left] (0.2,0.2) to[out=right,in=up] (0.4,-0.4);
            \region{0.6,0}{n};
        \end{tikzpicture}
    \right)
    \colon \TLbim{n+1}{\TLalg_{n+1} e_n}{n-1} \TLbimr{\TLalg_n}{n} \to \TLbim{n+1}{\TLalg_{n+1}}{n},\quad
    x \otimes_{n-1} y \mapsto xy,
    \\
    \bF \left( \splitNWreg{n} \right)
    =
    \bF
    \left(
        \begin{tikzpicture}[anchorbase,xscale=-1]
            \draw[D] (0,-0.4) -- (0,0);
            \draw[S,->] (0,0) -- (-0.3,0.3);
            \draw[S,->] (0,0) to[out=45,in=left] (0.2,0.2) to[out=right,in=up] (0.4,-0.4);
            \region{-0.2,-0.1}{n};
        \end{tikzpicture}
    \right)
    \colon \TLbim{n+1}{\TLalg_{n+2} e_{n+1}}{n} \to \TLbim{n+1}{\TLalg_{n+1}}{n},\quad
    x \mapsto \tr^{n+2}_{n+1}(x),
    \\
    \bF \left( \splitSEreg{n} \right)
    =
    \bF
    \left(
        \begin{tikzpicture}[anchorbase,yscale=-1]
            \draw[D] (0,0) -- (0,0.1) arc (0:180:0.2) -- (-0.4,-0.4);
            \draw[S,->] (0,0) to[out=-45, in=left] (0.2,-0.2) arc(-90:0:0.15) -- (0.35,0.4);
            \draw[S,->] (0,0) to[out=-135,in=up] (-0.2,-0.4);
            \region{0.55,0}{n};
        \end{tikzpicture}
    \right)
    \colon \TLbim{n-1}{\TLalg_n}{n} \to \TLbim{n-1}{e_n \TLalg_{n+1}}{n},\quad
    x \mapsto \delta^{-1} e_n x, 
    \\
    \begin{multlined}
        \bF \left( \splitSWreg{n} \right)
        =
        \bF
        \left(
            \begin{tikzpicture}[anchorbase,yscale=-1,xscale=-1]
                \draw[D] (0,0) -- (0,0.1) arc (0:180:0.2) -- (-0.4,-0.4);
                \draw[S,->] (0,0) to[out=-45, in=left] (0.2,-0.2) arc(-90:0:0.15) -- (0.35,0.4);
                \draw[S,->] (0,0) to[out=-135,in=up] (-0.2,-0.4);
                \region{0.2,0.1}{n};
            \end{tikzpicture}
        \right)
        \colon \TLbim{n-1}{\TLalg_n}{n} \to \TLbim{n-1}{\TLalg_n}{n-2} \TLbimr{e_{n-1} \TLalg_n}{n},
        \\
        x \mapsto \delta^{-2} \sum_{b \in \bB_n} \tr^n_{n-1}(x b e_{n-1}) e_{n-1} b^{\vee,2},
    \end{multlined}
    \\
    \bF \left( \mergeNWreg{n} \right)
    =
    \bF
    \left(
        \begin{tikzpicture}[anchorbase]
            \draw[D,->] (0,0) -- (0,0.4);
            \draw[S] (0,0) -- (-0.3,-0.3);
            \draw[S] (0,0) to[out=-45,in=left] (0.2,-0.2) to[out=right,in=down] (0.4,0.4);
            \region{0,-0.25}{n};
        \end{tikzpicture}
    \right)
    \colon \TLbim{n+1}{\TLalg_{n+1}}{n} \to \TLbim{n+1}{\TLalg_{n+1} e_n}{n-1} \TLbimr{\TLalg_n}{n},\quad
    x \mapsto x \sum_{b \in \bB_n}  b e_n \otimes_{n-1} b^\vee,
    \\
    \bF \left( \mergeNEreg{n} \right)
    =
    \bF
    \left(
        \begin{tikzpicture}[anchorbase,xscale=-1]
            \draw[D,->] (0,0) -- (0,0.4);
            \draw[S] (0,0) -- (-0.3,-0.3);
            \draw[S] (0,0) to[out=-45,in=left] (0.2,-0.2) to[out=right,in=down] (0.4,0.4);
            \region{-0.25,0.1}{n};
        \end{tikzpicture}
    \right)
    \colon \TLbim{n+1}{\TLalg_{n+1}}{n} \to \TLbim{n+1}{\TLalg_{n+2} e_{n+1}}{n},\quad
    x \mapsto x e_{n+1}.
\end{gather*}

We have
\begin{gather*}
    \bF
    \left(
        \begin{tikzpicture}[centerzero]
            \draw[D] (-0.3,-0.5) -- (0,-0.2);
            \draw[S,<-] (0.3,-0.5) -- (0,-0.2);
            \draw[S] (0,-0.2) -- (0,0.2);
            \draw[S] (0,0.2) -- (0.3,0.5);
            \draw[D,->] (0,0.2) -- (-0.3,0.5);
            \region{0.2,0}{n};
        \end{tikzpicture}
    \right)
    \colon \TLbim{n+1}{\TLalg_{n+1} e_n}{n-1} \TLbimr{\TLalg_n}{n} \to \TLbim{n+1}{\TLalg_{n+1}}{n-1} \TLbimr{\TLalg_n}{n},
    \\
    x e_n \otimes_{n-1} y \mapsto x e_n y \mapsto \sum_{b \in \bB_n} x e_n y b e_n \otimes_{n-1} b^\vee.
\end{gather*}
Noting that
\[
    x e_n y b e_n
    =
    \begin{tikzpicture}[anchorbase,TL]
        \draw[multi] (-0.05,1.8) -- (-0.05,2.3);
        \draw[multi] (0.3,0.7) -- (0.3,1.8);
        \draw[multi] (0.15,0) -- (0.15,0.7);
        \draw[multi] (0.3,-0.8) -- (0.3,-0.2);
        \draw (0,0.9) to[out=up,in=up,looseness=2] (-0.4,0.9) -- (-0.4,-0.2) to[out=down,in=down,looseness=2] (0,-0.2);
        \draw (-0.4,1.6) to[out=down,in=down,looseness=2] (0,1.6);
        \draw (-0.4,-0.8) to[out=up,in=up,looseness=2] (0,-0.8);
        \genbox{-0.2,-0.2}{0.5,0.2}{b};
        \genbox{-0.2,0.5}{0.5,0.9}{y};
        \genbox{-0.6,1.6}{0.5,2}{x};
    \end{tikzpicture}
    = x e_n \tr^n_{n-1}(yb),
\]
we see that
\[
    \bF
    \left(
        \begin{tikzpicture}[centerzero]
            \draw[D] (-0.3,-0.5) -- (0,-0.2);
            \draw[S,<-] (0.3,-0.5) -- (0,-0.2);
            \draw[S] (0,-0.2) -- (0,0.2);
            \draw[S] (0,0.2) -- (0.3,0.5);
            \draw[D,->] (0,0.2) -- (-0.3,0.5);
            \region{0.2,0}{n};
        \end{tikzpicture}
    \right)
    \colon x e_n \otimes_{n-1} y \mapsto \sum_{b \in \bB_n}  x e_n \tr^n_{n-1}(yb) \otimes_{n-1} b^\vee
    \overset{\cref{Frobfall}}{=} x e_n \otimes_{n-1} y.
\]
Thus, $\bF$ respects the first relation in \cref{fire}.
For $x \in \TLalg_{n+2}$, we have
\begin{gather*}
    \bF
    \left(
        \begin{tikzpicture}[centerzero]
            \draw[D] (0.3,-0.5) -- (0,-0.2);
            \draw[S,<-] (-0.3,-0.5) -- (0,-0.2);
            \draw[S] (0,-0.2) -- (0,0.2);
            \draw[S] (0,0.2) -- (-0.3,0.5);
            \draw[D,->] (0,0.2) -- (0.3,0.5);
            \region{0.2,0}{n};
        \end{tikzpicture}
    \right)
    \colon \TLbim{n+1}{\TLalg_{n+2} e_{n+1}}{n} \to \TLbim{n+1}{\TLalg_{n+2} e_{n+1}}{n},
    \\
    x e_{n+1} =
    \begin{tikzpicture}[centerzero,TL]
        \draw (-0.3,-0.2) to[out=down,in=down,looseness=2] (0,-0.2);
        \draw (-0.3,-0.7) to[out=up,in=up,looseness=2] (0,-0.7);
        \draw[multi] (0.3,-0.7) -- (0.3,0);
        \draw[multi] (0,0) -- (0,0.7);
        \genbox{-0.5,-0.2}{0.5,0.2}{x};
    \end{tikzpicture}
    \mapsto
    \begin{tikzpicture}[centerzero,TL]
        \draw (-0.3,-0.2) to[out=down,in=down,looseness=2] (0,-0.2);
        \draw (-0.3,0.2) to[out=up,in=up,looseness=2] (-0.7,0.2) -- (-0.7,-0.7) to[out=down,in=down,looseness=2] (-0.3,-0.7) to[out=up,in=up,looseness=2] (0,-0.7) -- (0,-1);
        \draw[multi] (0.3,-1) -- (0.3,0);
        \draw[multi] (0.1,0) -- (0.1,0.7);
        \genbox{-0.5,-0.2}{0.5,0.2}{x};
    \end{tikzpicture}
    \mapsto
    \begin{tikzpicture}[centerzero,TL]
        \draw (-0.3,-0.2) to[out=down,in=down,looseness=2] (0,-0.2);
        \draw (-0.3,0.2) to[out=up,in=up,looseness=2] (-0.7,0.2) -- (-0.7,-0.7) to[out=down,in=down,looseness=2] (-0.3,-0.7) to[out=up,in=up,looseness=2] (0,-0.7) to[out=down,in=down,looseness=2] (-1,-0.7) -- (-1,0.7);
        \draw (0,-1.5) to[out=up,in=up,looseness=2] (-0.3,-1.5);
        \draw[multi] (0.3,-1.5) -- (0.3,0);
        \draw[multi] (0.1,0) -- (0.1,0.7);
        \genbox{-0.5,-0.2}{0.5,0.2}{x};
    \end{tikzpicture}
    =
    \begin{tikzpicture}[centerzero,TL]
        \draw (-0.3,-0.2) to[out=down,in=down,looseness=2] (0,-0.2);
        \draw (-0.3,-0.7) to[out=up,in=up,looseness=2] (0,-0.7);
        \draw[multi] (0.3,-0.7) -- (0.3,0);
        \draw[multi] (0,0) -- (0,0.7);
        \genbox{-0.5,-0.2}{0.5,0.2}{x};
    \end{tikzpicture}
    \ .
\end{gather*}
Thus, $\bF$ respects the second relation in \cref{fire}.

Next, for $x \in \TLalg_{n+2}$, we compute
\begin{gather*}
    \bF
    \left(
        \begin{tikzpicture}[centerzero]
            \draw[D,<-] (-0.2,-0.5) -- (-0.2,-0.3) arc(180:0:0.2) -- (0.2,-0.5);
            \draw[D,->] (-0.2,0.5) -- (-0.2,0.3) arc(180:360:0.2) -- (0.2,0.5);
            \region{0.3,0}{n};
        \end{tikzpicture}
    \right)
    \colon \TLbim{n}{e_{n+1} \TLalg_{n+2} e_{n+1}}{n} \to \TLbim{n}{e_{n+1} \TLalg_{n+2} e_{n+1}}{n},
    \\
    e_{n+1} x e_{n+1}
    =
    \begin{tikzpicture}[centerzero,TL]
        \draw (-0.3,0.2) to[out=up,in=up,looseness=2] (0,0.2);
        \draw (-0.3,-0.2) to[out=down,in=down,looseness=2] (0,-0.2);
        \draw (-0.3,0.7) to[out=down,in=down,looseness=2] (0,0.7);
        \draw (-0.3,-0.7) to[out=up,in=up,looseness=2] (0,-0.7);
        \draw[multi] (0.3,-0.7) -- (0.3,0.7);
        \genbox{-0.5,-0.2}{0.5,0.2}{x};
    \end{tikzpicture}
    \mapsto
    \begin{tikzpicture}[centerzero,TL]
        \draw (-0.3,0.2) to[out=up,in=up,looseness=2] (0,0.2);
        \draw (-0.3,-0.2) to[out=down,in=down,looseness=2] (0,-0.2);
        \draw (-0.3,0.7) to[out=down,in=down,looseness=2] (0,0.7) to[out=up,in=up] (-1,0.7) -- (-1,-0.7) to[out=down,in=down] (0,-0.7) to[out=up,in=up,looseness=2] (-0.3,-0.7) to[out=down,in=down] (-0.7,-0.7) -- (-0.7,0.7) to[out=up,in=up] (-0.3,0.7);
        \draw[multi] (0.3,-1) -- (0.3,1);
        \genbox{-0.5,-0.2}{0.5,0.2}{x};
    \end{tikzpicture}
    = \delta\
    \begin{tikzpicture}[centerzero,TL]
        \draw (-0.3,0.2) to[out=up,in=up,looseness=2] (0,0.2);
        \draw (-0.3,-0.2) to[out=down,in=down,looseness=2] (0,-0.2);
        \draw[multi] (0.3,-0.7) -- (0.3,0.7);
        \genbox{-0.5,-0.2}{0.5,0.2}{x};
    \end{tikzpicture}
    \mapsto
    \begin{tikzpicture}[centerzero,TL]
        \draw (-0.3,0.2) to[out=up,in=up,looseness=2] (0,0.2);
        \draw (-0.3,-0.2) to[out=down,in=down,looseness=2] (0,-0.2);
        \draw (-0.3,0.7) to[out=down,in=down,looseness=2] (0,0.7);
        \draw (-0.3,-0.7) to[out=up,in=up,looseness=2] (0,-0.7);
        \draw[multi] (0.3,-0.7) -- (0.3,0.7);
        \genbox{-0.5,-0.2}{0.5,0.2}{x};
    \end{tikzpicture}
    \ .
\end{gather*}
Thus, $\bF$ respects the last relation in \cref{fire}.

For $x \in \TLbim{n}{\TLalg_n}{n}$, we have
\[
    \bF \left( \leftbubreg[S]{n} \right)
    \colon
    \begin{tikzpicture}[centerzero,TL]
        \draw[multi] (0,-0.4) -- (0,0.4);
        \coupon{0,0}{x};
    \end{tikzpicture}
    \mapsto
    \begin{tikzpicture}[centerzero,TL]
        \draw[multi] (0,-0.4) -- (0,0.4);
        \draw (-0.3,-0.4) -- (-0.3,0.4);
        \coupon{0,0}{x};
    \end{tikzpicture}
    \mapsto
    \begin{tikzpicture}[centerzero,TL]
        \draw[multi] (0,-0.4) -- (0,0.4);
        \coupon{0,0}{x};
        \draw (-0.3,0) arc(0:360:0.2);
    \end{tikzpicture}
    = \delta\
    \begin{tikzpicture}[centerzero,TL]
        \draw[multi] (0,-0.4) -- (0,0.4);
        \coupon{0,0}{x};
    \end{tikzpicture}
    \ .
\]
Thus, $\bF$ satisfies the first relation in \cref{smoke}.

For $x \in \TLbim{n}{\TLalg_n}{n}$, we have
\[
    \bF \left( \leftbubreg[D]{n} \right)
    \colon
    \begin{tikzpicture}[centerzero,TL]
        \draw[multi] (0,-0.4) -- (0,0.4);
        \coupon{0,0}{x};
    \end{tikzpicture}
    \mapsto \delta^{-1}\
    \begin{tikzpicture}[centerzero,TL]
        \draw[multi] (0,-0.4) -- (0,0.4);
        \coupon{0,0}{x};
        \draw (-0.3,0.4) -- (-0.3,0.3) to[out=down,in=down,looseness=2] (-0.6,0.3) -- (-0.6,0.4);
        \draw (-0.3,-0.4) -- (-0.3,-0.3) to[out=up,in=up,looseness=2] (-0.6,-0.3) -- (-0.6,-0.4);
    \end{tikzpicture}
    \mapsto \delta^{-1}\
    \begin{tikzpicture}[centerzero,TL]
        \draw[multi] (0,-0.7) -- (0,0.7);
        \coupon{0,0}{x};
        \draw (-0.3,0.4) -- (-0.3,0.3) to[out=down,in=down,looseness=2] (-0.6,0.3) -- (-0.6,0.4) to[out=up,in=up] (-0.9,0.4) -- (-0.9,-0.4) to[out=down,in=down] (-0.6,-0.4) -- (-0.6,-0.3) to[out=up,in=up,looseness=2] (-0.3,-0.3) -- (-0.3,-0.4) to[out=down,in=down] (-1.2,-0.4) -- (-1.2,0.4) to[out=up,in=up] (-0.3,0.4);
    \end{tikzpicture}
    =
    \begin{tikzpicture}[centerzero,TL]
        \draw[multi] (0,-0.4) -- (0,0.4);
        \coupon{0,0}{x};
    \end{tikzpicture}
    \ .
\]
Thus, $\bF$ respects the second relation in \cref{smoke}.

Before verifying the remaining relations, we prove two results.

\begin{prop} \label{ducky}
    We have
    \begin{equation} \label{duck}
        \bF( \rightbubmultreg[D]{r}{n} ) = \sum_{\substack{p \in P^n \\ p_n \le n-2r}} e_{p,p},\qquad r,n \in \N,\ n \ge 2r.
    \end{equation}
\end{prop}

\begin{proof}
    We prove the result by induction on $r$.  The case $r=0$ is trivial.  Now suppose that \cref{duck} holds for some $r \in \N$.  Then, for $x \in \TLalg_{n+2}$, we have, using \cref{salt,mudual},
    \begin{multline*}
        \bF \left( \rightbubmultreg[D]{r+1}{n+2} \right)
        \overset{\cref{bubrec}}{=} \bF \left( \rightbubmult[D]{\rightbubmultreg[D]{r}{n}} \right)
        \colon x \mapsto \delta^{-1} x \sum_{\substack{s,t \in P^{n+2} \\ s_{n+2} = t_{n+2}}} \sum_{\substack{p \in P^n \\ p_n \le n-2r}} \frac{\Delta_{t_n}}{\Delta_{t_{n+2}} |P^n_{t_{n}}|} e_{s,t} e_{n+1} e_{p,p} e_{t,s}
        \\
        \overset{\cref{mumultgen}}{=} \delta^{-1} x \sum_{\substack{s,t \in P^{n+2} \\ s_{n+2} = t_{n+2} \\ t_n \le n-2r}} \frac{\Delta_{t_n}}{\Delta_{t_{n+2}} |P^n_{t_{n}}|} e_{s,t} e_{n+1} e_{t,s}
        \overset{\cref{mumultgen}}{=} \delta^{-1} x \sum_{\substack{s,t \in P^{n+2} \\ s_{n+2} = t_{n+2} \\ t_n \le n-2r}} \frac{\Delta_{t_n}}{\Delta_s \Delta_t \Delta_{t_{n+2}} |P^n_{t_{n}}|} v_s v_t^\updownarrow e_{n+1} v_t v_s^\updownarrow.
    \end{multline*}
    
    Let $k=s_{n+2}=t_{n+2}$ and consider the four possibilities for the pair $(t_{n+1},t_n)$ in the final sum above.  By \cref{cloud}, $e_{n+1} v_t=0$ unless $t_{n+2} = t_n$.  If $(t_n,t_{n+1}) = (k,k-1)$, then
    \[
        e_{n+1} v_t
        \overset{\cref{v-recursion}}{=} 
        \begin{tikzpicture}[anchorbase,TL]
            \draw[multi] (0,-0.3) -- (0,0.8);
            \draw[multi] (0.2,-1) -- (0.2,-0.5);
            \draw[multi] (-0.4,-1.4) -- (-0.4,-1.7);
            \draw (-0.5,0.8) -- (-0.5,0.6) to[out=down,in=down,looseness=2] (-0.8,0.6) -- (-0.8,0.8);
            \draw (-0.2,-0.6) to[out=down,in=down,looseness=2] (-0.6,-0.6) -- (-0.6,-0.2) to[out=up,in=up,looseness=2] (-0.9,-0.2) -- (-0.9,-1);
            \genbox{-0.4,-0.6}{0.4,-0.2}{\JW_k};
            \genbox{-1.2,-1}{0.4,-1.4}{\JW_k};
            \coupon{0,0.3}{v_{t[n]}};
        \end{tikzpicture}
        \overset{\cref{JWabsorb}}{=}
        \begin{tikzpicture}[anchorbase,TL]
            \draw[multi] (0,-0.9) -- (0,0.8);
            \draw (-0.5,0.8) -- (-0.5,0.6) to[out=down,in=down,looseness=2] (-0.8,0.6) -- (-0.8,0.8);
            \coupon{0,-0.4}{\JW_k};
            \coupon{0,0.3}{v_{t[n]}};
        \end{tikzpicture}
        \ ,
    \]
    and so
    \[
        v_t^\updownarrow e_{n+1} v_t = v_{t[n]}^\updownarrow v_{t[n]}
        \overset{\cref{crab}}{=} \Delta_{t[n]} J_k
        \overset{\cref{crab}}{=} \frac{\Delta_{k-1}}{\Delta_k} \Delta_t J_k
        = \frac{\Delta_{t_{n+1}}}{\Delta_{t_n}} \Delta_t J_k.
    \]
    If $(t_n,t_{n+1}) = (k,k+1)$, then
    \[
        e_{n+1} v_t
        \overset{\cref{v-recursion}}{=}
        \begin{tikzpicture}[anchorbase,TL]
            \draw[multi] (0,-0.3) -- (0,0.8);
            \draw[multi] (-0.15,-0.9) -- (-0.15,-0.5);
            \draw (-0.5,0.8) -- (-0.5,0.6) to[out=down,in=down,looseness=2] (-0.8,0.6) -- (-0.8,0.8);
            \draw (-0.5,-0.2) -- (-0.5,-0.1) to[out=up,in=up,looseness=2] (-0.8,-0.1) -- (-0.8,-0.7) to[out=down,in=down,looseness=2] (-0.5,-0.7) -- (-0.5,-0.6);
            \genbox{-0.65,-0.6}{0.25,-0.2}{\JW_{k+1}};
            \coupon{0,0.3}{v_{t[n]}};
        \end{tikzpicture}
        \overset{\cref{JWtrace}}{=} \frac{\Delta_{k+1}}{\Delta_k}\
        \begin{tikzpicture}[anchorbase,TL]
            \draw[multi] (0,-0.9) -- (0,0.8);
            \draw (-0.5,0.8) -- (-0.5,0.6) to[out=down,in=down,looseness=2] (-0.8,0.6) -- (-0.8,0.8);
            \coupon{0,-0.4}{\JW_k};
            \coupon{0,0.3}{v_{t[n]}};
        \end{tikzpicture}
        \ ,
    \]
    and so
    \[
        v_t^\updownarrow e_{n+1} v_t = \frac{\Delta_{k+1}^2}{\Delta_k^2} v_{t[n]}^\updownarrow v_{t[n]}
        \overset{\cref{crab}}{=} \frac{\Delta_{k+1}^2}{\Delta_k^2} \Delta_{t[n]} J_k
        \overset{\cref{crab}}{=} \frac{\Delta_{k+1}}{\Delta_k} \Delta_t J_k
        = \frac{\Delta_{t_{n+1}}}{\Delta_{t_n}} \Delta_t J_k.
    \]
    Thus,
    \begin{multline*}
        \sum_{\substack{s,t \in P^{n+2} \\ s_{n+2} = t_{n+2} \\ t_n \le n-2r}} \frac{\Delta_{t_n}}{\Delta_s \Delta_t \Delta_{t_{n+2}} |P^n_{t_{n}}|} v_s v_t^\updownarrow e_{n+1} v_t v_s^\updownarrow
        = \sum_{k=0}^{n-2r} \sum_{\substack{s,t \in P^{n+2}_k \\ t_n=k}} \frac{\Delta_{t_{n+1}}}{\Delta_s \Delta_k |P^n_k|} v_s v_s^\updownarrow
        \\
        \overset{\cref{crab}}{=} \sum_{k=0}^{n-2r} \sum_{\substack{s,t \in P^{n+2}_k \\ t_n=k}} \frac{\Delta_{t_{n+1}}}{\Delta_s \Delta_k |P^n_k|} v_s v_s^\updownarrow
        = \sum_{k=0}^{n-2r} \sum_{s \in P^{n+2}_k} \frac{\Delta_{k-1} + \Delta_{k+1}}{\Delta_s \Delta_k} v_s v_s^\updownarrow
        \overset{\cref{Delta}}{\underset{\cref{matrixunits}}{=}} \delta \sum_{k=0}^{n-2r} \sum_{\substack{s \in P^{n+2} \\ s_{n+2} \le n-2r}} e_{s,s},
    \end{multline*}
    completing the proof of the induction step.
\end{proof}

\begin{cor} \label{goosey}
    We have
    \begin{equation} \label{goose}
        \bF( \freeprojectorreg{r}{n} ) = \sum_{p \in P^n_{n-2r}} e_{p,p},\qquad r,n \in \N,\ n \ge 2r.
    \end{equation}
    In particular,
    \begin{equation} \label{fiction}
        \bF
        \left(
            \begin{tikzpicture}[centerzero]
                \projector{0,0}{0};
                \region{0.4,0}{n};
            \end{tikzpicture}
        \right)
        \colon x \mapsto x \JW_n.
    \end{equation}
\end{cor}

\begin{proof}
    Equation \cref{goose} follows from \cref{duck,sigma}.  For \cref{fiction}, we use \cref{duck} to compute
    \[
        \bF
        \left(
            \begin{tikzpicture}[centerzero]
                \projector{0,0}{0};
                \region{0.4,0}{n};
            \end{tikzpicture}
        \right)
        \colon x \mapsto x e_{(0,1,\dotsc,n),(0,1,\dotsc,n)}
        \overset{\cref{funspiel}}{=} x \JW_n.
        \qedhere
    \]
\end{proof}

Next we verify that $\bF$ respects the third relation in \cref{fire}.  The left-hand side of this relation is
\[
    \begin{tikzpicture}[centerzero]
        \draw[S,->] (-0.2,-0.5) -- (-0.2,-0.2) -- (0.2,-0.2) -- (0.2,0.2) -- (-0.2,0.2) -- (-0.2,0.5);
        \draw[D] (0.2,-0.5) -- (0.2,-0.2);
        \draw[D] (-0.2,-0.2) -- (-0.2,0.2);
        \draw[D,->] (0.2,0.2) -- (0.2,0.5);
    \end{tikzpicture}
    \overset{\cref{fire}}{=}
    \begin{tikzpicture}[centerzero]
        \draw[D,->] (0.5,-0.5) -- (0.5,0.5);
        \draw[S] (0,-0.5) -- (0,-0.2);
        \draw[D] (0,-0.2) to[out=135,in=-135,looseness=2] (0,0.2);
        \draw[S] (0,-0.2) to[out=45,in=-45,looseness=2] (0,0.2);
        \draw[S,->] (0,0.2) -- (0,0.5);
    \end{tikzpicture}
    \overset{\cref{water}}{=}
    \begin{tikzpicture}[centerzero]
        \draw[S,->] (0,-0.5) -- (0,0.5);
        \draw[D,->] (0.3,-0.5) -- (0.3,0.5);
        \bubright[D]{-0.5,0};
    \end{tikzpicture}
    \ ,
\]
where the first equality above uses the second relation in \cref{fire}.  Since we have already verified that $\bF$ respects the first, second, and fourth relations in \cref{fire}, as well as all the relations used in the proof of \cref{water}, we see that
\[
    \bF
    \left(
        \begin{tikzpicture}[centerzero]
            \draw[S,->] (-0.2,-0.5) -- (-0.2,-0.2) -- (0.2,-0.2) -- (0.2,0.2) -- (-0.2,0.2) -- (-0.2,0.5);
            \draw[D] (0.2,-0.5) -- (0.2,-0.2);
            \draw[D] (-0.2,-0.2) -- (-0.2,0.2);
            \draw[D,->] (0.2,0.2) -- (0.2,0.5);
            \region{0.4,0}{n};
        \end{tikzpicture}
    \right)
    =
    \bF
    \left(
        \begin{tikzpicture}[centerzero]
            \draw[S,->] (0,-0.5) -- (0,0.5);
            \draw[D,->] (0.3,-0.5) -- (0.3,0.5);
            \bubright[D]{-0.5,0};
            \region{0.5,0}{n};
        \end{tikzpicture}
    \right)
    \colon \TLbim{n+3}{\TLalg_{n+3}e_{n+1}}{n} \to \TLbim{n+3}{\TLalg_{n+3}e_{n+1}}{n}
\]
is given, for $x \in \TLalg_{n+3}$, by
\[
    x e_{n+1} \xmapsto{\cref{duck}} x \sum_{\substack{p \in P^{n+3} \\ p_{n+3} \ne n+3}} e_{p,p} e_{n+1}
    \overset{\cref{funspiel}}{=} x (1-J_{n+3}) e_{n+1}
    \overset{\cref{JWkill}}{=} x e_{n+1}
    = \bF
    \left(
        \begin{tikzpicture}[centerzero]
            \draw[S,->] (-0.2,-0.4) -- (-0.2,0.4);
            \draw[D,->] (0.2,-0.4) -- (0.2,0.4);
        \end{tikzpicture}
    \right)
    (x e_{n+1}).
\]
Thus, $\bF$ respects the third relation in \cref{fire}, as desired.

Now we consider \cref{coal}.  For $x \in \TLalg_n$, we have
\[
    \bF
    \left(
        \leftbubmultreg[S]{\freeprojector{0}}{n}
    \right)
    \colon
    \begin{tikzpicture}[centerzero,TL]
        \draw[multi] (0,-0.4) -- (0,0.4);
        \coupon{0,0}{x};
    \end{tikzpicture}
    \mapsto
    \begin{tikzpicture}[anchorbase,TL]
        \draw[multi] (0,-1) -- (0,0.4);
        \draw (-0.3,-0.4) to[out=up,in=up,looseness=1.5] (-0.7,-0.4) -- (-0.7,-0.8) to[out=down,in=down,looseness=1.5] (-0.3,-0.8);
        \coupon{0,0}{x};
        \genbox{-0.55,-0.4}{0.25,-0.8}{\JW_{n+1}};
    \end{tikzpicture}
    \overset{\cref{JWtrace}}{=}
    \frac{\Delta_{n+1}}{\Delta_n}
    \begin{tikzpicture}[anchorbase,TL]
        \draw[multi] (0,-1) -- (0,0.4);
        \coupon{0,0}{x};
        \coupon{0,-0.6}{\JW_n};
    \end{tikzpicture}
    = 
    \bF
    \left(
        \begin{tikzpicture}[centerzero]
            \projector{0,0}{0};
            \region{0.4,0}{n};
        \end{tikzpicture}
    \right)
    \left(
        \begin{tikzpicture}[centerzero,TL]
            \draw[multi] (0,-0.4) -- (0,0.4);
            \coupon{0,0}{x};
        \end{tikzpicture}
    \right).
\]
Hence, $\bF$ respects \cref{coal}.

It remains to verify that $\bF$ respects \cref{wood}.  Suppose $p,p',r,r' \in P^n$ satisfy $p_n = p'_n$ and $r_n = r'_n$.  If $n \ge 2$, let $t = (0,1,\dotsc,n-2,n-1,n-2)$.  Then, using \cref{fiction},
\[
    \bF
    \left(
        \begin{tikzpicture}[centerzero]
            \draw[S,<-] (-0.4,0.8) -- (-0.4,0.6) to[out=down,in=down,looseness=2] (0.4,0.6) -- (0.4,0.8);
            \projector{0,0.5}{0};
            \draw[S,->] (-0.4,-0.8) -- (-0.4,-0.6) to[out=up,in=up,looseness=2] (0.4,-0.6) -- (0.4,-0.8);
            \projector{0,-0.5}{0};
            \region{0.4,0}{n};
        \end{tikzpicture}
    \right)
    \colon
    e_{p,p'} \otimes_{n-1} e_{r',r}
    \mapsto e_{p,p'} J_{n-1} e_{r',r} \sum_{b \in \bB_n} b J_{n-1} \otimes_{n-1} b^\vee.
\]
Using \cref{funspiel,mumultgen}, we have
\[
    e_{p,p'} J_{n-1} e_{r',r}
    =
    \begin{cases}
        J_n & \text{if } p=p'=r'=r = (0,1,\dotsc,n), \\
        e_{t,t} & \text{if } n \ge 2 \text{ and } p=p'=r'=r = t, \\
        0 & \text{otherwise}.
    \end{cases}
\]
Similarly,
\[
    \sum_{b \in \bB_n} b J_{n-1} \otimes_{n-1} b^\vee
    \overset{\cref{mudual}}{=} \frac{\Delta_{n-1}}{\Delta_n} J_n \otimes_{n-1} J_n + \sum_{s \in P^n_{n-2}} \frac{\Delta_{n-1}}{\Delta_{n-2}} e_{s,t} \otimes_{n-1} e_{t,s},
\]
where we interpret the last term as zero when $n=1$.  Thus,
\[
    \bF
    \left(
        \begin{tikzpicture}[centerzero]
            \draw[S,<-] (-0.4,0.8) -- (-0.4,0.6) to[out=down,in=down,looseness=2] (0.4,0.6) -- (0.4,0.8);
            \projector{0,0.5}{0};
            \draw[S,->] (-0.4,-0.8) -- (-0.4,-0.6) to[out=up,in=up,looseness=2] (0.4,-0.6) -- (0.4,-0.8);
            \projector{0,-0.5}{0};
            \region{0.4,0}{n};
        \end{tikzpicture}
    \right)
    \colon
    e_{p,p'} \otimes_{n-1} e_{r',r}
    \mapsto
    \begin{cases}
        \frac{\Delta_{n-1}}{\Delta_n} J_n \otimes_{n-1} J_n & \text{if } p=p'=r'=r = (0,1,\dotsc,n), \\
        \frac{\Delta_{n-1}}{\Delta_{n-2}} e_{t,t} \otimes_{n-1} e_{t,t} & \text{if } n \ge 2 \text{ and } p=p'=r'=r = t, \\
        0 & \text{otherwise}.
    \end{cases}
\]
On the other hand,
\begin{multline*}
    \bF
    \left(
        \begin{tikzpicture}[centerzero]
            \draw[S,->] (-0.4,-0.5) -- (-0.4,0.5);
            \draw[S,<-] (0.4,-0.5) -- (0.4,0.5);
            \projector{0,0}{0};
            \projector{-0.7,0}{0};
            \projector{0.7,0}{0};
            \region{0.8,0.5}{n};
        \end{tikzpicture}
    \right)
    e_{p,p'} \otimes_{n-1} e_{r',r} \mapsto e_{p,p'} J_n \otimes_{n-1} e_{r',r} J_n
    \\
    =
    \begin{cases}
        J_n \otimes_{n-1} J_n & \text{if } p=p'=r'=r = (0,1,\dotsc,n), \\
        0 & \text{otherwise},
    \end{cases}
\end{multline*}
and, if $n \ge 2$,
\begin{multline*}
    \bF
    \left(
        \begin{tikzpicture}[centerzero]
            \draw[S,->] (-0.4,-0.5) -- (-0.4,0.5);
            \draw[S,<-] (0.4,-0.5) -- (0.4,0.5);
            \projector{0,0}{0};
            \projector{-0.7,0}{1};
            \projector{0.7,0}{1};
            \region{0.8,0.5}{n};
        \end{tikzpicture}
    \right)
    \colon e_{p,p'} \otimes_{n-1} e_{r',r}
    \mapsto e_{p,p'} \left( \sum_{s \in P^n_{n-2}} e_{s,s} \right) J_{n-1} \otimes_{n-1} \left( \sum_{s \in P^n_{n-2}} e_{s,s} \right) e_{r',r}
    \\
    =
    \begin{cases}
        e_{t,t} \otimes_{n-1} e_{t,t} & \text{if } p=p'=r'=r = (0,1,\dotsc,n), \\
        0 & \text{otherwise}.
    \end{cases}
\end{multline*}
Thus, $\bF$ respects \cref{wood}.  This completes the proof that $\bF$ is well defined.

%-----------------------------
\subsection{Action on modules}
%-----------------------------

Consider the category
\begin{equation} \label{Mdef}
    \Mcat := \bigoplus_{n \in \N} \TLalg_n\md.
\end{equation}
Then we let $\tEnd(\Mcat)$ be the $\kk$-linear 2-category where
\begin{itemize}
    \item the set of objects is $\{ \obj{n} : n \in \N \}$, and
    \item for $m,n \in \N$, $\Mcat(\obj{m},\obj{n})$ is the category of $\kk$-linear functors $\TLalg_m\md \to \TLalg_n\md$.
\end{itemize}
We then have a natural $2$-functor
\[
    \bA \colon \Bcat \to \tEnd(\Mcat)
\]
sending $\obj{n} \mapsto \obj{n}$, sending the bimodule $M \in (\TLalg_m,\TLalg_n)\bimod$ to the functor
\[
    M \otimes_n - \colon \TLalg_n\md \to \TLalg_m\md,
\]
and sending a bimodule homomorphism to the corresponding natural transformation of functors.  The functor $\bA$ is full and faithful by the Eilenberg--Watts Theorem.  It induces an action of the monoidal category $\Dmon$ on $\Mcat$, which we denote by $\otimes$, given on $M \in \TLalg_n\md$ by
\begin{align*}
    \Sup \otimes M &= \bA \bF(\Sup 1_n) (M) = \TLbim{n+1}{\TLalg_{n+1}}{n} \otimes_n M, & n \in \N,
    \\
    \Sdown \otimes M &= \bA \bF(\Sdown 1_n) (M) = \TLbim{n-1}{\TLalg_n}{n} \otimes_n M, & n \ge 1,
    \\
    \Dup \otimes M &= \bA \bF(\Dup 1_n) (M) = \TLbim{n+2}{\TLalg_{n+2} e_{n+1}}{n} \otimes_n M, & n \in \N,
    \\
    \Ddown \otimes M &= \bA \bF(\Ddown 1_n) (M) = \TLbim{n-2}{e_{n-1} \TLalg_n}{n} \otimes_n M. & n \ge 2,
\end{align*}
and similarly for morphisms.

\begin{prop} \label{monkey}
    For $n,k \in \N$, $n-k \in 2\N$, we have
    \begin{align} \label{monkey+}
        \Sup \otimes L^n_k &\cong L^{n+1}_{k+1} \oplus L^{n+1}_{k-1},&
        \Dup \otimes L^n_k &\cong L^{n+2}_k,
        \\ \label{monkey-}
        \Sdown \otimes L^n_k &\cong L^{n-1}_{k+1} \oplus L^{n-1}_{k-1},\quad n \ge 1,&
        \Ddown \otimes L^n_k &\cong L^{n-2}_k,\quad n \ge 2.
    \end{align}
    where we interpret $L^m_l$ as the zero module if $l<0$ or $l>m$.  Furthermore, $\bA \bF (\rightbubmultreg[D]{r}{n})$ and $\bA \bF (\freeprojectorreg{r}{n})$ are the functors with components
    \begin{equation} \label{monkeybub}
        \bA \bF \left( \rightbubmultreg[D]{r}{n} \right)_{L^n_k} = \delta_{k \le n-2r} 1_{L^n_k}
        \qquad \text{and} \qquad
        \bA \bF \left( \freeprojectorreg{r}{n} \right)_{L^n_k} = \delta_{k,n-2r} 1_{L^n_k}.
    \end{equation}
\end{prop}

\begin{proof}
    The first isomorphisms in \cref{monkey+,monkey-} follow from \cref{indformula,resformula}.  For the second isomorphism in \cref{monkey+}, we have an isomorphism
    \[
        \Dup \otimes L^n_k
        = \TLalg_{n+2} e_{n+1} \otimes_n L^n_k
        =
        \left\{
            \begin{tikzpicture}[centerzero,TL]
                \draw[multi] (0,-0.7) -- (0,0.7) \toplabel{n+2};
                \coupon{0,0.3}{f};
                \coupon{0,-0.3}{\JW_k};
                \draw (-0.4,-0.7) -- (-0.4,-0.5) to[out=up,in=up,looseness=2] (-0.7,-0.5) -- (-0.7,-0.7);
            \end{tikzpicture}
            : f \in \TLcat(\Tobj^{\otimes k}, \Tobj^{\otimes n})
        \right\}
        \xrightarrow{\cong}
        L^{n+2}_k
    \]
    given by composing on the bottom with
    \(
        \delta^{-1}\
        \begin{tikzpicture}[centerzero,TL]
            \draw (0,0.1) -- (0,0) to[out=down,in=down,looseness=2] (0.3,0) -- (0.3,0.1);
            \draw[multi] (0.6,0.1) \toplabel{n} -- (0.6,-0.2);
        \end{tikzpicture}
    \).
    Its inverse is given by composing on the bottom with
    \(
        \begin{tikzpicture}[centerzero,TL]
            \draw (0,-0.1) -- (0,0) to[out=up,in=up,looseness=2] (0.3,0) -- (0.3,-0.1);
            \draw[multi] (0.6,-0.1) \botlabel{n} -- (0.6,0.2);
        \end{tikzpicture}
    \).
    The proof of the second isomorphism in \cref{monkey-} is similar.  The equations \cref{monkeybub} follow from \cref{ducky,goosey}.
\end{proof}

%======================
\section{Basis theorem}
%======================

In this section, we describe bases for the 2-morphism spaces in $\Dcat$.  For this purpose, using the pivotal structure, we will not distinguish between a string diagram and any other diagram obtained from it by planar isotopy.  The arguments in \cref{subsec:closed,subsec:reduced,subsec:scaffold} are independent of regional labels in string diagrams.  Thus, we will omit such labels in those subsections.

%------------------------------------------------
\subsection{Closed diagrams\label{subsec:closed}}
%------------------------------------------------

We begin by studying closed diagrams, i.e., diagrams with no endpoints at the top or bottom of the diagram.

\begin{lem} \label{cyanfloat}
    Every string diagram in $\Dcat(1_n,1_n)$ consisting of only cyan strands is a linear combination of $\freeprojectorreg{r}{n}$, $0 \le r \le \lfloor n/2 \rfloor$.
\end{lem}

\begin{proof}
    Let $D$ be a string diagram in $\Dcat(1_n,1_n)$ consisting of only cyan strands.  Since all vertices involve black strings, $D$ is a closed cyan diagram with no vertices.  Thus, $D$ is a union of cyan circles, some of which may be nested.  Repeated use of \cref{apple1,headroom,smoke} then shows that $D$ is, in fact, equal to $\rightbubmultreg[D]{r}{n}$ for some $r$.  Since
    \[
        \rightbubmultreg[D]{r}{n}
        \overset{\cref{sigma}}{\underset{\cref{freeze}}{=}}
        \freeprojectorreg{r}{n} + \freeprojectorreg{r+1}{n} + \dotsb + \freeprojectorreg{\lfloor n/2 \rfloor}{n},
    \]
    the result follows.
\end{proof}

Since all trivalent vertices have exactly two incident black edges, the black strings in any string diagram in $\Dcat$ are either cycles (i.e., closed curves) or they have both endpoints at the top or bottom of the diagram.

\begin{prop} \label{noloops}
    Every 2-morphism space in $\Dcat$ is spanned by string diagrams without black cycles.
\end{prop}

\begin{proof}
    It suffices to show that every string diagram can be written as a linear combination of diagrams without black cycles.  We do this by induction on the number of black cycles.  Suppose $D$ is a string diagram in $\Dcat$ with $b > 0$ black cycles.  Since black strings cannot cross, we may choose a black cycle $C$ not containing any black strings in its interior.  
    
    Our first step is to remove all the vertices from $C$.  For orientation reasons, $C$ contains an even number of vertices, each with an incident cyan string on either the inside or the outside of $C$.  Suppose first that at least one of the incident cyan strings is on the inside of $C$.  Each such cyan string must be attached to $C$ at both ends and thus partitions the interior of $C$ into two regions: its right and left sides, when heading in the direction of orientation of the cyan string.  Since the number of cyan strings in the interior of $C$ is finite, we may choose one cyan string, which we denote by $S$, whose right region does not contain any other cyan strings with endpoints on $C$.  Thus, if it is not empty, the right region contains a closed diagram consisting of only cyan strands.  By \cref{cyanfloat,smush}, these can be moved across $S$ to the left region.  Therefore, we may assume the right region of $S$ is empty.  The arc of $C$ bordering the right region may contain vertices, whose cyan strands must be on the outside of $C$.  For orientation reasons, there must be an even number of such vertices, which can all be removed by repeated application of the first two relations in \cref{fire}.  Then we may remove the two vertices where $S$ is attached to $C$ by again using the first two relations in \cref{fire}.  Continuing in this manner, we may remove from $C$ all vertices whose incident cyan strand is in the interior of $C$.
    
    By the above, we may now assume that all the vertices on $C$ have their incident cyan strand in the exterior of $C$.  For orientation reasons, there must be an even number of such vertices.   Then we can use the first two relations in \cref{fire} to remove all the vertices on $C$.
    
    We are now reduced to the situation where $C$ has no vertices.  Its interior may contain a closed cyan diagram which, by \cref{cyanfloat} is a linear combination of the 2-morphisms $\freeprojectorreg{r}{n}$, $0 \le r \le \lfloor n/2 \rfloor$, where $n$ is the label of the interior region immediately adjacent to $C$.  Then we can remove the cycle $C$ using \cref{wind}, which allows us to write the diagram $D$ as a linear combination of diagrams with fewer black cycles.  This completes the induction step.
\end{proof}

\begin{cor} \label{galaxy}
    For all $n \in \N$, the 2-endomorphism space $\Dcat(1_n,1_n)$ is spanned by $\freeprojectorreg{r}{n}$, $0 \le r \le \lfloor n/2 \rfloor$.
\end{cor}

\begin{proof}
    Any string diagram in $\Dcat(1_n,1_n)$ is a closed diagram.  Thus, all black strings must be cycles and so, by \cref{noloops}, $\Dcat(1_n,1_n)$ is spanned by diagrams with only cyan strings.  Then the result follows from \cref{cyanfloat}.
\end{proof}

%--------------------------------------------------
\subsection{Reduced diagrams\label{subsec:reduced}}
%--------------------------------------------------

We say that a string diagram in $\Dcat$ is \emph{reduced} if it satisfies the following conditions:
\begin{description}
    \item[\namedlabel{R1}{R1}] No black string has two adjacent trivalent vertices with the cyan string emanating from the same side.  In other words, the subdiagram
        
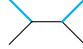
\[
            \begin{tikzpicture}[centerzero,rotate=-90]
                \draw[D] (-0.3,-0.5) -- (0,-0.2);
                \draw[S] (0.3,-0.5) -- (0,-0.2);
                \draw[S] (0,-0.2) -- (0,0.2);
                \draw[S] (0,0.2) -- (0.3,0.5);
                \draw[D] (0,0.2) -- (-0.3,0.5);
            \end{tikzpicture}
        \]
        does not appear with any orientation of the strands.  (Because of our convention of ignoring planar isotopy, this condition also applies to any rotation of the above subdiagram.)
    
    \item[\namedlabel{R2}{R2}] It contains no closed subdiagrams.
\end{description}

\begin{lem} \label{raspberry}
    Reduced diagrams have no black cycles.
\end{lem}

\begin{proof}
    This follows from the argument used in the proof of \cref{noloops}.
\end{proof}

We say that two reduced string diagrams $f$ and $g$ are \emph{equivalent}, and write $f \sim g$, if $f$ can be obtained from $g$ by repeated use of the last relation in \cref{fire} (equivalently, the relation \cref{spark}).

\begin{lem} \label{sky}
    If $f$ and $g$ are compositions of $\Dupdown$, then any two reduced string diagrams in $\Dcat(f,g)$ are equivalent.
\end{lem}

\begin{proof}
    By the assumption on $f$ and $g$, all black strings in diagrams in $\Dcat(f,g)$ must be cycles.  Thus, by \cref{raspberry}, any reduced string diagram in $\Dcat(f,g)$ contains only cyan strings.  Thus, it can be viewed as an oriented Temperley--Lieb diagram with fixed orientations at the endpoints, given by $f$ and $g$, and arcs given by the cyan strings. The endpoints can be arranged as vertices on a circle, with
    \begin{itemize}
        \item each occurrence of $\Dup$ in $f$, and $\Ddown$ in $g$, corresponding to a vertex oriented into the interior of the circle, and
        \item each occurrence of $\Ddown$ in $f$, and $\Dup$ in $g$, corresponding to a vertex oriented toward the exterior of the circle.
    \end{itemize}
    Then the lemma is equivalent to the statement that the graph of such oriented Temperley--Lieb diagrams, with adjacency given by a local flip (the last relation in \cref{fire}), is connected.  In the unoriented case, this is the standard flip graph of noncrossing perfect matchings; see \cite[Th.~2.2]{HHN02}.  The oriented version considered in the current paper can be viewed as the corresponding bichromatic subgraph.  Closely related results for bichromatic matching graphs can be found in \cite{ABLS15,ABHPV18}.   For the present local-flip graph with fixed orientations at the endpoints, the result can be proved by the following elementary induction.
    
    There are no oriented Temperley--Lieb diagrams unless the number of incoming vertices is equal to the number of outgoing vertices.  Thus, we assume that there are $k$ incoming vertices, $k$ outgoing vertices, and we induct on $k$.  In the base cases $k \in \{0,1\}$, there is only one oriented Temperley--Lieb diagram, and so the result is trivial.  Now suppose that $k \ge 2$ and that $D_1$ and $D_2$ are oriented Temperley--Lieb diagrams.  Choose adjacent vertices $v_1$ and $v_2$ with opposite orientations (which can always be done).  Suppose that, for $i \in \{1,2\}$, the vertices $v_1$ and $v_2$ in $D_i$ are not joined by an arc.  Then $v_1$ is joined to some $v_3$, and $v_2$ is joined to some $v_4$.  The distinct vertices $v_1,v_2,v_3,v_4$ bound the region incident to the boundary interval between $v_1$ and $v_2$.  Thus, we can use the last relation in \cref{fire} to replace the arcs $v_1 \leftrightarrow v_3$ and $v_2 \leftrightarrow v_4$ with the arcs $v_1 \leftrightarrow v_2$ and $v_3 \leftrightarrow v_4$.  So, we may assume that $D_1$ and $D_2$ both contain an arc joining $v_1$ and $v_2$.  Removing this arc from $D_1$ and $D_2$, the induction hypothesis implies that a series of flips converts $D_1$ into $D_2$.
\end{proof}

%--------------------------------------------
\subsection{Scaffolds\label{subsec:scaffold}}
%--------------------------------------------

For a $1$-morphism $f$ in $\Dcat$, let $\ell_\Sobj(f)$ be the number of factors of $\Sobj_\pm$ in $f$, and let $\ell(f)$ be the total length of $f$ (that is, the number of factors of $\Sobj_\pm$ and $\Dupdown$).

Suppose $f$ and $g$ are 1-morphisms in $\Dcat$ with $\ell_\Sobj(f)=r$ and $\ell_\Sobj(g)=s$.  We define an \emph{$(f,g)$-scaffold} to be a Temperley--Lieb diagram with $r$ nodes at the top and $s$ nodes at the bottom.  For example, if
\[
    f = \Dup \Sdown \Sdown \Ddown \Sup \Sdown \Dup
    \quad \text{and} \quad
    g = \Sup \Dup \Ddown \Sup,
\]
then $r=4$, $s=2$, and
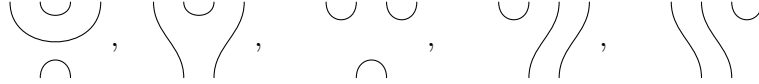
\[
    \begin{tikzpicture}[centerzero]
        \draw (-0.6,0.5) to[out=down,in=down,looseness=1.5] (0.6,0.5);
        \draw (-0.2,0.5) to[out=down,in=down,looseness=1.5] (0.2,0.5);
        \draw (-0.2,-0.5) to[out=up,in=up,looseness=2] (0.2,-0.5);
    \end{tikzpicture}
    \ ,\quad
    \begin{tikzpicture}[centerzero]
        \draw (-0.6,0.5) to[out=down,in=up] (-0.2,-0.5);
        \draw (-0.2,0.5) to[out=down,in=down,looseness=1.5] (0.2,0.5);
        \draw (0.6,0.5) to[out=down,in=up] (0.2,-0.5);
    \end{tikzpicture}
    \ ,\qquad
    \begin{tikzpicture}[centerzero]
        \draw (-0.6,0.5) to[out=down,in=down,looseness=2] (-0.2,0.5);
        \draw (0.2,0.5) to[out=down,in=down,looseness=2] (0.6,0.5);
        \draw (-0.2,-0.5) to[out=up,in=up,looseness=2] (0.2,-0.5);
    \end{tikzpicture}
    \ ,\qquad
    \begin{tikzpicture}[centerzero]
        \draw (-0.6,0.5) to[out=down,in=down,looseness=2] (-0.2,0.5);
        \draw (0.2,0.5) to[out=down,in=up] (-0.2,-0.5);
        \draw (0.6,0.5) to[out=down,in=up] (0.2,-0.5);
    \end{tikzpicture}
    \ ,\qquad
    \begin{tikzpicture}[centerzero]
        \draw (0.6,0.5) to[out=down,in=down,looseness=2] (0.2,0.5);
        \draw (-0.2,0.5) to[out=down,in=up] (0.2,-0.5);
        \draw (-0.6,0.5) to[out=down,in=up] (-0.2,-0.5);
    \end{tikzpicture}
\]
is a complete list of $(f,g)$-scaffolds.  We define the \emph{shadow}, $\sh(D)$, of a reduced string diagram $D \in \Dcat(f,g)$ to be the $(f,g)$-scaffold obtained from it by deleting all cyan strands and forgetting the orientation of all black strands.  Clearly, two equivalent reduced diagrams have the same shadow.

We define a grading on the 1-morphisms in $\Dcat$ by declaring
\[
    \deg \Sobj_\pm = \pm 1,\qquad
    \deg \Dupdown = \pm 2.
\]
Since the generating 2-morphisms of $\Dcat$ are all homogeneous, we have
\[
    \deg f \ne \deg g \implies \Dcat(f,g) = 0,
\]
for any 1-morphisms $f$ and $g$.  We say that 1-morphisms $f$ and $g$ are \emph{parallel} if they have the same domain and codomain.

\begin{prop} \label{noire}
    If $f$ and $g$ are parallel 1-morphisms in $\Dcat$, then taking shadows induces a bijection
    \[
        \tsh \colon \{ \text{reduced diagrams from $f$ to $g$} \}/{\sim} \xrightarrow{\cong} \{ \text{$(f,g)$-scaffolds} \}.
    \]
\end{prop}

\begin{proof}
    We first show that $\tsh$ is injective.  Suppose $f,g$ are parallel 1-morphisms in $\Dcat$.  We prove that, for reduced diagrams $D$ and $D'$ from $f$ to $g$,
    \[
        \sh(D) = \sh(D') \implies D \sim D',
    \]
    by induction on the number of strings in $T := \sh(D) = \sh(D')$ (equivalently, the number of black strings in $D$ and $D'$).  If $T$ has zero strings, then $D$ and $D'$ consist of only cyan strings, and the result follows from \cref{sky}.
    
    Now suppose that $T$ has $b > 0$ strings, and that the result holds for fewer than $b$ strings.  Choose one string $S$ in $T$, which corresponds to a black string $S_D$ in $D$.  Since black strings never terminate at trivalent vertices, there are three possibilities:
    \begin{itemize}
        \item One endpoint of $S$ is at the top of $T$ and the other endpoint is at the bottom of $T$.
        \item Both endpoints of $S$ are at the bottom of $T$.
        \item Both endpoints of $S$ are at the top of $T$.
    \end{itemize}
    
    Consider the case where one endpoint of $S$ is at the top and one is at the bottom.  Then we have
    \begin{equation} \label{mouse}
        f = f_1 f_3 f_2
        \qquad \text{and} \qquad
        g = g_1 g_3 g_2,
        \qquad f_3,g_3 \in \{\Sobj_\pm\},
    \end{equation}
    where $f_3,g_3$ are the endpoints of $S$.  By \ref{R1}, the cyan strings incident to $S_D$ must alternate sides of $S$.  Thus, if $\deg f_1 \le \deg g_1$ (which implies that $\deg f_2 \ge \deg g_2$), then $D$ and $D'$ look like
    \begin{equation} \label{spikey}
        \begin{tikzpicture}[centerzero]
            \filldraw[black!10!white] (0.4,-1.7) -- (0.4,1.7) -- (3,1.7) node[midway,anchor=south,black] {$g_2$} -- (3,-1.7) -- cycle node[midway,anchor=north,black] {$f_2$};
            \filldraw[black!10!white] (-0.4,-1.7) -- (-0.4,1.7) -- (-3,1.7) node[midway,anchor=south,black] {$g_1$} -- (-3,-1.7) -- cycle node[midway,anchor=north,black] {$f_1$};
            \draw[S,dotted] (0,-1.7) node[anchor=north] {$f_3$} -- (0,-1.4);
            \draw[S,->-=0.55] (0,-1.4) -- (0,-1.1);
            \draw[S,-<-=0.45] (0,-1.1) -- (0,-0.7);
            \draw[S,->-=0.75] (0,-0.7) -- (0,-0.3);
            \draw[S,dotted] (0,-0.3) -- (0,0.3);
            \draw[S,->-=0.45] (0,0.3) -- (0,0.7);
            \draw[S,-<-=0.45] (0,0.7) -- (0,1.1);
            \draw[S,->] (0,1.1) -- (0,1.4);
            \draw[S,dotted] (0,1.4) -- (0,1.7) node[anchor=south] {$g_3$};
            \draw[D,->] (0,-1.1) -- (-0.4,-1.1);
            \draw[D,<-] (0,-0.7) -- (0.4,-0.7);
            \draw[D,->] (0,0.7) -- (-0.4,0.7);
            \draw[D,<-] (0,1.1) -- (0.4,1.1);
            \node at (-1.7,0) {$D_1$};
            \node at (1.7,0) {$D_2$};
        \end{tikzpicture}
        \qquad \text{and} \qquad
        \begin{tikzpicture}[centerzero]
            \filldraw[black!10!white] (0.4,-1.7) -- (0.4,1.7) -- (3,1.7) node[midway,anchor=south,black] {$g_2$} -- (3,-1.7) -- cycle node[midway,anchor=north,black] {$f_2$};
            \filldraw[black!10!white] (-0.4,-1.7) -- (-0.4,1.7) -- (-3,1.7) node[midway,anchor=south,black] {$g_1$} -- (-3,-1.7) -- cycle node[midway,anchor=north,black] {$f_1$};
            \draw[S,dotted] (0,-1.7) node[anchor=north] {$f_3$} -- (0,-1.4);
            \draw[S,->-=0.55] (0,-1.4) -- (0,-1.1);
            \draw[S,-<-=0.45] (0,-1.1) -- (0,-0.7);
            \draw[S,->-=0.75] (0,-0.7) -- (0,-0.3);
            \draw[S,dotted] (0,-0.3) -- (0,0.3);
            \draw[S,->-=0.45] (0,0.3) -- (0,0.7);
            \draw[S,-<-=0.45] (0,0.7) -- (0,1.1);
            \draw[S,->] (0,1.1) -- (0,1.4);
            \draw[S,dotted] (0,1.4) -- (0,1.7) node[anchor=south] {$g_3$};
            \draw[D,->] (0,-1.1) -- (-0.4,-1.1);
            \draw[D,<-] (0,-0.7) -- (0.4,-0.7);
            \draw[D,->] (0,0.7) -- (-0.4,0.7);
            \draw[D,<-] (0,1.1) -- (0.4,1.1);
            \node at (-1.7,0) {$D_1'$};
            \node at (1.7,0) {$D_2'$};
        \end{tikzpicture}
        \ ,
    \end{equation}
    respectively, for some subdiagrams $D_1$ and $D_2$ of $D$, and subdiagrams $D_1'$ and $D_2'$ of $D'$.  For both $D$ and $D'$, the number of cyan strands incident to $S$ on the left side is $(\deg g_1 - \deg f_1)/2$, and the number of cyan strands incident to $S$ on the right side is $(\deg f_2 - \deg g_2)/2$.  We view $D_1$ and $D_1'$ as string diagrams in $\Dcat$ whose domains are the unions of their bottom and right edges, and whose codomains are their top edges.  Similarly, we view $D_2$ and $D_2'$ as string diagrams in $\Dcat$ whose domains are given by their bottom edges, and whose codomains are given by the unions of their left and top edges.  Since $\sh(D_1) = \sh(D_1')$ and $\sh(D_2)= \sh(D_2')$, and these all contain fewer than $b$ black strings, our induction hypothesis implies that $D_1 \sim D_1'$ and $D_2 \sim D_2'$.  This implies that $D \sim D'$, completing the proof of the induction step.  The case $\deg f_1 > \deg g_1$, where $D$ and $D'$ look like the reflections of \cref{spikey} in a vertical line, is analogous.  The arguments in the cases where the endpoints of $S$ are both at the bottom of $T$ or both at the top of $T$ are also similar.  This shows that $\tsh$ is injective.
    \details{
        When both endpoints are at the bottom of $T$, we have something like:
        \[
            \begin{tikzpicture}[centerzero]
                \filldraw[black!10!white] (-2.4,-1.5) -- (-2.4,-0.5) to[out=up,in=left] (-1.5,0.4) -- (1.5,0.4) to[out=right,in=up] (2.4,-0.5) -- (2.4,-1.5) -- (3.5,-1.5) -- (3.5,1.5) -- (-3.5,1.5) -- (-3.5,-1.5) -- cycle;
                \filldraw[black!10!white] (-1.6,-1.5) -- (-1.6,-0.5) to[out=up,in=left] (-1.5,-0.4) -- (1.5,-0.4) to[out=left,in=up] (1.6,-0.5) -- (1.6,-1.5) -- cycle;
                \draw[S,->] (-2,-1.5) -- (-2,-0.5);
                \draw[S] (-2,-0.5) to[out=up,in=left] (-1.5,0);
                \draw[S,-<-=0.45] (-1.5,0) -- (-1,0);
                \draw[S,->-=0.55] (-1,0) -- (-0.5,0);
                \draw[S,dotted] (-0.5,0) -- (0.5,0);
                \draw[S,-<-=-0.45] (0.5,0) -- (1,0);
                \draw[S,->-=0.55] (1,0) -- (1.5,0);
                \draw[S] (1.5,0) to[out=right,in=up] (2,-0.5);
                \draw[S,<-] (2,-0.5) -- (2,-1.5);
                \draw[D,->] (-1.5,0) -- (-1.5,0.4);
                \draw[D,<-] (-1,0) -- (-1,-0.4);
                \draw[D,<-] (1,0) -- (1,-0.4);
                \draw[D,->] (1.5,0) -- (1.5,0.4);
                \node at (0,0.95) {$D_1$};
                \node at (0,-0.95) {$D_2$};
            \end{tikzpicture}
        \]
        The case where both endpoints of $S$ are at the top of $T$ is analogous.
    }
    
    The proof that $\tsh$ is surjective proceeds by a similar induction on the number of strings $b$ in an $(f,g)$-scaffold $T$.  If $b=0$, then any matching of the $\Dupdown$ by oriented cyan strings is a preimage of $T$.  Now suppose $b > 0$, and fix one string $S$ in $T$.  Suppose that one endpoint of $S$ is at the top of $T$ and the other endpoint is at the bottom of $T$.  (As for injectivity above, the other cases are analogous.)  Then $f$ and $g$ can be written as in \cref{mouse}.  Let
    \[
        d_1 = \deg g_1 - \deg f_1
        \qquad \text{and} \qquad
        d_2 = \deg f_2 - \deg g_2.
    \]
    For $i \in \{1,2\}$, the sum of the number of $\Sobj_\pm$ appearing in $f_i$ and $g_i$ must be even (since arcs in $T$ cannot cross $S$).  Thus, $d_1, d_2 \in 2\Z$.  Furthermore,
    \[
        d_2 - d_1
        = (\deg f_1 + \deg f_2) - (\deg g_1 + \deg g_2)
        =
        \begin{cases}
            0 & \text{if } f_3=g_3, \\
            2 & \text{if } f_3 = \Sdown,\ g_3 = \Sup, \\
            -2 & \text{if } f_3 = \Sup,\ g_3 = \Sdown.
        \end{cases}
    \]
    By the induction hypothesis, there exist reduced diagrams $D_1$ and $D_2$ as in \cref{spikey}, with $d_1/2$ cyan strands entering the right side of $D_1$ and $d_2/2$ cyan strings exiting the left side of $D_2$, in the preimage of the corresponding subdiagrams of $T$ (on either side of the string $S$), which then yield a reduced diagram $D$ in the preimage of $T$.
\end{proof}

%-------------------
\subsection{Bridges}
%-------------------

Recall the definition of $\Lambda_n$, $n \in \N$, in \cref{Lambda}.

\begin{defin}[$\Lambda_f$]
    Suppose
    \[
        f = f_a \dotsm f_2 f_1 1_n \colon \obj{n} \to \obj{m},\qquad
        a \in \N,\ f_1,f_2,\dotsc,f_a \in \{\Sobj_\pm, \Dupdown\}.
    \]
    For $0 \le i \le a$, let
    \[
        n_i = n + \sum_{k=1}^i \deg f_k.
    \]
    (In particular, $n_a = m$.)  For $k,l \in \N$, define $\Lambda_f^{k,l}$ to be the set of all
    \[
        b=(b_0,b_1,\dotsc,b_a) \in \N^{a+1}
    \]
    such that
    \[
        b_0 = l,\quad
        b_a = k,\quad
        b_i \le n_i, \qquad 0 \le i \le a,
    \]
    and
    \[
        b_{i+1}
        =
        \begin{cases}
            b_i \pm 1 & \text{if } f_{i+1} = \Sobj_\pm,\\
            b_i & \text{if } f_{i+1} = \Dupdown,
        \end{cases}
    \]
    for all $0 \le i \le a-1$.  (Note that this implies that $b_i \in \Lambda_{n_i}$ for $0 \le i \le a$.)  We also define
    \[
        \Lambda_f = \bigsqcup_{k,l} \Lambda_f^{k,l}.
    \]
\end{defin}

\begin{defin}[$\Lambda_{f,g}$]
    For parallel 1-morphisms $f$ and $g$ in $\Dcat$,we define
    \begin{gather*}
        \Lambda_{f,g}^{k,l} = \{ b=(b_0,b_1,\dotsc,b_{\ell(f)+\ell(g)}) : (b_0,\dotsc,b_{\ell(f)}) \in \Lambda_f^{k,l},\ (b_{\ell(f)+\ell(g)},\dotsc,b_{\ell(f)}) \in \Lambda_g^{k,l} \},\ k,l \in \N,
        \\
        \Lambda_{f,g} := \bigsqcup_{k,l} \Lambda_{f,g}^{k,l}.
    \end{gather*}
    We call elements of $\Lambda_{f,g}$ \emph{bridges}, since they are closely related to \emph{nonnegative bridges} and \emph{Dyck paths}.  For a bridge $b \in \Lambda_{f,g}$, we let
    \[
        \bar{b} = \left( \bar{b}_0, \bar{b}_1, \dotsc, \bar{b}_{\ell_\Sobj(f) + \ell_\Sobj(g)} \right)
    \]
    denote the sequence obtained from $\bar{b}$ by deleting multiple consecutive occurrences of the same entry.  For example,
    \[
        \text{if } b = (3,4,4,3,2,2,3,4,5,5,5,4,3,3)
        \text{ then } \bar{b} = (3,4,3,2,3,4,5,4,3).
    \]
\end{defin}

The following is a slight modification of the usual method of associating Temperley--Lieb diagrams to Dyck paths.

\begin{defin}[Scaffold associated to a bridge] \label{pillar}
    Suppose $f$ and $g$ are parallel 1-morphisms in $\Dcat$ such that $l := \ell_\Sobj(f) + \ell_\Sobj(g)$ is even. (If $l$ is odd, then $\Dcat(f,g) = 0$.)  For $b \in \Lambda_{f,g}$, we define $b(f,g)$ to be the $(f,g)$-scaffold defined as follows:
    \begin{enumerate}
        \item Number the occurrences of $\Sobj_\pm$ in $f$ and $g$ from $1$ to $l$, starting from the right side of $f$ and moving left along $f$, then starting with the left side of $g$ and moving right along $g$.
            
        \item Choose $m$ such that $0 \le m \le l$ and $\bar{b}_m$ is a minimal entry of $\bar{b}$.
        
        \item Define an order $\prec$ on the set $\{ i : 0 \le i \le l\}$ by
            \[
                m \prec m+1 \prec \dotsb \prec l \prec 0 \prec 1 \dotsb \prec m-1.
            \]
            (If $m=0$, then $\prec$ is the order $<$.)
            
        \item Draw an arc connecting vertices $i \prec j$ whenever $\bar{b}_{i-1} = \bar{b}_j$ and
            \begin{equation} \label{pillarjoin}
                \bar{b}_k > \bar{b}_{i-1} \text{ for all } i \preceq k \prec j.
            \end{equation}
    \end{enumerate}
    Note that the choice of $m$ above is not necessarily unique, but that $b(f,g)$ is independent of this choice.  See \cref{sandbox} below.
\end{defin}

%-----------------
\subsection{Depth}
%-----------------

Suppose $f$ and $g$ are parallel 1-morphisms in $\Dcat$ and that $D \in \Dcat(f,g)$.  We define the \emph{depth} of a region of $D$ labelled $n$ and containing $\freeprojector{r}$ to be $n-2r$.  By \cref{absorb,abyss}, in a nonzero diagram, the depth of every region is a unique nonnegative integer.  We indicate depth by the notation
\begin{equation} \label{depthor}
    \freedepthorreg{n-2r}{n} := \freeprojectorreg{r}{n},
    \qquad r,n \in \Z.
\end{equation}
It follows from \cref{depthor,abyss} that
\begin{equation} \label{abyss2}
    \freedepthorreg{k}{n} = 0 \quad \text{if} \quad k<0 \text{ or } k > n.
\end{equation}
By \cref{absorb,smush,portis}
\begin{gather} \label{smush2}
    \begin{tikzpicture}[centerzero]
        \draw[D,->] (0,-0.4) -- (0,0.4);
        \depthor{0.4,0}{k};
    \end{tikzpicture}
    =
    \begin{tikzpicture}[centerzero]
        \draw[D,->] (0,-0.4) -- (0,0.4);
        \depthor{-0.4,0}{k};
    \end{tikzpicture}
    =
    \begin{tikzpicture}[centerzero]
        \draw[D,->] (0,-0.4) -- (0,0.4);
        \depthor{-0.4,0}{k};
        \depthor{0.4,0}{k};
    \end{tikzpicture}
    \ ,\qquad r \in \Z,
    \\ \label{portis2}
    \begin{tikzpicture}[centerzero]
        \draw[S,->] (0,-0.4) -- (0,0.4);
        \depthor{-0.4,0}{k};
    \end{tikzpicture}
    =
    \begin{tikzpicture}[centerzero]
        \draw[S,->] (0,-0.4) -- (0,0.4);
        \depthor{-0.4,0}{k};
    \end{tikzpicture}
    \left( \freedepthor{k-1} + \freedepthor{k+1} \right),
    \qquad
    \begin{tikzpicture}[centerzero]
        \draw[S,->] (0,-0.4) -- (0,0.4);
        \depthor{0.4,0}{k};
    \end{tikzpicture}
    = \left( \freedepthor{k-1} + \freedepthor{k+1} \right)
    \begin{tikzpicture}[centerzero]
        \draw[S,->] (0,-0.4) -- (0,0.4);
        \depthor{0.4,0}{k};
    \end{tikzpicture}
        \, , \qquad k \in \Z,
\end{gather}
In other words, depth remains unchanged across cyan strings and either increases or decreases by one across black strings.

%------------------------
\subsection{Spanning set}
%------------------------

Suppose $D$ is a string diagram in $\Dcat(f,g)$ obtained from a reduced string diagram by adding one $\freedepthor{k}$, $k \in \N$, in each region such that the depths of regions separated by a cyan string are equal and the depths of regions separated by a black string differ by one.  Then, since $D$ has no black cycles by \cref{raspberry}, the depths of the regions are uniquely determined by the regions along the top and bottom of the diagram.  Define $\beta(D) \in \Lambda_{f,g}$ to be the bridge given by recording the depths of the regions at the top and bottom of a diagram, starting from the rightmost bottom region, moving left across the bottom to the leftmost bottom region (which is the same as the leftmost top region), then moving right across the top. 

\begin{defin} \label{Bdef}
    For each pair of parallel 1-morphisms $f,g \colon \obj{n} \to \obj{m}$ in $\Dcat$, and each $b \in \Lambda_{f,g}$, fix a reduced diagram $D'_b(f,g)$ such that $\sh(D'(f,g)) = b(f,g)$.  (By \cref{noire}, this choice exists and is unique up to equivalence.)  Then let $D_b(f,g)$ be the string diagram obtained from $D'_b(f,g)$ by adding one $\freedepthor{k}$, $k \in \N$, in each region, such that
    \begin{itemize}
        \item regions separated by a cyan string contain $\freedepthor{k}$ for the same $k$, and
        \item $\beta(D_b(f,g)) = b$.
    \end{itemize}
    Let
    \[
        \bB(f,g) := \{ D_b(f,g) : b \in \Lambda_{f,g} \}.
    \]
\end{defin}

Our goal is to show that $\bB(f,g)$ is a $\kk$-basis for $\Dcat(f,g)$.

\begin{eg} \label{sandbox}
    Suppose
    \[
        f = \Sup \Sup \Sdown \Sdown \Sup \Dup \Sup \Sup \Ddown 1_3
        \quad \text{and} \quad
        g = \Dup \Sup \Ddown \Sup \Sup \Sup \Sup \Ddown \Dup \Ddown 1_3,
    \]
    Then $\ell_\Sobj(f) = 7$, $\ell(f) = 9$, $\ell_\Sobj(g) = 5$, and $\ell(g) = 10$.  Let
    \[
        b = (1,1,1,1,0,1,0,1,1,2,2,3,2,1,2,1,1,2,1,1) \in \Lambda_{f,g},
    \]
    so that
    \[
        \bar{b} = (1,0,1,0,1,2,3,2,1,2,1,2,1).
    \]
    Then we can choose $m=1$ or $m=3$ in \cref{pillar}, and either choice gives
    \[
        b(f,g)
        =
        \begin{tikzpicture}[centerzero,scale=0.8]
            \draw[S] (0,0.8) node[anchor=south] {\strandlabel{6}} to[out=down,in=down,looseness=2] (0.6,0.8) node[anchor=south] {\strandlabel{7}};
            \draw[S] (1.2,0.8) node[anchor=south] {\strandlabel{8}} -- (0.6,-0.8) node[anchor=north] {\strandlabel{5}};
            \draw[S] (1.8,-0.8) node[anchor=north] {\strandlabel{3}} to[out=up,in=up,looseness=2] (2.4,-0.8) node[anchor=north] {\strandlabel{2}};
            \draw[S] (1.2,-0.8) node[anchor=north] {\strandlabel{4}} to[out=up,in=up,looseness=1.5] (3,-0.8) node[anchor=north] {\strandlabel{1}};
            \draw[S] (1.8,0.8) node[anchor=south] {\strandlabel{9}} to[out=down,in=down,looseness=2] (2.4,0.8) node[anchor=south] {\strandlabel{10}};
            \draw[S] (3,0.8) node[anchor=south] {\strandlabel{11}} to[out=down,in=down,looseness=2] (3.6,0.8) node[anchor=south] {\strandlabel{12}};
        \end{tikzpicture}
        \ ,
    \]
    where we have indicated the labels of the vertices.  One possible choice for $D_b'(f,g)$ is
    \[
        D_b'(f,g) =
        \begin{tikzpicture}[centerzero]
            \draw[S,<->] (0,1.2) to[out=down,in=left] (0.3,0.7) to[out=right,in=down] (0.6,1.2);
            \draw[D,->-=0.55] (0,-1.2) to[out=up,in=down] (0.3,0.7);
            \draw[S,->-=0.55] (1.2,1.2) -- (0.9,0);
            \draw[S,->-=0.55] (0.6,-1.2) -- (0.9,0);
            \draw[D,->] (0.9,0) to[out=-45,in=up] (1.2,-1.2);
            \draw[S,->-=0.55] (2.4,-1.2) to[out=up,in=left] (2.7,-0.7);
            \draw[S,->-=0.55] (3,-1.2) to[out=up,in=right] (2.7,-0.7);
            \draw[S,->-=0.55] (1.8,-1.2) to[out=up,in=left] (2.2,0);
            \draw[S,->-=0.55] (2.7,0) -- (2.2,0);
            \draw[S,->-=0.55] (2.7,0) -- (3.2,0);
            \draw[S,->-=0.55] (3.6,-1.2) to[out=up,in=right] (3.2,0);
            \draw[D,->-=0.55] (2.7,-0.7) -- (2.7,0);
            \draw[S,->] (1.8,1.2) to[out=down,in=left] (2.1,0.7) to[out=right,in=down] (2.4,1.2);
            \draw[S,<->] (3.6,1.2) to[out=down,in=left] (3.9,0.7) to[out=right,in=down] (4.2,1.2);
            \draw[D,->] (2.2,0) to[out=up,in=down] (3,1.2);
            \draw[D,->-=0.55] (3.2,0) to[out=60,in=down] (3.9,0.7);
            \draw[D,<-] (4.2,-1.2) to[out=up,in=down] (4.8,1.2);
            \draw[D,->] (4.8,-1.2) to[out=up,in=left] (5.1,-0.7) to[out=right,in=up] (5.4,-1.2);
            \region{5.2,0}{3};
        \end{tikzpicture}
        \ ,
    \]
    where we have only labelled the rightmost region, since that uniquely determines the labels of the other regions.  With this choice, we have
    \begin{equation} \label{bday}
        D_b(f,g) =
        \begin{tikzpicture}[centerzero]
            \draw[S,<->] (0,1.2) to[out=down,in=left] (0.3,0.7) to[out=right,in=down] (0.6,1.2);
            \draw[D,->-=0.55] (0,-1.2) to[out=up,in=down] (0.3,0.7);
            \draw[S,->-=0.55] (1.2,1.2) -- (0.9,0);
            \draw[S,->-=0.55] (0.6,-1.2) -- (0.9,0);
            \draw[D,->] (0.9,0) to[out=-45,in=up] (1.2,-1.2);
            \draw[S,->-=0.55] (2.4,-1.2) to[out=up,in=left] (2.7,-0.7);
            \draw[S,->-=0.55] (3,-1.2) to[out=up,in=right] (2.7,-0.7);
            \draw[S,->-=0.55] (1.8,-1.2) to[out=up,in=left] (2.2,0);
            \draw[S,->-=0.55] (2.7,0) -- (2.2,0);
            \draw[S,->-=0.55] (2.7,0) -- (3.2,0);
            \draw[S,->-=0.55] (3.6,-1.2) to[out=up,in=right] (3.2,0);
            \draw[D,->-=0.55] (2.7,-0.7) -- (2.7,0);
            \draw[S,->] (1.8,1.2) to[out=down,in=left] (2.1,0.7) to[out=right,in=down] (2.4,1.2);
            \draw[S,<->] (3.6,1.2) to[out=down,in=left] (3.9,0.7) to[out=right,in=down] (4.2,1.2);
            \draw[D,->] (2.2,0) to[out=up,in=down] (3,1.2);
            \draw[D,->-=0.55] (3.2,0) to[out=60,in=down] (3.9,0.7);
            \draw[D,<-] (4.2,-1.2) to[out=up,in=down] (4.8,1.2);
            \draw[D,->] (4.8,-1.2) to[out=up,in=left] (5.1,-0.7) to[out=right,in=up] (5.4,-1.2);
            \region{5.1,0.3}{3};
            \depthor{5.1,-0.3}{1};
            \depthor{2.7,-1}{1};
            \depthor{3.2,-0.4}{0};
            \depthor{-0.1,0.2}{2};
            \depthor{0.3,1}{3};
            \depthor{2.1,1}{2};
            \depthor{3.9,1}{2};
        \end{tikzpicture}
        \ ,
    \end{equation}
    where we have only added $\freedepthor{k}$ to enough regions so that \cref{smush2} determines the ones in the remaining regions.
\end{eg}

\begin{lem} \label{tire}
    In $\Dcat$,
    \[
        \begin{tikzpicture}[centerzero]
            \draw[D] (-0.2,-0.4) to[out=45,in=down] (0.15,0) to[out=up,in=-45] (-0.2,0.4);
            \draw[S] (0.2,-0.4) to[out=135,in=down] (-0.15,0) to[out=up,in=225] (0.2,0.4);
            \region{0.5,0.2}{n};
            \depthor{0.5,-0.2}{k};
            \depthor{-0.7,-0.2}{k+1};
        \end{tikzpicture}
        =
        \begin{tikzpicture}[centerzero]
            \draw[D] (-0.2,-0.4) -- (-0.2,0.4);
            \draw[S] (0.2,-0.4) -- (0.2,0.4);
            \region{0.5,0.2}{n};
            \depthor{0.5,-0.2}{k};
            \depthor{-0.7,-0.2}{k+1};
        \end{tikzpicture}
        \ ,\qquad n \in \Z,\ k \in n - 2\Z,
    \]
    for all orientations of the strands.
\end{lem}

\begin{proof}
    This follows immediately from \cref{reidUU,reidUD} for all orientations of the strings except a downward blue string and an upward black string.  In that case, by \cref{abyss2}, the equation holds trivially unless $0 \le k < n$.  For $0 \le k < n$, setting $r=(n-k)/2  > 0$, we have
    \[
        \begin{tikzpicture}[centerzero]
            \draw[D,<-] (-0.2,-0.4) to[out=45,in=down] (0.15,0) to[out=up,in=-45] (-0.2,0.4);
            \draw[S,->] (0.2,-0.4) to[out=135,in=down] (-0.15,0) to[out=up,in=225] (0.2,0.4);
            \region{0.5,0.2}{n};
            \depthor{0.5,-0.2}{k};
            \depthor{-0.7,-0.2}{k + 1};
        \end{tikzpicture}
        \overset{\cref{reidUD}}{\underset{\cref{depthor}}{=}}
        \begin{tikzpicture}[centerzero]
            \draw[D,<-] (-0.2,-0.4) -- (-0.2,0.4);
            \draw[S,->] (0.2,-0.4) -- (0.2,0.4);
            \region{0.5,0.2}{n};
            \projector{0.6,-0.2}{r};
            \depthor{-0.8,-0.2}{k + 1};
            \bubright[D]{1.1,0};
        \end{tikzpicture}
        \overset{\cref{sigma}}{\underset{\cref{headroom}}{=}} 
        \begin{tikzpicture}[centerzero]
            \draw[D,<-] (-0.2,-0.4) -- (-0.2,0.4);
            \draw[S,->] (0.2,-0.4) -- (0.2,0.4);
            \region{0.5,0.2}{n};
            \projector{0.6,-0.2}{r};
            \depthor{-0.8,-0.2}{k + 1};
        \end{tikzpicture}
        \overset{\cref{depthor}}{=}
        \begin{tikzpicture}[centerzero]
            \draw[D,<-] (-0.2,-0.4) -- (-0.2,0.4);
            \draw[S,->] (0.2,-0.4) -- (0.2,0.4);
            \region{0.5,0.2}{n};
            \depthor{0.5,-0.2}{k};
            \depthor{-0.8,-0.2}{k + 1};
        \end{tikzpicture}
        \ . \qedhere
    \]
    \details{
        When $k=n$,
        \[
            \begin{tikzpicture}[centerzero]
                \draw[D,<-] (-0.2,-0.4) -- (-0.2,0.4);
                \draw[S,->] (0.2,-0.4) -- (0.2,0.4);
                \region{0.5,0.2}{n};
                \depthor{0.5,-0.2}{n};
                \depthor{-0.7,-0.2}{n+1};
            \end{tikzpicture}
            =
            \begin{tikzpicture}[centerzero]
                \draw[D,<-] (-0.4,-0.4) -- (-0.4,0.4);
                \draw[S,->] (0.4,-0.4) -- (0.4,0.4);
                \region{0.7,0.2}{n};
                \region{0,0.2}{n+1};
                \region{-0.9,0.2}{n-1};
                \depthor{0.7,-0.2}{n};
                \depthor{-0.9,-0.2}{n+1};
            \end{tikzpicture}
            \overset{\cref{abyss2}}{=} 0
            \overset{\cref{abyss2}}{=}
            \begin{tikzpicture}[centerzero]
                \draw[D,<-] (-0.4,-0.4) -- (-0.4,0.4);
                \draw[S,->] (0.4,-0.4) -- (0.4,0.4);
                \region{0.7,0.2}{n};
                \region{0,0.2}{n+1};
                \region{-0.9,0.2}{n-1};
                \depthor{0.7,-0.2}{n};
                \depthor{-0.9,-0.2}{n+1};
                \bubright[D]{1.1,0};
            \end{tikzpicture}
            \overset{\cref{reidUD}}{=}
            \begin{tikzpicture}[centerzero]
                \draw[D,<-] (-0.2,-0.4) to[out=45,in=down] (0.15,0) to[out=up,in=-45] (-0.2,0.4);
                \draw[S,->] (0.2,-0.4) to[out=135,in=down] (-0.15,0) to[out=up,in=225] (0.2,0.4);
                \region{0.5,0.2}{n};
                \depthor{0.5,-0.2}{n};
                \depthor{-0.7,-0.2}{n+1};
            \end{tikzpicture}
            \ .
        \]
    }
\end{proof}

\begin{lem}
    The following relations hold in $\Dcat$ for all $n \in \Z$, $k \in \N$:
    \begin{gather} \label{sushi1}
        \begin{tikzpicture}[centerzero]
            \draw[S,<-] (-0.25,-0.7) -- (-0.25,0.7);
            \draw[S,->] (0.25,-0.7) -- (0.25,0.7);
            \depthor{0,0}{k};
            \depthor{-0.7,0}{k \pm 1};
            \depthor{0.7,0}{k \pm 1};
            \region{0.8,0.5}{n};
        \end{tikzpicture}
        = \frac{\Delta_{k \pm 1}}{\Delta_k}\
        \begin{tikzpicture}[centerzero]
            \draw[S,->] (-0.3,0.7) -- (-0.3,0.5) to[out=down,in=down,looseness=2] (0.3,0.5) -- (0.3,0.7);
            \depthor{0,0.45}{k};
            \draw[S,<-] (-0.3,-0.7) -- (-0.3,-0.5) to[out=up,in=up,looseness=2] (0.3,-0.5) -- (0.3,-0.7);
            \depthor{0,-0.45}{k};
            \depthor{-0.6,0}{k \pm 1};
            \region{0.4,0}{n};
        \end{tikzpicture}
        \ ,\qquad
        \begin{tikzpicture}[centerzero]
            \draw[S,->] (-0.25,-0.7) -- (-0.25,0.7);
            \draw[S,<-] (0.25,-0.7) -- (0.25,0.7);
            \depthor{0,0}{k};
            \depthor{-0.7,0}{k \pm 1};
            \depthor{0.7,0}{k \pm 1};
            \region{0.8,0.5}{n};
        \end{tikzpicture}
        = \frac{\Delta_{k \pm 1}}{\Delta_k}\
        \begin{tikzpicture}[centerzero]
            \draw[S,<-] (-0.3,0.7) -- (-0.3,0.5) to[out=down,in=down,looseness=2] (0.3,0.5) -- (0.3,0.7);
            \depthor{0,0.45}{k};
            \draw[S,->] (-0.3,-0.7) -- (-0.3,-0.5) to[out=up,in=up,looseness=2] (0.3,-0.5) -- (0.3,-0.7);
            \depthor{0,-0.45}{k};
            \depthor{-0.6,0}{k \pm 1};
            \region{0.4,0}{n};
        \end{tikzpicture}
        \ ,
        \\ \label{sushi2}
        \begin{tikzpicture}[centerzero]
            \draw[S,->] (-0.25,-0.7) -- (-0.25,0.7);
            \draw[S,->] (0.25,-0.7) -- (0.25,0.7);
            \depthor{0,0}{k};
            \depthor{-0.7,0}{k \pm 1};
            \depthor{0.7,0}{k \pm 1};
            \region{0.8,0.5}{n};
        \end{tikzpicture}
        = \frac{\Delta_{k \pm 1}}{\Delta_k}\
        \begin{tikzpicture}[centerzero]
            \draw[S,<->] (-0.3,0.7) -- (-0.3,0.5) to[out=down,in=left] (0,0.2) to[out=right,in=down] (0.3,0.5) -- (0.3,0.7);
            \depthor{0,0.45}{k};
            \draw[S] (-0.3,-0.7) -- (-0.3,-0.5) to[out=up,in=left] (0,-0.2) to[out=right,in=up] (0.3,-0.5) -- (0.3,-0.7);
            \depthor{0,-0.45}{k};
            \draw[D] (0,-0.2) -- (0,0.2);
            \depthor{-0.6,0}{k \pm 1};
            \region{0.4,0}{n};
        \end{tikzpicture}
        \ .
    \end{gather}
\end{lem}

\begin{proof}
    Starting from the first relation in \cref{camel1}, tensoring on the left and right with $\freeprojectorreg{r-1}{n}$ and taking $k = n-2r+1$ gives the first equation in \cref{sushi1} where the $\pm$ there is chosen to be $+$.  The remaining relations are proved similarly, using \cref{camel1,camel2,camel3}.
\end{proof}

\begin{prop} \label{span}
    For all parallel 1-morphisms $f$ and $g$ in $\Dcat$, the set $\bB(f,g)$ spans $\Dcat(f,g)$.
\end{prop}

\begin{proof}
    Suppose $f$ and $g$ are parallel 1-morphisms in $\Dcat$.  By \cref{noloops}, it is enough to show that every string diagram in $\Dcat(f,g)$ with no black cycles is a linear combination of elements of $\bB(f,g)$.  Using the first two relations in \cref{fire}, it suffices to consider string diagrams satisfying \ref{R1}.  Any such diagram is a reduced diagram together with a closed diagram (possibly the empty diagram) in each region.  Thus, by \cref{galaxy,depthor}, it can be written as a linear combination of reduced diagrams with one $\freedepthor{k}$, $k \in \N$, added to each region.  Therefore, it remains to show that the latter diagrams are linear combinations of elements of $\bB(f,g)$.
    
    Suppose $D$ is a nonzero reduced diagram with one $\freedepthor{k}$, $k \in \N$, added to each region.  Let $b = \beta(D)$.  If $\sh(D) = b(f,g)$, then, by \cref{noire}, $D$ is equivalent to $D_b(f,g)$, and we are done, since equivalent diagrams are equal in $\Dcat$.  Now suppose that $\sh(D) \ne b(f,g)$.  Then, ordering the vertices of $\sh(D)$ as in \cref{pillar}, choose the minimal $i$ (with respect to the order $\prec$) such that there exists a $j \succ i$ such that $j$ is not connected to $i$ in $\sh(D)$, $b_{i-1} = b_j$, and \cref{pillarjoin} holds.  Let $i'$ and $j'$ be the vertices connected to $i$ and $j$, respectively, in $\sh(D)$.  Then, since we chose $i$ minimal, $\sh(D)$ is of the form
    
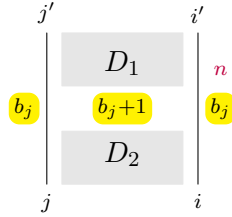
\[
        \begin{tikzpicture}[centerzero]
            \draw[S] (-1,-1) node[anchor=north] {\strandlabel{j}} -- (-1,1) node[anchor=south] {\strandlabel{j'}};
            \draw[S] (1,-1) node[anchor=north] {\strandlabel{i}} -- (1,1) node[anchor=south] {\strandlabel{i'}};
            \filldraw[black!10!white] (-0.8,-1) -- (-0.8,-0.3) -- (0.8,-0.3) -- (0.8,-1) -- cycle;
            \filldraw[black!10!white] (-0.8,1) -- (-0.8,0.3) -- (0.8,0.3) -- (0.8,1) -- cycle;
            \node at (0,0.6) {$D_1$};
            \node at (0,-0.6) {$D_2$};
            \depthor{0,0}{b_j+1};
            \depthor{1.3,0}{b_j};
            \depthor{-1.3,0}{b_j};
            \region{1.3,0.5}{n};
        \end{tikzpicture}
    \]
    where $D_1,D_2$ are Temperley--Lieb diagrams, and we have indicated the $\freedepthor{k}$ appearing in $D$ in the regions adjacent to the $i$-$i'$ and $j$-$j'$ strings.  (We have drawn the case where $i,j$ are at the bottom of $D$ and $i',j'$ are at the top; the other cases are analogous.)  In $D$, the region between these two strings may contain cyan strings.  By \cref{tire}, we may move the $i$-$i'$ and $j$-$j'$ strings past these cyan strings until the $i$-$i'$ and $j$-$j'$ strings are adjacent.  Then, depending on their orientation, we use \cref{sushi1} or \cref{sushi2} to replace the $i$-$i'$ and $j$-$j'$ strings by $i$-$j$ and $i'$-$j'$ strings, up to a multiplicative scalar.  Then repeated use of the first two relations in \cref{fire}, together with \cref{bluewind}, allows us to write the resulting diagram as a scalar multiple of a reduced diagram with one $\freedepthor{k}$, $k \in \N$, in each region, without changing its scaffold.  If necessary, we repeat the above procedure until we obtain a diagram with shadow $b(f,g)$.
\end{proof}

%-----------------------------------
\subsection{Bimodule decompositions}
%-----------------------------------

Because of our assumption \cref{generic}, we have a decomposition 
\[
    \TLalg_n \cong \bigoplus_{k \in \Lambda_n} L^n_k \otimes_\kk (L^n_k)^\vee,
\]
as $(\TLalg_n,\TLalg_n)$-bimodules, where $(L_n^k)^\vee$ is the right $\TLalg_n$-module dual to $L_n^k$.  

\begin{lem} \label{silent}
    If $f \colon \obj{n} \to \obj{m}$ is a 1-morphism in $\Dcat$, then we have a bimodule decomposition
    \begin{equation} \label{coolj}
        \bF(f) = \bigoplus_{b \in \Lambda_f} V_b,\qquad
        V_b \cong L^m_{b_{\ell(f)}} \otimes_\kk \left( L^n_{b_0} \right)^\vee.
    \end{equation}
\end{lem}

\begin{proof}
    This follows from \cref{monkey}.
\end{proof}

Suppose
\[
    f = f_a \dotsb f_2 f_1 \colon \obj{n} \to \obj{m},\qquad
    a \in \N,\ f_1,f_2,\dotsc,f_a \in \{\Sobj_\pm, \Dupdown\}.
\]
For $b \in \Lambda_f$, define
\begin{equation} \label{Edef}
    E_b(f) := \freedepthor{b_a} \otimes 1_{f_a} \otimes \freedepthor{b_{a-1}} \otimes 1_{f_{a-1}} \dotsm 1_{f_2} \otimes \freedepthor{b_1} \otimes 1_{f_1} \otimes \freedepthorreg{b_0}{n}
    \in \Dcat(f,f).
\end{equation}

\begin{lem} \label{sunny}
    Suppose $f \colon \obj{n} \to \obj{m}$ is a 1-morphism in $\Dcat$ and $b \in \Lambda_f$.  Then
    \[
        \bF(E_b(f)) \colon \bF(f) \to \bF(f)
    \]
    is projection onto the summand $V_b$ in \cref{coolj}.
\end{lem}

\begin{proof}
    This follows from \cref{depthor,monkeybub}.
\end{proof}

%-------------------------------------
\subsection{Fullness and faithfulness}
%-------------------------------------

Suppose $f,g \colon \obj{n} \to \obj{m}$ are 1-morphisms in $\Dcat$.  For $b \in \Lambda_{f,g}$, define
\begin{gather*}
    b^f := \left( b_0,b_1,\dotsc,b_{\ell(f)} \right)
    \qquad \text{and} \qquad
    b^g := \left( b_{\ell(f)+\ell(g)}, b_{\ell(f)+\ell(g)-1}, \dotsc, b_{\ell(f)} \right).
\end{gather*}
It follows from \cref{Edef,Bdef} that
\begin{equation} \label{meat}
    E_{b^g}(g) D_b(f,g) E_{b^f}(f) = D_b(f,g).
\end{equation}

\begin{eg}
    In the setting of \cref{sandbox}, we have
    \begin{gather*}
        b^f = (1,1,1,1,0,1,0,1,1,2,2),\qquad
        b^g = (1,1,2,1,1,2,1,2,3,2),
        \\
        E_{b^f}(f) = \freedepthor{2} \upstrand[D] \freedepthor{2} \upstrand[S] \freedepthor{1} \downstrand[D] \freedepthor{1} \upstrand[S] \freedepthor{0} \upstrand[S] \freedepthor{1} \upstrand[S] \freedepthor{0} \upstrand[S] \freedepthor{1} \downstrand[D] \freedepthor{1} \upstrand[D] \freedepthor{1} \downstrand[D] \freedepthorreg{1}{3}
        \ , \qquad
        E_{b^g}(g) = \freedepthor{2} \upstrand[S] \freedepthor{3} \upstrand[S] \freedepthor{2} \downstrand[S] \freedepthor{1} \downstrand[S] \freedepthor{2} \upstrand[S] \freedepthor{1} \upstrand[S] \freedepthor{1} \upstrand[S] \freedepthor{2} \upstrand[S] \freedepthor{1} \downstrand[D] \freedepthorreg{1}{3}
        \ .
    \end{gather*}
\end{eg}

It follows from \cref{silent} that, for 1-morphisms $f,g \colon \obj{n} \to \obj{m}$ in $\Dcat$,
\begin{multline} \label{boxy}
    \Bcat(\bF(f), \bF(g))
    = \bigoplus_{b \in \Lambda_{f,g}} \Bcat(V_{b^f}, V_{b^g})
    \\
    \cong \bigoplus_{(k,l) \in \Lambda_m \times \Lambda_n} \bigoplus_{b \in \Lambda_{f,g}^{k,l}} \Bcat \Bigl( L^m_k \otimes (L^n_l)^\vee, L^m_k \otimes (L^n_l)^\vee \Bigr)
    \cong \bigoplus_{(k,l) \in \Lambda_m \times \Lambda_n} \Mat_{|\Lambda_g^{k,l}| \times |\Lambda_f^{k,l}|} (\kk).
\end{multline}

\begin{prop} \label{skunk}
    Suppose $f,g \colon \obj{n} \to \obj{m}$ are 1-morphisms in $\Dcat$, and $b \in \Lambda_{f,g}$.  Then
    \[
        \bF \big( D_b(f,g) \big) \colon V_{b^f} \to V_{b^g}
    \]
    is a nonzero bimodule homomorphism, where we use \cref{sunny,meat} to view $\bF \big( D_b(f,g) \big)$ as a bimodule homomorphism from $V_{b^f}$ to $V_{b^g}$.
\end{prop}

\begin{proof}
    The spaces $\Bcat(V_{b^f},V_{b^g})$ are one-dimensional, and, for parallel 1-morphisms $f,g,h$, vertical composition
    \[
        \Bcat(V_{b^g},V_{b^h}) \times \Bcat(V_{b^f},V_{b^g}) \to \Bcat(V_{b^f},V_{b^h})  
    \]
    takes any pair of nonzero morphisms to a nonzero morphism.  (This corresponds to multiplication of matrix entries in the decomposition \cref{boxy}.)  Any $D_b(f,g)$ can be written as a composition of 2-morphisms of the form $1_Y \otimes h \otimes 1_Z$ for 1-morphisms $Y,Z$, and $h$ one of the 2-morphisms
    \begin{gather} \label{wonder1}
        \begin{tikzpicture}[centerzero]
            \draw[S,->] (-0.5,-0.4) -- (-0.5,0) to[out=up,in=up,looseness=1.5] (0.5,0) -- (0.5,-0.4);
            \depthor{0,0}{k \pm 1};
            \depthor{0.8,0}{k};
            \region{0.65,0.45}{n};
        \end{tikzpicture}
        \ ,\quad
        \begin{tikzpicture}[centerzero]
            \draw[S,->] (-0.5,0.4) -- (-0.5,0) to[out=down,in=down,looseness=1.5] (0.5,0) -- (0.5,0.4);
            \depthor{0,0}{k \pm 1};
            \depthor{0.8,0}{k};
            \region{0.65,-0.45}{n};
        \end{tikzpicture}
        \ ,\quad
        \begin{tikzpicture}[centerzero]
            \draw[S,<-] (-0.5,-0.4) -- (-0.5,0) to[out=up,in=up,looseness=1.5] (0.5,0) -- (0.5,-0.4);
            \depthor{0,0}{k \pm 1};
            \depthor{0.8,0}{k};
            \region{0.65,0.45}{n};
        \end{tikzpicture}
        \ ,\quad
        \begin{tikzpicture}[centerzero]
            \draw[S,<-] (-0.5,0.4) -- (-0.5,0) to[out=down,in=down,looseness=1.5] (0.5,0) -- (0.5,0.4);
            \depthor{0,0}{k \pm 1};
            \depthor{0.8,0}{k};
            \region{0.65,-0.45}{n};
        \end{tikzpicture}
        \ ,
        \\ \label{wonder2}
        \begin{tikzpicture}[centerzero]
            \draw[D,->] (-0.3,-0.4) -- (-0.3,0) to[out=up,in=up,looseness=2] (0.3,0) -- (0.3,-0.4);
            \depthor{0,0}{k};
            \region{0.5,0.2}{n};
        \end{tikzpicture}
        \ ,\quad
        \begin{tikzpicture}[centerzero]
            \draw[D,->] (-0.3,0.4) -- (-0.3,0) to[out=down,in=down,looseness=2] (0.3,0) -- (0.3,0.4);
            \depthor{0,0}{k};
            \region{0.5,-0.2}{n};
        \end{tikzpicture}
        \ ,\quad
        \begin{tikzpicture}[centerzero]
            \draw[D,<-] (-0.3,-0.4) -- (-0.3,0) to[out=up,in=up,looseness=2] (0.3,0) -- (0.3,-0.4);
            \depthor{0,0}{k};
            \region{0.5,0.2}{n};
        \end{tikzpicture}
        \ ,\quad
        \begin{tikzpicture}[centerzero]
            \draw[D,<-] (-0.3,0.4) -- (-0.3,0) to[out=down,in=down,looseness=2] (0.3,0) -- (0.3,0.4);
            \depthor{0,0}{k};
            \region{0.5,-0.2}{n};
        \end{tikzpicture}
        \ ,\quad
        \begin{tikzpicture}[centerzero]
            \draw[S] (-0.5,-0.6) to[out=up,in=left,looseness=1.5] (0,0.2) to[out=right,in=up,looseness=1.5] (0.5,-0.6);
            \draw[D,->] (0,0.2) -- (0,0.6);
            \depthor{0,-0.2}{k \pm 1};
            \depthor{0.7,0.3}{k};
            \region{0.7,-0.3}{n};
        \end{tikzpicture}
        \ ,\quad
        \begin{tikzpicture}[centerzero]
            \draw[S,<->] (-0.5,0.6) to[out=down,in=left,looseness=1.5] (0,-0.2) to[out=right,in=down,looseness=1.5] (0.5,0.6);
            \draw[D] (0,-0.2) -- (0,-0.6);
            \depthor{0,0.2}{k \pm 1};
            \depthor{0.7,-0.3}{k};
            \region{0.7,0.3}{n};
        \end{tikzpicture}
        \ ,
    \end{gather}
    where $0 \le l \le m$ for any $\freedepthorreg{l}{m}$ appearing in the diagram.  Thus, it suffices to prove that $\bF(h) \ne 0$ for these $h$.  For that, it is enough to show that $\bF$ sends the composition of $h$ with some other 2-morphism in $\Dcat$ to something nonzero.
    
    Setting $r=(n-k)/2$, we have
    \begin{gather*}
        \rightbubmultreg[S]{\freedepthor{k + 1}}{n} \freedepthor{k}
        \overset{\cref{depthor}}{=} \rightbubmultreg[S]{\freeprojector{r-1}}{n} \freeprojector{r}
        \overset{\cref{wind}}{=} \frac{\Delta_{k+1}}{\Delta_k} \freeprojectorreg{r}{n}
        \xmapsto[\cref{fiction}]{\bF} \frac{\Delta_{k+1}}{\Delta_k} \sum_{p \in P^n_k} e_{p,p}
        \ne 0,
        \\
        \rightbubmultreg[S]{\freedepthor{k - 1}}{n} \freedepthor{k}
        \overset{\cref{depthor}}{=} \rightbubmultreg[S]{\freeprojector{r}}{n} \freeprojector{r}
        \overset{\cref{wind}}{=} \frac{\Delta_{k-1}}{\Delta_k} \freeprojectorreg{r}{n}
        \xmapsto[\cref{fiction}]{\bF} \frac{\Delta_{k-1}}{\Delta_k} \sum_{p \in P^n_k} e_{p,p}
        \ne 0,
        \\
        \leftbubmultreg[S]{\freedepthor{k + 1}}{n} \freedepthor{k}
        \overset{\cref{depthor}}{=} \leftbubmultreg[S]{\freeprojector{r}}{n} \freeprojector{r}
        \overset{\cref{wind}}{=} \frac{\Delta_{k+1}}{\Delta_k} \freeprojectorreg{r}{n}
        \xmapsto[\cref{fiction}]{\bF} \frac{\Delta_{k+1}}{\Delta_k} \sum_{p \in P^n_k} e_{p,p}
        \ne 0,
        \\
        \leftbubmultreg[S]{\freedepthor{k - 1}}{n} \freedepthor{k}
        \overset{\cref{depthor}}{=} \leftbubmultreg[S]{\freeprojector{r+1}}{n} \freeprojector{r}
        \overset{\cref{wind}}{=} \frac{\Delta_{k-1}}{\Delta_k} \freeprojectorreg{r}{n}
        \xmapsto[\cref{fiction}]{\bF} \frac{\Delta_{k-1}}{\Delta_k} \sum_{p \in P^n_k} e_{p,p}
        \ne 0.
    \end{gather*}
    Thus, $\bF(h) \ne 0$ for $h$ any of the 2-morphisms \cref{wonder1}.  Similarly, by \cref{depthor,bluewind}
    \[
        \rightbubmultreg[D]{\freedepthor{k}}{n}
        = \leftbubmultreg[D]{\freedepthor{k}}{n}
        = \freedepthorreg{k}{n}
        \xmapsto[\cref{fiction}]{\bF} \sum_{p \in P^n_k} e_{p,p}
        \ne 0,
    \]
    and so $\bF(h) \ne 0$ for $h$ any of the first four 2-morphisms in \cref{wonder2}.  Finally,
    \[
        \begin{tikzpicture}[centerzero]
            \draw[S] (0,-0.3) to[out=left,in=down,looseness=1.5] (-0.5,0) to[out=up,in=left,looseness=1.5] (0,0.3) to[out=right,in=up,looseness=1.5] (0.5,0) to[out=down,in=right,looseness=1.5] (0,-0.3);
            \depthor{0,0}{k \pm 1};
            \draw[D] (0,-0.6) -- (0,-0.3);
            \draw[D,->] (0,0.3) -- (0,0.6);
            \depthor{0.8,0}{k};
            \region{0.3,0.5}{n};
        \end{tikzpicture}
        \overset{\cref{fire}}{=}
        \begin{tikzpicture}[centerzero]
            \draw[D,->] (0,-0.6) -- (0,0.6);
            \bubleftmult[S]{0.8,-0.2}{\freedepthor{k \pm 1}};
            \depthor{0.8,0.4}{k};
            \region{0.3,0.5}{n};
        \end{tikzpicture}
        = \frac{\Delta_{k \pm 1}}{\Delta_k}\
        \begin{tikzpicture}[centerzero]
            \draw[D,->] (0,-0.6) -- (0,0.6);
            \depthor{0.4,0}{k};
            \region{0.3,0.5}{n};
        \end{tikzpicture}
        \xmapsto{\bF} \frac{\Delta_{k \pm 1}}{\Delta_k} e_{n+1} \sum_{p \in P^n_k} e_{p,p}
        = \frac{\Delta_{k \pm 1}}{\Delta_k} \sum_{p \in P^{n+2}_k} e_{p,p}
        \ne 0,
    \]
    and so $\bF(h) \ne 0$ for $h$ any of the last two 2-morphisms in \cref{wonder2}.
\end{proof}

\begin{theo}[Basis Theorem] \label{basisthm}
    For all parallel 1-morphisms $f$ and $g$ in $\Dcat$, the set $\bB(f,g)$ is a basis for $\Dcat(f,g)$.
\end{theo}

\begin{proof}
    It follows from \cref{skunk,boxy} that the bimodule homomorphisms $\bF(D_b(f,g))$, $b \in \Lambda_{f,g}$, are linearly independent.  Thus, the result follows from \cref{span}
\end{proof}

Recall that a 2-functor is \emph{locally full} if the induced functors on morphism categories are all full.  Similarly, it is \emph{locally faithful} if the induced functors on morphism categories are all faithful.

\begin{theo} \label{fullfaithful}
    The 2-functor $\bF$ is locally full and faithful.
\end{theo}

\begin{proof}
    Local fullness follows immediately from \cref{skunk,boxy}.  Local faithfulness follows from \cref{basisthm,skunk,boxy}.
\end{proof}

\begin{rem} \label{HSbasis}
    As explained in the introduction, the 2-category defined in \cite[\S5]{HS25} is the locally full 2-subcategory of $\Dcat$ generated by the 1-morphisms $\Sobj_\pm 1_n$, $n \in \N$, which are denoted $Q_\pm 1_n$ there.  In \cite[Prop.~7.16, Cor.~7.17]{HS25}, Harper and Samuelson conjecture a basis for the endomorphism space of $Q_+^{\otimes r} 1_n$.  Their proposed basis is different than the basis $\bB(f,g)$.  We do not discuss the precise relationship between $\bB(f,g)$ and their conjectural basis, since the basis $\bB(f,g)$ seems to the author to be the most natural one.  Note that \cite[Cor.~7.11]{HS25}, which claims to describe a basis for the endomorphism algebra of the unit object of their monoidal category appears to be false, since their negatively-labelled boxes do not seem to be in the image of the map considered there.
\end{rem}

%============================================
\section{Decategorification\label{sec:decat}}
%============================================

In this final section, we give a precise description of the Grothendieck ring of the additive envelope of the monoidal category $\Dmon$, the Grothendieck group of the category $\Mcat$ of Temperley--Lieb modules, and the action of the former on the latter.  We also describe a positive, integral basis for this action.  Finally, we show that the additive Karoubi envelope of the 2-category $\Dcat$ is equivalent to the category $\Bcat$ of Temperley--Lieb bimodules.

%----------------------------------------------------------------------
\subsection{The ring and its polynomial module\label{subsec:Grothring}}
%----------------------------------------------------------------------

We first study the ring $R$ that we will show, in \cref{sump}, is isomorphic to the Grothendieck ring of the additive envelope of $\Dmon$.

\begin{defin}
    Let $\Ring = \bigoplus_{n \in \Z} \Ring_n$ be the $\Z$-graded ring generated by $s_-$, $s_+$, $d_+$, $d_-$, with degrees
    \[
        \deg(s_\pm) = \pm 1,\qquad \deg(d_\pm) = \pm 2,
    \]
    and subject to the following relations:
    \begin{gather} \label{refuge1}
        s_+ d_+ = d_+ s_+,\qquad
        s_- d_- = d_- s_-,
        \\ \label{refuge2}
        d_- d_+ = 1,\qquad
        s_- d_+ = s_+,\qquad
        d_- s_+ = s_-,\qquad
        s_- s_+ = s_+ s_- - d_+ d_- + 1.
    \end{gather}
\end{defin}

\begin{lem} \label{Gring-basis}
    We have an isomorphism of $\Z$-graded additive groups
    \[
        \Ring \cong \Z[s_+,d_+] \otimes_\Z \Z[s_-,d_-],
    \]
    with the two factors corresponding to subrings of $\Ring$.
\end{lem}

\begin{proof}
    This is a standard application of Bergman's Diamond Lemma.
\end{proof}

Let $\triv = \Z 1$ be the trivial one-dimensional $\Z[s_-,d_-]$-module, concentrated in degree zero, with action determined by
\[
    s_- 1 = 0 = d_- 1.
\]

\begin{defin}
    We define the \emph{polynomial module} to be the induced graded $\Ring$-module
    \[
        \Poly = \bigoplus_{n \in \N} \Poly_n := \Ring \otimes_{\Z[s_-,d_-]} \triv.
    \]
\end{defin}

It follows from \cref{Gring-basis} that we have an isomorphism of graded $\Z$-modules
\begin{equation} \label{gravel}
    \Poly \cong \Z[s_+,d_+].
\end{equation}
We will use this isomorphism to identify $\Poly$ with $\Z[s_+,d_+]$ in what follows.

\begin{prop} \label{birch}
    The action of $\Ring$ on $\Poly$ is given as follows:  $s_+$ and $d_+$ act by multiplication, while
    \begin{gather} \label{birch1}
        s_- ( d_+^a s_+^b )
        =
        \begin{cases}
            d_+^{a-1} s_+^{b+1} & \text{if } a > 0, \\
            \sum_{k=0}^{\lfloor (b-1)/2 \rfloor} (-1)^k \binom{b-k}{k+1} d_+^k s_+^{b-1-2k} & \text{if } a = 0,
        \end{cases}
        \\ \label{birch2}
        d_- ( d_+^a s_+^b )
        =
        \begin{cases}
            d_+^{a-1} s_+^b & \text{if } a>0, \\
            \sum_{k=0}^{\lfloor b/2 \rfloor - 1} (-1)^k \binom{b-k-1}{k+1} d_+^k s_+^{b-2-2k} & \text{if } a = 0.
        \end{cases}
    \end{gather}
\end{prop}

\begin{proof}
    Let $\mathbf{1} = 1 \otimes_{\Z[s_-,d_-]} \triv$.  First we prove \cref{birch1}.  The $a>0$ case follows immediately from the second equality in \cref{refuge2}.  When $a=0$, we prove the result by induction on $b$.  The base cases $b \in \{0,1\}$ follow from the last relation in \cref{refuge2}.  When $a=0$ and $b \ge 2$, we have
    \begin{align*}
        s_- &(s_+^b \mathbf{1})
        \overset{\cref{refuge2}}{=} ( s_+ s_- - d_+ d_- + 1 ) s_+^{b-1} \mathbf{1}
        \\
        &\overset{\mathclap{\cref{refuge2}}}{=}\ \sum_{k=0}^{\lfloor (b-2)/2 \rfloor} (-1)^k \binom{b-k-1}{k+1} d_+^k s_+^{b-1-2k} \mathbf{1} - d_+ s_- s_+^{b-2} \mathbf{1} + s_+^{b-1} \mathbf{1}
        \\
        &= \sum_{k=0}^{\lfloor (b-2)/2 \rfloor} (-1)^k \binom{b-k-1}{k+1} d_+^k s_+^{b-1-2k} \mathbf{1}
        - \sum_{k=0}^{\lfloor (b-3)/2 \rfloor} (-1)^k \binom{b-k-2}{k+1} d_+^{k+1} s_+^{b-3-2k} \mathbf{1} + s_+^{b-1} \mathbf{1}
        \\
        &= \sum_{k=0}^{\lfloor (b-2)/2 \rfloor} (-1)^k \binom{b-k-1}{k+1} d_+^k s_+^{b-1-2k} \mathbf{1}
        + \sum_{k=1}^{\lfloor (b-1)/2 \rfloor} (-1)^k \binom{b-k-1}{k} d_+^k s_+^{b-1-2k} \mathbf{1} + s_+^{b-1} \mathbf{1}
        \\
        &= \sum_{k=1}^{\lfloor (b-2)/2 \rfloor} (-1)^k \binom{b-k}{k+1} d_+^k s_+^{b-1-2k} \mathbf{1}
        + b s_+^{b-1} \mathbf{1} + \delta_{b \text{ odd}} (-1)^{(b-1)/2} d_+^{(b-1)/2} \mathbf{1}
        \\
        &= \sum_{k=0}^{\lfloor (b-1)/2 \rfloor} (-1)^k \binom{b-k}{k+1} d_+^k s_+^{b-1-2k} \mathbf{1},
    \end{align*}
    where we used the induction hypothesis in the second and third equalities.
    
    Next we prove \cref{birch2}.  The case $a>0$ follows immediately from the first relation in \cref{refuge2}.  When $a=0$ and $b=0$, the result is trivial.  When $a=0$ and $b>0$, we have
    \[
        d_- (s_+^b \mathbf{1})
        \overset{\cref{refuge2}}{=} s_- s_+^{b-1} \mathbf{1},
    \]
    and so the result follows from \cref{birch1}.
\end{proof}

%-------------------------------------------
\subsection{The homogenized Chebyshev basis}
%-------------------------------------------

Recall that the Chebyshev polynomials of the second kind $U_k(x) \in \Z[x]$, $k \ge -1$, are defined recursively by
\begin{equation} \label{Chebrec}
    U_{-1}(x) = 0,\qquad
    U_0(x) = 1,\qquad
    U_{k+1}(x) = x U_k - U_{k-1}(x).
\end{equation}
(We have replaced $x$ by $x/2$ compared to the usual definition of these polynomials.)  Explicitly, we have
\[
    U_k(x) = \sum_{j=0}^{\lfloor k/2 \rfloor} (-1)^j \binom{k-j}{j} x^{k-2j}.
\]
For later use, we note that
\begin{equation} \label{Cheblow}
    U_1(x) = x,\qquad
    U_2(x) = x^2 - 1.
\end{equation}
We have
\[
    x^b = \sum_{j=0}^{\lfloor b/2 \rfloor} C(b-j,j) U_{b-2j}(x),
\]
where
\[
    C(n,k) := \binom{n+k}{k} - \binom{n+k}{k-1}
    = \frac{n-k+1}{n+1} \binom{n+k}{k},\qquad n,k \in \N,
\]
are the entries of Catalan's triangle.  Now define the homogenized Chebyshev polynomials
\begin{equation} \label{canon}
    u_{n,k}
    := d_+^{n/2} U_k \left( \frac{s_+}{\sqrt{d_+}} \right)
    = \sum_{j=0}^{\lfloor k/2 \rfloor} \binom{k-j}{j} d_+^{\frac{n-k}{2}+j} s_+^{k-2j}
    \in \Poly_n,\quad k \in \N,\ n \in k + 2\N.
\end{equation}
We also define
\[
    u_{n,k} = 0 \text{ if } n < 0,\ k < 0, \text{ or } n - k \notin 2 \N.
\]

\begin{lem} \label{canbasis}
    The elements $u_{n,k}$, $k \in \N$, $n \in k + 2\N$, form a basis for $\Poly$.
\end{lem}

\begin{proof}
    For each $n$, the set $\{u_{n,k} : k \in n - 2\N,\ k \ge 0\}$ is unitriangular with respect to the monomial basis elements of degree $n$.
    \details{
        The element $u_{n,k}$ is homogeneous of degree $n$, and the number of $k$ satisfying $k \ge 0$, $k \in n - 2\N$, is $\lfloor n/2 \rfloor + 1$, which is the dimension of $\Poly_n$.  Thus, for each $n$, it suffices to show that the $u_{n,k}$, $k \ge 0$, $k \in n-2\N$, are linearly independent.  By \cref{canon}, we have
        \[
            u_{n,k} = d_+^{(n-k)/2} s_+^k + \text{(terms with strictly smaller powers of $s_+$)}.
        \]
        So, ordering the monomial basis by decreasing $s_+$-degree, the change-of-basis matrix from the $u_{n,k}$ to the monomial basis is unitriangular.
    }
\end{proof}

\begin{prop} \label{canaction}
    In the basis $u_{n,k}$, the action of $\Ring$ on $\Poly$ is given by
    \begin{align} \label{can+}
        d_+ u_{n,k} &= u_{n+2,k},&
        s_+ u_{n,k} &= u_{n+1,k+1} + u_{n+1,k-1},
        \\ \label{can-}
        d_- u_{n,k} &= u_{n-2,k},&
        s_- u_{n,k} &= u_{n-1,k+1} + u_{n-1,k-1}.
    \end{align}
\end{prop}

\begin{proof}
    We have
    \[
        d_+ u_{n,k}
        \overset{\cref{canon}}{=} d_+^{(n+2)/2} U_k \left( s_+ d_+^{-1/2} \right)
        \overset{\cref{canon}}{=} u_{n+2,k}.
    \]
    
    Next, we compute
    \begin{multline*}
        u_{n+1,k+1}
        \overset{\cref{canon}}{=} d_+^{(n+1)/2} U_{k+1}\left( s_+ d_+^{-1/2} \right)
        \overset{\cref{Chebrec}}{=} d_+^{(n+1)/2} \left( s_+ d_+^{-1/2} U_k\left( s_+ d_+^{-1/2} \right) -  U_{k-1}\left( s_+ d_+^{-1/2} \right) \right)
        \\
        \overset{\cref{canon}}{=} s_+ u_{n,k} - u_{n+1,k-1},
    \end{multline*}
    proving the second equation in \cref{can+}.
    
    When $n \le 1$, we have $d_- u_{n,k} = 0 = u_{n-2,k}$ for degree reasons.  When $n \ge 2$,
    \[
        d_- u_{n,k}
        \overset{\cref{canon}}{\underset{\cref{refuge2}}{=}} d_+^{(n-2)/2} U_k \left( s_+ d_+^{-1/2} \right)
        \overset{\cref{canon}}{=} u_{n-2,k}.
    \]
    
    Finally, we prove the second equation in \cref{can-} by induction on $k$.  Since the equation is trivial when $n=0$, we assume $n \ge 1$.  When $k=0$, we have
    \[
        s_- u_{n,0}
        \overset{\cref{canon}}{\underset{\cref{Chebrec}}{=}} s_- d_+^{n/2}
        \overset{\cref{refuge2}}{=} d_+^{(n-2)/2} s_+
        \overset{\cref{Chebrec}}{=} d_+^{(n-1)/2} U_1(s_+ d_+^{-1/2})
        \overset{\cref{canon}}{\underset{\cref{Cheblow}}{=}} u_{n-1,1}.
    \]
    When $k=1$ and $n=1$, we have
    \[
        s_- u_{1,1}
        \overset{\cref{canon}}{\underset{\cref{Cheblow}}{=}} s_- s_+
        \overset{\cref{birch1}}{=} 1
        \overset{\cref{canon}}{\underset{\cref{Chebrec}}{=}} u_{0,0}.
    \]
    When $k=1$ and $n \ge 3$, we have
    \begin{multline*}
        s_- u_{n,1}
        \overset{\cref{canon}}{\underset{\cref{Cheblow}}{=}} s_- d_+^{(n-1)/2} s_+
        \overset{\cref{refuge2}}{=} d_+^{(n-3)/2} s_+^2
        \\
        \overset{\cref{Cheblow}}{=} d_+^{(n-1)/2} U_2 \left( s_+ d_+^{-1/2} \right) + d^+{(n-1)/2}
        \overset{\cref{canon}}{=} u_{n-1,2} + u_{n-1,0}.
    \end{multline*}
    For the induction step, we compute
    \begin{multline*}
        s_- u_{n+1,k+1}
        = s_- \left( s_+ u_{n,k} - u_{n+1,k-1} \right)
        \overset{\cref{refuge2}}{=} s_+ s_- u_{n,k} - d_+ d_- u_{n,k} + u_{n,k} - s_- u_{n+1,k-1}
        \\
        = s_+ (u_{n-1,k+1} + u_{n-1,k-1}) - \delta_{n>k} u_{n,k} - u_{n,k-2}
        = u_{n,k+2} + u_{n,k},
    \end{multline*}
    as desired.
\end{proof}

\begin{prop} \label{Polyfaith}
    The action of $\Ring$ on $\Poly$ is faithful.
\end{prop}

\begin{proof}
    Let $0 \ne x \in \Ring$. Using \cref{Gring-basis,canbasis}, we may write
    \[
        x = \sum_{a,b,n,k} c_{a,b,n,k}\, u_{n,k}\, d_-^a s_-^b,
    \]
    with $c_{a,b,n,k} \in \mathbb{Z}$ and all but finitely many zero. Choose a quadruple $(a,b,n,k)$ with $c_{a,b,n,k} \neq 0$ such that $b-k$ is minimal, and, subject to that condition, $2a+b-n$ is minimal.
    
    By \cref{canaction}, for $N \gg 0$ we have
    \[
        d_-^a s_-^b\, u_{N,N} = u_{N-2a-b,\,N-b}.
    \]
    It follows from \cref{canon} that multiplication in the $\{u_{n,k}\}$-basis is triangular:
    \[
        u_{n,k}\, u_{p,q}
        = u_{n+p,\,k+q} + \text{(terms with smaller second index)}.
    \]
    It follows from this and the choice of $(a,b,n,k)$ that
    \[
        x\,u_{N,N}
        = c_{a,b,n,k}\,u_{N+n-2a-b,\,N-b+k}
        + \sum_{(p,q)\prec (N+n-2a-b,\,N-b+k)} c_{p,q}\,u_{p,q},
    \]
    for some $c_{p,q} \in \Z$, where $\prec$ denotes the lexicographic order defined by
    \[
        (p',q')  \prec (p,q) \iff (q' < q) \text{ or } (q'=q \text{ and } p'<p).
    \]
    Therefore $x$ does not act by zero on $\mathrm{Poly}$, and the action of $R$ on $\mathrm{Poly}$ is faithful.
\end{proof}

%------------------------------
\subsection{Decategorification}
%------------------------------

Let $K_\oplus(\cC)$ denote the split Grothendieck group of an additive category $\cC$.  For an object $X$ of $\cC$, we let $[X]$ denote its class in $K_\oplus(\cC)$.  If $\cC$ is also monoidal, then $K_\oplus(\cC)$ is a ring, with multiplication induced by $[X] [Y] = [X \otimes Y]$.  For $\Mcat$, defined in \cref{Mdef}, we have a natural grading
\[
    K_\oplus(\Mcat) = \bigoplus_{n \in \N} K_\oplus(\Mcat)_n,\qquad K_\oplus(\Mcat)_n := K_\oplus(\TLalg_n\md).
\]
If $\fC$ is an additive 2-category, then we let $K_\oplus(\fC)$ denote its split Grothendieck category.  This is the $\Z$-linear category with the same objects as $\fC$, and with the morphisms between objects given by the split Grothendieck group of the corresponding morphism categories of $\fC$.

\begin{lem} \label{Polyisom}
    We have an isomorphism of graded $\Z$-modules
    \[
        \Phi \colon \Poly \to K_\oplus(\Mcat),\qquad
        u_{n,k} \mapsto \left[ L^n_k \right],\qquad k \in \N,\ n \in k + 2\N.
    \]
\end{lem}

\begin{proof}
    This follows immediately from \cref{canbasis} and the fact that the algebras $\TLalg_n$ are split semisimple, with standard modules as described in \cref{subsec:TLstandards}.
\end{proof}

\begin{prop} \label{Grothsurj}
    We have a surjective homomorphism of graded rings
    \[
        \Psi \colon \Ring \to K_\oplus(\Add(\Dmon)),\qquad
        s_\pm \mapsto [\Sobj_\pm],\qquad
        d_\pm \mapsto [\Dupdown].
    \]
\end{prop}

\begin{proof}
    We must show that $\Psi$ respects the defining relations \cref{refuge1,refuge2}.
    \begin{itemize}
        \item Relations \cref{refuge1} follow from \cref{isoms1}.
        \item The first relation in \cref{refuge2} follows from the first isomorphism in \cref{hacky} with $r=1$.
        \item The second and third relations in \cref{refuge2} follow from \cref{isoms2}.
        \item The last relation in \cref{refuge2} follows from \cref{ccrmon}.
    \end{itemize}  
    Thus, $\Psi$ does indeed define a ring homomorphism.  The fact that it preserves gradings is clear.  Surjectivity follows from the fact that $\Sobj_\pm$ and $\Dupdown$ generate $\Dmon$ by definition.
\end{proof}

We can now prove one of the main results of the paper.

\begin{theo} \label{sump}
    The homomorphism $\Psi$ of \cref{Grothsurj} is an isomorphism of graded rings.  Furthermore, the diagram
    \[
        \begin{tikzcd}
            \Ring \times \Poly \arrow[d,"\cong","\Psi \times \Phi"'] \arrow[r] & \Poly \arrow[d,"\Phi"',"\cong"]
            \\
            K_\oplus(\Add(\Dmon)) \times K_\oplus(\Mcat) \arrow[r] & K_\oplus(\Mcat)
        \end{tikzcd}
    \]
    commutes, where the top horizontal map is the action of $\Ring$ on $\Poly$ and the bottom horizontal map is the action of $K_\oplus(\Add(\Dmon))$ on $K_\oplus(\Mcat)$ induced by the action of $\Dmon$ on $\Mcat$.
\end{theo}

\begin{proof}
    Commutativity of the diagram follows from \cref{canaction,monkey}.  Then, since $\Phi$ is an isomorphism by \cref{Polyisom}, it follows from \cref{Polyfaith} that $\Psi$ is injective.  By \cref{Grothsurj}, it is also surjective.
\end{proof}

%---------------------------------
\subsection{Idempotent completion}
%---------------------------------

Recall the definition of the additive Karoubi envelope from \cref{sec:isoms}.  Since the morphism categories in $\Bcat$ are idempotent complete, we have an induced 2-functor
\[
    \Kar(\bF) \colon \Kar(\Dcat) \to \Bcat.
\]

\begin{prop} \label{paint}
    The 2-functor $\Kar(\bF)$ is a 2-equivalence of 2-categories.
\end{prop}

\begin{proof}
    By definition, $\Kar(\bF)$ is an isomorphism on objects.  By \cref{fullfaithful}, it is full and faithful on morphism categories.  It remains to prove that it is essentially surjective on morphism categories.  For $m,n \in \N$, $\Bcat(\obj{n},\obj{m})$ is the category of $(\TLalg_m, \TLalg_n)$-bimodules.  Since every $(\TLalg_m,\TLalg_n)$-bimodule is a direct sum of bimodules of the form $L^m_k \otimes (L^n_l)^\vee$, $k,l \ge 0$, $k \in m - 2\N$, $l \in n - 2\N$, it suffices to show that these bimodules are in the image of $\Kar(\bF)$, up to isomorphism.
    
    Define
    \[
        f
        =
        \begin{cases}
            (\Sdown \Sup)^{|k-l|/2} \Dup[r] 1_n  & \text{if } m = n + 2r,\ r \in \N, \\
            (\Sdown \Sup)^{(|k-l|-1)/2} \Sup \Dup[r] 1_n & \text{if } m = n + 2r + 1,\ r \in \N, \\
            (\Sdown \Sup)^{|k-l|/2} \Ddown[r] 1_n & \text{if } m = n - 2r,\ r \in \N, \\
            (\Sdown \Sup)^{(|k-l|-1)/2} \Sdown \Ddown[r] 1_n & \text{if } m = n - 2r - 1,\ r \in \N.
        \end{cases}
    \]
    Then define
    \[
        b =
        \begin{cases}
            (l, \dotsc, l, l+1, \dotsc, k-1,k) & \text{if } k \ge l, \\
            (l, \dotsc, l, l-1, \dotsc, k+1,k) & \text{if } k < l,
        \end{cases}
    \]
    where the entry $l$ appears $r$ times.  Then $b \in \Lambda_f$ and $\Kar(\bF)$ sends $[(f,E_b(f))]$ to a bimodule isomorphic to $L^m_k \otimes (L^n_l)^\vee$ by \cref{sunny}.
\end{proof}

\begin{cor}
    The composite 2-functor $\bA \Kar(\bF)$ is a 2-equivalence from $\Kar(\Dcat)$ to the 2-category whose objects are $\obj{n}$, $n \in \N$, and whose category of morphisms from $\obj{m}$ to $\obj{n}$ is the category of additive, right-exact, direct-sum-preserving functors from $\TLalg_m$-mod to $\TLalg_n$-mod.
\end{cor}

\begin{proof}
    This follows from \cref{paint} and the Eilenberg--Watts Theorem.
\end{proof}

Our final two results state that passing to the Karoubi envelope does not enlarge the Grothendieck category of $\Dcat$ or the Grothendieck ring of $\Dmon$.  As the proof shows, one of the key reasons for this is the second isomorphism in \cref{hacky}, which shows that the idempotent $\rightbubmult[D]{r}$ corresponds to projection onto the object $\Dup[r] \Ddown[r]$.

\begin{prop} \label{convocation1}
    The natural inclusion $\Add(\Dcat) \hookrightarrow \Kar(\Dcat)$ induces an isomorphism of categories
    \[
        K_\oplus(\Add(\Dcat)) \cong K_\oplus(\Kar(\Dcat)).
    \]
\end{prop}

\begin{proof}
    It suffices to show that the induced functor $K_\oplus(\Add(\Dcat)) \to K_\oplus(\Kar(\Dcat))$ is surjective.  Let
    \[
        f = f_a \dotsm f_2 f_1 1_n \colon \obj{n} \to \obj{m},\qquad
        a \in \N,\ f_1,f_2,\dotsc,f_a \in \{\Sobj_\pm, \Dupdown\},
    \]
    be an arbitrary 1-morphism in $\Dcat$.  It follows from \cref{sunny,fullfaithful} that a complete set of primitive idempotents of $\Dcat(f,f)$ is given by $E_b(f)$, $b \in \Lambda_f$.  For $b \in \Lambda_f$, the class of $(f,E_b(f))$ in $K_\oplus(\Kar(\Dcat))$ is
    \begin{align*}
        [ f,& \freedepthor{b_a} \otimes 1_{f_a} \otimes \freedepthor{b_{a-1}} \otimes 1_{f_{a-1}} \dotsm 1_{f_2} \otimes \freedepthor{b_1} \otimes 1_{f_1} \otimes \freedepthorreg{b_0}{n} ]
        \\
        &\overset{\cref{depthor}}{\underset{\cref{sigma}}{=}} \Bigl( [\Dup[r_a] \Ddown[r_a] 1_{n_a}] - [\Dup[r_a-1] \Ddown[r_a-1] 1_{n_a}] \Bigr) [f_a] \dotsm
        \\
        &\qquad \qquad \qquad \dotsm \Bigl( [\Dup[r_1] \Ddown[r_1] 1_{n_1}] - [\Dup[r_1-1] \Ddown[r_1-1] 1_{n_1}] \Bigr) [f_1] \Bigl( [\Dup[r_0] \Ddown[r_0] 1_{n_0}] - [\Dup[r_0-1] \Ddown[r_0-1] 1_{n_0}] \Bigr),
    \end{align*}
    where
    \[
        n_i = n_0 + \sum_{k=1}^i \deg f_k,\qquad
        r_i = \frac{n_i - b_i}{2},\qquad
        0 \le i \le a.
    \]
    Thus, the class of $(f,E_b(f))$ is a linear combination of classes of 1-morphisms in $\Dcat$, and hence it lies in $K_\oplus(\Add(\Dcat))$.
\end{proof}

\begin{cor} \label{convocation2}
    The natural inclusion $\Add(\Dmon) \hookrightarrow \Kar(\Dmon)$ induces an isomorphism of rings
    \[
        K_\oplus(\Add(\Dmon)) \cong K_\oplus(\Kar(\Dmon)).
    \]
\end{cor}

\Cref{convocation1,convocation2} are in sharp contrast to many other categorifications appearing in the literature, such as the categorification of Heisenberg algebras and Kac--Moody algebras, where one \emph{needs} to pass to the Karoubi envelope to get the full categorification.

%===========
% References
%===========

\bibliographystyle{alphaurl}
\bibliography{TLtower}

\end{document}

%% file: TLmacros.tex
\usepackage{
    amsmath,
    amsfonts,
    amssymb,
    amsthm,
    amscd,
    comment,
    enumitem,
    etoolbox,
    textcomp,   % To avoid warning in gensymb
    gensymb,    % For \degree
    mathtools,
    mathdots,
    stmaryrd    % For \inplus
}
\usepackage[dvipsnames]{xcolor}

\usepackage[letterpaper,margin=1in]{geometry}
\usepackage[T1]{fontenc}
\usepackage{lmodern}            % To get better quality text
\usepackage{dsfont}             % For \mathbbm{1} (unit in a monoidal category)
\usepackage[
    unicode=true,
    pdfencoding=auto,
    psdextra,
    colorlinks=true,
    linkcolor=blue,
    citecolor=blue,
    urlcolor=blue,
    breaklinks=true
]{hyperref}

\DeclareFontFamily{OT1}{pzc}{}
\DeclareFontShape{OT1}{pzc}{m}{it}{<-> s * [1.10] pzcmi7t}{}
\DeclareMathAlphabet{\mathpzc}{OT1}{pzc}{m}{it}

\usepackage{zref-clever}

\zcsetup{
  cap,                      % To capitalize names of theorems, etc.
  nameinlink=false,         % Options are true, false, single, and tsingle
  lastsep = {, and }        % Oxford comma
}

\zcRefTypeSetup{assumption}{
    Name-sg = Assumption ,
    name-sg = assumption ,
    Name-pl = Assumptions ,
    name-pl = assumptions ,
}
\zcRefTypeSetup{cor}{
    Name-sg = Corollary ,
    name-sg = corollary ,
    Name-pl = Corollaries ,
    name-pl = corollaries ,
}
\zcRefTypeSetup{defin}{
    Name-sg = Definition ,
    name-sg = definition ,
    Name-pl = Definitions ,
    name-pl = definitions ,
}
\zcRefTypeSetup{enumi}{
    Name-sg = {} ,
    name-sg = {} ,
    Name-pl = {} ,
    name-pl = {} ,
}
\zcRefTypeSetup{equation}{
    Name-sg = {} ,
    name-sg = {} ,
    Name-pl = {} ,
    name-pl = {} ,
}
\zcRefTypeSetup{eg}{
    Name-sg = Example ,
    name-sg = example ,
    Name-pl = Examples ,
    name-pl = examples ,
}
\zcRefTypeSetup{lem}{
    Name-sg = Lemma ,
    name-sg = lemma ,
    Name-pl = Lemmas ,
    name-pl = lemmas ,
}
\zcRefTypeSetup{prop}{
    Name-sg = Proposition ,
    name-sg = proposition ,
    Name-pl = Propositions ,
    name-pl = propositions ,
}
\zcRefTypeSetup{rem}{
    Name-sg = Remark ,
    name-sg = remark ,
    Name-pl = Remarks ,
    name-pl = remarks ,
}
\zcRefTypeSetup{theo}{
    Name-sg = Theorem ,
    name-sg = theorem ,
    Name-pl = Theorems ,
    name-pl = theorems ,
}

\zcsetup{
  countertype = {
    enumi=item, enumii=item, enumiii=item, enumiv=item
  },
  counterresetby = {
    enumii=enumi, enumiii=enumii, enumiv=enumiii
  }
}
\zcRefTypeSetup{item}{
    Name-sg = {} ,
    name-sg = {} ,
    Name-pl = {} ,
    name-pl = {} ,
}

\AddToHook{env/assumption/begin}{\zcsetup{countertype={theo=assumption}}}
\AddToHook{env/cor/begin}{\zcsetup{countertype={theo=cor}}}
\AddToHook{env/defin/begin}{\zcsetup{countertype={theo=defin}}}
\AddToHook{env/eg/begin}{\zcsetup{countertype={theo=eg}}}
\AddToHook{env/lem/begin}{\zcsetup{countertype={theo=lem}}}
\AddToHook{env/prop/begin}{\zcsetup{countertype={theo=prop}}}
\AddToHook{env/rem/begin}{\zcsetup{countertype={theo=rem}}}

\newcommand{\cref}[1]{\zcref{#1}}
\newcommand{\Cref}[1]{\zcref[S]{#1}}

\makeatletter
\def\namedlabel#1#2{\begingroup
    (#2)%
    \def\@currentlabel{(#2)}%
    \phantomsection\label{#1}\endgroup
}
\makeatother

\newcommand\N{\mathbb{N}}

\newcommand\Z{\mathbb{Z}}
\newcommand\kk{\Bbbk}
\newcommand\one{\mathds{1}}
\newcommand\zero{\mathbf{0}}

\newcommand\bA{\mathbf{A}}

\newcommand\bB{\mathbf{B}}

\newcommand\cC{\mathcal{C}}

\newcommand\op{\textup{op}}
\newcommand\rev{\textup{rev}}

\newcommand\triv{\textup{triv}}

\newcommand{\fC}{\mathfrak{C}}

\newcommand\md{\textup{-mod}}
\newcommand\bimod{\textup{-bimod}}

\newcommand{\Ring}{R}
\newcommand{\Poly}{\mathrm{Poly}}
\newcommand{\TLalg}{\mathrm{TL}}
\newcommand{\TLbim}[3]{\sideset{_{#1}}{_{#3}}{\mathop{(#2)}\nolimits}}
\newcommand{\TLbimr}[2]{{\mathop{(#1)}\nolimits}_{#2}}
\newcommand{\JW}{J}

\newcommand{\tEnd}{\text{2-}\mathpzc{End}}         % 2-category of endofunctors
\newcommand{\Bcat}{\mathfrak{B}}         % Bimodules 2-category
\newcommand{\Mcat}{\mathcal{M}}         % Module ategory
\newcommand{\TLcat}{\mathpzc{TL}}        % Temperly-Lieb category
\newcommand{\Dcat}{\mathfrak{D}}        % Diagrammatic 2-category
\newcommand{\Dmon}{\mathcal{D}}         % Diagrammatic monoidal category

\newcommand{\obj}[1]{X_{#1}}
\newcommand{\Sobj}{\mathtt{S}}
\newcommand{\Sup}{{\Sobj_+}}
\newcommand{\Sdown}{{\Sobj_-}}
\newcommand{\Dobj}{{\color{cyan} \mathtt{D}}}
\newcommand{\Dup}[1][]{\Dobj^{#1}_{\color{cyan} +}}
\newcommand{\Ddown}[1][]{{\Dobj^{#1}_{\color{cyan} -}}}
\newcommand{\Dupdown}{{\Dobj_{\color{cyan} \pm}}}
\newcommand{\Tobj}{{\color{orange} \mathtt{T}}}

\newcommand\bF{\mathbf{F}}

\newcommand\bR{\mathbf{R}}

\DeclareMathOperator{\Add}{Add}                 % additive envelope

\DeclareMathOperator{\End}{End}

\DeclareMathOperator{\Ind}{Ind}                 % induction functor
\DeclareMathOperator{\Mat}{Mat}
\DeclareMathOperator{\Kar}{Kar}
\DeclareMathOperator{\Res}{Res}                 % restriction functor

\DeclareMathOperator{\sh}{sh}                   % shadow
\newcommand{\tsh}{\widetilde{\sh}}

\DeclareMathOperator{\tr}{tr}

\usepackage{tikz}
\usetikzlibrary{arrows.meta,patterns}
\usetikzlibrary{shapes.geometric,decorations.markings,decorations.pathreplacing}
\usepackage{tikz-cd}

\tikzset{multi/.style={very thick}}
\tikzset{S/.style={black}}
\tikzset{D/.style={thick,cyan}}
\tikzset{TL/.style={orange}}
\tikzset{
    anchorbase/.style={
        >=To,
        line cap = round, line join = round,
        baseline={([yshift=-0.5ex]current bounding box.center)},
    }
}
\tikzset{% Syntax: \begin{tikzpicture}[centerzero={0,0.2}]
    centerzero/.style={
        >=To,
        line cap = round, line join = round,
        baseline={([yshift=-0.5ex](#1))},
        },
    centerzero/.default={0,0}
}
\tikzset{   % To wipe a path for a string
    wipe/.style={
        white,
        line cap = butt,    % So that wipe doesn't extend beyond the endpoint
        line width=4pt
    }
} 
\tikzset{% For one string passing over another
    over/.style={
        preaction={
            draw=white,
            line cap = butt,    % So that wipe doesn't extend beyond the endpoint
            line width=4pt,
            -{},                % To omit arrowheads on wipe
        }
    },
}
\tikzset{->-/.style={decoration={
    markings,
    mark=at position #1 with {\arrow{>}}},postaction={decorate}}
}
\tikzset{-<-/.style={decoration={
    markings,
    mark=at position #1 with {\arrow{<}}},postaction={decorate}}
}

\newcommand{\strandlabel}[1]{$\scriptstyle{#1}$}
\newcommand{\regionlabel}[1]{\color{purple} $\scriptstyle{#1}$}
\newcommand{\region}[2]{ % \region{position}{label}
    \node at (#1) {\regionlabel{#2}}
}
\newcommand{\botlabel}[1]{node[anchor=north] {\strandlabel{#1}}}
\newcommand{\toplabel}[1]{node[anchor=south] {\strandlabel{#1}}}

\newcommand{\coupon}[2]{ % \coupon{position}{label}
    \draw (#1) node[minimum height=9pt,minimum width=4pt,inner sep=2pt,draw=black,rectangle,rounded corners,fill=black!20!white] (coup) {\color{black} $\scriptstyle{#2}$}
}
\newcommand{\projector}[2]{ % \projector{position}{subscript}
    \coupon{#1}{\sigma_{#2}}
}
\newcommand{\depthor}[2]{ % \depthor{position}{depth}
    \draw (#1) node[minimum height=9pt,minimum width=4pt,inner sep=2pt,rectangle,rounded corners,fill=yellow] (coup) {\color{black} $\scriptstyle{#2}$}
}
\newcommand{\genbox}[3]{% \genbox{lower left corner}{upper right corner}{label}
    \filldraw[draw=black,fill=black!20!white,rounded corners] (#1) rectangle (#2) node[midway] {\color{black} \strandlabel{#3}}
}

\newcommand{\bubright}[2][S]{% \bubright[S/D]{position}
    \draw[#1,->] (#2)+(0.2,0) arc(360:0:0.2)
}
\newcommand{\bubrightmult}[3][S]{ % \bubrightmult[S/D]{position}{number in center}
    \draw[#1] (#2)
        node[
            ellipse,draw,fill=white,
            inner sep=2pt,
            minimum height=12pt,
            minimum width=12pt,
            postaction={decorate,       % <- put an arrow on the border
                decoration={markings,
                mark=at position 0 with {\arrow{<}}}},
        ] (bub) {$\scriptstyle{#3}$}
}
\newcommand{\bubleftmult}[3][S]{ % \bubleftmult[S/D]{position}{number in center}
    \draw[#1] (#2)
        node[
            ellipse,draw,fill=white,
            inner sep=2pt,
            minimum height=12pt,
            minimum width=12pt,
            postaction={decorate,       % <- put an arrow on the border
                decoration={markings,
                mark=at position 0.5 with {\arrow{>}}}},
        ] (bub) {$\scriptstyle{#3}$}
}
\newcommand{\bubleft}[2][S]{% \bubleft[S/D]{position}
    \draw[#1,->] (#2)+(-0.2,0) arc(-180:180:0.2)
}
\newcommand{\bubun}[2][S]{% \bubun[S/D]{position}
    \draw[#1] (#2)+(0.2,0) arc(0:360:0.2)
}

\newcommand{\posupcross}{
    \begin{tikzpicture}[centerzero]
        \draw[->] (0.2,-0.2) -- (-0.2,0.2);
        \draw[over,->] (-0.2,-0.2) -- (0.2,0.2);
    \end{tikzpicture}
}
\newcommand{\negupcross}{
    \begin{tikzpicture}[centerzero]
        \draw[->] (-0.2,-0.2) -- (0.2,0.2);
        \draw[over,->] (0.2,-0.2) -- (-0.2,0.2);
    \end{tikzpicture}
}

\newcommand{\upcross}[2]{
    \begin{tikzpicture}[centerzero]
        \draw[#1,->] (-0.2,-0.2) -- (0.2,0.2);
        \draw[#2,->] (0.2,-0.2) -- (-0.2,0.2);
    \end{tikzpicture}
}
\newcommand{\leftcross}[2]{
    \begin{tikzpicture}[centerzero]
        \draw[#1,<-] (-0.2,-0.2) -- (0.2,0.2);
        \draw[#2,->] (0.2,-0.2) -- (-0.2,0.2);
    \end{tikzpicture}
}
\newcommand{\rightcross}[2]{
    \begin{tikzpicture}[centerzero]
        \draw[#1,->] (-0.2,-0.2) -- (0.2,0.2);
        \draw[#2,<-] (0.2,-0.2) -- (-0.2,0.2);
    \end{tikzpicture}
}
\newcommand{\downcross}[2]{
    \begin{tikzpicture}[centerzero]
        \draw[#1,<-] (-0.2,-0.2) -- (0.2,0.2);
        \draw[#2,<-] (0.2,-0.2) -- (-0.2,0.2);
    \end{tikzpicture}
}

\newcommand{\upcrossreg}[3]{
    \begin{tikzpicture}[centerzero]
        \draw[#1,->] (-0.2,-0.2) -- (0.2,0.2);
        \draw[#2,->] (0.2,-0.2) -- (-0.2,0.2);
        \region{0.35,0}{#3};
    \end{tikzpicture}
}

\newcommand{\upstrand}[1][S]{
    \begin{tikzpicture}[centerzero]
        \draw[->,#1] (0,-0.2) -- (0,0.2);
    \end{tikzpicture}
}
\newcommand{\downstrand}[1][S]{
    \begin{tikzpicture}[centerzero]
        \draw[<-,#1] (0,-0.2) -- (0,0.2);
    \end{tikzpicture}
}

\newcommand{\upstrandreg}[2][S]{
    \begin{tikzpicture}[centerzero]
        \draw[->,#1] (0,-0.2) -- (0,0.2);
        \node[anchor=west] at (0,0) {\regionlabel{#2}};
    \end{tikzpicture}
}
\newcommand{\downstrandreg}[2][S]{
    \begin{tikzpicture}[centerzero]
        \draw[<-,#1] (0,-0.2) -- (0,0.2);
        \node[anchor=west] at (0,0) {\regionlabel{#2}};
    \end{tikzpicture}
}

\newcommand{\leftcup}[1][S]{
    \begin{tikzpicture}[anchorbase]
        \draw[<-,#1] (-0.15,0.15) -- (-0.15,0) arc(180:360:0.15) -- (0.15,0.15);
    \end{tikzpicture}
}
\newcommand{\rightcap}[1][S]{
    \begin{tikzpicture}[anchorbase]
        \draw[->,#1] (-0.15,-0.15) -- (-0.15,0) arc(180:0:0.15) -- (0.15,-0.15);
    \end{tikzpicture}
}

\newcommand{\Tcup}{
    \begin{tikzpicture}[anchorbase]
        \draw[TL] (-0.15,0.15) -- (-0.15,0) arc(180:360:0.15) -- (0.15,0.15);
    \end{tikzpicture}
}
\newcommand{\Tcap}{
    \begin{tikzpicture}[anchorbase]
        \draw[TL] (-0.15,-0.15) -- (-0.15,0) arc(180:0:0.15) -- (0.15,-0.15);
    \end{tikzpicture}
}

\newcommand{\rightcupreg}[2][S]{
    \begin{tikzpicture}[anchorbase]
        \draw[->,#1] (-0.15,0.15) -- (-0.15,0) arc(180:360:0.15) -- (0.15,0.15);
        \node[anchor=west] at (0.15,0) {\regionlabel{#2}};
    \end{tikzpicture}
}
\newcommand{\leftcupreg}[2][S]{
    \begin{tikzpicture}[anchorbase]
        \draw[<-,#1] (-0.15,0.15) -- (-0.15,0) arc(180:360:0.15) -- (0.15,0.15);
        \node[anchor=west] at (0.15,0) {\regionlabel{#2}};
    \end{tikzpicture}
}
\newcommand{\rightcapreg}[2][S]{
    \begin{tikzpicture}[anchorbase]
        \draw[->,#1] (-0.15,-0.15) -- (-0.15,0) arc(180:0:0.15) -- (0.15,-0.15);
        \node[anchor=west] at (0.15,0) {\regionlabel{#2}};
    \end{tikzpicture}
}
\newcommand{\leftcapreg}[2][S]{
    \begin{tikzpicture}[anchorbase]
        \draw[<-,#1] (-0.15,-0.15) -- (-0.15,0) arc(180:0:0.15) -- (0.15,-0.15);
        \node[anchor=west] at (0.15,0) {\regionlabel{#2}};
    \end{tikzpicture}
}

\newcommand{\rightbub}[1][S]{
    \begin{tikzpicture}[centerzero]
        \bubright[#1]{0,0};
    \end{tikzpicture}
}
\newcommand{\rightbubmult}[2][S]{% \rightbubmult[S/D]{number in center}
    \begin{tikzpicture}[centerzero]
        \bubrightmult[#1]{0,0}{#2};
    \end{tikzpicture}
}

\newcommand{\leftbubmult}[2][S]{% \leftbubmult[S/D]{number in center}
    \begin{tikzpicture}[centerzero]
        \bubleftmult[#1]{0,0}{#2};
    \end{tikzpicture}
}
\newcommand{\unbub}[1][S]{
    \begin{tikzpicture}[centerzero]
        \bubun[#1]{0,0};
    \end{tikzpicture}
}

\newcommand{\rightbubreg}[2][S]{
    \begin{tikzpicture}[centerzero]
        \bubright[#1]{0,0};
        \node[anchor=west] at (0.2,0) {\regionlabel{#2}};
    \end{tikzpicture}
}
\newcommand{\rightbubmultreg}[3][S]{% \rightbubmult[S/D]{number in center}
    \begin{tikzpicture}[centerzero]
        \bubrightmult[#1]{0,0}{#2};
        \node[anchor=west] at (bub.east) {\regionlabel{#3}};
    \end{tikzpicture}
}
\newcommand{\leftbubreg}[2][S]{
    \begin{tikzpicture}[centerzero]
        \bubleft[#1]{0,0};
        \node[anchor=west] at (0.2,0) {\regionlabel{#2}};
    \end{tikzpicture}
}
\newcommand{\leftbubmultreg}[3][S]{% \leftbubmult[S/D]{number in center}
    \begin{tikzpicture}[centerzero]
        \bubleftmult[#1]{0,0}{#2};
        \node[anchor=west] at (bub.east) {\regionlabel{#3}};
    \end{tikzpicture}
}

\newcommand{\freecoupon}[1]{% Free-floating coupon
    \begin{tikzpicture}[centerzero]
        \coupon{0,0}{#1};
    \end{tikzpicture}
}
\newcommand{\freeprojector}[1]{% orthogonal idempotent
    \begin{tikzpicture}[centerzero]
        \projector{0,0}{#1};
    \end{tikzpicture}
}
\newcommand{\freedepthor}[1]{%
    \begin{tikzpicture}[centerzero]
        \depthor{0,0}{#1};
    \end{tikzpicture}
}

\newcommand{\freeprojectorreg}[2]{% orthogonal idempotent
    \begin{tikzpicture}[centerzero]
        \projector{0,0}{#1};
        \node[anchor=west,xshift=-2pt] at (coup.east) {\regionlabel{#2}};
    \end{tikzpicture}
}
\newcommand{\freedepthorreg}[2]{%
    \begin{tikzpicture}[centerzero]
        \depthor{0,0}{#1};
        \node[anchor=west,xshift=-2pt] at (coup.east) {\regionlabel{#2}};
    \end{tikzpicture}
}

\newcommand{\mergeup}{
    \begin{tikzpicture}[centerzero]
        \draw[S] (-0.3,-0.3) -- (0,0);
        \draw[S] (0.3,-0.3) -- (0,0);
        \draw[D,->] (0,0) -- (0,0.4);
    \end{tikzpicture}
}
\newcommand{\mergedown}{
    \begin{tikzpicture}[centerzero]
        \draw[D,<-] (0,-0.4) -- (0,0);
        \draw[S] (0,0) -- (-0.3,0.3);
        \draw[S] (0,0) -- (0.3,0.3);
    \end{tikzpicture}
}
\newcommand{\mergeSW}{
    \begin{tikzpicture}[centerzero]
        \draw[D,<-] (-0.3,-0.3) -- (0,0);
        \draw[S] (0.3,-0.3) -- (0,0);
        \draw[S] (0,0) -- (0,0.4);
    \end{tikzpicture}
}
\newcommand{\mergeSE}{
    \begin{tikzpicture}[centerzero]
        \draw[S] (-0.3,-0.3) -- (0,0);
        \draw[D,<-] (0.3,-0.3) -- (0,0);
        \draw[S] (0,0) -- (0,0.4);
    \end{tikzpicture}
}
\newcommand{\mergeNW}{
    \begin{tikzpicture}[centerzero]
        \draw[S] (0,-0.4) -- (0,0);
        \draw[D,->] (0,0) -- (-0.3,0.3);
        \draw[S] (0,0) -- (0.3,0.3);
    \end{tikzpicture}
}
\newcommand{\mergeNE}{
    \begin{tikzpicture}[centerzero]
        \draw[S] (0,-0.4) -- (0,0);
        \draw[S] (0,0) -- (-0.3,0.3);
        \draw[D,->] (0,0) -- (0.3,0.3);
    \end{tikzpicture}
}

\newcommand{\mergeupreg}[1]{
    \begin{tikzpicture}[centerzero]
        \draw[S] (-0.3,-0.3) -- (0,0);
        \draw[S] (0.3,-0.3) -- (0,0);
        \draw[D,->] (0,0) -- (0,0.4);
        \region{0.2,0.1}{#1};
    \end{tikzpicture}
}

\newcommand{\mergeSWreg}[1]{
    \begin{tikzpicture}[centerzero]
        \draw[D,<-] (-0.3,-0.3) -- (0,0);
        \draw[S] (0.3,-0.3) -- (0,0);
        \draw[S] (0,0) -- (0,0.4);
        \region{0.2,0.1}{n};
    \end{tikzpicture}
}
\newcommand{\mergeSEreg}[1]{
    \begin{tikzpicture}[centerzero]
        \draw[S] (-0.3,-0.3) -- (0,0);
        \draw[D,<-] (0.3,-0.3) -- (0,0);
        \draw[S] (0,0) -- (0,0.4);
        \region{0.2,0.1}{n};
    \end{tikzpicture}
}
\newcommand{\mergeNWreg}[1]{
    \begin{tikzpicture}[centerzero]
        \draw[S] (0,-0.4) -- (0,0);
        \draw[D,->] (0,0) -- (-0.3,0.3);
        \draw[S] (0,0) -- (0.3,0.3);
        \region{0.2,-0.1}{n};
    \end{tikzpicture}
}
\newcommand{\mergeNEreg}[1]{
    \begin{tikzpicture}[centerzero]
        \draw[S] (0,-0.4) -- (0,0);
        \draw[S] (0,0) -- (-0.3,0.3);
        \draw[D,->] (0,0) -- (0.3,0.3);
        \region{0.2,-0.1}{n};
    \end{tikzpicture}
}

\newcommand{\splitdown}{
    \begin{tikzpicture}[centerzero]
        \draw[S,<-] (-0.3,-0.3) -- (0,0);
        \draw[S,<-] (0.3,-0.3) -- (0,0);
        \draw[D] (0,0) -- (0,0.4);
    \end{tikzpicture}
}
\newcommand{\splitSE}{
    \begin{tikzpicture}[centerzero]
        \draw[S,<-] (0,-0.4) -- (0,0);
        \draw[D] (0,0) -- (-0.3,0.3);
        \draw[S,->] (0,0) -- (0.3,0.3);
    \end{tikzpicture}
}
\newcommand{\splitSW}{
    \begin{tikzpicture}[centerzero]
        \draw[S,<-] (0,-0.4) -- (0,0);
        \draw[S,->] (0,0) -- (-0.3,0.3);
        \draw[D] (0,0) -- (0.3,0.3);
    \end{tikzpicture}
}
\newcommand{\splitNE}{
    \begin{tikzpicture}[centerzero]
        \draw[D] (-0.3,-0.3) -- (0,0);
        \draw[S,<-] (0.3,-0.3) -- (0,0);
        \draw[S,->] (0,0) -- (0,0.4);
    \end{tikzpicture}
}
\newcommand{\splitNW}{
    \begin{tikzpicture}[centerzero]
        \draw[S,<-] (-0.3,-0.3) -- (0,0);
        \draw[D] (0.3,-0.3) -- (0,0);
        \draw[S,->] (0,0) -- (0,0.4);
    \end{tikzpicture}
}

\newcommand{\splitupreg}[1]{
    \begin{tikzpicture}[centerzero]
        \draw[D] (0,-0.4) -- (0,0);
        \draw[S,->] (0,0) -- (-0.3,0.3);
        \draw[S,->] (0,0) -- (0.3,0.3);
        \region{0.2,-0.1}{n};
    \end{tikzpicture}
}
\newcommand{\splitdownreg}[1]{
    \begin{tikzpicture}[centerzero]
        \draw[S,<-] (-0.3,-0.3) -- (0,0);
        \draw[S,<-] (0.3,-0.3) -- (0,0);
        \draw[D] (0,0) -- (0,0.4);
        \region{0.2,0.1}{n};
    \end{tikzpicture}
}
\newcommand{\splitSEreg}[1]{
    \begin{tikzpicture}[centerzero]
        \draw[S,<-] (0,-0.4) -- (0,0);
        \draw[D] (0,0) -- (-0.3,0.3);
        \draw[S,->] (0,0) -- (0.3,0.3);
        \region{0.2,-0.1}{n};
    \end{tikzpicture}
}
\newcommand{\splitSWreg}[1]{
    \begin{tikzpicture}[centerzero]
        \draw[S,<-] (0,-0.4) -- (0,0);
        \draw[S,->] (0,0) -- (-0.3,0.3);
        \draw[D] (0,0) -- (0.3,0.3);
        \region{0.2,-0.1}{n};
    \end{tikzpicture}
}
\newcommand{\splitNEreg}[1]{
    \begin{tikzpicture}[centerzero]
        \draw[D] (-0.3,-0.3) -- (0,0);
        \draw[S,<-] (0.3,-0.3) -- (0,0);
        \draw[S,->] (0,0) -- (0,0.4);
        \region{0.2,0.1}{n};
    \end{tikzpicture}
}
\newcommand{\splitNWreg}[1]{
    \begin{tikzpicture}[centerzero]
        \draw[S,<-] (-0.3,-0.3) -- (0,0);
        \draw[D] (0.3,-0.3) -- (0,0);
        \draw[S,->] (0,0) -- (0,0.4);
        \region{0.2,0.1}{n};
    \end{tikzpicture}
}

\newtheorem{theo}{Theorem}[section]

\newtheorem{prop}[theo]{Proposition}
\newtheorem{lem}[theo]{Lemma}
\newtheorem{cor}[theo]{Corollary}

\theoremstyle{definition}

\newtheorem{defin}[theo]{Definition}
\newtheorem{rem}[theo]{Remark}
\newtheorem{eg}[theo]{Example}

\numberwithin{equation}{section}
\allowdisplaybreaks

\setenumerate[1]{label=(\alph*)}          % Only need to reference enumerate items with \ref{}

\newtoggle{comments}
\newtoggle{details}
\newtoggle{detailsnote}

\iftoggle{comments}{%
    \usepackage[notref,notcite]{showkeys}   % To show labels
    \newcommand{\acomments}[1]{
        \ \\
        {\color{red}
            \textbf{AS:} #1
        }
        \ \\
    }

    \newcommand{\ycomments}[1]{
        \ \\
        {\color{blue}
            \textbf{YS:} #1
        }
        \ \\
    }
    \newcommand{\question}[1]{
        \ \\
        {\color{blue}
            \textbf{Question:} #1
        }
        \ \\
    }
    }{%
        \newcommand{\acomments}[1]{}
        \newcommand{\ycomments}[1]{}
        \newcommand{\question}[1]{}
    }

\iftoggle{details}{%
    \newcommand{\details}[1]{
        \ \\
        {\color{OliveGreen}
            \textbf{Details:} #1
        }
        \\
    }
}{%
    \newcommand{\details}[1]{\ignorespaces}
}